\DeclareSymbolFont{AMSb}{U}{msb}{m}{n}
\documentclass[11pt,a4paper,oneside,leqno,noamsfonts]{amsbook}
\usepackage[english]{babel}
\usepackage[dvipsnames]{xcolor}
\definecolor{britishracinggreen}{rgb}{0.0, 0.26, 0.15}
\definecolor{cobalt}{rgb}{0.0, 0.28, 0.67}
\usepackage[utopia]{mathdesign}
    \DeclareSymbolFont{usualmathcal}{OMS}{cmsy}{m}{n}
    \DeclareSymbolFontAlphabet{\mathcal}{usualmathcal}
\usepackage[a4paper,top=3.5cm,bottom=4cm,left=3.25cm,
           right=3.25cm,bindingoffset=5mm]{geometry}
\savegeometry{Mem}
\usepackage[utf8]{inputenc}
\usepackage{braket,caption,comment,mathtools,stmaryrd}
\usepackage{multirow,booktabs,microtype}
\usepackage{latexsym}
\usepackage{todonotes}
\usepackage{fancyhdr}
\usepackage{young}
\usepackage{nomencl}
\usepackage{soul} 
\usepackage[colorlinks,bookmarks]{hyperref} %
\hypersetup{colorlinks,%
            citecolor=britishracinggreen,%
            filecolor=black,%
            linkcolor=cobalt,%
            urlcolor=black}
\numberwithin{section}{chapter}
\numberwithin{equation}{chapter}
\numberwithin{figure}{chapter}
\usepackage{titlesec}
\titleformat{\chapter}[display]{ \normalsize \huge  \color{black}}{\flushleft \Huge \MakeUppercase { } \hspace{1ex} {\fontsize{60}{70}\selectfont \bfseries \color{cobalt}  \thechapter }}{12 pt}{\huge}
\titleformat{\section}{\normalfont \Large \MakeUppercase{ }\selectfont\bfseries }{\thesection}{12pt}{}
\titleformat{\subsection}{\fontsize{12pt}{14pt}\selectfont\bfseries}{\thesubsection}{12pt}{}
\titleformat{\subsubsection}{\fontsize{11pt}{12pt}\selectfont\itshape\bfseries}{}{0.5pt}{}

\makeatletter
\def\l@subsection{\@tocline{1}{0,2pt}{2pc}{10mm}{\ \ }}   
\def\l@section{\@tocline{1}{0,2pt}{2pc}{8mm}{\ \ }}



\def\into{\hookrightarrow}

\def\ra{\rightarrow}

\def\Sets{\mathrm{Sets}}
\def\Def{\mathsf{Def}}
\def\KS{\mathsf{KS}}
\def\ad{\mathsf{ad}}

\def\m{\mathrm{m}}

\def\A{\mathcal A}

\def\R{\mathbb R}
\def\N{\mathbb N}
\def\C{\mathbb C}

\def\P{\mathbb P}
\def\Q{\mathbb Q}

\def\X{\mathcal X}
\def\E{\mathcal E}
\def\Z{\mathbb Z}

\def\FF{\mathscr{F}}

\def\O{\mathscr O}

\def\pt{\mathrm{pt}}

\DeclareMathOperator{\Mor}{Mor}

\DeclareMathOperator{\ev}{ev}

\DeclareMathOperator{\Ob}{Ob}

\DeclareMathOperator{\id}{id}

\DeclareMathOperator{\vir}{vir}

\DeclareMathOperator{\Spec}{Spec\,}

\DeclareMathOperator{\GL}{GL}

\DeclareMathOperator{\dd}{d}

\DeclareMathOperator{\Aut}{Aut}

\DeclareMathOperator{\Hom}{Hom}

\DeclareMathOperator{\End}{End}

\DeclareMathOperator{\pr}{pr}
\DeclareMathOperator{\Art}{Art}

\DeclareMathOperator{\Der}{Der}

\theoremstyle{definition}

\newtheorem*{lemma*}{Lemma}
\newtheorem*{theorem*}{Theorem}
\newtheorem*{example*}{Example}
\newtheorem*{fact*}{Fact}
\newtheorem*{notation*}{Notation}
\newtheorem*{definition*}{Definition}
\newtheorem*{prop*}{Proposition}
\newtheorem*{remark*}{Remark}
\newtheorem*{corollary*}{Corollary}
\newtheorem*{conventions*}{Conventions}
\newtheorem*{caution*}{Caution}

\newtheorem{definition}{Definition}[section]

\newtheorem{example}[definition]{Example}

\newtheorem{remark}[definition]{Remark}
\newtheorem{conjecture}[definition]{Conjecture}

\newtheoremstyle{thm} 
        {3mm}
        {3mm}
        {\slshape}
        {0mm}
        {\bfseries}
        {.}
        {1mm}
        {}
\theoremstyle{thm}
\newtheorem{theorem}[definition]{Theorem}
\newtheorem{corollary}[definition]{Corollary}
\newtheorem{lemma}[definition]{Lemma}
\newtheorem{prop}[definition]{Proposition}

\newtheoremstyle{sol} 
        {3mm}
        {3mm}
        {\normalfont}
        {0mm}
        {\scshape}
        {.}
        {1mm}
        {}
\theoremstyle{sol}

\newtheorem{notation}[definition]{Notation}
\newtheorem*{ssolution*}{Solution (sketch)}
\newtheorem*{solution*}{Solution}

\usepackage{tikz}
\usepackage{tikz-cd}
\usepackage{rotating}

\usetikzlibrary{matrix,shapes,arrows,decorations.pathmorphing}
\tikzset{commutative diagrams/arrow style=math font}
\tikzset{commutative diagrams/.cd,
mysymbol/.style={start anchor=center,end anchor=center,draw=none}}

\tikzset{
shift up/.style={
to path={([yshift=#1]\tikztostart.east) -- ([yshift=#1]\tikztotarget.west) \tikztonodes}
}
}

\DeclareMathAlphabet{\mathpzc}{OT1}{pzc}{m}{it}

\newcommand*{\defeq}{\mathrel{\vcenter{\baselineskip0.5ex \lineskiplimit0pt
                     \hbox{\scriptsize.}\hbox{\scriptsize.}}}%
                     =}

\makeindex

\begin{document}

\frontmatter

\title{}
\cleardoublepage
\thispagestyle{empty}
\newgeometry{margin=0pt}
\begin{figure}
\includegraphics[scale=1]{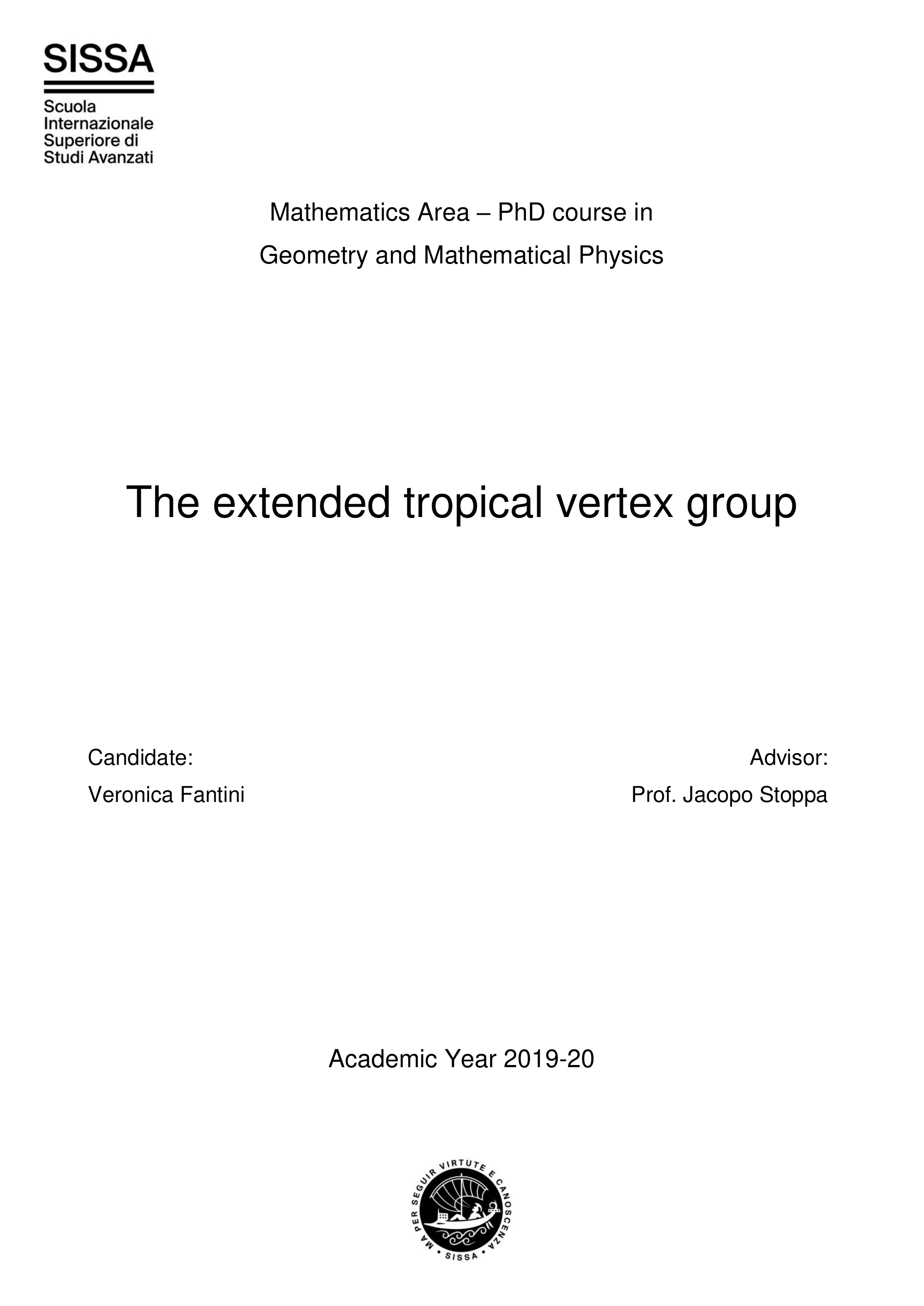}
\end{figure}
\vspace{\stretch{2}}
\cleardoublepage

\clearpage
\loadgeometry{Mem}
%
%




\date{}

\begin{abstract}
In this thesis we study the relation between scattering diagrams and deformations of holomorphic pairs, building on a recent work of Chan--Conan Leung--Ma \cite{MCscattering}. The new feature is the \textit{extended tropical vertex} group where the scattering diagrams are defined. In addition, the extended tropical vertex provides interesting applications: on one hand we get a geometric interpretation of the wall-crossing formulas for coupled $2d$-$4d$ systems, previously introduced by Gaiotto--Moore--Neitzke \cite{WCF2d-4d}. On the other hand, Gromov--Witten invariants of toric surfaces relative to their boundary divisor appear in the commutator formulas, along with certain absolute invariants due to Gross--Pandharipande--Siebert \cite{GPS}, which suggests a possible connection to open/closed theories in geometry and mathematical physics.  


\end{abstract}

\maketitle

\cleardoublepage
\thispagestyle{empty}
\begin{flushright}
\itshape Ai nonni\\
Sesta e Olindo,\\
Aldo e Rina
\end{flushright}
\vspace{\stretch{2}}
\cleardoublepage

\section*{Acknowledgements}

I am grateful to my advisor Jacopo Stoppa, for his constant support, fruitful discussions, suggestions and corrections.   

I wish to thank Kwokwai Chan, Mark Gross and Andrew Neitzke for suggesting corrections and further directions and applications of the results of the thesis. 
  
I thank SISSA for providing a welcoming and professional enviroment for PhD studies and research. I thank the PhD students for innumerous discussions, and among them Guilherme Almeida, Nadir Fasola, Xiao Han, Vitantonio Peragine, Carlo Scarpa, Michele Stecconi and Boris Stupovski, with whom I started the PhD at SISSA. 

Last but not least, I would like to thank my family and Angelo, my flatmates Monica Nonino e Federico Pichi, and my friends Matteo Wauters, Andrea Ricolfi, Luca Franzoi, Raffaele Scandone, Alessio Lerose, Matteo Zancanaro, Maria Strazzullo, Alessandro Nobile, Saddam Hijazi and Daniele Agostinelli. They have always encouraged me in my studies and I have spent great time with them in Trieste.    

\tableofcontents

\mainmatter

\chapter{Introduction}

In this thesis we are going to apply many techniques and ideas which have been developed by studying mirror symmetry with different approaches. This introduction aims to present this circle of ideas, paying special attention to wall crossing formulas and Gromov--Witten invariants. 
\\
Mirror symmetry predicts that Calabi--Yaus come in pairs, i.e. that type IIB string theory compactified on a Calabi--Yau $\check{X}$ gives the same physical theory of type IIA string theory compactified on the mirror Calabi--Yau $X$. Calabi--Yau manifolds are complex K\"ahler manifolds with trivial canonical bundle, and they admit a Ricci flat K\"ahler metric. In particular, type A string refers to $X$ as a symplectic manifold, while type B refers to $\check{X}$ as a complex manifold. Hence mirror symmetry can be considered as map between $X$ and $\check{X}$ which exchanges the two structures.   
The first intrinsic formulation of mirror symmetry was the \textit{Homological Mirror Symmetry} proposed by Kontsevich \cite{HMS}, who conjectures existence of an equivalence between the derived category of coherent sheaves on $\check{X}$ and the Fukaya category of Lagrangian submanifolds on $X$.        
The Strominger--Yau--Zaslow Mirror Symmetry \cite{SYZ} asserts that mirror symmetry is T-duality, i.e. that if $X$ has a mirror and it has a special Lagrangian torus fibrations, then the moduli space of Lagrangian torus fibration with flat $U(1)$ connections is the mirror $\check{X}$. Actually this statement holds in the \textit{large radius limit}: indeed the metric on the moduli space has to be modified by adding quantum corrections (coming from open Gromov--Witten invariants). In particular, these corrections exponentially decay in the large radius limit at generic points, but near the singular points the corrections can not be neglected.   
This limit version of mirror symmetry was then studied by Fukaya \cite{Fuk05}. In particular, assuming there are mirror pairs of dual torus fibrations $\check{X}\to M$ and $X\to M$ he proposed that in the large volume limit (i.e. rescaling the symplectic structure on $X$ by $\hbar^{-1}$ and taking the limit as $\hbar\to0$), quantum corrections arise from studying the semi-classical limit of the Fourier expansion along the torus fibers of deformations of the complex structure $J_{\hbar}$, of $\check{X}$. In order to study deformations of the complex structure of $\check{X}$ he adopted the analytic approach of studying solutions of the Maurer--Cartan equation in Kodaira--Spencer theory. 

To overcome convergence issues, Kontsevich--Soibelman \cite{KS01} and Gross--Siebert \cite{GS1}, \cite{GS2} come up with new approaches. On one side Kontsevich--Soibelman study the Gromov--Hausdorff collapse of degenerating families of Calabi--Yaus obtaining a limit structure which is an integral affine structure either coming from an analytic manifold over $\C((t))$ (B model) or from the base of a fibration of a Calabi--Yau by Lagrangian tori (A model). In addition, in \cite{KS} the authors study the problem of \textit{reconstructing} the family from the integral affine structure of the base and they solve it for analytic K3 surfaces. In particular they define combinatorial objects, called \textit{scattering diagrams}, which prescribe how to deform the sheaf of functions on the smooth locus and how to perform the gluing near the singularities.   

On the other side Gross--Siebert apply techniques of logarithmic geometry and from an integral affine manifold with singularities $B$ together with a polyhedral decomposition $\mathcal{P}$, they recover a degeneration of toric Calabi--Yau varieties (i.e. with singular fibers which are union of toric varietis). In this contest they define mirror pair of polarized log Calabi--Yau $(X,\mathcal{L})$ and $(\check{X},\check{\mathcal{L}})$ such that their degeneration data $(B,\mathcal{P})$ and $(\check{B},\check{\mathcal{P}})$ are mirror under a discrete Legendre transform. They also verify the expected computation of Hodge numbers for K3 surfaces.  

A common feature of Kontsevich--Soibelman and Gross--Siebert approach is the combinatorial structure which governs the construction of the mirror manifold. Apart from the reconstruction problem, scattering diagrams have been studied in relation with some enumerative problems, such as wall crossing formulas for Donaldson--Thomas invariants \cite{KS4d} and Gromov--Witten invariants for toric surfaces \cite{GPS}, \cite{Pierrick}. We are going to explain these relations in the following sections.

%

  
\section{Wall Crossing Formulas}

The first appearance of wall-crossing-formulas (WCFs) has been in the context of studying certain classes of two dimensional $\mathcal{N}=2$ supersymmetric fields theories, by Cecotti and Vafa in \cite{CV}. In particular, their WCFs compute how the number of Bogomolony solitons jumps when the central charge crosses a wall of marginal stability: let $\lbrace i,j,k...\rbrace$ be the vacua, if deforming the central charge the j-th critical value crosses the lines which connect the i-th and the k-th vacua, then the number of solitons $\mu_{ik}$ between $i$ and $k$ becomes $\mu_{ik}\pm\mu_{ij}\mu_{jk}$. From a mathematical viewpoint, these WCFs look like braiding identities for the matrices representing monodromy data (see e.g. \cite{dubrovin}). Indeed deformations of massive $\mathcal{N}=2$ SCFTs can be studied via isomonodromic deformations of a linear operator with rational coefficients, and the BPS spectrum is encoded in the Stokes matrices of the ODE associated to the linear operator.

Later on Kontsevich--Soibelman, studying \textit{numerical} Donaldson--Thomas invariants on $3$ dimensional Calabi--Yau categories, come up with new WCFs (\cite{WCFKS}). In particular, their WCFs encode the jumping behaviour of some semistable objects when stability conditions cross a codimension one subvariety (\textit{wall}). It has been remarkably studied in a series of papers by Gaiotto--Moore--Neitzke  \cite{WCF4d}, \cite{GMNWKB}, \cite{GMNframed} that Kontsevich--Soibelman WCFs have the same algebraic structure as WCFs for BPS states in four dimensional $\mathcal{N}=2$ SCFTs. These $4d$ WCFs can be defined starting from the datum of the (``charge") lattice $\Gamma$ endowed with an antisymmetric (``Dirac") pairing $\langle\cdot,\cdot\rangle_D\colon\Gamma\times\Gamma\to\Z$, and a graded Lie algebra closely related to the Poisson algebra of functions on the algebraic torus $(\C^*)^{\operatorname{rk} \Gamma}$. Then WCFs are expressed in terms of formal Poisson automorphisms of this algebraic torus. 

For the purpose of this work, we are interested in WCFs of so called \textit{coupled $2d$-$4d$ systems}, namely of $\mathcal{N}=2$ super symmetric field theories in four dimension coupled with a surface defect, introduced by Gaiotto--Moore--Neitzke in \cite{WCF2d-4d}. This generalizes both the formulas of Cecotti--Vafa in the pure $2d$ case and those of Kontsevich--Soibelman in the pure $4d$ case. In the coupled $2d$-$4d$ the setting becomes rather more complicated: the lattice $\Gamma$ is upgraded to a pointed groupoid $\mathbb{G}$, whose objects are indices $\lbrace i,j,k\cdots\rbrace$ and whose morphisms include the charge lattice $\Gamma$ as well as arrows parameterized by $\amalg_{i,j}\Gamma_{ij}$, where $\Gamma_{ij}$ is a $\Gamma$-torsor. Then the relevant wall-crossing formulas involve two types of formal automorphisms of the groupoid algebra $\C[\mathbb{G}]$: type $S$, corresponding to Cecotti--Vafa monodromy matrices, and type $K$, which generalize the formal torus automorphisms of Kontsevich--Soibelman. The main new feature is the non trivial interaction of automorphisms of type $S$ and $K$. 

It is worth mentioning that, despite the lack of a categorical description for pure $2d$ WCFs, the $2d$-$4d$ formulas have been recently studied within a categorical framework by Kerr and Soibelman in \cite{2d-4dSK}.


\section{Relative Gromov--Witten Invariants}

From an algebraic geometric view point, Gromov--Witten theory usually requires the definition of a compact moduli space (or a proper, separated Deligne--Mumford stack with a virtual fundamental class) parameterizing smooth curves of genus $g$ and class $\beta\in H_2(X,\Z)$: Kontsevich introduced the notion of stable maps to compactify the moduli space of genus $g$, degree $d$ curves in $\P^2$. 
Here we are interested in target manifolds which are either complete toric surfaces $X^{\mathrm{toric}}$ or log Calabi--Yau pairs $(X,D)$.\footnote{A log Calabi Yau pair $(X,D)$ is by definition (Definition 2.1 \cite{HK}) a pair of a smooth projective variety $X$ and a reduced normal crossing divisor $D\subset X$ such that $D+K_X=0$.} In addition we are going to count rational curves with tangency conditions along the boundary divisor (namely, either of the union of all toric boundary divisors $\partial X^{\mathrm{toric}} $ or of the divisor $D$).     

A first approach to properly define Gromov--Witten invariants for $X^{\mathrm{toric}}$ and $(X,D)$ comes with the notion of \textit{relative stable morphisms}\footnote{A relative map $f\colon (C_g,p_1,...,p_n,q_1,...,q_s)\to (X,D)$ is such that $C_g$ is smooth genus $g$ curve with $p_1,...,p_n,q_1,...,q_s$ marked points and $f(q_i)\in D$. A relative map $(f,C_g,p_1,...,p_n,q_1,...,q_s)$ is stable if it is stable map and $f^{-1}(D)=\sum_{j=1}^sw_jq_j$ for some weight $w_j$.} introduced by Li \cite{Li}\cite{Linotes}. However Li's theory requires  $D$ to be smooth, hence in the toric case it can be applied on the open locus $(X^{\mathrm{toric},o},\partial X^{\mathrm{toric},o})$ where the zero torus orbits have been removed. In addition the moduli space of relative stable maps is not  proper in general, and Li introduces the notion of expanded degenerations, which require changing the target variety blowing it up.  
In \cite{GPS}, relative stable maps have been used to define genus zero Gromov--Witten invariants for open toric surface $X^{\mathrm{toric},o}$ with tangency conditions relative to the toric boundary divisors $\partial X^{\mathrm{toric},o}$. In addition by Li's degeneration formula \cite{Lideg} the authors get an expression to compute genus zero Gromov--Witten for the projective surface $(X,D)$ in terms of the invariants of $(X^{\mathrm{toric},o},\partial X^{\mathrm{toric},o})$, where $(X,D)$ is the blowup of the toric surface along a fixed number of generic points on the toric divisors $\partial X^{\mathrm{toric},o}$ and $D$ is the strict transform of $\partial X^{\mathrm{toric},o}$.

An alternative approach introduced by Gross--Siebert \cite{logGW} relies on logarithmic geometry. Compared to Li's theory, log theory allows to consider $(X,D)$ with $D$ a reduced normal crossing divisor. In addition every complete toric variety has a log structure over the full toric boundary divisor (without removing the zero torus orbits).

\subsection{Relation to open invariants}

Parallel to the existence of open and closed string theories, on one hand \textit{open} Gromov--Witten theory concerns ``counting'' of holomorphic maps from a genus $g$ curve with boundary components to a target manifold $X$ which admits a Lagrangian submanifold $L$, such that the image of boundary of the source lies on a Lagrangian fiber of the target. On the other hand \textit{closed} Gromov--Witten invariants aims to ``count'' holomorphic maps from a genus $g$ projective curve to a target $X$. In particular, open Gromov--Witten are rare in algebraic geometry, because it is not clear in general how to construct moduli space of maps between manifold with boundaries. However in \cite{LiSong06} by using relative stable maps the authors obtain the analogue results of the open string amplitudes computed by Ooguri--Vafa \cite{OV00}. Furthermore in \cite{Pierrick}, \cite{bousseau20} Bousseau explains how higher genus Gromov--Witten for log Calabi--Yau surfaces (which are higher genus generalizations of invariants for the blow-up surface of \cite{GPS}) offer a rigorous mathematical interpretation of the open topological string amplitudes computed in \cite{CV09}. 

From a different perspective, there have been many successfully results on computing open Gromov--Witten invariants in terms of holomorphic discs with boundary on a Lagrangian fiber (see e.g. \cite{CsemiFano} and the reference therein). The main advantage is the existence of a well defined notion of moduli space of stable discs, introduced by \cite{FOOO}.

\section{Main results}

Let $M$ be an affine tropical two dimensional manifold. Let $\Lambda$ be a lattice subbundle of $TM$ locally generated by $\frac{\partial}{\partial x_1},\frac{\partial}{\partial x_2} $, for a choice of affine coordinates $x=(x_1,x_2)$ on a contractible open subset $U\subset M$. We denote by $\Lambda^*=\Hom_\Z(\Lambda,\Z)$ the dual lattice and by \[\langle\cdot,\cdot\rangle\colon\Lambda^*\times\Lambda\to\C\] the natural pairing. 

Define $\check{X}\defeq TM/\Lambda$ to be the total space of the torus fibration $\check{p}\colon\check{X}\to M$ and similarly define $X\defeq T^*M/\Lambda^* $ as the total space of the dual torus fibration $p\colon X\to M$. Then, let $\lbrace \check{y}_1, \check{y}_2\rbrace$ be the coordinates on the fibres of  $\check{X}(U)$ with respect to the basis $\frac{\partial}{\partial x_1},\frac{\partial}{\partial x_2}$, and define a one-parameter family of complex structures on $\check{X}$:\[J_{\hbar}=\begin{pmatrix}
0 & \hbar I\\
-\hbar^{-1}I & 0
\end{pmatrix}\] with respect to the basis $\left\lbrace \frac{\partial}{\partial{x_1}},\frac{\partial}{\partial{x_2}},\frac{\partial}{\partial{\check{y}_1}},\frac{\partial}{\partial{\check{y}_2}}\right\rbrace$, parameterized by $\hbar\in\R_{>0}$. Notice that a set of holomorphic coordinates with respect to $J_{\hbar}$ is defined by \[z_j\defeq\check{y}_j+i\hbar x_j\] $j=1,2$; in particular we will denote by $\textbf{w}_j\defeq e^{2\pi iz_j}$. On the other hand $X$ is endowed with a natural symplectic structure \[\omega_{\hbar}\defeq\hbar^{-1}dy_j\wedge dx_j\] where $\lbrace y_j\rbrace$ are coordinates on the fibres of $X(U)$.

Motivated by Fukaya approach to mirror symmetry in \cite{MCscattering} the authors show how consistent scattering diagrams, in the sense of Kontsevich--Soibelman and Gross--Siebert can be constructed via the asymptotic analysis of deformations of the complex manifold $\check{X}$. Since the complex structure depends on a parameter $\hbar$, the asymptotic analysis is performed in the semiclassical limit $\hbar\to 0$. The link between scattering diagrams and solutions of Maurer--Cartan equations comes from the fact that the gauge group acting on the set of solutions of the Maurer--Cartan equation (which governs deformations of $\check{X}$) contains the \textit{tropical vertex group} $\mathbb{V}$ of Gross--Pandharipande--Siebert \cite{GPS}. 

In \cite{Fan19} and in the thesis (Section \ref{sec:extended tropical vertex}), we introduce the extended tropical vertex group $\tilde{\mathbb{V}}$ which is an extension of the tropical vertex group $\mathbb{V}$ via the general linear group $GL(r,\C)$. Hence the elements of $\tilde{\mathbb{V}}$ are pairs with a matrix component and a derivation component. Moreover, $\tilde{\mathbb{V}}$ is generated by a Lie algebra $\tilde{\mathfrak{h}}$, with a twisted Lie bracket and its definition is modelled on the deformation theory of holomorphic pairs $(\check{X}, E)$. In our applications, $\check{X}$ is defined as above and $E$ is a holomorphically trivial vector bundle on $\check{X}$. We always assume $\check{X}$ has complex dimension $2$, but we believe that this restriction can be removed along the lines of \cite{MCscattering}.    
Our first main result gives the required generalization of the construction of Chan, Conan Leung and Ma.
\begin{theorem}[Theorem \ref{thm:asymptotic_gauge}, Theorem \ref{thm:consistentD}]\label{thm:main1}
Let $\mathfrak{D}$ be an initial scattering diagram, with values in the extended tropical vertex group $\tilde{\mathbb{V}}$, consisting of two non-parallel walls. Then there exists an associated solution $\Phi$ of the Maurer-Cartan equation, which governs deformations of the holomorphic pair $(\check{X}, E)$, such that the asymptotic behaviour of  $\Phi$ as $\hbar \to 0$ defines uniquely a scattering diagram $\mathfrak{D}^{\infty}$ with values in $\tilde{\mathbb{V}}$. The scattering diagram $\mathfrak{D}^{\infty}$ is consistent.
\end{theorem}

Elements of the tropical vertex group are formal automorphisms of an algebraic torus and are analogous to the type $K$ automorphisms, and \textit{consistent} scattering diagrams with values the tropical vertex group reproduce wall-crossing formulas in the pure $4d$ case. This is a motivation for our second main result, namely the application to $2d$-$4d$ wall-crossing. As we mentioned WCFs for coupled $2d$-$4d$ systems involve automorphisms of type $S$ and $K$. By considering their infinitesimal generators (i.e. elements of the Lie algebra of derivations of $\Aut(\C[\mathbb{G}][\![ t ]\!])$), we introduce the Lie ring $\mathbf{L}_\Gamma$ which they generate as a $\C[\Gamma]$-module. On the other hand we construct a Lie ring $\tilde{\mathbf{L}}$ generated as $\C[\Gamma]$-module by certain special elements of the extended tropical vertex Lie algebra for holomorphic pairs, $\tilde{\mathfrak{h}}$. Then we compare these two Lie rings:
\begin{theorem}[Theorem \ref{thm:homomLiering}]\label{thm:main3}
Let $\left(\mathbf{L}_{\Gamma},[-,-]_{\Der(\C[\mathbb{G}])}\right)$ and $\left(\tilde{\mathbf{L}},[\cdot,\cdot]_{\tilde{\mathfrak{h}}}\right)$ be the $\C[\Gamma]$-modules discussed above (see Section \ref{sec:WCF}). Under an assumption on the BPS spectrum, there exists a homomorphism of $\C[\Gamma]$-modules and of Lie rings $\Upsilon\colon \mathbf{L}_{\Gamma}\to\tilde{\mathbf{L}}$.
\end{theorem}
This result shows that a consistent scattering diagram with values in (the formal group of) $\tilde{\mathbf{L}}$ is the same as a $2d$-$4d$ wall-crossing formula. Thus, applying our main construction with suitable input data, we can recover a large class of WCFs for coupled $2d$-$4d$ systems from the deformation theory of holomorphic pairs $(\check{X}, E)$.  
\\

In \cite{GPS} the authors show that computing commutators in the tropical vertex group allows one to compute genus zero Gromov--Witten invariants for weighted projective surfaces. Recently, Bousseau defines the \textit{quantum tropical vertex} \cite{Pierrick} and he shows that higher genus invariants (with insertion of Lambda classes) can be computed from commutators in the quantum tropical vertex. 
In the spirit of the previous works, we show that genus zero, relative Gromov--Witten invariants for some toric surfaces appear in the matrix component of the automorphisms of consistent scattering diagrams in $\tilde{\mathbb{V}}$. Let $\mathbf{m}=\left((-1,0), (0,-1), (a,b)\right)\in\Lambda^3$ and let $\overline{Y}_{\mathbf{m}}$ be the toric surface associated to the complete fan generated by $(-1,0), (0,-1), (a,b)$, where $(a,b)$ is a primitive vector. In section \ref{sec:toric GW} we define relative Gromov--Witten invariants $N_{0,{\mathbf{w}}}(\overline{Y}_{\mathbf{m}})$ counting curves of class $\beta_{\mathbf{w}}$ with tangency conditions at the boundary divisors, specified by the vector $\mathbf{w}\in\Lambda^s$ for some positive integer $s$. Then we consider the blow-up surface $Y_{\mathbf{m}}$ along a finite number of points on the toric boundary divisors associated to $(-1,0)$ and $(0,-1)$. Let $N_{0,{\mathbf{P}}}(Y_{\mathbf{m}})$ be the Gromov--Witten invariants with full tangency at a point on the strict transform of the boundary divisor. Then our main result is the following:    

 
\begin{theorem}[ Theorem \ref{thm:m_out}]\label{thm:main4}
Let $\mathfrak{D}$ be a standard scattering diagram in $\tilde{\mathbb{V}}$ which consists of n initial walls. Assuming the matrix contributions of the initial scattering diagram are all commuting, then the automorphism associated to a ray $(a,b)$ in the consistent scattering diagram $\mathfrak{D}^{\infty}$ is explicitly determined by $N_{0,{\mathbf{P}}}(Y_{\mathbf{m}})$ and $N_{0,\mathbf{w}}(\overline{Y}_{\mathbf{m}})$.
\end{theorem} 

Furthermore we conjecture that the automorphism associated to a ray $(a,b)$ in the consistent scattering diagram $\mathfrak{D}^{\infty}$ is a generating function of  $N_{0,\mathbf{w}}(\overline{Y}_{\mathbf{m}})$ (Conjecture \ref{conj}), i.e. it determines them completely. We show this in some special cases (see Theorem \ref{thm:partial}). Thus scattering diagrams in the extended tropical vertex group are closely related to both relative and absolute invariants. It seems that this may be compared with a much more general expectation in the physical literature that ``holomorphic Chern--Simons/BCOV'' theories, coupling deformations of a complex structure to an auxiliary bundle, contain the information of certain open/closed Gromov--Witten invariants \cite{open_closed}, \cite{CL12}.


\section{Plan of the thesis}

The first chapter is organized as follows: in Section \ref{sec:deformation} we provide some background on deformations of complex manifolds and of holomorphic pairs in terms of differential graded Lie algebras. Then in Section \ref{sec:scattering} we recall definitions and properties of scattering diagrams.  
In Chapter \ref{sec:MC and scattering} we introduce the tools which finally lead to the proof of Theorem \ref{thm:main1} in Section \ref{sec:two_walls}. 
\\
Then in Chapter \ref{sec:WCF} we recall the setting of wall-crossing formulas in coupled $2d$-$4d$ systems and prove Theorem \ref{thm:main3}. We also include two examples to show how the correspondence works explicitly. 
\\
Finally in Chapter \ref{cha:enumertaive geom} we exploit the relation between scattering diagrams in the extended tropical vertex group and relative Gromov--Witten invariants. In particular, in Section \ref{sec:tropical curve count} we give a first interpretation of commutators in the extended tropical vertex in terms of tropical curves counting. Then in Section \ref{sec:GW} we review definitions and properties of relative Gromov--Witten invariants for toric surfaces and for log Calabi--Yau surfaces. Finally the proof of the Theorem \ref{thm:main4} is given in Section \ref{sec:GW and commutators}.   

\chapter{Preliminaries}



\section{Deformations of complex manifolds and holomorphic pairs}
\label{sec:deformation}

In this section we review some background materials about infinitesimal deformations of complex manifolds and of holomorphic pairs with a differential geometric approach. We try to keep the material self-consistent and we refer the reader to Chapter 6 of Huybrechts's book \cite{Huybrechts}, Manetti's lectures note \cite{Lecture_Manetti} and Chan--Suen's paper \cite{defpair} for more detailed and complete discussions. 

Classically infinitesimal deformations of compact complex manifolds were studied as small parametric variations of their complex structure. Let $B\subset\C^m$ be an open subset which contains the origin and let $\check{X}$ be a compact complex manifold of dimension $n$.  

\begin{definition} A deformation of $\check{X}$ is a proper holomorphic submersion $\pi\colon\X\to B$ such that:
\begin{itemize}
\item $\X$ is a complex manifold;
\item $\pi^{-1}(0)=\check{X}$;
\item $\pi^{-1}(t)=:\check{X}_t$ is a compact complex manifold.
\end{itemize}
\end{definition}
Two deformations $\pi\colon\X\to B$ and $\pi'\colon\X'\to B$ over the same base $B$ are isomorphic if and only if there exists a holomorphic morphism $f\colon\X\to X'$ which commutes with $\pi, \pi'$. A deformation $\X\to B$ is said \textit{trivial} if it is isomorphic to the product $\check{X}\times B\to B$. If $\X\to B$ is a deformation of $\check{X}$, then the fibers $\check{X}_t$ are diffeomorphic to $\check{X}$ (in general they are not biholomorphic, see Ehresmann's theorem, Proposition 6.2.2 \cite{Huybrechts}). 
 
A complex structure on $\check{X}$ is an integrable almost complex structure $J\in\End(T\check{X})$ such that $J^2=\id$ and the holomorphic tangent bundle $T^{1,0}\check{X}\subset T_{\C}\check{X}\defeq T\check{X}\otimes_\R\C$ is an integrable distribution $[T^{1,0}\check{X}, T^{1,0}\check{X}]\subset T^{1,0}\check{X}$, where $[-,-]$ is the Lie bracket on vector fields in $T_\C\check{X}$. Analogously a complex structure on $\check{X}_t$ is an integrable almost complex structure $J_t$, and we denote by $T^{1,0}_t\check{X}$ ($T^{0,1}_t\check{X}$) the holomorphic (antiholomorphic) tangent bundle with respect to the splitting induced by $J_t$. If $t$ is small enough, then the datum of $J_t$ is equivalent to the datum of $\phi(t)\colon T^{0,1}\check{X}\to T^{1,0}\check{X}$ with $v+\phi(t)v\in T^{0,1}_t\check{X}$  such that $\phi(0)=0$\footnote{Indeed if $J_t$ is know, then $\phi(t)=-pr_{T^{1,0}_t\check{X}}\circ j$ where $pr_{T^{1,0}}\colon T^{0,1}\to T_\C\check{X}$ is the projection and $j\colon T^{0,1}\hookrightarrow T_\C\check{X}$ is the inclusion. Conversely if $\phi(t)$ is given, then $T^{1,0}_t\defeq\left(\id+\phi(t)\right)T^{0,1}\check{X}$.}. In addition, the integrability of $J_t$ is equivalent to the so called Maurer--Cartan equation, namely 
\begin{equation}
\label{eq:MC X}
\bar{\partial}\phi(t)+\frac{1}{2}[\phi(t),\phi(t)]=0
\end{equation} 
where $\bar{\partial}$ is the Dolbeaut differential with respect to the complex structure $J$, and $[-,-]$ is the standard Lie bracket on $T_\C\check{X}$ (see Theorem 1.1 Chapter 4 \cite{Chap4}). 


Let us assume $\phi(t)$ has a formal power expansion in $t\in B_\varepsilon\subset\C^m$, $\phi(t)=\sum_{k\geq 1}\phi_kt^k$, where $t^k$ is a homogeneous polynomial of degree $k$ in $t_1,...,t_m$: the Maurer--Cartan equation can be written order by order in $t$ as follows:
\begin{equation}\label{eq:hierarchy}
\begin{split}
&\bar{\partial}\phi_1=0 \\
&\bar{\partial}\phi_2+\frac{1}{2}[\phi_1,\phi_1]=0\\
&\vdots \\
&\bar{\partial}\phi_j+\frac{1}{2}\sum_{i=1}^{j-1}[\phi_i,\phi_{j-i}]=0
\end{split}
\end{equation} 
In particular, the first equation says that $\phi_1$ is a $\bar{\partial}$-closed $1$-form, hence it defines a cohomology class $[\phi_1]\in H^1(\check{X},T^{1,0}\check{X})$.

Let $\X\to B_\varepsilon$ and $\X'\to B_\varepsilon$ be two isomorphic deformations of $\check{X}$ and let $f\colon\X\to\X'$ be a holomorphic morphism such that $f\vert_{\check{X}}=\id_{\check{X}}$. In particular, for $t$ small enough, we denote by $f_t$ the one parameter family of diffeomorphisms of $\check{X}$ such that $f_t\colon X_t\to X_t'$ with $f_0=\id$. If the complex structures of $X_t$, $X_t'$ are $J_t$, $J_t'$ respectively, then $J_t'\circ f_t=f_t\circ J_t$. Let $\phi(t),\phi(t)'\in\Omega_{0,1}(\check{X},T^{1,0}\check{X})$ be such that $J_t=\bar{\partial}+\phi(t)\lrcorner\partial$ and $J_t'=\bar{\partial}+\phi(t)'\lrcorner\partial$. Then the difference between two isomorphic first order deformations is
\begin{align*}
\frac{d}{dt}\left(J_t'-J_t\right)\big\vert_{t=0}=\frac{d}{dt}\left(\phi(t)'-\phi(t)\right)\big\vert_{t=0}\lrcorner\partial =\left(\phi_1'-\phi_1\right)\lrcorner\partial.
\end{align*} 
In addition,  
\begin{align*}
\frac{d}{dt}\left(J_t'\circ f_t\right)\big\vert_{t=0}=\frac{d}{dt}\left(f_t\circ J_t\right)\big\vert_{t=0}
\end{align*}
which is equivalent to 
\begin{align*}
\frac{d}{dt}\left(J_t'\right)\Big\vert_{t=0}\circ f_0+J_0'\circ\frac{d}{dt}\left(f_t\right)\Big\vert_{t=0}=\frac{d}{dt}\left(f_t\right)\Big\vert_{t=0}\circ J_0+f_0\circ\frac{d}{dt}\left( J_t\right)\Big\vert_{t=0}\\
\frac{d}{dt}\left(J_t'\right)\Big\vert_{t=0}+\bar{\partial}\circ\frac{d}{dt}\left(f_t\right)\Big\vert_{t=0}=\frac{d}{dt}\left(f_t\right)\Big\vert_{t=0}\circ \bar{\partial}+\frac{d}{dt}\left( J_t\right)\Big\vert_{t=0}.
\end{align*}
Since $\frac{d}{dt}\left(f_t\right)\big\vert_{t=0}\in\Omega^0(\check{X},T\check{X})$, we define $h\defeq\left(\frac{d f_t}{dt}\big\vert_{t=0}\right)^{1,0}\in\Omega^0(\check{X},T^{1,0}\check{X})$ and from the previous computations we get 
\begin{align*}
\left(\phi_1'-\phi_1\right)T^{1,0}\check{X}=\left(\frac{d}{dt}\left(J_t'-J_t\right)\big\vert_{t=0}\right)T^{1,0}\check{X}=\left(\frac{d}{dt}\left(f_t\right)\Big\vert_{t=0}\circ \bar{\partial}-\bar{\partial}\circ\frac{d}{dt}\left(f_t\right)\Big\vert_{t=0}\right)T^{1,0}\check{X}=\bar{\partial}h T^{1,0}\check{X}.
\end{align*}
We have proved the following proposition:
\begin{prop}
Let $\X\to B_\varepsilon$ and $\X'\to B_\varepsilon$ be two isomorphic deformations of $\check{X}$. Then their first order deformations differ from a $\bar{\partial}$-exact form. 
\end{prop} 
Hence first order deformations are characterized as follows:
\begin{prop}
There exists a natural bijection between first order deformations of a compact complex manifold $\check{X}$ (up to isomorphism) and $H^1(\check{X},T^{1,0}\check{X})$.  
\end{prop}
In particular, if $H^1(\check{X},T^{1,0}\check{X})=0$ then every deformation is trivial. 

Let us now study the existence of solution of the Maurer--Cartan equation. Let $g$ be a hermitian metric on $\check{X}$, define the formal adjoint $\bar{\partial}^*$ of the Dolbeaut operator $\bar{\partial}$ with respect to the metric $g$ and the Laplace operator $\Delta_{\bar{\partial}}\defeq\bar{\partial}\bar{\partial}^*+\bar{\partial}^*\bar{\partial}$. Then we denote the set of harmonic form as $\mathbb{H}^p(\check{X})\defeq\lbrace \alpha\in\Omega^p(\check{X})\vert \Delta_{\bar{\partial}}\alpha=0\rbrace$ and let us choose a basis $\lbrace\eta_1,...,\eta_m\rbrace$ of $\mathbb{H}^1(\check{X})$, $m=\dim_\C\mathbb{H}^1(\check{X})$. Let $L^p$ be the completion of $\Omega^p(\check{X},T^{1,0}\check{X})$ with respect to the metric $g$ and recall that the Green operator $G\colon L^q(\check{X})\to L^{q}(\check{X})$ is a linear operator such that \[\id=\mathbf{H}+ \Delta_{\bar{\partial}}G \] where $\mathbf{H}\colon L^{q}(\check{X})\to\mathbb{H}^q(\check{X})$ is the harmonic projector. 
\begin{lemma}[Kuranishi's method]
Let $\eta=\sum_{j=1}^m\eta_jt_j$ be a harmonic form $\eta\in\mathbb{H}^1(\check{X})$. Then there exists a unique $\phi(t)\in\Omega^{0,1}(\check{X},T^{1,0}\check{X})$ such that 
\begin{align}\label{eq:Kur phi}
\phi(t)=\eta -\frac{1}{2}\bar{\partial}^*G([\phi(t),\phi(t)])
\end{align}
and for $|t|$ small enough it is holomorphic in $t$.

In addition, such $\phi(t)$ is a solution of the Maurer--Cartan equation \eqref{eq:MC X} if and only if $\mathbf{H}([\phi(t),\phi(t)])=0$.  
\end{lemma}
\begin{proof}
Let $\phi(t)$ be a solution of \eqref{eq:Kur phi} and assume it is a solution of the Maurer--Cartan equation. Then by property of Green's operator $G$: 
\begin{align*}
[\phi(t),\phi(t)]=\mathbf{H}([\phi(t),\phi(t)])+\Delta_{\bar{\partial}}G([\phi(t),\phi(t)])
\end{align*} 
which is equivalent to
\begin{align*}
-2\bar{\partial}\phi(t)&=\mathbf{H}([\phi(t),\phi(t)]+\bar{\partial}^*\bar{\partial}G([\phi(t),\phi(t)])\bar{\partial}\bar{\partial}^*G([\phi(t),\phi(t)])\\
&=\mathbf{H}([\phi(t),\phi(t)]+\bar{\partial}^*G\bar{\partial}([\phi(t),\phi(t)])\bar{\partial}(\eta -\phi(t))\\
&=\mathbf{H}([\phi(t),\phi(t)]-2\bar{\partial}(\phi(t))
\end{align*}
hence $\mathbf{H}([\phi(t),\phi(t)])=0$. 

Conversely, assume $\phi(t)$ is a solution of \eqref{eq:Kur phi} and $\mathbf{H}([\phi(t),\phi(t)])=0$. Then 
\begin{align*}
\bar{\partial}\phi(t)&=\bar{\partial}(\eta)-\frac{1}{2}\bar{\partial}\bar{\partial}^*G([\phi(t),\phi(t)])\\
&=-\frac{1}{2}\Delta_{\bar{\partial}}G([\phi(t),\phi(t)])+\frac{1}{2}\bar{\partial}^*\bar{\partial}G([\phi(t),\phi(t)])\\
&=\frac{1}{2}\mathbf{H}([\phi(t),\phi(t)])-\frac{1}{2}[\phi(t),\phi(t)]+\frac{1}{2}\bar{\partial}^*G\bar{\partial}([\phi(t),\phi(t)])\\
&=-\frac{1}{2}[\phi(t),\phi(t)]+\bar{\partial}^*G([\bar{\partial}\phi(t),\phi(t)]) 
\end{align*}
and by Jacobi identity
\[
\bar{\partial}\phi(t)+\frac{1}{2}[\phi(t),\phi(t)]=\bar{\partial}^*G([\bar{\partial}\phi(t)+\frac{1}{2}[\phi(t),\phi(t)],\phi(t)]). 
\]  
Let $\psi(t)=\bar{\partial}\phi(t)\frac{1}{2}[\phi(t),\phi(t)]$ and let us introduce the Holder norm $|| \bullet ||_{k,\alpha}$ with respect to the metric $g$. From analytic estimates in the Holder norm (see Chapter 4, Proposition 2.2, Proposition 2.3, Proposition 2.4 \cite{Chap4}) it follows that 
\begin{align*}
||\psi(t)||_{k,\alpha}=||\bar{\partial}^*G([\psi(t),\phi(t)]) ||_{k,\alpha}\leq C_1 ||G([\psi(t),\phi(t)]) ||_{k+1,\alpha}\leq C_2 ||[\psi(t),\phi(t)]||_{k-1,\alpha}\\
\leq C_3 ||\psi(t)||_{k,\alpha}||\phi(t)||_{k,\alpha} 
\end{align*} 
and by choosing $t$ small enough such that $C_3||\phi(t)||_{k,\alpha}<1$ (Proposition 2.4 \cite{Chap4}), we get a contradiction unless $\psi(t)=0$.  

Existence and uniqueness of $\phi(t)$ solution of \eqref{eq:Kur phi} relies on implicit function theorem for Banach spaces and some analytic estimates in the Holder norm (we refer to \cite{Kuranishi}).  
\end{proof}
\begin{definition}
A deformation $\pi\colon\X\to B$ of a compact complex manifold $\check{X}$ is said \textit{complete }if any other deformation $\pi'\colon\X'\to B'$ is the pull-back under some $f\colon B'\to B$, namely $\X'=\X\times_{B}B'$.
Moreover if $df_0\colon T_{0'}B'\to T_0B$ is always unique, the deformation $\X\to B$ is called \textit{versal}. 
\end{definition}  
Deformations of $\check{X}$ are called \textit{unobstructed} if $\check{X}$ admits a versal deformation $\X\to B$ and $B$ is smooth. 
\begin{theorem}[\cite{Kuranishi}]
Any compact complex manifold admits a versal deformation. 
\end{theorem}
\begin{definition}
Let $\mathbf{S}\subset\C^m$ be the set $\mathbf{S}\defeq\lbrace t\in\C^m\vert |t|<\varepsilon,\,\, \mathbf{H}([\phi(t),\phi(t)])=0\rbrace$. A family $\X\to\mathbf{S}$ parametrizing deformations of the manifold $\check{X}$ is the so called \textit{Kuranishi family}.
\end{definition}
In particular if $H^2(\check{X},T^{1,0}\check{X})=0$ then the solution $\phi(t)$ defined in \eqref{eq:Kur phi} solves Maurer--Cartan euqation \eqref{eq:MC X} and deformations are \textit{unobstructed} (this result was originally proved by Kodaira, Nirenberg and Spencer, see Chapter 4, Theorem 2.1 \cite{Chap4}). However, even if $H^2(\check{X},T^{1,0}\check{X})\neq0$, deformations may be unobstructed. Indeed let $\check{X}$ be Calabi--Yau manifolds, namely compact K\"ahler manifolds with trivial canonical bundle and let $\Omega_{\check{X}}$ be the holomorphic volume form of $\check{X}$, then deformations are unobstructed: 
\begin{theorem}[Bogomolov--Tian--Todorov\footnote{This result was first annaunced by Bogomolov \cite{Bogomolov} and then it has been proved independently by Tian \cite{Tian87} and Todorov \cite{Todorov}.}]
Let $\check{X}$ be a Calabi--Yau and let $g$ be the K\"ahler--Einstein metric. Let $\eta\in\mathbb{H}^1(\check{X})$ be a harmonic form, then there exists a unique convergent power series $\phi(t)=\sum_{j\geq 1}\phi_jt^j\in\Omega^{0,1}(\check{X},T^{1,0}\check{X})$ such that for $|t|<\varepsilon$
\begin{itemize}
\item[$(a)$] $[\phi_1]=\eta$;
\item[$(b)$] $\bar{\partial}^*\phi(t)=0$; 
\item[$(c)$] $\phi_j\lrcorner\Omega_{\check{X}}$ is $\partial$-exact for every $j>1$;
\item[$(d)$] $\bar{\partial}\phi(t)+\frac{1}{2}[\phi(t),\phi(t)]=0$. 
\end{itemize} 
\end{theorem}
The proof is by induction on the order of the power series, the convergence follows by analytic estimates and it realyies on the Kuranishi method (see \cite{Todorov}). If we relax the assumption on the metric and we let $g$ be a generic K\"ahler metric, then $\phi(t)$ will be a formal power series satisfying $(a),(c)$ and $(d)$ (see Huybrechts Proposition 6.1.11 \cite{Huybrechts}). Indeed first order deformation $\phi_1\in\Omega^{1,0}(\check{X},T^{1,0}\check{X})$ must be $\bar{\partial}$-closed hence by Hodge theorem 
\[[\phi_1]\in H^{0,1}(\check{X},T^{1,0}\check{X})\simeq H_{\bar{\partial}}^1(\check{X})\simeq\mathbb{H}^1(\check{X})\]      
and we can choose $\phi_1=\eta$. Then at order two in the formal parameter $t$, we need to find a solution $\phi_2\in\Omega^{0,1}(\check{X},T^{1,0}\check{X})$ such that
\begin{enumerate}
\item[$(d)$] $\bar{\partial}(\phi_2)=-\frac{1}{2}[\phi_1,\phi_1]$ and
\item[$(c)$] $\phi_2\lrcorner\Omega_{\check{X}}$ is $\partial$-exact.
\end{enumerate} 
These are consequences of the Tian--Todorov lemma\footnote{The Tian--Todorov lemma says the following: Let $\alpha\in\Omega^{0,p}(\check{X},T^{1,0}\check{X})$ and $\beta\in\Omega^{0,q}(\check{X},T^{1,0}\check{X})$, then \[(-1)^p[\alpha,\beta]=\Delta(\alpha\wedge\beta)-\Delta(\alpha)\wedge\beta-(-1)^{p+1}\alpha\wedge\Delta(\beta)\] where $\Delta\colon\Omega^{0,q}\left(\bigwedge^p T^{1,0}\check{x}\right)\to\Omega^{0,q}\left(\bigwedge^{p-1}T^{1,0}\check{X}\right)$ is defined as 
\[
\Delta\colon\Omega^{0,q}\left(\bigwedge^pT^{1,0}\check{X}\right)\xrightarrow{\lrcorner\Omega_{\check{X}}} \Omega^{n-p,q}\left(\check{X}\right)\xrightarrow{\partial} \Omega^{n-p+1,q}\left(\check{X}\right)\xrightarrow{\lrcorner\Omega_{\check{X}}} \Omega^{0,q}\left(\bigwedge^{p-1}T^{1,0}\check{X}\right).
\]
In addition the operator $\Delta$ anti-commutes with the differential $\bar{\partial}$, i.e. $\bar{\partial}\circ\Delta=-\Delta\circ\bar{\partial}$ (see Lemma 6.1.8 \cite{Huybrechts}).}: indeed it follows that $\bar{\partial}([\phi_1,\phi_1])=0$ and $[\phi_1,\phi_1]\lrcorner\Omega_{\check{X}}$ is $\partial$-exact. Hence, by Hodge decomposition, $[\phi_1,\phi_1]$ has no non-trivial harmonic part  and it is $\bar{\partial}$-exact. Moreover $\phi_2$ can be chosen such that $\phi_2\lrcorner\Omega_{\check{X}}$ is $\partial$-exact; indeed by $\partial\bar{\partial}$-lemma there exists $\gamma\in\Omega^{n-2,0}(\check{X},T^{1,0}\check{X})$ such that $\bar{\partial}\partial\gamma=\phi_2\lrcorner\Omega_{\check{X}}$ hence we choose $\phi_2\lrcorner \Omega_{\check{X}}=\partial\gamma$.

A modern approach to study deformations is via differential graded Lie algebras (DGLA) and we are going to define the Kodaira--Spencer DGLA which govern deformations of a complex manifold $\check{X}$. 
 
\begin{definition}      
A differential graded Lie algebra is the datum of a differential graded vector space $(L, \dd)$ together a with bilinear map $[-,-]\colon L\times L \to L$ (called bracket) of degree $0$ such that the following properties are satisfied:
\begin{enumerate}
\item[$-$] (graded skewsymmetric) $[a, b]= -(-1)^{\deg(a)\deg(b)}[b, a]$
\item[$-$] (graded Jacobi identity) $[a, [b, c]]=[[a, b], c] + (-1)^{\deg(a)\deg(b)}[b, [a, c]]$
\item[$-$] (graded Leibniz rule) $\dd[a, b]= [\dd a, b] + (-1)^{\deg(a)}[a, \dd b]$.
\end{enumerate}
\end{definition}
Let $\Art$ be the category of Artinian rings and for every $A\in\Art$ let $\m_A$ be the maximal ideal of $A$. Then we define the functor of deformations of a DGLA:
\begin{definition}
Let $(L, \dd, [-,-])$ be a DGLA, deformations of $(L, \dd, [-,-])$ are defined to be a functor \[\Def_L\colon\Art\to\Sets\] from the category of Aritinian rings to the category of sets, such that 
\begin{equation}
\Def_L(A)\defeq\left\lbrace\phi\in L^1\otimes\m_A\vert \dd\phi+[\phi,\phi]=0\right\rbrace/\text{gauge}
\end{equation}
where the gauge action is defined by $h\in L^0\otimes\m_A$ such that 
\begin{equation}\label{eq:gauge action X}
e^h\ast\phi\defeq\phi-\sum_{k\geq 0}\frac{1}{(k+1)!}\ad_h^k(\dd h- [h,\phi]), 
\end{equation}
and $\ad_h(-)=[h,-]$.
\end{definition}
We usually restrict to $A=\C[\![t]\!]$ so that $\phi\in L^1\otimes \mathsf{m}_t$ can be expanded as a formal power series in the formal parameter $t$. Then the Kodaira--Spencer DGLA $\KS(\check{X})$ which governs deformations of the complex manifold $\check{X}$ is defined as follows: let $\bar{\partial}_{\check{X}}$ be the Dolbeaut operator of the complex manifold $\check{X}$ and let $[-,-]$ be the Lie bracket such that $[\alpha^J d\bar{z}_J,\beta^K d\bar{z}_K]=d\bar{z}_J\wedge d\bar{z}_K [\alpha^J, \beta^K]$ for every $\alpha^Jd\bar{z}_J\in\Omega^{0,p}(\check{X},T^{1,0}\check{X})$ and $\beta^Kd\bar{z}_K\in\Omega^{0,q}(\check{X},T^{1,0}\check{X})$ with $|J|=p$ and $|K|=q$.
\begin{definition}
The Kodaira--Spencer DGLA is 
\begin{equation}
\KS(\check{X})\defeq\left(\Omega^{0,\bullet}(\check{X},T^{1,0}\check{X}), \bar{\partial}_{\check{X}}, [-,-]\right).
\end{equation}
\end{definition}
The gauge group action \eqref{eq:gauge action X} corresponds to the infinitesimal action of the diffeomorphisms group of $\check{X}$: indeed let $h=\left(\frac{d f_t}{dt}\vert_{t=0}\right)^{1,0}\in\Omega^0(\check{X},T^{1,0}\check{X})$ be a one parameter family of diffeomorphism $f_t$, then
\begin{align*}
e^{h}\circ(\bar{\partial}_{\check{X}}+\phi(t))\circ e^{-h}&=\sum_{j\geq 0}\frac{1}{j!}h^k\left(\bar{\partial}_{\check{X}}+\phi(t)\right)\sum_{l\geq 0}\frac{1}{l!}(-h)^l \\
&=\phi(t)+\sum_{n\geq 1}\frac{1}{n!}\sum_{j=0}^n{n \choose j} h^{n-j}\left(\bar{\partial}_{\check{X}}+\phi(t)\right)(-h)^{j}\\
&=\phi(t)+\sum_{n\geq 1}\frac{1}{n!}\ad_h^{n}\left(\bar{\partial}_{\check{X}}+\phi(t)\right)\\
&=\phi(t)+\sum_{n\geq 0}\frac{1}{(n+1)!}\ad_h^{n}\left(\left[h,\bar{\partial}_{\check{X}}+\phi(t)\right]\right)\\
&=\phi(t)+\sum_{n\geq 0}\frac{1}{(n+1)!}\ad_h^{n}\left(-\bar{\partial}_{\check{X}}h+\left[h,\phi(t)\right]\right)\\
&=e^{h}\ast \phi(t).
\end{align*}
Hence \[\Def_{\KS(\check{X)}}(\C[\![t]\!])=\left\lbrace \phi(t)=\sum_{j\geq 1}\phi_jt^j\in\Omega^{0,1}(\check{X},T^{1,0}\check{X})[\![t]\!]\big\vert \bar{\partial}_{\check{X}}\phi(t)+\left[\phi(t),\phi(t)\right]=0 \right\rbrace/\text{gauge}.\] 

\subsection{Infinitesimal deformations of holomorphic pairs}

Let $E$ be a rank $r$ holomorphic vector bundle on a compact complex manifold $\check{X}$ with fixed hermitian metric $h_E$. Then let $\bar{\partial}_E$ be the complex structure on $E$.
\begin{definition}
A holomorphic pair $(\check{X},E)$ is the datum of a complex manifold $(\check{X},\bar{\partial}_{\check{X}})$ and of a holomorphic vector bundle $(E,\bar{\partial}_E)$ on $\check{X}$. 
\end{definition}

Then let $\nabla^E$ be the Chern connection on $E$ with respect to $(h_E,\bar{\partial}_E)$ and let $F_E$ be the Chern curvature. The class $[F_E]\in H^{1,1}(\check{X},\End E)$ is called the \textit{Atiyah class} of $E$ and it does not depend on the metric $h_E$. Moreover it allows to define an extension $\mathbf{A}(E)$ of $\End E$ by $T^{1,0}\check{X}$: indeed $\mathbf{A}(E)\defeq \End E\oplus T^{1,0}\check{X}$ as a complex vector bundle on $\check{X}$ and it has an induced a holomorphic structure defined by $\bar{\partial}_{\mathbf{A}(E)}=\begin{pmatrix}
\bar{\partial}_E & B\\
0 & \bar{\partial}_{\check{X}}
\end{pmatrix}$, where $B\colon\Omega^{0,q}(\check{X},T^{1,0}\check{X})\ra\Omega^{0,q}(\check{X},\End E)$ acts on $\varphi\in\Omega^{0,q}(\check{X}, T^{1,0}\check{X})$ as $B\varphi\defeq\varphi\lrcorner F_E$.  

\begin{definition}[Definition 3.4 \cite{defpair}]
Let $(\check{X},E)$ be a holomorphic pair. A deformation of $(\check{X},E)$ consists of a holomorphic proper submersion $\pi\colon\X\to B_\varepsilon$ such that 
\begin{enumerate}
\item[-] $\pi^{-1}(0)=\check{X}$
\item[-] $\pi^{-1}(t)=\check{X}_t$ is a compact complex manifold 
\end{enumerate}
and of a holomorphic vector bundle $\E\to\X$ such that
\begin{enumerate}
\item[-] $\E\vert_{\pi^{-1}(0)}=E$
\item[-] for every $t\in B_\varepsilon$ the pair $(\check{X}_t,E_t)$ is the holomorphic pair parametrized by $t$. 
\end{enumerate}
\end{definition}
In particular, deformations of a holomorphic pair $(\check{X},E)$ are deformations of both the complex structure on $\check{X}$ and on $E$. 
\begin{definition}
Two deformations $\E\to\X$ and $\E'\to\X'$ of $(\check{X},E)$ on the same base $B$ are isomorphic if there exists a diffeomorphism $f\colon\X\to\X'$ and a bundle isomorphism $\Phi\colon\E\to\E'$ such that $f\vert_{\check{X}}=\id$, $\Phi\vert_E=\id$ and the following diagram
\[
\begin{tikzcd}
\E\arrow[r,"\Phi"]\arrow[d] & \E'\arrow[d]\\
\X\arrow[r, swap, "f"] & \X'
\end{tikzcd}
\]
commutes.
\end{definition}
We have already discussed how to characterize deformations of the complex structure of $\check{X}$. Indeed for $t$ small enough we can assume the $\check{X}_t=\lbrace t\rbrace\times\check{X}$ and the complex structure on $\check{X}_t$ is parametrized by $\phi(t)\in\Omega^{0,1}(\check{X},T^{1,0}\check{X})$. In addition we can trivialize $\mathcal{E}$ as $B_{\varepsilon}\times E$ so that the holomorphic structure on $E_t=\lbrace t\rbrace\times E$ is induced from $\mathcal{E}\vert_{X_t}$. Hence, following \cite{defpair} we define a differential operator on $\End E$ which defines a deformation of $\bar{\partial}_E$ in terms of deformations of $\bar{\partial}_{\check{X}}$. 
\begin{definition}
Let $(\check{X},E)$ be a holomorphic pair and let $(A(t),\phi(t))\in\Omega^{0,1}(\check{X},\mathbf{A}(E))$, then define  
\[\overline{D}_t\defeq \bar{\partial}_E+A_t+\phi_t\lrcorner\nabla^E\]
\end{definition} 
\begin{lemma}
$\overline{D}_t\colon\Omega^{0,q}(\check{X},E)\to\Omega^{0,q+1}(\check{X},E) $ is a well-defined operator and it satisfies the Leibniz rule. 
\end{lemma}
\begin{proof}
Let $s$ be a section of $E$ and let $f$ be a smooth function on $\check{X}$, then 
\begin{align*}
\overline{D}_t(fs)&=\bar{\partial}_E(fs)+A_tfs+\phi(t)\lrcorner\nabla^E(fs)\\
&=(\bar{\partial}_{\check{X}}f)s+f\bar{\partial}_Es+fA_ts+(\phi(t)\lrcorner df)s+f\phi(t)\nabla^Es\\
&=(\bar{\partial}_{\check{X}}f+\phi(t)\lrcorner \partial f)s+f\overline{D}_ts\\
&=(\bar{\partial}_tf)s+f\overline{D}_ts.
\end{align*}
\end{proof}
\begin{theorem}[Theorem 3.12 \cite{defpair}]
Let $(A(t),\phi(t))\in\Omega^{0,1}(\check{X},\mathbf{A}(E))$ and assume $\phi(t)$ is a deformation of $\check{X}$. If $\overline{D}_t^2=0$ then it induces a holomorphic structure on $E$ over $\check{X}_t$ which is denoted by $E_t\to\check{X}_t$. 
\end{theorem}
\begin{definition}
Deformations of a holomorphic pair $(\check{X},E)$ are defined by pairs $(A(t),\phi(t))\in\Omega^{0,1}(\check{X},\mathbf{A}(E))$ such that $\phi(t)$ is a deformation of $\check{X}$ and $\overline{D}_t^2=0$.
\end{definition}

\begin{lemma}
Let $(A(t),\phi(t))\in\Omega^{0,1}(\check{X},\mathbf{A}(E))$ and assume \[\bar{\partial}_{\check{X}}\phi(t)+\frac{1}{2}[\phi(t),\phi(t)]=0\] then  integrability of $\overline{D}_t$ is equivalent to the following system of equations:
\begin{equation}\label{eq: D_t^2=0}
\begin{cases}
\bar{\partial}_EA(t)+\phi(t)\lrcorner F_E+\frac{1}{2}[A(t),A(t)]+\phi(t)\lrcorner\nabla^EA(t)=0\\
\bar{\partial}_{\check{X}}\phi(t)+\frac{1}{2}[\phi(t),\phi(t)]=0
\end{cases}
\end{equation} 
\end{lemma}
\begin{proof}
Let $\lbrace e_k\rbrace$ be a holomrphic frame for $(E,\bar{\partial}_E)$ and let $s=s^ke_k$ be a smooth section of $E$, then 
\begin{align*}
\overline{D}_t^2s&=(\bar{\partial}_E+A(t)+\phi(t)\lrcorner\nabla^E)(\bar{\partial}_Es+A(t)s+\phi(t)\lrcorner\nabla^Es)\\
&=\bar{\partial}_E^2s+\bar{\partial}_E(A(t)s)+\bar{\partial}_E(\phi(t)\lrcorner\nabla^Es)+A(t)\wedge\bar{\partial}_E(s)+A(t)\wedge A(t)s+A(t)\wedge\phi(t)\lrcorner\nabla^Es+\\
&\quad+\phi(t)\lrcorner\nabla^E(\bar{\partial}_Es)+\phi(t)\lrcorner\nabla^E(A(t)s) +\phi(t)\lrcorner\nabla^E(\phi(t)\lrcorner\nabla^Es)\\
&=\bar{\partial}_E(A(t))s-A(t)\wedge\bar{\partial}_Es+(\bar{\partial}\phi(t))\lrcorner\nabla^Es+\phi(t)\lrcorner \bar{\partial}_E\nabla^Es+A(t)\wedge\bar{\partial}_E(s)+A(t)\wedge A(t)s+\\
&\quad+A(t)\wedge\phi(t)\lrcorner\nabla^Es+\phi(t)\lrcorner\nabla^E(\bar{\partial}_Es)+\phi(t)\lrcorner\nabla^E(A(t)s) +\phi(t)\lrcorner\nabla^E(\phi(t)\lrcorner\nabla^Es)\\
&=\bar{\partial}_E(A(t))s+(\bar{\partial}\phi(t))\lrcorner\nabla^Es+\phi(t)\lrcorner (\nabla^E\bar{\partial}_Es+ \bar{\partial}_E\nabla^Es)+A(t)\wedge A(t)s+A(t)\wedge\phi(t)\lrcorner\nabla^Es+\\
&\quad+\phi(t)\lrcorner\nabla^E(A(t))s-\phi(t)\lrcorner A(t)\wedge\nabla^E(s) +\phi(t)\lrcorner\nabla^E(\phi(t)\lrcorner\nabla^Es)\\
&=\bar{\partial}_E(A(t))s+(\bar{\partial}\phi(t))\lrcorner\nabla^Es+\phi(t)\lrcorner F^Es+A(t)\wedge A(t)s+\phi(t)\lrcorner\nabla^E(A(t))s+\phi(t)\lrcorner\nabla^E(\phi(t)\lrcorner\nabla^Es)\\
&=\left(\bar{\partial}_E(A(t))+\phi(t)\lrcorner F^E+A(t)\wedge A(t)+\phi(t)\lrcorner\nabla^E(A(t))\right)s+\phi(t)\lrcorner\nabla^E(\phi(t)\lrcorner\nabla^Es)+(\bar{\partial}\phi(t))\lrcorner\nabla^Es.
\end{align*}
Now we claim $\phi(t)\lrcorner\nabla^E(\phi(t)\lrcorner\nabla^Es)=[\phi(t),\phi(t)]\lrcorner\nabla^Es$ and since $\phi(t)$ is a solution of the Maurer--Cartan equation, we get the expected equivalence. Let us prove the claim in local coordinates: let $\phi(t)=\phi_{ij}d\bar{z}_i\partial_j$ and let $\nabla^E=d +\Theta^ldz_l$ (where we are summing over repeated indexes), then

\begin{align*}
\phi(t)\lrcorner\nabla^E(\phi(t)\lrcorner\nabla^Es)&=\phi(t)\lrcorner\nabla^E\left(\phi_{ij}d\bar{z}_i\left(\frac{\partial s^a}{\partial z_j}+\Theta^j_{ab}s^b\right)e_a\right)\\
&=\phi_{ml}\left(\frac{\partial \phi_{ij}}{\partial z_l}\frac{\partial s^a}{\partial z_j}+\phi_{ij}\frac{\partial^2 s^a}{\partial z_j\partial z_l}+\frac{\partial \phi_{ij}}{\partial z_l}\Theta^j_{ab}s^b\right)d\bar{z}_m\wedge d\bar{z}_i e_a+\\
&\quad+\phi_{ml}\left(\phi_{ij}\frac{\partial\Theta_{ab}^j}{\partial z_l}s^b+\phi_{ij}\Theta_{ab}^j\frac{\partial s^b}{\partial z_l} \right)d\bar{z}_m\wedge d\bar{z}_i e_a\\
&\quad+\phi_{ml} \Theta_{hk}^l\left(\phi_{ij}\frac{\partial s^k}{\partial z_j}+\phi_{ij}\Theta^j_{kb}s^b\right)d\bar{z}_m\wedge d\bar{z}_i e_h\\
&=\phi_{ml}\frac{\partial \phi_{ij}}{\partial z_l}\left(\frac{\partial s^a}{\partial z_j}+\Theta^j_{ab}s^b\right) d\bar{z}_m\wedge d\bar{z}_i e_a+\\
&\quad+\phi_{ml}\phi_{ij}\left(\frac{\partial\Theta_{ab}^j}{\partial z_l}+\Theta_{ak}^l\Theta^j_{kb}\right)s^b d\bar{z}_m\wedge d\bar{z}_i e_a
\end{align*}
where in the last step we notice that $\phi_{ml}\phi_{ij}\frac{\partial^2 s^a}{\partial z_j\partial z_l}d\bar{z}_m\wedge d\bar{z}_i e_a=0$. Finally we have
\begin{align*}
\phi_{ml}\frac{\partial \phi_{ij}}{\partial z_l}\left(\frac{\partial s^a}{\partial z_j}+\Theta^j_{ab}s^b\right) d\bar{z}_m\wedge d\bar{z}_i e_a&=2[\phi_{ml}d\bar{z}^m\partial_l ,\phi_{ij}d\bar{z}^i\partial_j]\lrcorner\left(\frac{\partial s^a}{\partial z_p}e_a+\Theta^p_{ab}s^be_a dz_p\right)\\
&=2\left[\phi(t),\phi(t)\right]\lrcorner \nabla^Es
\end{align*}
while, since $F_E$ is of type $(1,1)$, $\frac{\partial\Theta_{ab}^j}{\partial z_l}+\Theta_{ak}^l\Theta^j_{kb}=0$ and the claim is proved. 
\end{proof}
Recall the definition of $\bar{\partial}_{\mathbf{A}(E)}$, then we define the following operator:
\begin{align*}
\left[-,-\right]&\colon \Omega^{0,q}(\check{X},\mathbf{A}(E))\times\Omega^{0,r}(\check{X},\mathbf{A}(E))\to\Omega^{0,q+r}(\check{X},\mathbf{A}(E))\\
\left[(A(t),\phi(t)), (A'(t),\phi'(t))\right]&\defeq\left([A(t),A'(t)]+\phi(t)\lrcorner\nabla^EA'(t)+\phi'(t)\lrcorner\nabla^EA(t), [\phi(t),\phi'(t)]\right).
\end{align*}
It is possible to prove that $[-,-]$ is indeed a Lie bracket. Therefore deformations of a holomorphic pair $(\check{X},E)$ are governed by $(A(t),\phi(t))\in\Omega^{0,1}(\check{X},\mathbf{A}(E))$ which is a solution of the Maurer--Cartan equation
\begin{equation}\label{eq:MC}
\bar{\partial}_{\mathbf{A}(E)}(A(t),\phi(t)) +\frac{1}{2}\big[(A(t),\phi(t)),(A(t),\phi(t))\big]=0.
\end{equation}
These lead to the following definition:
\begin{definition}
The Kodaira-Spencer DGLA which governs deformations of a holomorphic pair $(\check{X},E)$ is defined as follows
\begin{equation}
\KS(\check{X},E)\defeq (\Omega^{0,\bullet}(\check{X},\mathbf{A}(E)),\bar{\partial}_{\mathbf{A}(E)},[-,-]).
\end{equation}
\end{definition}
The proof that $\KS(\check{X},E)$ is indeed a DGLA follows from the definitions of $\bar{\partial}_{\mathbf{A}(E)}$ and of the Lie bracket $[-,-]$. Moreover, it can be proved that $\KS(\check{X},E)$ does not depend on the choice of the hermitian metric $h_E$ (see Appendix A \cite{defpair}). Thus deformations of $\KS(\check{X},E)$ are defined by  
\begin{multline}
\Def_{\KS(\check{X},E)}(\C[\![t]\!])=\Big\lbrace (A(t),\phi(t))\in\Omega^{0,1}(\check{X},\mathbf{A}(E))[\![t]\!]\,\big\vert\\ \bar{\partial}_{\mathbf{A}(E)}(A(t),\phi(t))+\left[\left(A(t),\phi(t)\right),\left(A(t),\phi(t)\right)\right]=0\Big\rbrace/\text{gauge}
\end{multline}
where the gauge group acting on the set of solutions of \eqref{eq:MC} is defined by $(\Theta,h)\in\Omega^{0}(\check{X},\End E\oplus T^{1,0}\check{X})[\![ t ]\!]$ such that it acts on $(A,\phi)\in\Omega^{0,1}(\check{X},\End E\oplus T^{1,0}\check{X})[\![ t ]\!]$ as follows 
\begin{equation}
e^{(\Theta,h)}\ast (A,\phi)\defeq \left(A-\displaystyle\sum_{k=0}^\infty \frac{\ad_\Theta^k}{(k+1)!}(\bar{\partial} \Theta-[\Theta,A]), \phi-\displaystyle\sum_{k=0}^\infty \frac{\ad_h^k}{(k+1)!}(\bar{\partial} h-[h,\phi])\right).
\end{equation}



Let us now assume $(A(t),\phi(t))=\sum_{j\geq 1}(A_j,\phi_j)t^j$, then from the Maurer--Cartan equation \eqref{eq:MC} it is immediate that first order deformations $(A_1,\phi_1)\in\Omega^{0,1}(\check{X},T^{1,0}\check{X})$ are $\bar{\partial}_{\mathbf{A}(E)}$-closed $1$-forms hence they define a cohomology class $[(A_1,\phi_1)]\in H^1(\check{X},\mathbf{A}(E))$. Furthermore we claim that two isomorphic deformations differ by a $\bar{\partial}_{\mathbf{A}(E)}$-exact form. Indeed let $\E\to\X$ and $\E'\to\X'$ be two isomorphic families of deformations of $(\check{X},E)$ which are represented by $(A(t),\phi(t)),(A(t)',\phi(t)')\in\Omega^{0,1}(\check{X},\mathbf{A}(E))$ respectively. Then, for $t$ small enough, we denote by $f_t\colon \check{X}_t\to\check{X}_t'$ a one parameter family of diffeomorphisms of $\check{X}$ such that $f_0=\id_{\check{X}}$ and by $\Phi_t\colon E_t\to E_t'$ a one parameter family of bundle isomorphisms such that $\Phi_0=\id_E$ and they cover $f_t$, i.e. $\Phi_t\colon E_x\to E_{f_t(x)}$ for every $x\in\check{X}$ where $E_x$ is the fiber over $x$. In addition if $\overline{D}_t'$ and $\overline{D}_t$ denote the deformed complex structure of $E_t$ and $E_t'$ respectively, which are defined in terms of $(A(t),\phi(t))$ and $(A'(t),\phi'(t))$, then 
\begin{equation}
\label{eq:D_tPhi}
\overline{D}_t'\Phi_t=\Phi_t\overline{D}_t.
\end{equation} 
Let $x\in\check{X}$ and denote by $\mathcal{P}_{\gamma_x(t)}\colon E_x\to E_{f_t(x)}$ the parallel transport along $\gamma_x(t)=f_t(x)$ associated with the connection $\nabla^E$. Then we define $\Theta_t\defeq \mathcal{P}^{-1}_{\gamma_x(t)}\Phi_t\colon E_x\to E_x$. Since $\Phi_t$ is a holomorphic bundle isomorphism, it follows that $\Theta_t$ is a well defined holomorphic bundle isomorphism too and $\Theta_0=\id_E$. We claim that
\begin{equation}\label{eq:claim A}
(A_1',\phi_1')-(A_1,\phi_1)=\left(\bar{\partial}_E\Theta_1+h\lrcorner F_E, \bar{\partial}_{\check{X}}h\right)=\bar{\partial}_{\mathbf{A}(E)}(\Theta_1,h).
\end{equation}
Recall that two first order deformations of $\check{X}$ differ by a $\bar{\partial}$-exact form, i.e. there is $h\in\Omega^{0}(\check{X},T^{1,0}\check{X})$ such that $\bar{\partial}_{\check{X}}h=\phi_1'-\phi_1$. Let us now compute $A_1'-A_1$: from definition of $\overline{D}_t'$ and $\overline{D}_t$ we have
\begin{align*}
A_1'-A_1&=\frac{d}{dt}\left(\overline{D}_t'-\overline{D}_t\right)\Big\vert_{t=0}-\bar{\partial}_{\check{X}}h\lrcorner\nabla_E.
\end{align*}  
In addition, taking derivatives from \eqref{eq:D_tPhi} we get
\begin{align*}
\frac{d}{dt}\left(\overline{D}_t'\right)\Big\vert_{t=0}\Phi_0+\bar{\partial}_E\frac{d}{dt}\left(\Phi_t\right)\Big\vert_{t=0}=\frac{d}{dt}\left(\Phi_t\right)\Big\vert_{t=0}\bar{\partial}_E+\Phi_0\frac{d}{dt}\left(\overline{D}_t\right)\Big\vert_{t=0}
\end{align*}
hence $\frac{d}{dt}\left(\overline{D}_t'-\overline{D}_t\right)\Big\vert_{t=0}=\frac{d}{dt}\left(\Phi_t\right)\Big\vert_{t=0}\bar{\partial}_E- \bar{\partial}_E\frac{d}{dt}\left(\Phi_t\right)\Big\vert_{t=0}$. Furthermore, if $s_t$ is a holomorphic section of $E_t$, then $\Phi_ts_t'=\mathcal{P}_{\gamma(t)}\Theta_ts_t$ and 
\begin{align*}
\frac{d}{dt}\left(\Phi_ts_t\right)\Big\vert_{t=0}&=\frac{d}{dt}\left( \mathcal{P}_{\gamma (t)}\Theta_t s_t \right)\Big\vert_{t=0} \\
\frac{d}{dt}\left(\Phi_t\right)\Big\vert_{t=0}s_0+\frac{d}{dt}\left(s_t\right)\Big\vert_{t=0}&=-h\lrcorner\nabla^Es_0+\Theta_1s_0+\frac{d}{dt}\left(s_t\right)\Big\vert_{t=0}
\end{align*}
thus $\frac{d}{dt}\left(\Phi_t\right)\Big\vert_{t=0}s_0=-h\lrcorner\nabla^Es_0+\Theta_1s_0$. Collecting these results together we end up with 
\begin{align*}
(A_1'-A_1)s_0&=-\bar{\partial}_E\frac{d}{dt}\left(\Phi_t\right)\Big\vert_{t=0}s_0-\bar{\partial}_{\check{X}}h\lrcorner\nabla_Es_0\\
&=-\bar{\partial}_E\left(-h\lrcorner\nabla^Es_0+\Theta_1s_0\right)-\bar{\partial}_{\check{X}}h\lrcorner\nabla_Es_0\\
&=\bar{\partial}_{\check{X}}h\lrcorner\nabla^Es_0+h\lrcorner\bar{\partial}_E\nabla^Es_0+\bar{\partial}_E\Theta_1s_0-\bar{\partial}_{\check{X}}h\lrcorner\nabla_Es_0\\
&=h\lrcorner F_Es_0+\bar{\partial}_E\Theta_1s_0
\end{align*}  
and we have proved the claim \eqref{eq:claim A}. 
\begin{lemma}
There is a bijective correspondence between the first order deformations up to gauge equivalence and the first cohomology group $H^1(\check{X},\mathbf{A}(E))$.
\end{lemma}
 
The obstruction to extend the solution of Maurer--Cartan to higher order in the formal parameter $t$ is encoded in the second cohomology $H^2(\check{X},\mathbf{A}(E))$. In particular, let us fix an hermintian metric $g_{\mathbf{A}(E)}$ on $\mathbf{A}(E)$, then the Kuranishi method applies:
\begin{prop}
Let $\eta=\sum_{j=1}^m\eta_jt_j\in\mathbb{H}^1(\mathbf{A}(E))$ be a harmonic form, where $\lbrace\eta_1,...\eta_m\rbrace$ is a basis for $\mathbb{H}^1(\check{X})$. There exists a unique solution $(A(t),\phi(t))\in\Omega^{0,1}(\check{X},\mathbf{A}(E))$ of
\[(A(t),\phi(t))= \eta +\bar{\partial}^*_{\mathbf{A}(E)}G([(A(t),\phi(t)),(A(t),\phi(t)))] \]
In addition it solves the Maurer--Cartan equation \eqref{eq:MC} if and only if \[\mathbf{H}([(A(t),\phi(t)),(A(t),\phi(t))])=0\] where $\mathbf{H}$ is the harmonic projection on $\mathbb{H}(\mathbf{A}(E))$ and $G$ is the Green operator of the Laplacian $\Delta_{\bar{\partial}_{\mathbf{A}(E)}}=\bar{\partial}_{\mathbf{A}(E)}\bar{\partial}_{\mathbf{A}(E)}^{*_{g_{\mathbf{A}(E)}}}+\bar{\partial}_{\mathbf{A}(E)}^{*_{g_{\mathbf{A}(E)}}}\bar{\partial}_{\mathbf{A}(E)}$.
\end{prop}

As a consequence of Bogomolov--Tian--Todorov theorem the following result has been proved in \cite{defpair}:
\begin{theorem}[Proposition 7.7 \cite{defpair}]
Let $\check{X}$ be a compact Calabi--Yau surface and let $E$ be a holomorphic vector bundle on $\check{X}$ such that $c_1(E)\neq 0$ and the trace free second cohomology vanishes $H^2(\check{X},\End_0 E)=0$. Then infinitesimal deformations of the pair $(\check{X},E)$ are unobstructed.
\end{theorem}

\section{Scattering diagrams}\label{sec:scattering}

Scattering diagrams have been introduced by Kontsevich--Soibelman in \cite{KS} and they usually encode a combinatorial structure. Naively we may define scattering diagrams as a bunch of co-dimension one subspace in $\R^n$ decorated with some automorphisms. In our application we will restrict to $\R^2$ and the automorphisms group will be a generalization of the group of formal automorphisms of an algebraic torus $\C^*\times\C^*$. Let us first introduce some notation: 
\begin{notation}
Let $\Lambda\simeq\Z^2$ be a rank two lattice and choose $e_1$ and $e_2$ being a basis for $\Lambda$. Then the group ring $\mathbb{C}[\Lambda]$ is the ring of Laurent polynomial in the variable $z^m$, with the convention that $z^{e_1}=:x$ and $z^{e_2}=:y$. 
\end{notation}

We define the Lie algebra $\mathfrak{g}$ as follows:
\begin{equation}
\mathfrak{g}\defeq\mathsf{m}_t\big(\mathbb{C}[\Lambda]{\otimes}_{\mathbb{C}}\C[\![ t ]\!]\big)\otimes_{\mathbb{Z}} \Lambda^*
\end{equation}
where every $n\in\Lambda^*$ is associated to a derivation $\partial_n$ such that $\partial_n(z^m)=\langle m,n\rangle z^m$ and the natural Lie bracket on $\mathfrak{g}$ is 
\begin{equation}
[z^m\partial_n,z^{m'}\partial_{n'}]\defeq z^{m+m'}\partial_{\langle m',n\rangle n'-\langle m, n'\rangle n}.
\end{equation}
In particular $\mathfrak{g}$ has a Lie sub-algebra $\mathfrak{h}\subset\mathfrak{g}$ defined by:
\begin{equation}
\mathfrak{h}\defeq\bigoplus_{m\in \Lambda\smallsetminus\lbrace0\rbrace}z^m\cdot\big(\mathsf{m}_t\otimes m^{\perp}\big),
\end{equation}
where $m^{\perp}\in\Lambda^*$ is identified with the derivation $\partial_n$ and $n$ the unique primitive vector such that $\langle m,n\rangle=0$ and it is positive oriented according with the orientation induced by $\Lambda_{\mathbb{R}}\defeq\Lambda\otimes_{\Z}\mathbb{R}$. 
\begin{definition}
The tropical vertex group $\mathbb{V}$ is the sub-group of $\Aut_{\C[\![ t ]\!]}\big(\mathbb{C}[\Lambda]{\otimes}_{\mathbb{C}}\C[\![ t ]\!]\big)$, such that $\mathbb{V}\defeq\exp(\mathfrak{h})$. The product on $\mathbb{V}$ is defined by the Baker-Campbell-Hausdorff (BCH) formula, namely  
\begin{equation}
g\circ g'=\exp(h)\circ \exp(h')\defeq\exp(h\bullet h')=\exp(h+h'+\frac{1}{2}[h,h']+\cdots)
\end{equation}
where $g=\exp(h), g'=\exp(h')\in\mathbb{V}$.  
\end{definition}
The tropical vertex group has been introduced by Kontsevich--Soibelman in \cite{KS} and in the simplest case its elements are formal one parameter families of symplectomorphisms of the algebraic torus $\C^*\times\C^*=\Spec\C[x,x^{-1},y,y^{-1}]$ with respect to the holomorphic symplectic form $\omega=\frac{dx}{x}\wedge \frac{dy}{y}$. Indeed let $f_{(a,b)}=1+tx^ay^b\cdot g(x^ay^b,t)\in\C[\Lambda]\otimes_\C\C[\![t]\!]$ for some $(a,b)\in\Z^2$ and for a polynomial $g(x^ay^b,t)$, then $\theta_{(a,b),f_{(a,b)}}$ is defined as follows:
\begin{equation}
\theta_{(a,b),f_{(a,b)}}(x)=f^{-b}x \qquad \theta_{(a,b),f_{(a,b)}}(y)=f^ay.
\end{equation}
In particular $\theta_{(a,b),f_{(a,b)}}^*\omega=\omega$. 

We can now state the definition of scattering diagrams according to \cite{GPS}: \begin{definition}[Scattering diagram]\label{def:scattering}
A scattering diagram $\mathfrak{D}$ is a collection of \textit{walls} $\textsf{w}_i=(m_i,\mathfrak{d}_i,\theta_i)$, where 
\begin{itemize}
    \item $m_i\in \Lambda$,
    \item $\mathfrak{d}_i$ can be either a \textit{line} through $\xi_0$, i.e. $\mathfrak{d}_i=\xi_0-m_i\mathbb{R}$ or a \textit{ray} (half line) $\mathfrak{d}_i=\xi_0-m_i\mathbb{R}_{\geq 0}$,
    \item $\theta_i\in\mathbb{V}$ is such that $\log(\theta_i)=\sum_{j,k}a_{jk}t^j z^{km_i}\partial_{n_i}$. 
\end{itemize}
Moreover for any $k>0$ there are finitely many $\theta_i$ such that $\theta_i\not\equiv 1$ mod $t^k$.
\end{definition}
 
As an example, the scattering diagram \[\mathfrak{D}=\lbrace \mathsf{w}_1=\big(m_1=(1,0), \mathfrak{d}_1=m_1\R, \theta_1\big), \mathsf{w}_2=\big(m_2=(0,1), \mathfrak{d}_2=m_2\R, \theta_2\big)\rbrace\] can be represented as if figure \ref{fig:D1}.
\begin{figure}[h]
\center
\begin{tikzpicture}
\draw (2,0) -- (2,4);
\draw (0,2) -- (4,2);
\node [below right, font=\tiny] at (2,2) {0};
\node [below right, font=\tiny] at (2,2) {0};
\node [font=\tiny,above right] at (4,2) {$\theta_1$};
\node [font=\tiny,above left] at (2,4) {$\theta_2$};
\end{tikzpicture}
\caption{A scattering diagram with only two walls $\mathfrak{D}=\lbrace\mathsf{w}_1,\mathsf{w}_2\rbrace$}
\label{fig:D1}
\end{figure}
 
Denote by $\text{Sing}(\mathfrak{D})$ the singular set of $\mathfrak{D}$: 
\[\text{Sing}(\mathfrak{D})\defeq\bigcup_{\mathsf{w}\in\mathfrak{D}}\mathsf{\partial}\mathfrak{d}_\mathsf{w}\cup\bigcup_{\mathsf{w}_1,\mathsf{w}_2}\mathfrak{d}_{\mathsf{w}_1}\cap \mathfrak{d}_{\mathsf{w}_2}\] where $\partial \mathfrak{d}_\mathsf{w}=\xi_0$ if $\mathfrak{d}_{\mathsf{w}}$ is a ray and zero otherwise. There is a notion of ordered product for the automorphisms associated to each lines of a given scattering diagram, and it is defined as follows:
\begin{definition}[Path ordered product]\label{def:pathorderedprod}
Let $\gamma:[0,1]\rightarrow \Lambda\otimes_{\Z}\R\setminus\text{Sing}(\mathfrak{D})$ be a smooth immersion with starting point that does not lie on a ray of the scattering diagram $\mathfrak{D}$ and such that it intersects transversally the rays of $\mathfrak{D}$ (as in figure \ref{fig:Dloop}). For each power $k>0$, there are times $0< \tau_{1}\leq\cdots\leq\tau_{s} <1$ and rays $\mathfrak{d}_i\in\mathfrak{D}$ such that $\gamma(\tau_j)\cap \mathfrak{d}_j\neq 0$. Then, define $\Theta_{\gamma,\mathfrak{D}}^k\defeq\prod_{j=1}^s\theta_{j}$. The path ordered product is given by: 
\begin{equation}
\Theta_{\gamma,\mathfrak{D}}\defeq\lim_{k\rightarrow\infty}\Theta_{\gamma,\mathfrak{D}}^k
\end{equation}
\end{definition}

\begin{figure}[h]
\begin{tikzpicture}
\draw[thick] (2,0) -- (2,4);
\draw[thick] (0,2) -- (4,2);
\node [below right, font=\tiny] at (2,2) {0};
\node [font=\small, right] at (4,2) {$\theta_1$};
\node [font=\small, right] at (2,4) { $\theta_2$};
\node [font=\small,left] at (0,2) {$\theta_1^{-1}$};
\node [font=\small, right] at (2,0) {$\theta_2^{-1}$};
\draw[red, thick] (2,2) -- (4,4);
\node [red, font=\small, right] at (4,4) {$\theta_m$ };
\draw [-latex, blue] (3,1) arc (-45:310:1.41);
\node [blue, below left, font=\tiny] at (1,1) {$\gamma$};
\end{tikzpicture}
\caption{$\Theta_{{\gamma}, \mathfrak{D}^{\infty}}=\theta_{1}\circ{\theta_{m}}\circ\theta_{2}\circ\theta_{1}^{-1}\circ\theta_{2}^{-1}$}
\label{fig:Dloop}
\end{figure}
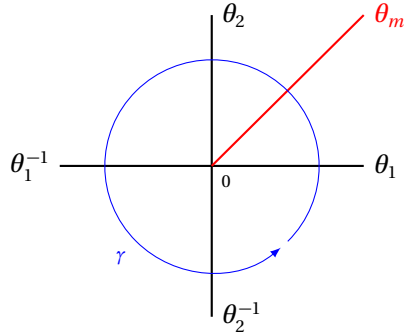

\begin{definition}[Consistent scattering diagram]A scattering diagram $\mathfrak{D}$ is \textit{consistent} if for any closed path $\gamma$ intersecting $\mathfrak{D}$ generically, $\Theta_{\gamma,\mathfrak{D}}=\id_{\mathbb{V}}$. 
\end{definition}
The following theorem by Kontsevich and Soibelman is an existence (and uniqueness) result of consistent scattering diagrams: 
\begin{theorem}[\cite{KS}]\label{thm:KS}
Let $\mathfrak{D}$ be a scattering diagram with two non parallel walls. There exists a unique minimal scattering diagram\footnote{The diagram is minimal meaning that we do not consider rays with trivial automorphisms and no rays with the same support.} $\mathfrak{D}^{\infty}\supseteq\mathfrak{D}$ such that $\mathfrak{D}^{\infty}\setminus\mathfrak{D}$ consists only of rays, and it is {consistent}.
\end{theorem}
\begin{proof}
Let $\mathfrak{D}_0=\mathfrak{D}$ and let $\mathfrak{D}_k$ be the order $k$ scattering diagram obtained from $\mathfrak{D}_{k-1}$ adding only rays emanating from $p\in\text{Sing}(\mathfrak{D}_{k-1})$ such that for any generic loop around $p$, $\Theta_{\gamma,\mathfrak{D}_k}=\id$ mod $t^{k+1}$.  
The proof goes by induction on $k$, and we need to prove that there exists a scattering diagram $\mathfrak{D}_k$. Assume $\mathfrak{D}_{k-1}$ satisfies the inductive assumption, then $\mathfrak{D}_k$ is constructed as follows: let $p\in\text{Sing}(\mathfrak{D}_{k-1})$ and compute $\Theta_{\gamma_p,\mathfrak{D}_{k-1}}$ for a generic loop $\gamma_p$ around $p$. By inductive assumption $\Theta_{\gamma_p,\mathfrak{D}_{k-1}}=\exp(\sum_{i=1}^s a_i(p)z^{m_i}\partial_{n_i})$ mod $t^{k+1}$, for some $a_i\in\C[\![t]\!]/(t)^k$, $m_i\in\Lambda$ and $\langle m_i,n_i\rangle=0$. Hence \[\mathfrak{D}_{k}=\mathfrak{D}_{k-1}\cup\bigcup_{p\in\text{Sing}(\mathfrak{D}_{k-1})}\left\lbrace \left(\mathfrak{d}_i=p+m_{i}\R_{\geq 0}, \theta_{i}=\exp(-a_i(p)z^{m_i}\partial_{n_i})\right)\vert i=1,...,s\right\rbrace.\] Notice that we have to chose the opposite sign for the automorphisms of $\mathfrak{D}_{k}$ in order to cancel the contribution from $\Theta_{\gamma_p,\mathfrak{D}_{k-1}}$.         
\end{proof}
There are examples in which the final configuration of rays (for which the diagram is consistent) it is known explicitly. For instance let $\mathfrak{D}$ as in Figure \ref{fig:D1} with automorphisms $\theta_1,\theta_2\in\mathbb{V}$ 
\begin{align*}
\theta_1\colon & x\to x& \theta_2\colon & x\to x/(1+ty) \\
 &y\to y(1+tx) &  &y\to y
\end{align*}
Then the consistent scattering diagram $\mathfrak{D}^{\infty}$ consists of one more ray $\mathfrak{d}=(1,1)$ with automorphism $\theta$ such that
\begin{align*}
\theta_m\colon& x\to x/(1+t^2xy)\\
& y\to y(1+t^2xy)
\end{align*}
as it is represented in Figure \ref{fig:Dloop}. 
\subsection{Extension of the tropical vertex group}\label{sec:extended tropical vertex}

In \cite{MCscattering} the authors prove that some elements of the gauge group acting on $\Omega^{0,1}(\check{X},T^{1,0}\check{X})$ can be represented as elements of the tropical vertex group $\mathbb{V}$. Here we are going to define an extension of the Lie algebra $\mathfrak{h}$, which will be related with the infinitesimal generators of the gauge group acting on $\Omega^{0,1}(\check{X},\End E\oplus T^{1,0}\check{X})$.
Let $\mathfrak{gl}(r,\mathbb{C})$ be the Lie algebra of the Lie group $\GL(r,\mathbb{C})$, then we define 
\begin{equation}
\tilde{\mathfrak{h}}\defeq\bigoplus_{m\in \Lambda\smallsetminus\lbrace 0\rbrace}\C z^{m}\cdot\left(\mathsf{m}_t\mathfrak{gl}(r,\mathbb{C})\oplus \left(\mathsf{m}_t\otimes m^{\perp}\right)\right). 
\end{equation}
\begin{lemma}\label{lem:dgLa}
$\Big(\tilde{\mathfrak{h}} , [\cdot ,\cdot]_{\sim}\Big)$ is a Lie algebra, where the bracket $[\cdot,\cdot]_{\sim}$ is defined by: 
\begin{equation}\label{eq:Liebracket}
[(A,\partial_n)z^m , (A',\partial_{n'})z^{m'} ]_{\sim}\defeq([A,A']_{\mathfrak{gl}}z^{m+m'}+A'\langle m',n\rangle z^{m+m'}- A\langle m,n'\rangle z^{m+m'}, [z^m\partial_n,z^{m'}\partial_{n'}]_{\mathfrak{h}} ).
\end{equation}
\end{lemma}

The definition of the Lie bracket $[\cdot, \cdot]_\sim$ is closely related with the Lie bracket of $\KS(\check{X}, E)$ and we will explain it below, in \eqref{subsec:relation}. 
\begin{definition}
The extended tropical vertex group $\tilde{\mathbb{V}}$ is the sub-group of $\GL(r,\C)\times\Aut_{\C[\![ t ]\!]}\big(\mathbb{C}[\Lambda]{\otimes}_{\mathbb{C}}\C[\![ t ]\!]\big)$, such that $\tilde{\mathbb{V}}\defeq\exp(\tilde{\mathfrak{h}})$. The product on $\tilde{\mathbb{V}}$ is defined by the BCH formula. 
\end{definition}

\begin{proof}
First of all the the bracket is antisymmetric: 
\begin{equation*}
\begin{split}
[(A,\partial_n)z^m , (A',\partial_{n'})z^{m'} ]_{\sim}&=([A,A']_{\mathfrak{gl}}z^{m+m'}+A'\langle m',n\rangle z^{m+m'}- A\langle m,n'\rangle z^{m+m'}, [z^m\partial_n,z^{m'}\partial_{n'}]_{\mathfrak{h}} )\\
&=(-[A',A]_{\mathfrak{gl}}z^{m+m'}+A'\langle m',n\rangle z^{m+m'}- A\langle m,n'\rangle z^{m+m'}, -[z^{m'}\partial_{n'},z^{m}\partial_{n}]_{\mathfrak{h}} )\\
&=-([A',A]_{\mathfrak{gl}}z^{m+m'}+A\langle m,n'\rangle z^{m+m'}- A'\langle m',n\rangle z^{m+m'}, [z^{m'}\partial_{n'},z^{m}\partial_{n}]_{\mathfrak{h}} ).
\end{split}
\end{equation*}
Moreover the Jacobi identity is satisfied: 
\begin{align*}
&\left[[(A_1,\partial_{n_1})z^{m_1} , (A_2,\partial_{n_2})z^{m_2} ]_{\sim},(A_3,\partial_{n_3})z^{m_3} \right]_{\sim}=\Big[\big([A_1,A_2]_{\mathfrak{gl}}+A_2\langle m_2,n_1\rangle-A_1\langle m_1,n_2\rangle,\\
&\qquad\qquad\qquad\qquad\qquad\qquad\qquad\qquad\qquad\qquad\qquad\partial_{\langle m_2,n_1\rangle n_2-\langle m_1,n_2\rangle n_1}\big)z^{m_1+m_2}, (A_3,\partial_{n_3})z^{m_3}\Big]_{\sim}\\
&=\Bigg(\big[([A_1,A_2]_{\mathfrak{gl}}+A_2\langle m_2, n_1\rangle -A_1\langle m_1,n_2\rangle),A_3\big]_{\mathfrak{gl}}+A_3\langle m_2,n_1\rangle\langle m_3,n_2\rangle -A_3\langle m_1,n_2\rangle\langle m_3,n_1\rangle +\\
&\quad-\big([A_1,A_2]+A_2\langle m_2, n_1\rangle -A_1\langle m_1, n_2\rangle\big)\langle m_1+ m_2,n_3\rangle, (m_1+m_2+m_3)^{\perp})\Bigg)z^{m_1+m_2+m_3}\\
&=\Bigg([[A_1,A_2]_{\mathfrak{gl}}, A_3]_{\mathfrak{gl}}+[A_2,A_3]_{\mathfrak{gl}}\langle m_2,n_1\rangle-\langle m_1, n_2\rangle [A_1,A_3]_{\mathfrak{gl}}-[A_1,A_2]_{\mathfrak{gl}}\langle m_1+m_2,n_3\rangle+\\
&\quad+A_3\langle m_2,n_1\rangle\langle m_3,n_2\rangle-A_3\langle m_1,n_2\rangle\langle m_3,n_1\rangle-A_2\langle m_2,n_1\rangle\langle m_1+m_2,n_3\rangle+A_1\langle m_1,n_2\rangle\langle m_1+m_2,n_3\rangle,\\
&\qquad\qquad\qquad\qquad\qquad\qquad\qquad\qquad\qquad\qquad\qquad (m_1+m_2+m_3)^{\perp}\Bigg)z^{m_1+m_2+m_3}
\end{align*} 
Then by cyclic permutation we compute also the other terms:
\begin{align*}
&\left[\left[(A_2,\partial_{n_2})z^{m_2}, (A_3,\partial_{n_3})z^{m_3} \right]_{\sim},(A_1,\partial_{n_1})z^{m_1}\right]_{\sim}=\Bigg(\big[[A_2,A_3]_{\mathfrak{gl}}, A_1]\big]+[A_3,A_1]_{\mathfrak{gl}}\langle m_3,n_2\rangle+\\
&\qquad-\langle m_2,n_3\rangle [A_2,A_1]_{\mathfrak{gl}}-[A_2,A_3]_{\mathfrak{gl}}\langle m_3+m_2,n_1\rangle+A_1\langle m_3,n_2\rangle\langle m_1,n_3\rangle+\\
&\qquad-A_1\langle m_2,n_3\rangle\langle m_1,n_2\rangle-A_3\langle m_3,n_2\rangle\langle m_2+m_3, n_1\rangle+A_2\langle m_2,n_3\rangle\langle m_2+m_3, n_1\rangle,\\
&\qquad\qquad\qquad\qquad\qquad\qquad\qquad\qquad\qquad\qquad\qquad (m_1+m_2+m_3)^{\perp}\Bigg)z^{m_1+m_2+m_3}
\end{align*} 
\begin{align*}
&\left[\left[(A_1,\partial_{n_3})z^{m_3}, (A_1,\partial_{n_1})z^{m_1} \right]_{\sim},(A_2,\partial_{n_2})z^{m_2} \right]_{\sim}=\Big(\big[[A_3,A_1]_{\mathfrak{gl}}, A_2]\big]+[A_1,A_2]\langle m_1,n_3\rangle\\
&\qquad -\langle m_3,n_1\rangle [A_3,A_2]_{\mathfrak{gl}}-[A_3,A_1]_{\mathfrak{gl}}\langle m_1+m_3, n_2\rangle+A_2\langle m_1,n_3\rangle\langle m_2,n_1\rangle+\\
&\qquad-A_2\langle m_3,n_1\rangle\langle m_2,n_3\rangle-A_1\langle m_1,n_3\rangle\langle m_3+m_1,n_2\rangle+A_3\langle m_3, n_1\rangle\langle m_3+m_1,n_2\rangle ,\\
&\qquad\qquad\qquad\qquad\qquad\qquad\qquad\qquad\qquad\qquad\qquad(m_1+m_2+m_3)^{\perp}\Big)z^{m_1+m_2+m_3}
\end{align*}
Since Jacobi identity holds for $[\cdot,\cdot]_{\mathfrak{gl}}$ and $[\cdot,\cdot]_{\mathfrak{h}}$, we are left to check that the remaining terms sum to zero. Indeed the coefficient of $[A_2,A_3]_{\mathfrak{gl}}$ is $\langle n_1,m_2\rangle-\langle n_1,m_2+m_3\rangle-\langle n_1,m_3\rangle$, and it is zero. By permuting the indexes, the same hold true for the coefficients in front of the other bracket $[A_1,A_3]_{\mathfrak{gl}}$ and $[A_2,A_1]_{\mathfrak{gl}}$. In addition the coefficient of $A_3$ is $\langle m_2,n_1\rangle\langle m_3,n_2\rangle-\langle m_1,n_2\rangle\langle m_3,n_1\rangle-\langle m_3,n_2\rangle\langle m_2,n_1\rangle-\langle m_3,n_2\rangle\langle m_3,n_1\rangle+ \langle m_3,n_1\rangle\langle m_1,n_2\rangle+\langle m_3,n_1\rangle\langle m_3,n_2\rangle$ and it is zero. By permuting the indexes the same holds true for the coefficient in front of $A_1$ and $A_2$.  
\end{proof}

\begin{notation}
Let $\mathfrak{d}_i=\xi_i +m_i\R_{\geq 0}$ and let $\overrightarrow{f}_{i}$ be the function 
\begin{align*}
\overrightarrow{f}_{i}&\defeq\left(1+A_it_iz^{m_i}, f_i\right)\\
f_i&=1+c_it_iz^{m_i}
\end{align*}
such that $\log\overrightarrow{f}_{i}\defeq\left(\log(1+A_it_iz^{m_i}),\log f_i\partial_{n_i}\right)\in\tilde{\mathfrak{h}}$, where $n_i$ is the unique primitive vector in $\Lambda^*$ orthogonal to $m_i$ and positively oriented. 
\end{notation}
We define scattering diagrams in the extended tropical vertex group $\tilde{\mathbb{V}}$ by replacing the last assumption of Definition \ref{def:scattering} with $\overrightarrow{f}_i$ such that \[\log\overrightarrow{f}_{i}=\left(\log(1+A_it_iz^{m_i}),\log f_i\partial_{n_i}\right)\in\tilde{\mathfrak{h}}.\]
In Chapter \ref{cha:enumertaive geom} we will introduce other definitions about scattering diagrams which generalizes that of \cite{GPS} to scattering diagrams defined in the extended tropical vertex group.  


\chapter{Holomorphic pairs and scattering diagrams}\label{sec:MC and scattering}

This section is devoted to study the relation between scattering diagrams in the extended tropical vertex group and the asymptotic behaviour of the solutions of the Maurer--Cartan equation which governs deformations of holomorphic pairs. 

We briefly highlight the main steps of the construction, which follows closely that of \cite{MCscattering}, adapting it to pairs $(\check{X}, E)$.   

\textbf{Step 1} We first introduce a \textit{symplectic DGLA} as the Fourier-type transform of the Kodaira-Spencer DGLA $\KS(\check{X},E)$ which governs deformation of the pair $(\check{X}, E)$. Although the two DGLAs are isomorphic, we find that working on the symplectic side makes the results more transparent. In particular we define the Lie algebra $\tilde{\mathfrak{h}}$ as a subalgebra, modulo terms which vanish as $\hbar \to 0$, of the Lie algebra of infinitesimal gauge transformations on the symplectic side. 

\textbf{Step 2.a} Starting from the data of a wall in a scattering diagram, namely from the automorphism $\theta$ attached to a line $\mathfrak{d}$, we construct a solution $\Pi$ supported along the wall, i.e. such that there exists a unique normalised infinitesimal gauge transformation $\varphi$ which takes the trivial solution to $\Pi$ and has asymptotic behaviour with leading ordered term given by $\log(\theta)$ (see Proposition \ref{prop:asymp1}). The gauge-fixing condition $\varphi$ is given by choosing a suitable homotopy operator $H$.    

\textbf{Step 2.b} Let $\mathfrak{D}=\lbrace\mathsf{w}_1, \mathsf{w}_2\rbrace$ be an initial scattering diagram with two non-parallel walls. By \textbf{Step 2.a.}, there are Maurer-Cartan solutions $\Pi_1$, $\Pi_2$, which are respectively supported along the walls $\mathsf{w}_1$, $\mathsf{w}_2$. Using Kuranishi's method we construct a solution $\Phi$ taking as input $\Pi_1 +\Pi_2$, of the form $\Phi=\Pi_1+\Pi_2+\Xi$, where $\Xi$ is a correction term. In particular $\Xi$ is computed using a different homotopy operator $\mathbf{H}$. 

\textbf{Step 2.c } By using \textit{labeled ribbon trees} we write $\Phi$ as a sum of contributions $\Phi_a$ over ${a\in\left(\Z^2_{\geq 0}\right)_{\text{prim}}}$, each of which turns out to be independently a Maurer-Cartan equation (Lemma \ref{lem:Phi_aMC}). Moreover we show that each $\Phi_a$ is supported on a ray of rational slope, meaning that for every $a$, there is a unique normalised infinitesimal gauge transformation $\varphi_a$ whose asymptotic behaviour is an element of our Lie algebra $\tilde{\mathfrak{h}}$ (Theorem \ref{thm:asymptotic_gauge}). The transformations $\varphi_a$ allow us to define the scattering diagram $\mathfrak{D}^{\infty}$ (Definition \ref{def:D_infty}) from the solution $\Phi$.

\textbf{Step 3 } The consistency of the scattering diagram $\mathfrak{D}^{\infty}$ is proved by a monodromy argument.    

Note that in fact the results of \cite{MCscattering} have already been extended to a large class of DGLAs (see \cite{CLMaY19}). For our purposes however we need a more ad hoc study of a specific differential-geometric realization of $\KS(\check{X},E)$: for example, there is a background Hermitian metric on $E$ which needs to be chosen carefully.\\  

\section{Symplectic DGLA}

In order to construct scattering diagrams from deformations of holomorphic pairs, it is more convenient to work with a suitable Fourier transform $\FF$ of the DGLA $\KS(\check{X},E)$. Following \cite{MCv1} we start with the definition of $\FF(\KS(\check{X}, E))$ (see also Section 3.2.1 of \cite{Ma}). Let $\mathcal{L}$ be the space of fibre-wise homotopy classes of loops with respect to the fibration $p\colon X\to M$ and the zero section $s\colon M\to X$ 
\[\mathcal{L}=\bigsqcup_{x\in M}\pi_1(p^{-1}(x), s(x)).\]
Define a map $\ev\colon \mathcal{L}\to X$, which maps a homotopy class $[\gamma]\in\mathcal{L}$ to $\gamma(0)\in X$ and define $\pr\colon \mathcal{L}\to M$ the projection, such that the following diagram commutes:
\[
\begin{tikzcd}
\mathcal{L}\arrow{rr}{\ev}\arrow[swap]{rd}{{\pr}}& & X\arrow{ld}{p}\\
& M&
\end{tikzcd}
\] 
In particular $\pr$ is a local diffeomorphism and on a contractible open subset $U\subset M$ it induces an isomorphism $\Omega^\bullet(U,TM)\cong\Omega^\bullet(U_{\textbf{m}}, T\mathcal{L})$, where $U_{\textbf{m}}\defeq \lbrace\textbf{m}\rbrace\times U\in \pr^{-1}(U)$, $\textbf{m}\in\Lambda$. In addition, there is a one-to-one correspondence between $\Omega^0(U,T^{1,0}\check{X})$ and $\Omega^0(U,TM)$, 
\[
\frac{\partial}{\partial z_j}\longleftrightarrow\frac{\hbar}{4\pi}\frac{\partial}{\partial x_j}
\] 
which leads us to the following definition:
\begin{definition} 
The Fourier transform is a map $\FF\colon\Omega^{0,k}(\check{X},T^{1,0}\check{X})\ra\Omega^k(\mathcal{L},T\mathcal{L})$, such that  
\begin{equation}\label{def:F}
\big(\FF(\varphi)\big)_{\textbf{m}}(x)\defeq\bigg(\frac{4\pi}{\hbar}\bigg)^{|I|-1}\int_{\check{p}^{-1}(x)}\varphi^{I}_j(x,\check{y})e^{-2\pi i(\textbf{m},\check{y})}d\check{y}\,\, dx_{I}\otimes \frac{\partial}{\partial x_j},
\end{equation}
where $\textbf{m}\in\Lambda$ represents an affine loop in the fibre $p^{-1} (x)$ with tangent vector $\sum_{j=1}^2m_j\frac{\partial}{\partial y_j}$ and $\varphi$ is locally given by $\varphi=\varphi^{I}_j(x,\check{y})d\bar{z}_I\otimes\frac{\partial}{\partial z_j}$, $|I|=k$.
\end{definition}
The inverse Fourier transform is then defined by the following formula, providing the coefficients have enough regularity: 
\begin{equation}
\FF^{-1}\big(\alpha\big)(x,\check{y})=\bigg(\frac{4\pi}{\hbar}\bigg)^{-|I|+1}\sum_{\textbf{m}\in\Lambda}\alpha_{j,\textbf{m}}^I e^{2\pi i(\textbf{m},\check{y})} d\bar{z}_I\otimes \frac{\partial}{\partial z_j}
\end{equation}
where $\alpha_{j,\textbf{m}}^I(x)dx_I\otimes\frac{\partial}{\partial x^j}\in\Omega^k(U_{\textbf{m}},T\mathcal{L})$ is the $\textbf{m}$-th Fourier coefficient of $\alpha\in\Omega^k(\mathcal{L},T\mathcal{L})$ and $|I|=k$.


The Fourier transform can be extended to $\KS(\check{X},E)$ as a map
\[\FF\colon\Omega^{0,k}(\check{X},\End E\oplus T^{1,0}\check{X})\ra\Omega^{k}(\mathcal{L}, \End E\oplus T\mathcal{L})\] 
\begin{equation}
\begin{split}
\FF\left(\left(A^{I}d\bar{z}_I,\varphi^{I}_jd\bar{z}_I\otimes\frac{\partial}{\partial z_j}\right)\right)_{\textbf{m}}\defeq&\left(\frac{4\pi}{\hbar}\right)^{|I|}\Big(\int_{\check{p}^{-1}(x)}A^{I}(x,\check{y})e^{-2\pi i (\textbf{m},\check{y})}d\check{y} dx_I,\\
&\left(\frac{4\pi}{\hbar}\right)^{-1}\int_{\check{p}^{-1}(x)}\varphi^{I}_j(x,\check{y})e^{-2\pi i(\textbf{m},\check{y})}d\check{y} dx_{I}\otimes\frac{\partial}{\partial x_j}\Big)
\end{split}
\end{equation} 
where the first integral is meant on each matrix element of $A^I$.
\\

In order to define a DGLA isomorphic to $\KS(\check{X},E)$, we introduce the so called Witten differential $d_W$ and the Lie bracket $\lbrace\cdot, \cdot\rbrace_\sim$, acting on $\Omega^\bullet(\mathcal{L},\End E\oplus T\mathcal{L})$. It is enough for us to consider the case in which $E$ is holomorphically trivial $E=\O_{\check{X}}\oplus\O_{\check{X}}\oplus\cdots\oplus\O_{\check{X}}$ and the hermitian metric $h_E$ is diagonal, $h_E=\mathrm{diag}(e^{-\phi_1},\cdots, e^{-\phi_r)}$ and $\phi_j\in\Omega^0(\check{X})$ for $j=1,...,r$.  The differential $d_W$ is defined as follows:
\begin{equation*}
\begin{split}
d_W&\colon\Omega^k(\mathcal{L},\End E\oplus T\mathcal{L})\to\Omega^{k+1}(\mathcal{L},\End E\oplus T\mathcal{L})\\
d_W&\defeq\begin{pmatrix}
d_{W,E} & \hat{B}\\
0 & d_{W,\mathcal{L}}
\end{pmatrix}.
\end{split}
\end{equation*} 
In particular, $d_{W,E}$ is defined as:
\begin{equation}
\begin{split}
\left(d_{W,E}\left(A^Jdx_J\right)\right)_{\textbf{n}}&\defeq\FF(\bar{\partial}_E(\FF^{-1}(A^Jdx_J))_{\textbf{n}}\\
&=\FF\left(\bar{\partial}_E\left(\left(\frac{4\pi}{\hbar}\right)^{-|J|}\sum_{\textbf{m}\in\Lambda}e^{2\pi i(\textbf{m},\check{y})}A^J_{\textbf{m}}d\bar{z}_J\right)\right)_{\textbf{n}} \\
&=\left(\frac{4\pi}{\hbar}\right)^{-|J|}\FF\left(\sum_{\textbf{m}}e^{2\pi i(\textbf{m},\check{y})}\left(2\pi i\textbf{m}_kA^J_{\textbf{m}}+i\hbar\frac{\partial A^J_{\textbf{m}}}{\partial x_k}\right)d\bar{z}_k\wedge d\bar{z}_J\right)_{\textbf{n}} \\
&=\frac{4\pi}{\hbar}\int_{\check{p}^{-1}(x)}\left(\sum_{\textbf{m}}e^{2\pi i(\textbf{m},\check{y})}\left(2\pi i\textbf{m}_kA^J_{\textbf{m}}+i\hbar\frac{\partial A^J_{\textbf{m}}}{\partial x_k}\right)e^{-2\pi i(\textbf{n},\check{y})}\right)d\check{y} dx_k\wedge dx_J\\
&=\frac{4\pi}{\hbar}\left(2\pi i\textbf{n}_kA^J_{\textbf{n}}+i\hbar\frac{\partial A^J_{\textbf{n}}}{\partial x_k}\right)dx_k\wedge dx_J.
\end{split}
\end{equation} 
The operator $\hat{B}$ is then defined by
\begin{equation}
\begin{split}
\left(\hat{B}(\psi_{j}^I dx_I\otimes\frac{\partial}{\partial x_j})\right)_{\textbf{n}}&\defeq\FF(\FF^{-1}(\psi_{j}^I dx_I\otimes\frac{\partial}{\partial x_j})\lrcorner F_E)\\
&=\bigg(\frac{4\pi}{\hbar}\bigg)^{1-|I|}\FF\Big(\sum_\textbf{m}\psi_{\textbf{m},j}^Ie^{2\pi i(\textbf{m},\check{y})}d\bar{z}_I\otimes\frac{\partial}{\partial z_j}\lrcorner F_{pq}(\phi)dz_p\wedge d\bar{z}_q)\Big)\\
&=\bigg(\frac{4\pi}{\hbar}\bigg)^{1-|I|}\bigg(\frac{4\pi}{\hbar}\bigg)^{1+|I|}\int_{\check{p}^{-1}(x)}\sum_\textbf{m}\psi_{\textbf{m},j}^Ie^{2\pi i (\textbf{m}-\textbf{n},\check{y})}F_{jq}(\phi)d\check{y} dx_I\wedge dx_q 
\end{split}
\end{equation}
where $F_{pq}(\phi)$ is the curvature matrix. Then the $d_{W,\mathcal{L}}$ is defined by:
\begin{equation}
\left(d_{W,\mathcal{L}}(\psi_{j}^I dx_I\otimes\frac{\partial}{\partial x_j})\right)_{\textbf{n}}\defeq e^{-2\pi\hbar^{-1}({\textbf{n}},x)}d\left(\psi_{j}^I dx_I\otimes\frac{\partial}{\partial x_j}e^{2\pi\hbar^{-1}({\textbf{n}},x)}\right),
\end{equation}
where $d$ is the de Rham differential on the base $M$. Notice that by definition $d_W=\FF(\bar{\partial})\FF^{-1}$. 

Analogously we define the Lie bracket $\lbrace\cdot ,\cdot\rbrace_\sim\defeq\FF([\cdot,\cdot]_\sim)\FF^{-1}$. If we compute it explicitly in local coordinates we find:
\begin{align*}
\lbrace\cdot ,\cdot\rbrace_\sim\colon\Omega^p(\mathcal{L},\End E\oplus T\mathcal{L})\times\Omega^q(\mathcal{L},\End E\oplus T\mathcal{L})\to\Omega^{p+q}(\mathcal{L},\End E\oplus T\mathcal{L})\\
\lbrace(A,\varphi),(N,\psi)\rbrace_{\sim } = \big(\lbrace A,N\rbrace+\textbf{ad}(\varphi,N)-(-1)^{pq}\textbf{ad}(\psi,A), \lbrace\varphi, \psi\rbrace\big).
\end{align*}

In particular locally on $U_{\textbf{m}}$ we consider \[(A,\varphi)=\big(A_{\textbf{m}}^Idx_I,\varphi_{j,\textbf{m}}^{I}dx_{I}\otimes\frac{\partial}{\partial x_j}\big)\in\Omega^p(U_{\textbf{m}},\End E\oplus T\mathcal{L})\] and on $U_{\textbf{m}'}$ we consider \[(N,\psi)=(N_{\textbf{m}'}^Jdx_J,\psi_{J,\textbf{m}'}^{J}dx_{J}\otimes\frac{\partial}{\partial x\_{\textbf{m}}})\in\Omega^q(U_{\textbf{m}'},\End E\oplus T\mathcal{L})\] then   
\begin{equation}
\lbrace A,N\rbrace_{\textbf{n}}\defeq \sum_{\textbf{m}+\textbf{m'}=\textbf{n}}[A_{\textbf{m}}^I,N_{\textbf{m'}}^J]dx_I\wedge dx_J
\end{equation}
where the sum over $\textbf{m}+\textbf{m'}=\textbf{n}$ makes sense under the assumption of enough regularity of the coefficients. The operator $\textbf{ad}\colon\Omega^{p}(\mathcal{L},T\mathcal{L})\times\Omega^q(\mathcal{L},\End E)\to \Omega^{p+q}(\mathcal{L},\End E)$ is explicitly 
\begin{multline*}
\left(\textbf{ad}\left(\varphi,N\right)\right)_{\textbf{n}}\defeq\frac{4\pi}{\hbar}\Big(\int_{\check{p}^{-1}(x)}\Big(\sum_{\textbf{m}+\textbf{m}'=\textbf{n}}\varphi_{I,\textbf{m}}^j \Big(\frac{\partial N_{\textbf{m}'}^J}{\partial z_j}+2\pi i{\textbf{m}'}_j+A_j(\phi)N_{\textbf{m}'}^{J}\Big)\cdot\\
\cdot e^{2\pi i(\textbf{m}+\textbf{m}'-\textbf{n},\check{y})}\Big)d\check{y}\Big) dx_{I}\wedge dx_{J}
\end{multline*}
where $A_j(\phi)dz_j$ is the connection $\nabla^E$ one-form matrix. Finally the Lie bracket $\lbrace\psi,\phi\rbrace_\sim$ is 
\begin{equation}
\begin{split}
\lbrace\varphi, \psi\rbrace_{\textbf{n}}&\defeq\bigg(\sum_{\textbf{m}'+\textbf{m}=\textbf{n}}e^{-2\pi\hbar^{-1}(\textbf{n},x)}\Big(\varphi_{j,\textbf{m}}^I e^{2\pi (\textbf{m},x)}\nabla_{\frac{\partial}{\partial x_j}}\big(e^{2\pi (\textbf{m}',x)}\psi_{k,\textbf{m}'}^J\frac{\partial}{\partial x_k}\big)\\
&-(-1)^{pq}\psi_{k,\textbf{m}'}^Je^{2\pi\hbar^{-1} (\textbf{m}',x)}\nabla_{\frac{\partial}{\partial x_k}}\big(e^{2\pi\hbar^{-1}(\textbf{m},x)}\varphi_{j,\textbf{m}}^I\frac{\partial}{\partial x_j}\big)\Big)\bigg)dx_I\wedge dx_J
\end{split}
\end{equation}
where $\nabla$ is the flat connection on $M$.

\begin{definition}\label{def:mirror DGLA} 
The symplectic DGLA is defined as follows:
\begin{equation*}
G\defeq\left(\Omega^\bullet(\mathcal{L},\End E\oplus T\mathcal{L}), d_W, \lbrace\cdot, \cdot\rbrace_\sim\right)
\end{equation*}
and it is isomorphic to $\KS(\check{X},E)$ via $\FF$. 
\end{definition} 
As we mention above, the gauge group on the symplectic side $\Omega^0(\mathcal{L},\End E\oplus T\mathcal{L})$ is related with the extended Lie algebra $\tilde{\mathfrak{h}}$. However to figure it out, some more work has to be done, as we prove in the following subsection.  
\subsection{Relation with the Lie algebra $\tilde{\mathfrak{h}}$}\label{subsec:relation}

Let $Aff_M^{\Z}$ be the sheaf of affine linear transformations over $M$ defined for any open affine subset $U\subset M$ by $f_{{m}}(x)=({m},x)+b\in Aff_M^{\Z}(U)$ where $x\in U$, $m\in\Lambda$ and $b\in\R$. Since there is an embedding of $Aff_M^{\Z}(U)$ into $\O_{\check{X}}(\check{p}^{-1}(U))$ which maps $f_{{m}}(x)=({m},x)+b\in Aff_M^{\Z}(U)$ to $e^{2\pi i({m}, z)+2\pi i b}\in\O_{\check{X}}(\check{p}^{-1}(U))$, we define $\O_{aff}$ the image sub-sheaf of $Aff_M^{\Z}$ in $\O_{\check{X}}$. Then consider the embedding of the dual lattice $\Lambda^*\into T^{1,0}\check{X}$ which maps
\[
n\to n^j\frac{\partial}{\partial z_j}=:\check{\partial}_n.\] 
It follows that the Fourier transform $\FF$ maps \[\left(Ne^{2\pi i((m,z)+b)}, e^{2\pi i((m,z)+b)}\check{\partial}_n\right)\in\O_{aff}\left(\check{p}^{-1}(U), \mathfrak{gl}(r,\C)\oplus T^{1,0}\check{X} \right)\] to \[\left(Ne^{2\pi ib}\mathfrak{w}^m, \frac{\hbar}{4\pi}e^{2\pi i b}\mathfrak{w}^mn^j\frac{\partial}{\partial x_j}\right)\in\mathfrak{w}^m\cdot \underline{\C}\left(U_m,\mathfrak{gl}(r,\C)\oplus T\mathcal{L}\right)\] where $\mathfrak{w}^m\defeq\FF(e^{2\pi i({m}, z)})$, i.e. on $U_k$
\[
\mathfrak{w}^m=\begin{cases} 
e^{2\pi\hbar^{-1} (m,x)} & \text{if $k=m$}\\
0 & \text{if $k\neq m$}
\end{cases}
\] and we define \[\partial_n\defeq \frac{\hbar}{4\pi}n^j\frac{\partial}{\partial x_j}.\]
Let $\mathcal{G}$ be the sheaf over $M$ defined as follows: for any open subset $U\subset M$
\[
\mathcal{G}(U)\defeq\bigoplus_{m\in\Lambda\setminus 0}\mathfrak{w}^m\cdot\underline{\C}(U,\mathfrak{gl}(r,\C)\oplus TM).
\]  
In particular $\tilde{\mathfrak{h}}$ is a subspace of $\mathcal{G}(U)$ once we identify $z^m$ with $\mathfrak{w}^m$ and $m^\perp$ with $\partial_{m^\perp}$. In order to show how the Lie bracket on $\tilde{\mathfrak{h}}$ is defined, we need to make another assumption on the metric: assume that the metric $h_E$ is constant along the fibres of $\check{X}$, i.e. in an open subset $U\subset M$ $\phi_j=\phi_j(x_1,x_2)$, $j=1,\cdots, r$. Hence, the Chern connection becomes $\nabla^E=d+\hbar A_j(\phi)dz_j$ while the curvature becomes $F_E=\hbar^2 F_{jk}(\phi)dz_j\wedge dz_k$. 
We now show $\mathcal{G}(U)$ is a Lie sub-algebra of $\left(\Omega^0(\pr^{-1}(U)),\End E\oplus T\mathcal{L}), \lbrace\cdot,\cdot\rbrace_\sim\right)\subset G(U)$ and we compute the Lie bracket $\lbrace\cdot,\cdot\rbrace_{\sim}$ explicitly on functions of $\mathcal{G}(U)$.

\begin{multline*}
\lbrace \left( A\mathfrak{w}^m, \mathfrak{w}^m\partial_n\right), \left(N\mathfrak{w}^{m'}, \mathfrak{w}^{m'}\partial_{n'}\right)\rbrace_\sim=\Big([A,N]\mathfrak{w}^{m+m'}+\textbf{ad}(\mathfrak{w}^m\partial_n,N\mathfrak{w}^{m'})-\textbf{ad}(\mathfrak{w}^{m'}\partial_{n'},A\mathfrak{w}^{m}),\\
\lbrace\mathfrak{w}^m\partial_n,\mathfrak{w}^{m'}\partial_{n'}\rbrace\Big)
\end{multline*} 
\begin{equation}
\begin{split}
\textbf{ad}(\mathfrak{w}^m\partial_n,N\mathfrak{w}^{m'})_s&=\frac{4\pi}{\hbar}\sum_{k+k'=s}\mathfrak{w}^m\frac{\hbar}{4\pi} n^j\left(\frac{\partial N\mathfrak{w}^{m'}}{\partial x_j}+2\pi im'_j+\hbar A_j(\phi)N\right)\\
&=\mathfrak{w}^{m+m'}n^j\left( 2\pi im'_j+i\hbar\frac{\partial\phi}{\partial x_j}N\right)\\
&=\mathfrak{w}^{m+m'}\left(2\pi i\langle m',n\rangle+i\hbar n^jA_j(\phi)N\right)
\end{split}
\end{equation}
where in the second step we use the fact that $\mathfrak{w}^m$ is not zero only on $U_m$, and in the last step we use the pairing of $\Lambda$ and $\Lambda^*$ given by $\langle m,n'\rangle =\sum_jm_jn'^j $. Thus
\begin{equation*}
\begin{split}
\lbrace \left( A\mathfrak{w}^m, \mathfrak{w}^m\partial_n\right), \left(N\mathfrak{w}^{m'}, \mathfrak{w}^{m'}\partial_{n'}\right)\rbrace_\sim &=([A,N]\mathfrak{w}^{m+m'}+N(2\pi im'_j)n^j\mathfrak{w}^{m+m'}+\\
&+\hbar NA_j(\phi)n^j\mathfrak{w}^{m+m'}-A(2\pi im_j){n'}^j \mathfrak{w}^{m+m'}+\\
&-\hbar A A_j(\phi){n'}^j\mathfrak{w}^{m+m'},\lbrace\mathfrak{w}^m\partial_n,\mathfrak{w}^{m'}\partial_{n'}\rbrace)
\end{split}
\end{equation*} 
Taking the limit as $\hbar\to 0$, the Lie bracket $\lbrace\left( A\mathfrak{w}^m, \mathfrak{w}^m\partial_n\right), \left(N\mathfrak{w}^{m'}, \mathfrak{w}^{m'}\partial_{n'}\right)\rbrace_\sim$ converges to
\begin{equation*}
\left([A,N]\mathfrak{w}^{m+m'}+2\pi i\langle m',n\rangle N\mathfrak{w}^{m+m'}-2\pi i\langle m,n'\rangle A\mathfrak{w}^{m+m'}, \lbrace\mathfrak{w}^m\partial_n,\mathfrak{w}^{m'}\partial_{n'}\rbrace\right)
\end{equation*} 
and we finally recover the definition of the Lie bracket of $[\cdot,\cdot]_{\tilde{\mathfrak{h}}}$ \eqref{eq:Liebracket}, up to a factor of $2\pi i$. Hence $(\tilde{\mathfrak{h}},[\cdot,\cdot]_\sim)$ is the \textit{asymptotic} subalgebra of $\left(\Omega^0(\pr^{-1}(U)),\End E\oplus T\mathcal{L}), \lbrace\cdot,\cdot\rbrace_\sim\right)$.

\section{Deformations associated to a single wall diagram}\label{sec:single wall}
In this section we are going to construct a solution of the Maurer-Cartan equation from the data of a single wall. We work locally on a contractible, open affine subset $U\subset M$.

Let $(m,\mathfrak{d}_m,\theta_m)$ be a wall and assume $\log(\theta_m)=\sum_{j,k}\big(A_{jk}t^j\mathfrak{w}^{km}, a_{jk}t^j\mathfrak{w}^{km}\partial_n\big)$, where $A_{jk}\in{\mathfrak{gl}(r,\C)}$ and $a_{jk}\in\C$, for every $j,k$. 

\begin{notation}
We need to introduce a suitable set of local coordinates on $U$, namely $(u_m,u_{m,\perp})$, where $u_m$ is the coordinate in the direction of $\mathfrak{d}_m$, while $u_{m^\perp}$ is normal to $\mathfrak{d}_m$, according with the orientation of $U$. We further define $H_{m,+}$ and $H_{m,-}$ to be the half planes in which $\mathfrak{d}_m$ divides $U$, according with the orientation. 
\end{notation}

\begin{notation}
We will denote by the superscript \textit{CLM} the elements already introduced in \cite{MCscattering}. 
\end{notation}

\subsection{Ansatz for a wall}\label{sec:ansatz}

Let $\delta_m\defeq\frac{e^{-\frac{u_{m^\perp}^2}{\hbar}}}{\sqrt{\pi \hbar}}du_{m^\perp}$ be a normalized Gaussian one-form, which is \textit{supported} on $\mathfrak{d}_m$.
Then, let us define \[\Pi\defeq(\Pi_E,\Pi^{CLM})\]
where $\Pi_E=-\sum_{j,k}A_{jk}t^j\delta_m\mathfrak{w}^{km}$ and $\Pi^{CLM}=-\sum_{j,k\geq 1}a_{kj}\delta_mt^j\mathfrak{w}^{km}\partial_n$. 

From section 4 of \cite{MCscattering} we are going to recall the definition of \textit{generalized Sobolev space} suitably defined to compute the asymptotic behaviour of Gaussian k-forms like $\delta_m$ which depend on $\hbar$. 
Let $\Omega_{\hbar}^k(U)$ denote the set of $k$-forms on $U$ whose coefficients depend on the real positive parameter $\hbar$.  
\begin{definition}[Definition 4.15 \cite{MCscattering}]
\[\mathcal{W}_k^{-\infty}(U)\defeq\big\lbrace \alpha\in\Omega_{\hbar}^k(U)\vert \forall q\in U \exists V\subset U\,,q\in V\,\text{s.t.} \sup_{x\in V}\left|\nabla^j\alpha(x)\right|\leq C(j,V)e^{-\frac{c_V}{\hbar}}, C(j,V),\textcolor{britishracinggreen}{c_V}>0\big\rbrace\] is the set of exponential k-forms.
\end{definition}
\begin{definition}[Definition 4.16 \cite{MCscattering}]
\[\mathcal{W}_k^\infty(U)\defeq\big\lbrace \alpha\in\Omega_{\hbar}^k(U)\vert \forall q\in U \exists  V\subset U\,,q\in V\,\text{s.t.} \sup_{x\in V}\left|\nabla^j\alpha(x)\right|\leq C(j,V)\hbar^{-N_{j,V}},\, C(j,V),N_{j,V}\in\Z_{>0}\big\rbrace\] is the set of polynomially growing k-forms.
\end{definition}

\begin{definition}[Definition 4.19 \cite{MCscattering}]
Let $\mathfrak{d}_m$ be a ray in $U$. The set $\mathcal{W}_{\mathfrak{d}_m}^s(U)$ of $1$-forms $\alpha$ which have \textit{asymptotic support} of order $s\in\Z$ on $\mathfrak{d}_m$ is defined by the following conditions:
\begin{enumerate}
\item for every $q_*\in U\setminus \mathfrak{d}_m$, there is a neighbourhood $V\subset U\setminus \mathfrak{d}_m$ such that $\alpha|_V\in\mathcal{W}_{1}^{-\infty}(V)$;
\item for every $q_*\in \mathfrak{d}_m$ there exists a neighbourhood $q_*\in W\subset U$ where in local coordinates $u_q=(u_{q,m},u_{q,m^\perp})$ centred at $q_*$, $\alpha$ decomposes as
\[ 
\alpha=f(u_q,\hbar)du_{q,m^\perp}+\eta
\]
$\eta\in\mathcal{W}_{1}^{-\infty}(W)$ and for all $j\geq 0$ and for all $\beta\in\Z_{\geq 0}$
\begin{equation}\label{eq:stima W^s}
\int_{(0,u_{q,m^\perp})\in W}(u_{m^\perp})^\beta\big(\sup_{(u_{q,m},u_{m^\perp})\in W}\left|\nabla^j(f(u_q,\hbar))\right|\big)du_{m^\perp}\leq C(j,W,\beta)\hbar^{-\frac{j+s-\beta-1}{2}}
\end{equation}
for some positive constant $C(\beta, W,j)$.
\end{enumerate}
\end{definition}

\begin{remark}
A simpler way to figure out what is the space $\mathcal{W}_{\mathfrak{d}_m}^s(U)$, is to understand first the case of a $1$-form $\alpha\in\Omega_{\hbar}^1(U)$ which depends only on the coordinate $u_{m^\perp}$. Indeed $\alpha=\alpha(u_{m^\perp},\hbar)du_{m^\perp}$ has asymptotic support of order $s$ on a ray $\mathfrak{d}_m$ if for every $q\in \mathfrak{d}_m$, there exists a neighbourhood $q\in W\subset U$ such that
\[
\int_{(0,u_{q,m^\perp})\in W} u_{q,m^\perp}^{\beta} \left|\nabla^j \alpha(u_{q,m^\perp},\hbar)\right|du_{q,m^\perp} \leq C(W,\beta,j)\hbar^{-\frac{\beta+s-1-j}{2}}
\]
for every $\beta\in\Z_{\geq 0}$ and $j\geq 0$.

In particular for $\beta=0$ the estimate above reminds to the definition of the usual Sobolev spaces $L_1^j(U)$.  
\end{remark}

\begin{lemma}
The one-form $\delta_m$ defined above, has asymptotic support of order $1$ along $\mathfrak{d}_m$, i.e. $\delta_m\in\mathcal{W}_{\mathfrak{d}_m}^1(U)$. 
\end{lemma}
\begin{proof}
We claim that
\begin{equation}\label{eq:estimate}
\int_{-a}^b(u_{m^\perp})^\beta \nabla^j\left(\frac{e^{-\frac{u_{m^\perp}^2}{\hbar}}}{\sqrt{\hbar\pi}}\right)du_{m^\perp}\leq C(\beta,W,j)\hbar^{-\frac{j-\beta}{2}}
\end{equation}
for every $j\geq 0$, $\beta\in\Z_{\geq 0}$, for some $a,b> 0$.
This claim holds for $\beta=0=j$, indeed $\int_{-a}^b\frac{e^{-\frac{u_{m^\perp}^2}{\hbar}}}{\sqrt{\hbar\pi}}du_{m^\perp}$ is bounded by a constant $C=C(a,b)>0$. 

Then we prove the claim by induction on $\beta$, at $\beta=0$ it holds true by the previous computation. Assume that
\begin{equation}
\int_{-a}^b(u_{m\perp})^\beta \frac{e^{-\frac{u_{m^\perp}^2}{\hbar}}}{\sqrt{\hbar\pi}}du_{m^\perp} \leq C(\beta,a,b)\hbar^{\beta/2}
\end{equation} 
holds for $\beta$, then 
\begin{equation}
\begin{split}
\int_{-a}^b(u_{m^\perp})^{\beta+1}\frac{e^{-\frac{u_{m^\perp}^2}{\hbar}}}{\sqrt{\hbar\pi}}du_{m^\perp}&=-\frac{\hbar}{2}\int_{-a}^b(u_{m^\perp})^\beta\left(-2\frac{u_{m^\perp}}{\hbar}\frac{e^{-\frac{u_{m^\perp}^2}{\hbar}}}{\sqrt{\hbar\pi}}\right)du_{m^\perp}\\
&=-\frac{\hbar}{2}\left[(u_{m^\perp})^\beta\frac{e^{-\frac{u_{m^\perp}^2}{\hbar}}}{\sqrt{\hbar\pi}}\right]_{-a}^b +\beta\frac{\hbar}{2}\int_{-a}^b(u_{m^\perp})^{\beta-1}\frac{e^{-\frac{u_{m^\perp}^2}{\hbar}}}{\sqrt{\hbar\pi}}du_{m^\perp}\\
&\leq C(\beta, a,b)\hbar^{\frac{1}{2}}+\tilde{C}(\beta,a,b)\hbar^{1+\frac{\beta-1}{2}}\\
&\leq C(a,b,\beta)\hbar^{\frac{\beta+1}{2}}.
\end{split}
\end{equation}
Analogously let us prove the estimate by induction on $j$. At $j=0$ it holds true, and assume that 
\begin{equation}
\int_{-a}^b(u_{m^\perp})^\beta \nabla^j\left( \frac{e^{-\frac{u_{m^\perp}^2}{\hbar}}}{\sqrt{\hbar\pi}}\right)du_{m^\perp}\leq C(a,b,\beta,j)\hbar^{-\frac{j-\beta}{2}}
\end{equation} 
holds for $j$. Then at $j+1$ we have the following
\begin{equation*}
\begin{split}
\int_{-a}^b(u_{m^\perp})^\beta\nabla^{j+1}\left(\frac{e^{-\frac{u_{m^\perp}^2}{\hbar}}}{\sqrt{\hbar\pi}}\right)du_{m^\perp}&=\left[u_{m^\perp}^{\beta}\nabla^j\left(\frac{e^{-\frac{u_{m^\perp}^2}{\hbar}}}{\sqrt{\hbar\pi}}\right)\right]_{-a}^b-\beta\int_{N_V}(u_{m^\perp})^{\beta-1}\nabla^j\left(\frac{e^{-\frac{u_{m^\perp}^2}{\hbar}}}{\sqrt{\hbar\pi}}\right)du_{m^\perp}\\
&\leq \tilde{C}(\beta,a,b,j)\hbar^{-j-\frac{1}{2}}+ C(a,b,\beta, j)\hbar^{-\frac{j-\beta+1}{2}}\\
&\leq C(a,b,\beta,j)\hbar^{-\frac{j+1-\beta}{2}}
\end{split}
\end{equation*}
This ends the proof.  
\end{proof}

\begin{notation}
We say that a function $f(x,\hbar)$ on an open subset $U\times\R_{\geq 0}\subset M\times\R_{\geq 0}$ belongs to $O_{loc}(\hbar^l)$ if it is bounded by $C_K\hbar^l$ on every compact subset $K\subset U$, for some constant $C_K$ (independent on $\hbar$), $l\in\R$. 
\end{notation}
In order to deal with $0$-forms ``asymptotically supported on $U$'', we define the following space $\mathcal{W}_{0}^{s}$: 
\begin{definition}\label{def:I(P)}
A function $f(u_q,\hbar)\in\Omega^0_{\hbar}(U)$ belongs to $\mathcal{W}_{0}^{s}(U)$ if and only if for every $q_*\in U$ there is a neighbourhood $q_*\in W\subset U$ such that  
\[\sup_{q\in W}\left|\nabla^j f(u_q,\hbar)\right|\leq C(W,j)\hbar^{-\frac{s+j}{2}}\]
for every $j\geq 0$.  

\end{definition}
 
\begin{notation}
Let us denote by $\Omega_{\hbar}^k(U,TM)$ the set of $k$-forms valued in $TM$, which depends on the real parameter $\hbar$ and analogously we denote by $\Omega_{\hbar}^k(U,\End E)$ the set of $k$-forms valued in $\End E$ which also depend on $\hbar$. 
We say that $\alpha=\alpha_K(x,\hbar)dx^K\otimes\partial_n\in\Omega_{\hbar}^k(U,TM)$ belongs to $\mathcal{W}_P^s(U,TM)/ \mathcal{W}_k^\infty(U, TM)/ \mathcal{W}_k^{-\infty}(U, TM)$ if $\alpha_K(x,\hbar)dx^K\in \mathcal{W}_P^s(U)/ \mathcal{W}_k^\infty(U)/ \mathcal{W}_k^{-\infty}(U)$. Analogously we say that $A=A_K(x,\hbar)dx^K\in\Omega_{\hbar}^k(U,\End E)$ belongs to $\mathcal{W}_P^s(U, \End E)/ \mathcal{W}_k^\infty(U, \End E)/ \mathcal{W}_k^{-\infty}(U, \End E)$ if for every $p,q=1,\cdots, r$ then $(A_K)_{ij}(x,\hbar)dx^K\in \mathcal{W}_P^s(U)/ \mathcal{W}_k^\infty(U)/ \mathcal{W}_k^{-\infty}(U)$.
\end{notation}

\begin{prop}\label{rmk:closed}
$\Pi$ is a solution of the Maurer-Cartan equation $d_W\Pi+\frac{1}{2}\lbrace\Pi,\Pi\rbrace_\sim=0$, up to higher order term in $\hbar$, i.e. there exists $\Pi_{E,R}\in\Omega^1(U,\End E\oplus T\mathcal{L})$ such that $\bar{\Pi}\defeq (\Pi_E+\Pi_{E,R}, \Pi^{CLM})$ is a solution of Maurer-Cartan and $\Pi_{E,R}\in\mathcal{W}_{\mathfrak{d}_m}^{-1}(U)$. 
\end{prop}

\begin{proof}
First of all let us compute $d_W\Pi$:
\[
\begin{split}
d_W\Pi&=\big(d_{W,E}\Pi_E+\hat{B}\Pi^{CLM},d_{W,\mathcal{L}}\Pi^{CLM}\big)\\
&=\big(-A_{jk}t^j\mathfrak{w}^{km}d\delta_m+\hat{B}\Pi^{CLM}, -a_{jk}t^j\mathfrak{w}^{km}d(\delta_m)\otimes\partial_n\big)
\end{split}
\]
and notice that $d(\delta_m)=0$. Then, let us compute $\hat{B}\Pi^{CLM}$:
\[
\begin{split}
\hat{B}\Pi^{CLM}&=\FF(\FF^{-1}(\Pi^{CLM})\lrcorner F_E)\\
&=-\FF\left(\left(\frac{4\pi}{\hbar}\right)^{-1}a_{jk}t^j\textbf{w}^{km}\check{\delta}_m\otimes\check{\partial}_n\lrcorner \left(\hbar^2F_j^q(\phi)dz^j\wedge d\bar{z}_q\right)\right)\\
&=-\left(\frac{4\pi}{\hbar}\right)^{-1}\FF\big(a_{jk}t^j\textbf{w}^{km} n^l\hbar^2F_{jq}(\phi)\check{\delta}_m\wedge d\bar{z}^q\big)\\
&=-\hbar^{2}\Big(\frac{4\pi}{\hbar}\Big)^{-1}\Big(\frac{4\pi}{\hbar}\Big)^{2}a_{jk}t^j\mathfrak{w}^{km} n^lF_{lq}(\phi)\delta_m\wedge dx^q\\
&=-4\pi\hbar (a_{jk}t^j\mathfrak{w}^{km} n^lF_{lq}(\phi)\delta_m\wedge dx^q)
\end{split}
\]
where we denote by $\check{\delta}_m$ the Fourier transform of $\delta_m$. Notice that $\hat{B}\Pi^{CLM}$ is an exact two form, thus since $F_{lq}(\phi)dx^q=dA_l(\phi)$ (recall that the hermitian metric on $E$ is diagonal)we define 
\[\Pi_{E,R}\defeq 4\pi\hbar (a_{jk}t^j\mathfrak{w}^{km} n^lA_l(\phi)\delta_m)\] 
i.e. as a solution of $d_W\Pi_{E,R}=-\hat{B}\Pi^{CLM}$.

In particular, since $\delta_m\in\mathcal{W}_{\mathfrak{d}_m}^1(U)$ then $\hbar\delta_m\in\mathcal{W}_{\mathfrak{d}_m}^{-1}(U)$. Therefore $\Pi_{E,R}$ has the expected asymptotic behaviour and $d_W\bar{\Pi}=0$.  
Let us now compute the commutator:
\[
\begin{split}
\lbrace\bar{\Pi}, \bar{\Pi}\rbrace_\sim&=\big(2\FF\big(\FF^{-1}\Pi^{CLM}\lrcorner\nabla^E\FF^{-1}(\Pi_E+\Pi_{E,R})\big)+\lbrace\Pi_E+\Pi_{E,R},\Pi_E+\Pi_{E,R}\rbrace,\lbrace\Pi^{CLM},\Pi^{CLM}\rbrace\big)\\
&=\big(2\FF\big(\FF^{-1}\Pi^{CLM}\lrcorner\nabla^E\FF^{-1}(\Pi_E+\Pi_{E,R})\big)+2(\Pi_E+\Pi_{E,R})\wedge(\Pi_E+\Pi_{E,R}),0\big)
\end{split}
\] 
Notice that, since both $\Pi_E$ and $\Pi_{E,R}$ are matrix valued one forms where the form part is given by $\delta_m$, the wedge product $(\Pi_E+\Pi_{E,R})\wedge(\Pi_E+\Pi_{E,R})$ vanishes as we explicitly compute below
\[
\begin{split}
&(\Pi_E+\Pi_{E,R})\wedge(\Pi_E+\Pi_{E,R})=A_{jk}A_{rs}t^{j+r}\mathfrak{w}^{km+sm}\delta_m\wedge\delta_m+\\
&+8\pi\hbar a_{jk}t^{j+r}\mathfrak{w}^{km+sm} n^lA_l(\phi)A_{rs}\delta_m\wedge\delta_m+4\pi\hbar (a_{jk}t^j\mathfrak{w}^{km} n^lA_l(\phi))^2\delta_m\wedge\delta_m=0.
\end{split}
\]
 
Hence we are left to compute $\FF\left(\FF^{-1}\Pi^{CLM}\lrcorner\nabla^E\FF^{-1}\left((\Pi_E+\Pi_{E,R})\right)\right)$:
\[
\begin{split}
&\FF\big(\FF^{-1}\Pi^{CLM}\lrcorner\nabla^E\FF^{-1}(\Pi_E+\Pi_{E,R})\big)=\\
&=\FF\bigg(\Big(\frac{4\pi}{\hbar}\Big)^{-1}a_{jk}t^j\textbf{w}^{km}\check{\delta}_m\check{\partial}_n\lrcorner d \Big(\Big(\frac{4\pi}{\hbar}\Big)^{-1}\big(At\textbf{w}^m\check{\delta}_m+4\pi\hbar a_{rs}t^r\textbf{w}^{sm}n^lA_l(\phi)\check{\delta}_m\big)\Big)+\\
&\quad+\Big(\frac{4\pi}{\hbar}\Big)^{-1}a_{jk}t^j\textbf{w}^{km}\check{\delta}_m\check{\partial}_n\lrcorner \Big(i\hbar A_q(\phi)dz^q\wedge \Big(\Big(\frac{4\pi}{\hbar}\Big)^{-1}\big(At\textbf{w}^m\check{\delta}_m+4\pi\hbar a_{rs}t^r\textbf{w}^{sm}n^lA_l(\phi)\check{\delta}_m\big)\Big)\Big)\bigg)\\
&=\Big(\frac{4\pi}{\hbar}\Big)^{-1}\FF\bigg(a_{jk}t^j\textbf{w}^{km}\check{\delta}_m\check{\partial}_n\lrcorner\Big(At\partial_l(\textbf{w}^m)dz^l\wedge\check{\delta}_m+At\textbf{w}^md(\check{\delta}_m)+\\
&\quad+4\pi\hbar a_{rs}t^r\partial_l(n^qA_q(\phi)\textbf{w}^{sm})\check{\delta}_m+ 4\pi\hbar a_{rs}t^rn^qA_q(\phi)\textbf{w}^{sm}d(\check{\delta}_m)\Big)\bigg)\\
&=\Big(\frac{4\pi}{\hbar}\Big)^{-1}\FF\bigg(a_{jk}At^{j+1}\textbf{w}^{km}\check{\delta}_mn^l\partial_l(\textbf{w}^m)\wedge\check{\delta}_m+a_{jk}At^{j+1}\textbf{w}^{km+m}n^l\gamma_l(i\hbar^{-1}\gamma_pz^p)\check{\delta}_m\wedge\check{\delta}_m+\\
&\quad+4\pi\hbar a_{jk}a_{rs}t^{j+r}\textbf{w}^{km}\check{\delta}_m n^l\partial_l(n^qA_q(\phi)\textbf{w}^{sm})\check{\delta}_m+\\
&\quad+4\pi\hbar a_{jk}a_{rs}t^{j+r}\textbf{w}^{km+sm}n^qA_q(\phi)n^l\gamma_l(i\hbar^{-1}\gamma_pz^p)\check{\delta}_m\wedge\check{\delta}_m\bigg)\\
&=0
\end{split}
\]  
where $\check{\delta}_m=\frac{e^{-\frac{u_{m^\perp}^2}{\hbar}}}{\sqrt{\pi\hbar}}\gamma_pd\bar{z}^p$ for some constant $\gamma_p$ such that $u_{m^\perp}=\gamma_1x^1+\gamma_2 x^2$, and $\partial_l$ is the partial derivative with respect to the coordinate $z^l$. In the last step we use that $\check{\delta}_m\wedge\check{\delta}_m=0$.  
\end{proof}
\begin{remark}
In the following it will be useful to consider $\bar{\Pi}$ in order to compute the solution of Maurer-Cartan from the data of two non-parallel walls (see section \eqref{sec:two_walls}). However, in order to compute the asymptotic behaviour of the gauge it is enough to consider $\Pi$. 
\end{remark}
Since $\check{X}(U)\cong U\times \C/\Lambda$ has no non trivial deformations and $E$ is holomorphically trivial, then also the pair $(\check{X}(U),E)$ has no non trivial deformations. Therefore there is a gauge $\varphi\in\Omega^0(U, \End E\oplus TM)[\![ t ]\!]$ such that 
\begin{equation}
e^\varphi\ast 0=\bar{\Pi}
\end{equation}
namely $\varphi$ is a solution of the following equation
\begin{equation}\label{eq:gauge}
d_W\varphi=-\bar{\Pi}-\sum_{k\geq 0}\frac{1}{(k+1)!}\textsf{ad}_{\varphi}^kd_W\varphi.
\end{equation}
In particular the gauge $\varphi$ is not unique, unless we choose a gauge fixing condition (see Lemma \ref{lem:uniq}). In order to define the gauge fixing condition we introduce the so called homotopy operator.
  
\subsection{Gauge fixing condition and homotopy operator}  

Since $\mathcal{L}(U)=\bigsqcup_{m\in\Lambda}U_m$, it is enough to define the homotopy operator $H_m$ for every frequency $m$. Let us first define morphisms $p\defeq\bigoplus_{m\in \Lambda\setminus\lbrace 0\rbrace}p_m$ and $\iota\defeq\bigoplus_{m\in\Lambda\setminus\lbrace 0\rbrace}\iota_m$. We define $p_m:\mathfrak{w}^m\cdot\Omega^{\bullet}(U)\rightarrow\mathfrak{w}^m\cdot H^{\bullet}(U)$ which acts as $p_m(\alpha \mathfrak{w}^m)=\alpha({q_0})\mathfrak{w}^m$ if $\alpha\in\Omega^0(U)$ and it is zero otherwise. 

Then $\iota_m\colon\mathfrak{w}^m\cdot H^{\bullet}(U)\rightarrow\mathfrak{w}^m\Omega^{\bullet}(U)$ is the embedding of constant functions on $\Omega^{\bullet}(U)$ at degree zero, and it is zero otherwise. Then let $q_0\in H_{-}$ be a fixed base point, then since $U$ is contractible, there is a homotopy $\varrho\colon [0,1]\times U\to U$ which maps $(\tau,u_m,u_{m,\perp})$ to $(\varrho_1(\tau,u_m,u_{m,\perp}),\varrho_2(\tau,u_m,u_{m,\perp}))$ and such that $\varrho(0,\cdot)=q_0=(u_0^1,u_0^2)$ and $\varrho(1,\cdot)=\text{Id}$. We define $H_m$ as follows:
\begin{equation}\label{def:homotopy_1wall}
\begin{split}
H_m &\colon\mathfrak{w}^m\cdot \Omega^{\bullet}(U)\rightarrow \mathfrak{w}^m\cdot\Omega^{\bullet}(U)[-1]\\
&H_m(\mathfrak{w}^m\alpha)\defeq\mathfrak{w}^m\int_0^1 d\tau\wedge\frac{\partial}{\partial \tau}\lrcorner\varrho^*(\alpha)
\end{split}
\end{equation} 

\begin{lemma}
The morphism $H$ is a homotopy equivalence of $\text{id}_{\Omega^{\bullet}}$ and $\iota\circ p$, i.e. the identity
\begin{equation}\label{eq:PHi}
\text{id}-\iota\circ p=d_WH+Hd_W
\end{equation}
holds true.
\end{lemma}
\begin{proof}
At degree zero, let $f\in\Omega^0(U)$: then $\iota_m\circ p_m(f\mathfrak{w}^m)=f(q_0)\mathfrak{w}^m$. By degree reason $H_m(f\mathfrak{w}^m)=0$ and 
\begin{equation*}
H_md_W(\mathfrak{w}^mf)=\mathfrak{w}^m\int_0^1d\tau\wedge\frac{\partial}{\partial\tau}\lrcorner(d_M (f(\varrho))+d\tau\frac{\partial f(\varrho)}{\partial\tau})=\mathfrak{w}^m\int_0^1d\tau\frac{\partial f(\varrho)}{\partial\tau}=\mathfrak{w}^m(f(q)-f(q_0)).
\end{equation*} 
At degree $k=1$, let $\alpha=f_idx^i\in \Omega^1(U)$ then: $\iota_m\circ p_m(\alpha\mathfrak{w}^m)=0$,
\begin{equation*}
\begin{split}
H_md_W(\alpha\mathfrak{w}^m)&=\mathfrak{w}^m\int_0^1d\tau\wedge\frac{\partial}{\partial\tau}\lrcorner (d(\varrho^*(\alpha))\\
&=\mathfrak{w}^m\int_0^1d\tau\wedge\frac{\partial}{\partial\tau}\lrcorner (d_M(\varrho^*(\alpha))+d\tau\wedge\frac{\partial}{\partial\tau}(f_i(\varrho)\frac{\partial\varrho_i}{\partial x^i})dx^i)\\
&=-\mathfrak{w}^md_M\left(\int_0^1d\tau\wedge\frac{\partial}{\partial\tau}\lrcorner (\varrho^*(\alpha))\right)+\mathfrak{w}^m\int_0^1d\tau \frac{\partial}{\partial\tau}(f_i(\varrho)\frac{\partial\varrho_i}{\partial x^i})dx^i\\
&=-\mathfrak{w}^md_M\left(\int_0^1d\tau\wedge\frac{\partial}{\partial\tau}\lrcorner (\varrho^*(\alpha))\right)+\mathfrak{w}^m\alpha
\end{split}
\end{equation*}
and 
\begin{equation*}
d_WH_m(\mathfrak{w}^m\alpha)=\mathfrak{w}^md_M\left(\int_0^1d\tau\wedge\frac{\partial}{\partial\tau}\lrcorner\varrho^*(\alpha)\right).
\end{equation*}
Finally let $\alpha\in\Omega^2(U)$, then $p_m(\alpha\mathfrak{w}^m)=0$ and $d_W(\alpha\mathfrak{w}^m)=0$. Then it is easy to check that $d_WH_m(\mathfrak{w}^m\alpha)=\alpha$.
\end{proof}

\begin{lemma}[Lemma 4.7 in \cite{MCscattering}]\label{lem:uniq}
Among all solution of $e^{\varphi}\ast 0=\bar{\Pi}$, there exists a unique one such that $p(\varphi)=0$. 
\end{lemma} 
\begin{proof}
First of all, let $\sigma\in\Omega^0(U)$ such that $d\sigma=0$. Then $e^{\varphi\bullet\sigma}\ast 0=\bar{\Pi}$, indeed $e^{\sigma}\ast 0=0-\sum_{k}\frac{[\sigma,\cdot]^k}{k!}(d\sigma)=0$. Thus $e^{\varphi\bullet\sigma}\ast 0=e^{\varphi}\ast(e^{\sigma}\ast 0)=e^{\varphi}\ast 0=\bar{\Pi}$. Thanks to the BCH formula \[\varphi\bullet\sigma=\varphi+\sigma+\frac{1}{2}\lbrace\varphi,\sigma\rbrace_\sim +\cdots\] we can uniquely determine $\sigma$ such that $p(\varphi\bullet\sigma)=0$. Indeed working order by order in the formal parameter $t$, we get:
\begin{enumerate}
\item $p(\sigma_1+\varphi_1)=0$, hence by definition of $p$, $\sigma_1(q_0)=-\varphi_1(q_0)$;
\item $p(\sigma_2+\varphi_2+\frac{1}{2}\lbrace\varphi_1,\sigma_1\rbrace_\sim)=0$, hence $\sigma_2(q_0)=-\big(\varphi_2(q_0)+\frac{1}{2}\lbrace\varphi_1,\sigma_1\rbrace_\sim (q_0)\big)$;
\end{enumerate} 
and any further order is determined by the previous one. 
\end{proof}

Now that we have defined the homotopy operator and the gauge fixing condition (as in Lemma \ref{lem:uniq}), we are going to study the asymptotic behaviour of the gauge $\varphi$ such that it is a solution of \eqref{eq:gauge} and $p(\varphi)=0$. Equations \eqref{eq:gauge}, \eqref{eq:PHi} and $p(\varphi)=0$ together say that the unique gauge $\varphi$ is indeed a solution of the following equation:
\begin{equation}\label{eq:sol_gauge}
\varphi=-Hd_W(\varphi)=-H\big(\bar{\Pi}+\sum_{k}{\textsf{ad}_{\varphi}^k\over (k+1)!}d_W\varphi\big).
\end{equation}

Up to now we have used a generic homotopy $\varrho$, but from now on we are going to choose it in order to get the expected asymptotic behaviour of the gauge $\varphi$. In particular we choose the homotopy $\varrho$ as follows: for every $q=(u_{q,m},u_{q,m^\perp})\in U$   
\begin{equation}\label{varro}
\varrho(\tau, u_q)=\begin{cases}
\left((1-2\tau)u_{1}^0+2\tau u_{q,m},u_{2}^0\right) & \text{if} \tau\in [0,\frac{1}{2}] \\
\left(u_{q,m},(2\tau-1)u_{q,m^\perp}+(2-2\tau)u_{2}^0\right) & \text{if} \tau\in [\frac{1}{2},1]
\end{cases}
\end{equation}

where $(u_0^1, u^2_0)$ are the coordinates for the fixed point $q_0$ on $U$. Then we have the following result:

\begin{lemma}\label{lem:H(W^s)}
Let $\mathfrak{d}_m$ be a ray in $U$ and let $\alpha\in\mathcal{W}_{\mathfrak{d}_m}^s(U)$. Then $H(\alpha\mathfrak{w}^m)$ belongs to $\mathcal{W}_0^{s-1}(U)$. 
\end{lemma}
\begin{proof}
Let us first consider $q_*\in U\setminus \mathfrak{d}_m$. By assumption there is a neighbourhood of $q_*$, $V\subset U$ such that $\alpha\in\mathcal{W}_1^{-\infty}(V)$. Then by definition  
\begin{equation*}
H(\alpha\mathfrak{w}^m)=\mathfrak{w}^m\int_0^1d\tau\wedge\frac{\partial}{\partial \tau}\lrcorner\varrho^*(\alpha)=\int_0^1d\tau\alpha(\varrho)\left(\frac{\partial\varrho_1}{\partial\tau}+\frac{\partial\varrho_2}{\partial\tau} \right)
\end{equation*}
hence, since $\varrho$ does not depend on $\hbar$ 
\begin{equation*}
\sup_{q\in V}\nabla^j\left|\int_0^1d\tau\alpha(\varrho)\left(\frac{\partial\varrho_1}{\partial\tau}+\frac{\partial\varrho_2}{\partial\tau} \right)\right|\leq\int_0^1d\tau\sup_{q\in V}\left|\nabla^j(\alpha(\varrho)\left(\frac{\partial\varrho_1}{\partial\tau}+\frac{\partial\varrho_2}{\partial\tau} \right))\right|\leq C(V,j)e^{-\frac{c_v}{\hbar}}.
\end{equation*}

Let us now consider $q_*\in \mathfrak{d}_m$. By assumption there is a neighbourhood of $q$, $W\subset U$ such that for all $q=(u_{q,m},u_{q,m^\perp})\in W$  $\alpha=h(u_q, \hbar)du_{m_q^\perp}+\eta$ and $\eta\in\mathcal{W}_{1}^{-\infty}(W)$. By definition 

\begin{equation*}
\begin{split}
H(\alpha\mathfrak{w}^m)&=\mathfrak{w}^m\int_0^1d\tau\wedge\frac{\partial}{\partial \tau}\lrcorner\varrho^*(\alpha)\\
&=2\int_{\frac{1}{2}}^1d\tau h(u_{q,m},(2\tau-1)u_{q,m^\perp}+(2-2\tau)u_0^2)(u_{q,m^\perp}-u_0^2)+ \int_0^1d\tau\eta(\varrho)\frac{\partial\varrho_1}{\partial\tau}\\
&=\int_{u_0^2}^{u_{q,m^\perp}}du_m^\perp h(u_{q,m},u_{m^\perp})+ \int_0^1d\tau\eta(\varrho)\frac{\partial\varrho_1}{\partial\tau}
\end{split}
\end{equation*} 

and since $\eta\in\mathcal{W}_1^{-\infty}(W)$ the second term $\int_0^1d\tau\eta(\varrho)\frac{\partial\varrho_1}{\partial\tau}$ belongs to $\mathcal{W}_0^{\infty}$. The first term is computed below:

\begin{equation}
\begin{split}
\sup_{q\in W}\left|\nabla^j\left(\int_{u_0^2}^{u_{q,m^\perp}}du_m^\perp h(u_{q,m},u_{m^\perp})\right)\right|&=\sup_{q\in W}\Big|\int_{u_0^2}^{u_{q,m^\perp}}du_m^\perp \nabla^j(h(u_{q,m},u_{m^\perp}))+\\
&\quad+\left[\frac{\partial^{j-1}}{\partial u_{m^\perp}^{j-1}}(h(u_{q,m},u_{m^\perp)})\right]_{u_{m^\perp}=u_{q,m^\perp}}\Big|\\
&\leq \sup_{u_{q,m^\perp}}\int_{u_0^2}^{u_{q,m^\perp}}du_m^\perp \sup_{u_{q,m}}\left|\nabla^j(h(u_{q,m},u_{m^\perp}))\right|+\\
&\quad +\left[\sup_{q\in W}\left|\frac{\partial^{j-1}}{\partial u_{m^\perp}^{j-1}}(h(u_{q,m},u_{m^\perp)})\right|\right]_{u_{m^\perp}=u_{0}^2}\\
&\leq C(j,W)\hbar^{-\frac{s+j-1}{2}}
\end{split}
\end{equation}
where in the last step we use that $\left[\frac{\partial^{j-1}}{\partial u_{m^\perp}^{j-1}}(h(u_{q,m},u_{m^\perp)})\right]_{u_{m^\perp}=u_{0}^2} $ is outside the support of $\mathfrak{d}_m$.

\end{proof}

\begin{corollary}\label{lem:H(delta_m)}
Let $\mathfrak{d}_m$ be a ray in $U$, then $H(\delta_m\mathfrak{w}^m)\in\mathcal{W}_{0}^{0}(U)\mathfrak{w}^m$.
\end{corollary}

\subsection{Asymptotic behaviour of the gauge $\varphi$}
We are going to compute the asymptotic behaviour of $\varphi=\sum_j\varphi^{(j)}t^j\in\Omega^0(U,\End E\oplus TM)[\![ t ]\!]$ order by order in the formal parameter $t$. In addition since $\Pi_{E,R}$ gives a higher $\hbar$-order contribution in the definition of $\bar{\Pi}$ we get rid of it by replacing $\bar{\Pi}$ with $\Pi$ in equation \eqref{eq:sol_gauge}. 

\begin{prop}\label{prop:asymp1}
Let $(m,\mathfrak{d}_m,\theta_m)$ be a wall with $\log\theta_m=\sum_{j,k\geq 1}\big(A_{jk}t^j\mathfrak{w}^{km},a_{jk}\mathfrak{w}^{km}t^j\partial_n\big)$. Then, the unique gauge $\varphi=(\varphi_E,\varphi^{CLM})$ such that $e^\varphi\ast 0=\Pi$ and $P(\varphi)=0$, has the following  asymptotic jumping behaviour along the wall, namely 
\begin{equation}\label{eq:lim gauge}
\varphi^{(s+1)}\in\begin{cases} \sum_{k\geq 1}\big(A_{s+1,k}t^{s+1}\mathfrak{w}^{km},a_{s+1,k}\mathfrak{w}^{km}t^{s+1}\partial_n\big)+\bigoplus_{k\geq 1}\mathcal{W}_{0}^{-1}(U,\End E\oplus TM)\mathfrak{w}^{km}t^{s+1} & H_{m,+}\\
\bigoplus_{k\geq 1}\mathcal{W}_0^{-\infty}(U, \End E\oplus TM)\mathfrak{w}^{km} t^{s+1}& H_{m,-}.
\end{cases}
\end{equation} 
\end{prop}


Before giving the proof of Proposition \ref{prop:asymp1}, let us introduce the following Lemma which are useful to compute the asymptotic behaviour of one-forms asymptotically supported on a ray $\mathfrak{d}_m$.
\begin{lemma}\label{lem:wedge_sobolev}
Let $\mathfrak{d}_m$ be a ray in $U$. Then $\mathcal{W}_{\mathfrak{d}_m}^s(U)\wedge\mathcal{W}_{0}^{r}(U)\subset\mathcal{W}^{r+s}_{\mathfrak{d}_{m}}(U)$.
\end{lemma}
\begin{proof}
Let $\alpha\in\mathcal{W}_{\mathfrak{d}_m}^s(U)$ and let $f\in\mathcal{W}_{0}^{r}(U)$. Pick a point $q_*\in \mathfrak{d}_m$ and let $W\subset U$ be a neighbourhood of $q_*$ where $\alpha=h(u_q,\hbar)du_{m^\perp}+\eta$, we claim 
\begin{equation}
\int_{-a}^b u_{m^\perp}^\beta\sup_{u_m}\left|\nabla^j(h(u_q,\hbar)f(u_q\hbar))\right|du_{m^\perp}\leq C(a,b,j,\beta)\hbar^{-\frac{r+s+j-\beta-1}{2}}
\end{equation}
for every $\beta\in\Z_{\geq 0}$ and for every $j\geq0$.

\begin{equation*}
\begin{split}
&\int_{-a}^b u_{m^\perp}^{\beta} \sup_{u_m}\left|\nabla^j\left(h(u_q,\hbar)f(u_q\hbar)\right)\right|du_{m^\perp}=\\
&=\sum_{j_1+j_2=j}\int_{-a}^b u_{m^\perp}^{\beta}  \sup_{u_m}\left|\nabla^{j_1}(h(u_q,\hbar))\nabla^{j_2}(f(u_q\hbar))\right|du_{m^\perp}\\
&\leq \sum_{j_1+j_2=j}C(a,b,j_2)\hbar^{-\frac{r+j_2}{2}}\int_{-a}^b u_{m^\perp}^{\beta} \sup_{u_m}\left|\nabla^{j_1}(h(u_q,\hbar))\right|du_{m^\perp}\\
&\leq \sum_{j_1+j_2=j}C(a,b,j_2,j_1)\hbar^{-\frac{r+j_2}{2}}\hbar^{-\frac{s+j_1-\beta-1}{2}}\\
&\leq C(a,b,j)\hbar^{-\frac{r+s+j-\beta-1}{2}}
\end{split}
\end{equation*}
Finally, since $\eta\in\mathcal{W}_1^{-\infty}(W)$ also $f(x,\hbar)\eta$ belongs to $\mathcal{W}_1^{-\infty}(W)$.  
\end{proof}

\begin{lemma}\label{lem:bracket} 
Let $\mathfrak{d}_{m}$ be a ray in $U$. If $(A\mathfrak{w}^{m}, \varphi\mathfrak{w}^{m}\partial_{n})\in\mathcal{W}_{\mathfrak{d}_{m}}^{r}(U,\End E\oplus TM)\mathfrak{w}^{m}$ for some $r\geq 0$ and $(T\mathfrak{w}^{m}, \psi\mathfrak{w}^{m}\partial_{n})\in\mathcal{W}_{0}^{s}(U,\End E\oplus TM)\mathfrak{w}^{m}$ for some $s\geq$, then 
\begin{equation}
\lbrace (A\mathfrak{w}^{m}, \alpha\mathfrak{w}^{m}\partial_{n}),(T\mathfrak{w}^{m}, f\mathfrak{w}^{m}\partial_{n})\rbrace_\sim\in \mathcal{W}_{\mathfrak{d}_{m}}^{r+s}(U,\End E\oplus TM)\mathfrak{w}^{2m}.
\end{equation}
\end{lemma}
\begin{proof}
We are going to prove the following:
\begin{equation*}
\begin{split}
(1)&\lbrace A\mathfrak{w}^{m},T\mathfrak{w}^{m}\rbrace_{\End E}\subset \mathcal{W}_{\mathfrak{d}_{m}}^{r+s}(U,\End E)\mathfrak{w}^{2m} \\
(2)&\textbf{ad}\left(\varphi\mathfrak{w}^{m}\partial_{n}, T\mathfrak{w}^{m}\right)\subset \mathcal{W}_{\mathfrak{d}_{m}}^{r+s-1}(U,\End E)\mathfrak{w}^{2m}\\
(3)&\textbf{ad}\left(\psi\mathfrak{w}^{m}\partial_{n}, A\mathfrak{w}^{m}\right)\subset \mathcal{W}_{\mathfrak{d}_{m}}^{r+s-1}(U,\End E)\mathfrak{w}^{2m}\\
(4)&\lbrace\varphi\mathfrak{w}^{m}\partial_{n},\psi\mathfrak{w}^{m}\partial_{n}\rbrace\subset \mathcal{W}_{\mathfrak{d}_{m}}^{r+s}(U, TM)\mathfrak{w}^{m}.
\end{split}
\end{equation*}  
The first one is a consequence of Lemma \ref{lem:wedge_sobolev}, indeed by definition \[\lbrace A_k(x)dx^k, T(x)\rbrace_{\End E}=[A_k(x), T(x)]dx^k\]
which is an element in $\End E$ with coefficients in $\mathcal{W}_{\mathfrak{d}_{m}}^{r+s}(U)$.
 
The second one is less straightforward and need some explicit computations to be done. 
\[
\begin{split}
&\textbf{ad}(\varphi_k(x)\mathfrak{w}^mdx^k\partial_n,T(x)\mathfrak{w}^m)=\FF\big(\FF^{-1}(\varphi_k(x)\textbf{w}^m dx^k\otimes\partial_n)\lrcorner\nabla^E\FF^{-1}(T(x)\mathfrak{w}^m)\big)\\
&=\FF\bigg(\Big(\frac{4\pi}{\hbar}\Big)^{-1}\varphi_k(x)\textbf{w}^m d\bar{z}^k\otimes\check{\partial}_n\lrcorner\nabla^E(T(x)\textbf{w}^m)\bigg)\\
&=\Big(\frac{4\pi}{\hbar}\Big)^{-1}\FF\bigg(\varphi_k(x)\textbf{w}^m d\bar{z}^k\check{\partial}_n\lrcorner\Big(\partial_j(T(x)\textbf{w}^m)dz^j \Big)\bigg)\\
&\quad+i\hbar\Big(\frac{4\pi}{\hbar}\Big)^{-1}\FF\bigg(\varphi_k(x)\textbf{w}^m d\bar{z}^k\check{\partial}_n\lrcorner\Big(A_j(\phi)T(x) \textbf{w}^mdz^j\Big)\bigg)\\
&=\Big(\frac{4\pi}{\hbar}\Big)^{-1}\FF\bigg(\varphi_k(x)\textbf{w}^md\bar{z}^kn^l\frac{\partial}{\partial z_l}\lrcorner\Big(\partial_j(T(x))\textbf{w}^{m}dz^j +m_jA(x)\textbf{w}^{m}dz^j\Big)\bigg)\\
&\quad+i\hbar\Big(\frac{4\pi}{\hbar}\Big)^{-1}\FF\bigg(\varphi_k(x)\textbf{w}^m d\bar{z}^kn^l\frac{\partial}{\partial z_l}\lrcorner\Big(A_j(\phi)T(x)\textbf{w}^mdz^j\Big)\bigg)\\
&=\Big(\frac{4\pi}{\hbar}\Big)^{-1}\FF\bigg(\varphi_k(x)d\bar{z}^k
n^lm_lT(x)\textbf{w}^{2m}\bigg)+\\
&\quad+i\hbar\Big(\frac{4\pi}{\hbar}\Big)^{-1}\FF\bigg(\varphi_k(x)d\bar{z}^k\big(n^lT(x)A_l(\phi)+n^l\frac{\partial T(x)}{\partial x^l} \big)\textbf{w}^{2m}\Big)\bigg)\\
&=i\hbar \varphi_k(x)
\big(n^lT(x)A_l(\phi)+n^l\frac{\partial T(x)}{\partial x^l}\big) \mathfrak{w}^{2m} dx^k
\end{split}
\]
Notice that as a consequence of Lemma \ref{lem:wedge_sobolev}, $\hbar\varphi_k(x)dx^k A_l(\phi)T(x)\in\mathcal{W}_{\mathfrak{d}_{m}}^{s+r-2}(U)$ while $\hbar\varphi_k(x)\frac{\partial T(x)}{\partial x^l}dx^k\in\mathcal{W}_{\mathfrak{d}_m}^{r+s-1}$.

The third one is
\begin{align*}
&\textbf{ad}(\psi(x)\mathfrak{w}^m\partial_n,A_k(x)\mathfrak{w}^mdx^k)=\\
&=\FF\big(\FF^{-1}(\psi(x)\partial_n)\lrcorner\nabla^E\FF^{-1}(A_k(x)\mathfrak{w}^mdx^k)\big)\\
&=\FF\bigg(\psi(x)\textbf{w}^m\check{\partial}_n\lrcorner\nabla^E(\Big(\frac{4\pi}{\hbar}\Big)^{-1}A_k(x)\textbf{w}^md\bar{z}^k)\bigg)\\
&=\Big(\frac{4\pi}{\hbar}\Big)^{-1}\FF\bigg(\psi(x)\textbf{w}^m\check{\partial}_n\lrcorner\Big(\partial_j(A_k(x)\textbf{w}^m)\Big)dz^j\wedge d\bar{z}^k\bigg)\\
&\quad+i\hbar\Big(\frac{4\pi}{\hbar}\Big)^{-1}\FF\bigg(\psi(x)\textbf{w}^m\check{\partial}_n\lrcorner\Big(A_j(\phi)A_k(x) \textbf{w}^m\Big)\wedge d\bar{z}^k\bigg)\\
&=\Big(\frac{4\pi}{\hbar}\Big)^{-1}\FF\bigg(\psi(x)\textbf{w}^mn^l\frac{\partial}{\partial z_l}\lrcorner\Big(\partial_j(A_k(x))\textbf{w}^mdz^j+m_jA_k(x)\textbf{w}^m dz^j\Big)\wedge d\bar{z}^k\bigg)\\
&\quad+i\hbar\Big(\frac{4\pi}{\hbar}\Big)^{-1}\FF\bigg(\psi(x)\textbf{w}^m\check{\partial}_n\lrcorner\Big(A_j(\phi)A_k(x) \textbf{w}^mdz^j\Big)\wedge d\bar{z}^k\bigg)\\
&=\Big(\frac{4\pi}{\hbar}\Big)^{-1}\FF\bigg(\psi(x)\textbf{w}^{2m}n^lm_lA(x) d\bar{z}^k\bigg)+\\
&\quad+i\hbar\Big(\frac{4\pi}{\hbar}\Big)^{-1}\FF\bigg(\psi(x)\textbf{w}^{2m}n^l
A_k(x)A_l(\phi)+\psi(x)n^l\frac{\partial A_k(x)}{\partial x^l}\textbf{w}^{2m}d\bar{z}^k\bigg)\\
&=i\hbar \psi(x)\Big(n^lA_k(x)A_l(\phi)+n^l\frac{\partial A_k(x)}{\partial x^l}\Big)\mathfrak{w}^{2m}dx^k
\end{align*}

Notice that $\hbar\psi(x)A_k(x)A(\phi)dx^k\in\mathcal{W}_{\mathfrak{d}_m}^{r+s-2}(W)$ and $\hbar\psi(x)\frac{\partial A_k(x)}{\partial x^l}dx^k\in\mathcal{W}_{\mathfrak{d}_m}^{r+s-1}(W)$.

In the end $\lbrace\varphi\mathfrak{w}^{m}\partial_{n},\psi\mathfrak{w}^{m}\partial_{n}\rbrace$ is equal to zero, indeed by definition 
\[
\lbrace\varphi_{k}(x)dx^k{\partial_n}, \psi(x)\partial_n \rbrace=\left(\varphi_{k}dx^k\wedge\psi\right)[\mathfrak{w}^m\partial_n, \mathfrak{w}^m\partial_n]
\]
and $[\mathfrak{w}^m\partial_n,\mathfrak{w}^m\partial_n]=0$.
\end{proof}

\begin{proof}{(Proposition \ref{prop:asymp1})}

It is enough to show that for every $s\geq 0$
\begin{equation}
\label{claim}
\Big(\sum_{k\geq 1}{\ad_{\varphi^s}^k\over (k+1)!} d_W\varphi^s\Big)^{(s+1)}\in\mathcal{W}_{\mathfrak{d}_m}^{0}(U,\End E\oplus TM),
\end{equation}
where $\varphi^s=\sum_{j=1}^s\varphi^{(j)}t^j$.
Indeed from equation \eqref{eq:sol_gauge}, at the order $s+1$ in the formal parameter $t$, the solution $\varphi^{(s+1)}$ is:
\begin{equation}
\varphi^{(s+1)}=-H(\Pi^{(s+1)})-H\left(\left(\sum_{k\geq 1}{\ad_{\varphi^s}^k\over (k+1)!} d_W\varphi^s\right)^{(s+1)}\right).    
\end{equation}

In particular, if we assume equation \eqref{claim} then by Lemma \ref{lem:H(W^s)} \[H\left(\left(\sum_{k\geq 1}{\ad_{\varphi^s}^k\over (k+1)!} d_W\varphi^s\right)^{(s+1)}\right)\in\mathcal{W}_{0}^{-1}(U,\End E\oplus TM).\] By definition of $H$, \[H(\Pi^{(s+1)})=\sum_{k}(A_{s+1,k}t^{s+1}H(\mathfrak{w}^{km}\delta_m),a_{s+1,k}t^{s+1}H(\mathfrak{w}^{km}\delta_m)\partial_n)\] and by Corollary \ref{lem:H(delta_m)} $H(\delta_m\mathfrak{w}^{km})\in\mathcal{W}_{0}^{0}(U, \End E\oplus TM)\mathfrak{w}^{km}$ for every $k\geq 1$. Hence $H(\Pi^{(s+1)})$ is the leading order term and $\varphi^{(s+1)}$ has the expected asymptotic behaviour.   

Let us now prove the claim \eqref{claim} by induction on $s$. At $s=0$,
\begin{equation}
\varphi^{(1)}=-H(\Pi^{(1)})
\end{equation}

and there is nothing to prove. Assume that \eqref{claim} holds true for $s$, then at order $s+1$ we get contributions for every $k=1,\cdots,s$. Thus let start at $k=1$ with $\ad_{\varphi^s}d_W\varphi^s$:
\begin{equation}
\begin{split}
\ad_{\varphi^s}d_W\varphi^s&=\lbrace\varphi^s,d_W\varphi^s\rbrace_\sim\\
&\in\lbrace H(\Pi^s)+\mathcal{W}_{0}^{-1}(U), \Pi^s+\mathcal{W}_{\mathfrak{d}_m}^0(U)\rbrace_\sim\\
&=\lbrace H(\Pi^s), \Pi^s\rbrace_\sim +\lbrace H(\Pi^s), \mathcal{W}_{\mathfrak{d}_m}^0(U)\rbrace_\sim +\lbrace\mathcal{W}_{0}^{-1}(U), \Pi^s\rbrace_\sim +\lbrace\mathcal{W}_{0}^{-1}(U), \mathcal{W}_{\mathfrak{d}_m}^0(U)\rbrace_\sim\\
&\in \lbrace H(\Pi^s), \Pi^s\rbrace_\sim +\mathcal{W}_{\mathfrak{d}_m}^{0}(U)
\end{split}
\end{equation}
where in the first step we use the inductive assumption on $\varphi^s$ and $d_W\varphi^s$ and the identity \eqref{eq:PHi}. In the last step since $H(\Pi^s)\in\mathcal{W}_{0}^{0}(U)$ then by Lemma \ref{lem:bracket} $\lbrace H(\Pi^s), \mathcal{W}_{\mathfrak{d}_m}^0(U)\rbrace_\sim\in\mathcal{W}_{\mathfrak{d}_m}^{0}(U)$. Then, since $\Pi^s\in\mathcal{W}_{\mathfrak{d}_m}^1(U)$, still by Lemma \ref{lem:bracket} $\lbrace\mathcal{W}_{0}^{-1}(U), \Pi^s\rbrace_\sim\in\mathcal{W}_{\mathfrak{d}_m}^{0}(U)$ and $\lbrace\mathcal{W}_{0}^{-1}(U), \mathcal{W}_{\mathfrak{d}_m}^0(U)\rbrace_\sim\in\mathcal{W}_{\mathfrak{d}_m}^{-1}(U)\subset\mathcal{W}_{\mathfrak{d}_m}^0(U)$. In addition $\lbrace H(\Pi^s), \Pi^s\rbrace_\sim\in\mathcal{W}_{\mathfrak{d}_m}^0(U) $, indeed 
\begin{multline*}
\lbrace H(\Pi^s), \Pi^s\rbrace_\sim =\Big(\lbrace H(\Pi_E^s), \Pi_E^s\rbrace_{\End E}+\textbf{ad}(H(\Pi^{CLM,s}), \Pi_E^s)- \textbf{ad}(\Pi^{CLM,s},H(\Pi_E^s)),\\ \lbrace H(\Pi^{CLM,s}), \Pi^{CLM,s}\rbrace\Big)
\end{multline*}
Notice that since $[A,A]=0$ then $\lbrace H(\Pi_E^s), \Pi_E^s\rbrace_{\End E}=0$ and because of the grading \[\lbrace H(\Pi^{CLM,s}), \Pi^{CLM,s}\rbrace=0.\] 
Then by the proof of Lemma \ref{lem:bracket} identities $(2)$ and $(3)$ we get \[\ad(H(\Pi^{CLM,s}), \Pi_E^s),  \ad(\Pi^{CLM,s},H(\Pi_E^s))\in\mathcal{W}_{\mathfrak{d}_m}^{0}(U)\] therefore 
\begin{equation}
\lbrace H(\Pi^s), \Pi^s\rbrace_\sim\in\mathcal{W}_{\mathfrak{d}_m}^{0}(U).
\end{equation}   

Now at $k>1$ we have to prove that:
\begin{equation}
\ad_{\varphi^s}\cdots \textsf{ad}_{\varphi^s}d_W\varphi^s\in \mathcal{W}_{\mathfrak{d}_m}^{0}(U)
\end{equation}
By the fact that $ H(\Pi^s)\in\mathcal{W}_{0}^{0}(U)$, applying Lemma \ref{lem:bracket} $k$ times we finally get:
\begin{equation}
\ad_{\varphi^s}\cdots \ad_{\varphi^s}d_W\varphi^s\in\lbrace H(\Pi^s),\cdots, \lbrace H(\Pi^s), \lbrace H(\Pi^s), \Pi^s\rbrace_\sim\rbrace_\sim\cdots\rbrace_\sim+\mathcal{W}_{\mathfrak{d}_m}^{0}(U)\in\mathcal{W}_{\mathfrak{d}_m}^{0}(U).
\end{equation}
\end{proof}

\section{Scattering diagrams from solutions of Maurer-Cartan}
\label{sec:two_walls}
In this section we are going to construct consistent scattering diagrams from solutions of the Maurer-Cartan equation. In particular we will first show how to construct a solution $\Phi$ of the Maurer-Cartan equation from the data of an initial scattering diagram $\mathfrak{D}$ with two non parallel walls. Then we will define its completion $\mathfrak{D}^{\infty} $ by the solution $\Phi$ and we will prove it is consistent. 
\subsection{From scattering diagram to solution of Maurer-Cartan}
Let the initial scattering diagram $\mathfrak{D}=\lbrace\mathsf{w}_1, \mathsf{w}_2\rbrace$ be such that $\mathsf{w}_1=(m_1,\mathfrak{d}_1,\theta_1)$ and $\mathsf{w}_2=(m_2,\mathfrak{d}_2,\theta_2)$ are two non-parallel walls and \[\log(\theta_i)=\sum_{j_i,k_i}\Big(A_{j_i,k_i}\mathfrak{w}^{k_im_i}t^{j_i},  a_{j_i,k_i}\mathfrak{w}^{k_im_i}t^{j_i}\partial_{n_i}\Big)\]
for $i=1,2$. 
As we have already done in Section \ref{sec:single wall}, we can define $\bar{\Pi}_1$ and $\bar{\Pi}_2$ to be solutions of Maurer-Cartan equation, respectively supported on $\mathsf{w}_1$ and $\mathsf{w}_2$. 

Although $\bar{\Pi}\defeq\bar{\Pi}_1+\bar{\Pi}_2$ is not a solution of Maurer-Cartan, by Kuranishi's method we can construct $\Xi=\sum_{j\geq 2}\Xi^{(j)}t^j$ such that the one form $\Phi\in\Omega^1(U,\End E\oplus TM)[\![ t ]\!]$ is $\Phi=\bar{\Pi}+\Xi$ and it is a solution of Maurer-Cartan up to higher order in $\hbar$. Indeed let us we write $\Phi$ as a formal power series in the parameter $t$, $\Phi=\sum_{j\geq 1}\Phi^{(j)}t^j$, then it is a solution of Maurer-Cartan if and only if:
\begin{equation*}
\begin{split}
&d_W\Phi^{(1)}=0\\
&d_W\Phi^{(2)}+\frac{1}{2}\lbrace\Phi^{(1)},\Phi^{(1)}\rbrace_\sim=0\\
&\vdots\\
&d_W\Phi^{(k)}+\frac{1}{2}\left(\sum_{s=1}^{k-1}\lbrace\Phi^{(s)},\Phi^{(k-s)}\rbrace_\sim\right)=0
\end{split}
\end{equation*}
 Moreover, recall from \eqref{rmk:closed} that $\bar{\Pi}_i\defeq\big(\Pi_{E,i}+\Pi_{E,R,i},\Pi^{CLM}_i\big)$, $i=1,2$ are solutions of the Maurer-Cartan equation and they are $d_W$-closed. Therefore at any order in the formal parameter $t$, the solution $\Phi=\bar{\Pi}+\Xi$ is computed as follows:
\begin{equation}\label{eq:recPhi}
\begin{split}
\Phi^{(1)}&=\bar{\Pi}^{(1)}\\
\Phi^{(2)}&=\bar{\Pi}^{(2)}+\Xi^{(2)}, \text{where } d_W\Xi^{(2)}=-\frac{1}{2}(\lbrace\Phi^{(1)},\Phi^{(1)}\rbrace_\sim) \\
\Phi^{(3)}&=\bar{\Pi}^{(3)}+\Xi^{(3)}, \text{where }  d_W\Xi^{(3)}=-\frac{1}{2}\left(\lbrace\Phi^{(1)},\bar{\Pi}^{(2)}+\Xi^{(2)}\rbrace_\sim+\lbrace\bar{\Pi}^{(2)}+\Xi^{(2)},\Phi^{(1)}\rbrace_\sim\right)\\
&\vdots\\
\Phi^{(k)}&=\bar{\Pi}^{(k)}+\Xi^{(k)}, \text{where } d_W\Xi^{(k)}=-\frac{1}{2}\left(\lbrace\Phi,\Phi\rbrace_\sim\right)^{(k)}. 
\end{split}
\end{equation}  
In order to explicitly compute $\Xi$ we want to ``invert'' the differential $d_W$ and this can be done by choosing a homotopy operator. Let us recall that a homotopy operator is a homotopy $H$ of morphisms $p$ and $\iota$, namely $H\colon \Omega^{\bullet}(U)\to \Omega^{\bullet}[-1](U)$, $p\colon\Omega^{\bullet}(U)\to H^{\bullet}(U)$ and $\iota\colon H^{\bullet}(U)\to\Omega^{\bullet}(U)$ such that $\text{id}_{\Omega^\bullet}-\iota\circ p=d_WH+Hd_W$. 
Let us now explicitly define the homotopy operator $\mathbf{H}$. Let $U$ be an open affine neighbourhood of $\xi_0=\mathfrak{d}_1\cap \mathfrak{d}_2$, and fix $q_0\in \left(H_{-,m_1}\cap H_{-,m_2}\right)\cap U$. Then choose a set of coordinates centred in $q_0$ and denote by $(u_{m},u_{m^\perp})$ a choice of such coordinates such that with respect to a ray $\mathfrak{d}_m=\xi_0+\R_{\geq 0}m$, $u_{m^\perp}$ is the coordinate orthogonal to $\mathfrak{d}_m$ and $u_{m}$ is tangential to $\mathfrak{d}_m$. 
Moreover recall the definition of morphisms $p$ and $\iota$, namely 
$p\defeq\bigoplus_m p_m$ and $p_m$ maps functions $\alpha\mathfrak{w}^m\in\Omega^0(U)\mathfrak{w}^m$
to $\alpha(q_0)\mathfrak{w}^m$, while $\iota\defeq\bigoplus_m\iota_m$ and $\iota_m$ is the embedding of constant function at degree zero, and it is zero otherwise. 

\begin{definition}\label{def:homotopyH}
The homotopy operator $\mathbf{H}=\bigoplus_m\mathbf{H}_m\colon\bigoplus_m\Omega^{\bullet}(U)\mathfrak{w}^m\to\bigoplus_m\Omega^{\bullet}(U)[-1]\mathfrak{w}^m$ is defined as follows. 
For any $0$-form $\alpha\in\Omega^0(U)$, $\mathbf{H}(\alpha\mathfrak{w}^m)=0$, since there are no degree $-1$-forms. 
For any $1$-form $\alpha\in\Omega^1(U)$, in local coordinates we have $\alpha=f_0(u_m,u_{m^\perp})du_m+f_1(u_m,u_{m^\perp})du_{m^\perp}$ and 
\[
\mathbf{H}(\alpha\mathfrak{w}^m)\defeq\mathfrak{w}^m\left(\int_0^{u_m}f_0(s,u_{m^\perp})ds+\int_0^{u_{m^\perp}}f_1(0,r)dr\right)
\] 

Finally since any $2$-forms $\alpha\in\Omega^2(U)$ in local coordinates can be written $\alpha=f(u_m,u_{m^\perp})du_m\wedge du_{m^\perp}$, then 
\[
\mathbf{H}(\alpha\mathfrak{w}^m)\defeq\mathfrak{w}^m\left(\int_0^{u_m}f(s,u_{m^\perp})ds\right)du_{m^\perp}.
\]
\end{definition}

The homotopy $\mathbf{H}$ seems defined \textit{ad hoc} for each degree of forms, however it can be written in an intrinsic way for every degree, as in Definition 5.12 \cite{MCscattering}. We have defined $\mathbf{H}$ in this way because it is clearer how to compute it in practice. 

\begin{lemma}
The following identity 
\begin{equation}\label{eq:idH2}
\text{id}_{\Omega^\bullet}-\mathbf{\iota}_m\circ p_m=\mathbf{H}_md_W+d_W\mathbf{H}_m
\end{equation}
holds true for all $m\in\Lambda$. 
\end{lemma}
\begin{proof} We are going to prove the identity separately for $0, 1$ and $2$ forms.

Let $\alpha=\alpha_0$ be of degree zero, then by definition $\mathbf{H}_m(\alpha\mathfrak{w}^m)=0$ and $\mathbf{\iota}_m\circ p_m(\alpha\mathfrak{w}^m)=\alpha_0(q_0)$. Then $d_W(\alpha\mathfrak{w}^m)=\mathfrak{w}^m\big(\frac{\partial\alpha_0}{\partial u_m}du_m+\frac{\partial\alpha_0}{\partial u_{m^\perp}}du_{m^\perp}\big)$. Hence
\[
\begin{split}
\mathbf{H}_md_W(\alpha_0\mathfrak{w}^m)&=\mathfrak{w}^m\int_{0}^{u_m}\frac{\partial\alpha_0(s,u_{m^\perp})}{\partial s}ds+\mathfrak{w}^m\int_0^{u_{m^\perp}}\frac{\partial\alpha_0}{\partial u_{m^\perp}}(0,r)dr \\
&=\mathfrak{w}^m[\alpha_0(u_m,u_{m^\perp})-\alpha_0(0,u_{m^\perp})+\alpha_0(0,u_{m^\perp})-\alpha_0(0,0)].
\end{split}
\]

Then consider $\alpha\in\Omega^1(U)\mathfrak{w}^m$. By definition $\mathbf{\iota}_m\circ p_m(\alpha\mathfrak{w}^m)=0$ and 
\[
\begin{split}
\mathbf{H}_md_W(\alpha\mathfrak{w}^m)&=\mathbf{H}d_W\left(\mathfrak{w}^m(f_0du_m+f_1du_{m^\perp})\right)\\
&=\mathbf{H}_m\left(\mathfrak{w}^m\left(\frac{\partial f_0}{\partial u_{m^\perp}}du_{m^\perp}\wedge du_m+\frac{\partial f_1}{\partial u_{m}}du_m\wedge du_{m^\perp}\right)\right)\\
&=\mathfrak{w}^m\left(\int_0^{u_m}\left(-\frac{\partial f_0}{\partial u_{m^\perp}}+\frac{\partial f_1}{\partial s}\right)ds\right)du_{m^\perp}\\
&=\left(-\int_0^{u_m}\frac{\partial f_0}{\partial u_{m^\perp}}(s,u_{m^\perp})ds +f_1(u_m,u_{m^\perp})-f_1(0,u_{m^\perp})\right)\mathfrak{w}^mdu_{m^\perp}
\end{split}
\] 
\[
\begin{split}
d_W\mathbf{H}_m(\alpha\mathfrak{w}^m)&=d_W\mathbf{H}(\mathfrak{w}^m\left(f_0du_m+f_1du_{m^\perp}\right))\\
&=d_W\left(\mathfrak{w}^m\left(\int_0^{u_m}f_0(s,u_{m^\perp})ds +\int_0^{u_{m^\perp}}f_1(0,r)dr\right)\right)\\
&=\mathfrak{w}^m d\left(\int_0^{u_m}f_0(s,u_{m^\perp})ds +\int_0^{u_{m^\perp}}f_1(0,r)dr\right)\\
&=\mathfrak{w}^m\left(f_0(u_m,u_{m^\perp})du_m+\frac{\partial}{\partial u_{m^\perp}}\left(\int_0^{u_m}f_0(s,u_{m^\perp})ds\right)du_{m^\perp}+f_1(0,u_{m^\perp})du_{m^\perp}\right).
\end{split}
\] 

We are left to prove the identity when $\alpha$ is of degree two: by degree reasons $d_W(\alpha\mathfrak{w}^m)=0$ and $\iota_m\circ p_m(\alpha\mathfrak{w}^m)=0$. Then 
\[
\begin{split}
d_W\mathbf{H}_m(\alpha\mathfrak{w}^m)&=\mathfrak{w}^md\left(\left(\int_{0}^{u_m}f(s,u_{m^\perp})ds\right)du_{m^\perp}\right)\\
&=\mathfrak{w}^mf(u_m,u_{m^\perp})du_m\wedge du_{m^\perp}. 
\end{split}
\] 
\end{proof} 

\begin{prop}[{see prop 5.1 in \cite{MCscattering}}]\label{prop:Kuranishi}
Assume that $\Phi$ is a solution of 
\begin{equation}\label{eq:Phi}
\Phi=\bar{\Pi}-\frac{1}{2}\mathbf{H}\left(\lbrace\Phi,\Phi\rbrace_\sim\right)   
\end{equation}
Then $\Phi$ is a solution of the Maurer-Cartan equation.
\end{prop}
\begin{proof}
First notice that by definition $p(\lbrace\Phi,\Phi\rbrace_\sim)=0$ and by degree reasons $d_W(\lbrace\Phi,\Phi\rbrace_\sim)=0$ too. Hence by identity \eqref{eq:idH2} we get that $\lbrace\Phi,\Phi,\rbrace_\sim=d_W\mathbf{H}(\lbrace\Phi,\Phi\rbrace_\sim)$, and if $\Phi$ is a solution of equation \eqref{eq:Phi} then $d_W\Phi=d_W\bar{\Pi}-\frac{1}{2}d_W\mathbf{H}(\lbrace\Phi,\Phi\rbrace_\sim)=-\frac{1}{2}d_W\mathbf{H}(\lbrace\Phi,\Phi\rbrace_\sim)$. 
\end{proof}
From now on we will look for solutions $\Phi$ of equation \eqref{eq:Phi} rather than to the Maurer-Cartan equation. The advantage is that we have an integral equation instead of a differential equation, and $\Phi$ can be computed by its expansion in the formal parameter $t$, namely $\Phi=\sum_{j\geq 1}\Phi^{(j)}t^j$.  
\begin{notation}
Let $\mathfrak{d}_{m_1}=\xi_0-m_1\mathbb{R}$ and $\mathfrak{d}_{m_2}=\xi_0-m_2\mathbb{R}$ and let $(u_{m_1},u_{m_1^\perp})$ and $(u_{m_2},u_{m_2^\perp})$ be respectively two basis of coordinates in $U$, centred in $q_0$ as above. Let $m_a\defeq a_1m_1+a_2m_2$, consider the ray $\mathfrak{d}_{m_a}\defeq \xi_0-m_a\mathbb{R}_{\geq 0}$ and choose coordinates $u_{m_{a}}\defeq \left(-a_2u_{m_1^\perp}+a_1u_{m_2^\perp}\right)$ and $u_{m_a^\perp}\defeq\left(a_1u_{m_1^\perp}+a_2u_{m_2^\perp}\right)$.    
\end{notation}
\begin{remark}\label{rmk:H(delta wedge delta)}
If $\alpha=\delta_{m_1}\wedge\delta_{m_2}$, then by the previous choice of coordinates 
\[\delta_{m_1}\wedge\delta_{m_2}=\frac{e^{-\frac{u_{m_1^\perp}^2+u_{m_2^\perp}^2}{\hbar}}}{{\hbar\pi}}du_{m_1^\perp}\wedge du_{m_2^\perp}=\frac{e^{-\frac{u_{m_{a}}^2+u_{m_a^\perp}^2}{(a_1^2+a_2^2)\hbar}}}{{\hbar\pi}}du_{u_{m_a}^\perp}\wedge du_{m_{a}}.\]
In particular we explicitly compute $\mathbf{H}(\delta_{m_1}\wedge\delta_{m_2}\mathfrak{w}^{lm_a})$:
\begin{equation}
\begin{split}
\mathbf{H}(\alpha\mathfrak{w}^{lm_a})&=\mathfrak{w}^{lm_a}\left(\int_{0}^{u_{m_{a}}}\frac{e^{-\frac{s^2+u_{m_a^\perp}^2}{(a_1^2+a_2^2)\hbar}}}{\hbar\pi}ds\right) du_{u_a^\perp}=\mathfrak{w}^{lm_a}\left(\int_{0}^{u_{m_{a}}}\frac{e^{-\frac{s^2}{(a_1^2+a_2^2)\hbar}}}{\sqrt{\hbar\pi}}ds\right) \frac{e^{-\frac{u_{m_a^{\perp}}^2}{(a_1^2+a_2^2)\hbar}}}{\sqrt{\hbar\pi}} du_{m_a^{\perp}}
\end{split}
\end{equation}
Hence $\mathbf{H}\big(\delta_{m_1}\wedge\delta_{m_2}\mathfrak{w}^{lm_a}\big)=f(\hbar,u_{m_{a}})\delta_{m_a}$ where $f(\hbar,u_{m_{a}})=\int_{0}^{u_{m_{a}}}\frac{e^{-\frac{s^2}{(a_1^2+a_2^2)\hbar}}}{\sqrt{\hbar\pi}}ds\in O_{loc}(1)$.
\end{remark}

In order to construct a consistent scattering diagram from the solution $\Phi$ we introduce labeled ribbon trees. Indeed via the combinatorial of such trees we  can rewrite $\Phi$ as a sum over primitive Fourier mode, coming from the contribution of the out-going edge of the trees. 
\subsubsection{Labeled ribbon trees} 
Let us briefly recall the definition of labeled ribbon trees, which was introduced in \cite{MCscattering}. 
\begin{definition}[Definition 5.2 in \cite{MCscattering}]\label{def:trees}
A k-tree $T$ is the datum of a finite set of vertices $V$, together with a decomposition $V=V_{in}\sqcup V_0\sqcup\lbrace v_T\rbrace$, and a finite set of edges $E$, such that, given the two boundary maps $\partial_{in},\partial_{out}:E\rightarrow V$ (which respectively assign to each edge its incoming and outgoing vertices), satisfies the following assumption:
\begin{enumerate}
\item $\#V_{in}=k$ and for any vertex $v\in V_{in}$, $\#\partial_{in}^{-1}(v)=0$ and $\#\partial_{out}^{-1}(v)=1$; 
\item for any vertex $v\in V_0$, $\#\partial_{in}^{-1}(v)=1$ and $\#\partial_{out}^{-1}(v)=2$;
\item $v_T$ is such that $\#\partial_{in}^{-1}(v_T)=0$ and $\#\partial_{out}^{-1}(v)=1$.
\end{enumerate}
We also define $e_T=\partial_{in}^{-1}(v_T)$. 
\end{definition}
\begin{figure}[h]
\center
\begin{tikzpicture}
\draw (0,4) -- (1,3);
\draw (2,4) -- (1,3);
\draw (1,3) -- (2,2);
\draw (3,3) -- (2,2);
\draw (2,2) -- (2,1);
\node[below, font=\tiny] at (2,1) {$v_T$};
\node[above, font=\tiny] at (0,4) {$v_1$};
\node[above, font=\tiny] at (2,4) {$v_2$};
\node[below, font=\tiny] at (1,3) {$v_3$};
\node[above, font=\tiny] at (3,3) {$v_4$};
\node[right, font=\tiny] at (2,2) {$v_5$};
\node[font=\tiny] at (2,1) {$\bullet$};
\node[ font=\tiny] at (0,4) {$\bullet$};
\node[font=\tiny] at (2,4) {$\bullet$};
\node[font=\tiny] at (1,3) {$\bullet$};
\node[font=\tiny] at (3,3) {$\bullet$};
\node[font=\tiny] at (2,2) {$\bullet$};
\node[right, font=\tiny] at (2,1.5) {$e_T$};
\end{tikzpicture}
\caption{This is an example of a 3-tree, where the set of vertices is decomposed by $V_{in}=\lbrace v_1, v_2, v_4\rbrace$, $V_0=\lbrace v_3, v_5\rbrace$.}
\label{t:3tree}
\end{figure}
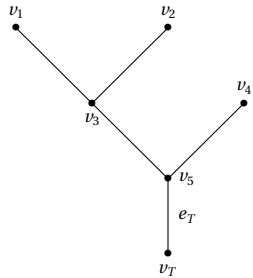
Two k-trees $T$ and $T'$ are \textit{isomorphic} if there are bijections $V\cong V'$ and $E\cong E'$ preserving the decomposition $V_0\cong V'_0$, $V_{in}\cong V'_{in}$ and $\lbrace v_T\rbrace\cong\lbrace v_{T'}\rbrace$ and the boundary maps $\partial_{in}, \partial_{out}$. 
  
It will be useful in the following to introduce the definition of topological realization $\mathcal{T}(T)$ of a k-tree $T$, namely $\mathcal{T}(T)\defeq\big(\coprod_{e\in E}[0,1]\big)/\thicksim$, where $\thicksim$ is the equivalence relation that identifies boundary points of edges with the same image in $V\setminus\lbrace v_T\rbrace$. 

Since we need to keep track of all the possible combinations while we compute commutators (for instance for $\Phi^{(3)}$  there is the contribution of $\lbrace\Phi^{(1)},\Phi^{(2)}\rbrace_\sim$ and $\lbrace\Phi^{(2)},\Phi^{(1)}\rbrace_\sim$), we introduce the notion of ribbon trees:
\begin{definition}[Definition 5.3 in \cite{MCscattering}]\label{def:ribbon trees}
A ribbon structure on a k-tree is a cyclic ordering of the vertices. It can be viewed as an embedding $\mathcal{T}(T)\hookrightarrow\mathbb{D}$, where $\mathbb{D}$ is the disk in $\mathbb{R}^2$, and the cyclic ordering is given according to the anticlockwise orientation of $\mathbb{D}$. 
\end{definition}
Two ribbon k-trees $T$ and $T'$ are \textit{isomorphic} if they are isomorphic as k-trees and the isomorphism preserves the cyclic order. The set of all isomorphism classes of ribbon k-trees will be denoted by $\mathbb{R}\mathbb{T}_k$. 
As an example, the following two 2-trees are not isomorphic:
\begin{figure}[h]
\begin{tikzpicture}
\draw (0,2) -- (1,1);
\draw (2,2) -- (1,1);
\draw (1,1) -- (1,0);
\node[below, font=\tiny] at (1,0) {$v_T$};
\node[above, font=\tiny] at (0,2) {$v_1$};
\node[above, font=\tiny] at (2,2) {$v_2$};
\node[right, font=\tiny] at (1,1) {$v_3$};
\node[font=\tiny] at (1,0) {$\bullet$};
\node[ font=\tiny] at (0,2) {$\bullet$};
\node[font=\tiny] at (2,2) {$\bullet$};
\node[font=\tiny] at (1,1) {$\bullet$};
\node[right, font=\tiny] at (1,0.5) {$e_T$};
\end{tikzpicture}
\begin{tikzpicture}
\draw (0,2) -- (1,1);
\draw (2,2) -- (1,1);
\draw (1,1) -- (1,0);
\node[below, font=\tiny] at (1,0) {$v_T'$};
\node[above, font=\tiny] at (0,2) {$v_2$};
\node[above, font=\tiny] at (2,2) {$v_1$};
\node[right, font=\tiny] at (1,1) {$v_4$};
\node[font=\tiny] at (1,0) {$\bullet$};
\node[ font=\tiny] at (0,2) {$\bullet$};
\node[font=\tiny] at (2,2) {$\bullet$};
\node[font=\tiny] at (1,1) {$\bullet$};
\node[right, font=\tiny] at (1,0.5) {$e_T'$};
\end{tikzpicture}
\end{figure}

In order to keep track of the $\hbar$ behaviour while we compute the contribution from the commutators, let us decompose the bracket on the DGLA as follows:
\begin{definition}\label{label}
Let $(A,\alpha)=(A_Jdx^J\mathfrak{w}^{m_1},\alpha_Jdx^J\mathfrak{w}^{m_1}\partial_{n_1})\in\Omega^p(U,\End E\oplus TM)\mathfrak{w}^{m_1}$ and $(B,\beta)=(B_Kdx^K\mathfrak{w}^{m_2},\beta_Kdx^K\mathfrak{w}^{m_2}\partial_{n_2})\in\Omega^q(U,\End E\oplus TM)\mathfrak{w}^{m_2}$. Then we decompose $\lbrace\cdot ,\cdot\rbrace_\sim$ as the sum of: 
\begin{enumerate}
\item[$\natural$] $\lbrace(A,\alpha),(B,\beta)\rbrace_{\natural}\defeq\left(\alpha\wedge B\langle n_1,m_2\rangle- \beta\wedge A\langle n_2,m_1\rangle+ \lbrace A,B\rbrace_{\End E},\lbrace\alpha,\beta\rbrace\right)$
\item[$\flat$] $\lbrace(A,\alpha),(B,\beta)\rbrace_{\flat}\defeq\left(i\hbar \beta_Kn_2^q\frac{\partial A_J}{\partial x_q}dx^J\wedge dx^K\mathfrak{w}^{m_1+m_2}, \beta(\nabla_{\partial_{n_2}}\alpha)\mathfrak{w}^{m_1+m_2}\partial_{n_1}\right)$
\item[$\sharp$] $\lbrace(A,\alpha),(B,\beta)\rbrace_{\sharp}\defeq\left(i\hbar \alpha_Jn_1^q\frac{\partial B_K}{\partial x_q}dx^K\wedge dx^J\mathfrak{w}^{m_1+m_2}, \alpha(\nabla_{\partial_{n_1}}\beta)\mathfrak{w}^{m_1+m_2}\partial_{n_2}\right)$
\item[$\star$] $\lbrace(A,\alpha),(N,\beta)\rbrace_{\star}\defeq i\hbar\left(\alpha_Jn_1^qB_KA_q(\phi)dx^J\wedge dx^K-n_2^q\beta_KA_q(\phi)A_Jdx^K\wedge  dx^J,0\right)$.
\end{enumerate}
\end{definition}

The previous definition is motivated by the following observation: the label $\natural$ contains terms of the Lie bracket $\lbrace\cdot,\cdot\rbrace_\sim$ which leave unchanged the behaviour in $\hbar$. Then both the labels $\flat$ and $\sharp$ contain terms which contribute with an extra $\hbar$ factor and at the same time contain derivatives. The last label $\star$ contains terms which contribute with an extra $\hbar$ but do not contain derivatives. 
 
\begin{definition}\label{def:labeled ribbon tree}
A labeled ribbon k-tree is a ribbon k-tree $T$ together with:
\begin{enumerate}
\item[(i)] a label $\natural$, $\sharp$, $\flat$, $\star$ -as defined in Definition \ref{label}- for each vertex in $V_0$;
\item[(ii)] a label $(m_e,j_e)$ for each incoming edge $e$, where $m_e$ is the Fourier mode of the incoming vertex and $j_e\in\Z_{>0}$ gives the order in the formal parameter $t$.
\end{enumerate} 
\end{definition}
There is an induced labeling of all the edges of the trees defined as follows: at any trivalent vertex with incoming edges $e_1,e_2$ and outgoing edge $e_3$ we define $(m_{e_3},j_{e_3})=(m_{e_1}+m_{e_2}, j_{e_1}+j_{e_2})$. We will denote by $(m_T,j_T)$ the label corresponding to the unique incoming edge of $\nu_T$. 
Two labeled ribbon k-trees $T$ and $T'$ are \textit{isomorphic} if they are isomorphic as ribbon k-trees and the isomorphism preserves the labeling. The set of equivalence classes of labeled ribbon k-trees will be denoted by $\mathbb{L}\mathbb{R}\mathbb{T}_k$. We also introduce the following notation for equivalence classes of labeled ribbon trees:
\begin{notation}
We denote by $\mathbb{LRT}_{k,0}$ the set of equivalence classes of $k$ labeled ribbon trees such that they have only the label $\natural$. We denote by $\mathbb{LRT}_{k,1}$ the complement set, namely $\mathbb{LRT}_{k,1}=\mathbb{LRT}_k-\mathbb{LRT}_{k,0}$.
\end{notation}
Let us now define the operator $\mathfrak{t}_{k,T}$ which allows to write the solution $\Phi$ in terms of labeled ribbon trees.
\begin{definition}
Let $T$ be a labeled ribbon k-tree, then the operator 
\begin{equation}
\mathfrak{t}_{k,T}:\Omega^1(U,\End E\oplus TM)^{\otimes k}\rightarrow\Omega^{1}(U, \End E\oplus TM)
\end{equation}
is defined as follows: it aligns the input with the incoming vertices according with the cyclic ordering and it labels the incoming edges (as in part (ii) of Definition \ref{def:labeled ribbon tree}). Then it assigns at each vertex in $V_0$ the commutator according with the part (i) of Definition \ref{label}. Finally it assigns the homotopy operator $-\mathbf{H}$ to each outgoing edge. 
\end{definition}
In particular the solution $\Phi$ of equation \eqref{eq:Phi} can be written as a sum on labeled ribbon k-trees as follows:
\begin{equation}\label{def:Phitree}
\Phi=\sum_{k\geq 1}\sum_{T\in\mathbb{LRT}_k}\frac{1}{2^{k-1}}\mathfrak{t}_{k,T}(\bar{\Pi},\cdots ,\bar{\Pi}).
\end{equation} 
Recall that by definition 
\[\lbrace\left(A,\alpha\partial_{n_1}\right)\mathfrak{w}^{k_1m_1},\left(B,\beta\partial_{n_2}\right)\mathfrak{w}^{k_2m_2}\rbrace_\sim=\left(C,\gamma\partial_{\langle k_2m_2,n_1\rangle n_2-\langle k_1m_1,n_2\rangle n_1}\right)\mathfrak{w}^{k_1m_1+k_2m_2}\]
for some $(A,\alpha)\in\Omega^s(U,\End E\oplus TM), (B,\beta)\Omega^r(U,\End E\oplus TM)$ and $(C,\gamma)\in\Omega^{r+s}(U,\End E\oplus TM)$, hence the Fourier mode of any labeled brackets has the same frequency $m_e=k_1m_1+k_2m_2$ independently of the label $\natural, \flat, \sharp,\star$. In particular each $m_e$ can be written as $m_e=l (a_1m_1+a_2m_2)$ for some primitive elements $(a_1,a_2)\in\big(\Z^2_{\geq 0}\big)_{\text{prim}}$. Let us introduce the following notation: 
\begin{notation}
Let $a=(a_1,a_2)\in\left(\Z_{\geq0}^2\right)_{\text{prim}}$ and define $m_a\defeq a_1m_1+a_2m_2$. Then we define $\Phi_a$ to be the sum over all trees of the contribution to $\mathfrak{t}_{k,T}(\bar{\Pi},\cdots, \bar{\Pi})$ with Fourier mode $\mathfrak{w}^{lm_a}$ for every $l\geq1$. In particular we define $\Phi_{(1,0)}\defeq\bar{\Pi}_1$ and $\Phi_{(0,1)}\defeq\bar{\Pi}_2$. 
\end{notation}
It follows that the solution $\Phi$ can be written as a sum on primitive Fourier mode as follows:
\begin{equation}\label{eq:Phi_primitive}
    \Phi=\sum_{a\in\big(\Z^2_{\geq 0}\big)_{\text{prim}}}\Phi_{a}.
\end{equation}

As an example, let us consider $\Phi^{(2)}$. From equation \eqref{eq:Phi} we get   
\[
\Phi^{(2)}=\bar{\Pi}^{(2)}-\frac{1}{2}\mathbf{H}\left(\lbrace\bar{\Pi}^{(1)},\bar{\Pi}^{(1)}\rbrace_\sim\right)\]
and the possible 2-trees $T\in\mathbb{LRT}_2$, up to choice of the initial Fourier modes, are represented in figure \ref{fig:2trees}.    
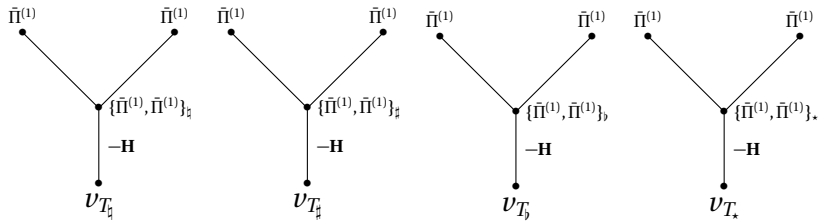
\begin{figure}[h]
\begin{tikzpicture}
\draw (0,2) -- (1,1);
\draw (2,2) -- (1,1);
\draw (1,1) -- (1,0);
\node[below] at (1,0) {$v_{T_\natural}$};
\node[above, font=\tiny] at (0,2) {$\bar{\Pi}^{(1)}$};
\node[above, font=\tiny] at (2,2) {$\bar{\Pi}^{(1)}$};
\node[right, font=\tiny] at (1,1) {$\lbrace \bar{\Pi}^{(1)},\bar{\Pi}^{(1)}\rbrace_\natural $};
\node[font=\tiny] at (1,0) {$\bullet$};
\node[ font=\tiny] at (0,2) {$\bullet$};
\node[font=\tiny] at (2,2) {$\bullet$};
\node[font=\tiny] at (1,1) {$\bullet$};
\node[right, font=\tiny] at (1,0.5) {$-\mathbf{H}$};
\end{tikzpicture}
\begin{tikzpicture}
\draw (0,2) -- (1,1);
\draw (2,2) -- (1,1);
\draw (1,1) -- (1,0);
\node[below] at (1,0) {$v_{T_\sharp}$};
\node[above, font=\tiny] at (0,2) {$\bar{\Pi}^{(1)}$};
\node[above, font=\tiny] at (2,2) {$\bar{\Pi}^{(1)}$};
\node[right, font=\tiny] at (1,1) {$\lbrace \bar{\Pi}^{(1)},\bar{\Pi}^{(1)}\rbrace_\sharp $};
\node[font=\tiny] at (1,0) {$\bullet$};
\node[ font=\tiny] at (0,2) {$\bullet$};
\node[font=\tiny] at (2,2) {$\bullet$};
\node[font=\tiny] at (1,1) {$\bullet$};
\node[right, font=\tiny] at (1,0.5) {$-\mathbf{H}$};
\end{tikzpicture}
\begin{tikzpicture}
\draw (0,2) -- (1,1);
\draw (2,2) -- (1,1);
\draw (1,1) -- (1,0);
\node[below] at (1,0) {$v_{T_\flat}$};
\node[above, font=\tiny] at (0,2) {$\bar{\Pi}^{(1)}$};
\node[above, font=\tiny] at (2,2) {$\bar{\Pi}^{(1)}$};
\node[right, font=\tiny] at (1,1) {$\lbrace \bar{\Pi}^{(1)},\bar{\Pi}^{(1)}\rbrace_\flat $};
\node[font=\tiny] at (1,0) {$\bullet$};
\node[ font=\tiny] at (0,2) {$\bullet$};
\node[font=\tiny] at (2,2) {$\bullet$};
\node[font=\tiny] at (1,1) {$\bullet$};
\node[right, font=\tiny] at (1,0.5) {$-\mathbf{H}$};
\end{tikzpicture}
\begin{tikzpicture}
\draw (0,2) -- (1,1);
\draw (2,2) -- (1,1);
\draw (1,1) -- (1,0);
\node[below] at (1,0) {$v_{T_\star}$};
\node[above, font=\tiny] at (0,2) {$\bar{\Pi}^{(1)}$};
\node[above, font=\tiny] at (2,2) {$\bar{\Pi}^{(1)}$};
\node[right, font=\tiny] at (1,1) {$\lbrace \bar{\Pi}^{(1)},\bar{\Pi}^{(1)}\rbrace_\star $};
\node[font=\tiny] at (1,0) {$\bullet$};
\node[ font=\tiny] at (0,2) {$\bullet$};
\node[font=\tiny] at (2,2) {$\bullet$};
\node[font=\tiny] at (1,1) {$\bullet$};
\node[right, font=\tiny] at (1,0.5) {$-\mathbf{H}$};
\end{tikzpicture}
\caption{$2$-trees labeled ribbon trees, which contribute to the solution $\Phi$.}\label{fig:2trees}
\end{figure}
Hence 
\begin{equation*}
\begin{split}
\Phi^{(2)}&=\bar{\Pi}_1^{(2)}+\bar{\Pi}_2^{(2)}-\frac{1}{2}\mathbf{H}\left(\lbrace \bar{\Pi}^{(1)},\bar{\Pi}^{(1)}\rbrace_\natural+\lbrace \bar{\Pi}^{(1)},\bar{\Pi}^{(1)}\rbrace_\sharp+\lbrace \bar{\Pi}^{(1)},\bar{\Pi}^{(1)}\rbrace_\flat+\lbrace \bar{\Pi}^{(1)},\bar{\Pi}^{(1)}\rbrace_\star \right)\\
&=\Phi_{(1,0)}^{(2)}+\Phi_{(0,1)}^{(2)}+\\
&-\Big(\left(a_{1,k_1}A_{1,k_2}\langle n_1,k_2m_2\rangle- a_{1,k_2}A_{1,k_1}\langle n_2,k_1m_1\rangle+[A_{1,k_1},A_{1,k_2}]\right), a_{1,k_1} a_{1,k_2}\partial_{\langle k_2m_2,n_1\rangle n_2-\langle k_1m_1,n_2\rangle n_1}\Big)\cdot\\
&\cdot\mathbf{H}(\delta_{m_1}\wedge \delta_{m_2}\mathfrak{w}^{k_1m_1+k_2m_2})\\
&+\Big(a_{1,k_2}A_{1,k_1}, -a_{1,k_2}a_{1,k_1}\partial_{n_1}\Big)\mathbf{H}\left(i\hbar n_2^q\frac{\partial \delta_{m_1}}{\partial x_q}\wedge \delta_{m_2}\mathfrak{w}^{k_1m_1+k_2m_2}\right)\\
&+\Big(a_{1,k_1}A_{1,k_2}, a_{1,k_1}a_{1,k_2}\partial_{n_2}\Big)\mathbf{H}\left(i\hbar n_1^q\frac{\partial \delta_{m_2}}{\partial x_q}\wedge \delta_{m_1}\mathfrak{w}^{k_1m_1+k_2m_2}\right)\\
&+i\hbar\Big(a_{1,k_1}A_{1,k_2}n_1^q\mathbf{H}\left(A_q(\phi)\mathfrak{w}^{k_1m_1+m_2}\delta_{m_1}\wedge \delta_{m_2}\right)-a_{1,k_2} A_{1}n_2^q\mathbf{H}\left(A_q(\phi)\mathfrak{w}^{k_1m_1+k_2m_2}\delta_{m_1}\wedge \delta_{m_2}\right) ,0\Big)
\end{split}
\end{equation*}
By Remark \ref{rmk:H(delta wedge delta)}
\[\mathbf{H}(\delta_{m_1}\wedge \delta_{m_2}\mathfrak{w}^{k_1m_1+k_2m_2})=f(\hbar,u_{m_{a}})\mathfrak{w}^{lm_a}\delta_{m_a}\] and $f(\hbar,u_{m_{a}})\in O_{loc}(1)$, for $k_1m_1+k_2m_2=lm_a$. Analogously \[\mathbf{H}\left(A_q(\phi)\mathfrak{w}^{k_1m_1+k_2m_2}\delta_{m_1}\wedge \delta_{m_2}\right)=f(\hbar,A(\phi),u_{m_{a}})\mathfrak{w}^{lm_a}\delta_{m_a}\] and $f(\hbar,A(\phi),u_{m_{a}})\in O_{loc}(1)$. Then \[\mathbf{H}\left(\hbar\frac{\partial \delta_{m_1}}{\partial x_q}\wedge \delta_{m_2}\mathfrak{w}^{k_1m_1+k_2m_2}\right)=\mathfrak{w}^{lm_a} f(\hbar,u_{m_{a}})\delta_{m_a}\] and $f(\hbar,u_{m_{a}})\in O_{loc}(\hbar^{1/2})$. This shows that every term in the sum above, is a function of some order in $\hbar$ times a delta supported along a ray of slope $m_{(a_1,a_2)}=a_1m_1+a_2m_2$. For any given $a\in\left(\Z^2_{\geq 0}\right)_{\text{prim}}$, these contributions are by definition $\Phi_{(a_1,a_2)}^{(2)}$, hence 
\[\Phi^{(2)}=\Phi_{(1,0)}^{(2)}+\sum_{a\in\left(\Z^2_{\geq0}\right)_{\text{prim}}}\Phi_{(a_1,a_2)}^{(2)}+\Phi_{(0,1)}^{(2)}.\]
In general, the expression of $\Phi_a$ will be much more complicated, but as a consequence of the definition of $\mathbf{H}$ it always contains a delta supported on a ray of slope $m_a$.

  
\subsection{From solutions of Maurer-Cartan to the consistent diagram $\mathfrak{D}^{\infty}$}

Let us first introduce the following notation:
\begin{notation}
Let $A\defeq U\setminus\lbrace \xi_0\rbrace$ be an annulus and let $\tilde{A}$ be the universal cover of $A$ with projection $\varpi\colon\tilde{A}\to A$. Then let us denote by $\tilde{\Phi}$ the pullback of $\Phi$ by $\varpi$, in particular by the decomposition in its primitive Fourier mode $\tilde{\Phi}=\sum_{a\in\left(\Z^2_{\geq 0}\right)}\varpi^*(\Phi_a)\defeq\sum_{a\in\left(\Z^2_{\geq 0}\right)}\tilde{\Phi}_a$. 
\end{notation}
\begin{notation}
We introduce polar coordinates in $\xi_0$, centred in $\xi_0=\mathfrak{d}_{m_1}\cup \mathfrak{d}_{m_2}$, denoted by $(r, \vartheta)$ and we fix a reference angle $\vartheta_0$ such that the ray with slope $\vartheta_0$ trough $\xi_0$ contains the base point $q_0$ (see Figure \ref{fig:rif_theta}). 

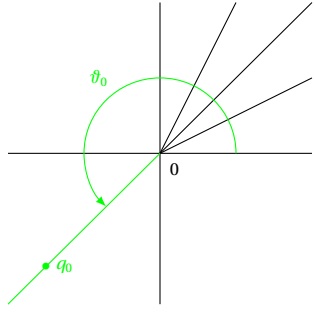
\begin{figure}[h]
\begin{tikzpicture}
\draw (2,0) -- (2,4);
\draw (0,2) -- (4,2);
\draw (2,2) -- (4,4) ;
\draw (2,2) -- (3,4) ;
\draw (2,2) -- (4,3) ;
\draw [green] (2,2) -- (0,0);
\node [below right, font=\tiny] at (2,2) {0};
\node[green, right, font=\tiny] at (0.5,0.5) {$q_0$};
\node[ green, font=\tiny] at (0.5,0.5) {$\bullet$};
\draw [-latex, green] (3,2) arc (0:225:1);
\node[green,font=\tiny] at (1.2,3) {$\vartheta_0$}; 
\end{tikzpicture}
\caption{The reference angle $\vartheta_0$.}
\label{fig:rif_theta}
\end{figure}

Then for every $a\in\left(\Z_{\geq0}^2\right)_{\text{prim}}$ we associate to the ray $\mathfrak{d}_{m_a}\defeq \xi_0+\R_{\geq 0}m_a$ an angle $\vartheta_a\in\left(\vartheta_0,\vartheta_0+2\pi\right)$. We identify $\mathfrak{d}_{m_a}\cap A$ with its lifting in $\tilde{A}$ and by abuse of notation we will denote it by $\mathfrak{d}_{m_a}$. We finally define $\tilde{A}_0\defeq\lbrace (r, \vartheta)\vert \vartheta_0-\epsilon_0<\vartheta <\vartheta_0+2\pi\rbrace$, for some small positive $\epsilon_0$.  
\end{notation}
   
\begin{lemma}[see Lemma 5.40 \cite{MCscattering}]\label{lem:Phi_aMC}
Let $\Phi$ be a solution of equation \eqref{eq:Phi} which has been decomposed as a sum over primitive Fourier mode, as in \eqref{eq:Phi_primitive}. Then for any $a\in\big(\Z_{\geq 0}^2\big)_{\text{prim}}$, $\tilde{\Phi}_{a}$ is a solution of the Maurer-Cartan equation in $\tilde{A}_0$, up to higher order in $\hbar$, namely $\lbrace\tilde{\Phi}_{a},\tilde{\Phi}_{a'}\rbrace_\sim\in\mathcal{W}_2^{-\infty}(\tilde{A}_0,\End E\oplus TM)\mathfrak{w}^{km_a+k'm_{a'}}$ and $d_W\tilde{\Phi}_{a}\in\mathcal{W}_{2}^{-\infty}(\tilde{A}_0,\End E\oplus TM)\mathfrak{w}^{km_a}$ for some $k,k'\geq 1$.
\end{lemma}

\begin{proof}
Recall that $\Phi$ is a solution of Maurer-Cartan $d_W\Phi+\lbrace\Phi,\Phi\rbrace_\sim=0$, hence its pullback $\varpi^*(\Phi)$ is such that $\sum_{a\in\big(\Z_{\geq0}^2\big)_{\text{prim}}}d_W\tilde{\Phi}_{a}+\sum_{a,a'\in\big(\Z_{\geq0}^2\big)_{\text{prim}}}\lbrace\tilde{\Phi}_{a},\tilde{\Phi}_{a'}\rbrace_\sim=0$. 
Looking at the bracket there are two possibilities: first of all, if $a\neq a'$ then $\lbrace\tilde{\Phi}_{a},\tilde{\Phi}_{a'}\rbrace_\sim$  is proportional to $\delta_{m_{a}}\wedge\delta_{m_{a'}}\mathfrak{w}^{km+k'm'}$. Since $\mathfrak{d}_{m_a}\cap \mathfrak{d}_{m{_a'}}\cap \tilde{A}=\varnothing$ then $\delta_{m_{a}}\wedge\delta_{m_{a'}}\in\mathcal{W}_2^{-\infty}(\tilde{A}_0)$, indeed writing $\delta_{m_{a}}\wedge\delta_{m_{a'}}$ in polar coordinates it is a $2$-form with coefficient a 
Gaussian function in two variables centred in $\xi_0\not\in\tilde{A}_0$. Hence it is bounded by $e^{-\frac{c_V}{\hbar}}$ in the open subset $V\subset\tilde{A}_0$.     
Secondly, if $a=a'$ then by definition $\lbrace\tilde{\Phi}_{a},\tilde{\Phi}_{a}\rbrace_\sim=0$. Finally, by the fact that $d_W\tilde{\Phi}_a=-\sum_{a',a''}\lbrace\tilde{\Phi}_{a'},\tilde{\Phi}_{a''}\rbrace_\sim$ it follows that $d_W\tilde{\Phi}_{a}\in\mathcal{W}_{2}^{-\infty}(\tilde{A}_0)\mathfrak{w}^{km_a}$ for some $k\geq 1$. 
\end{proof}


Now recall that the homotopy operator we have defined in Section \ref{sec:single wall} gives a gauge fixing condition, hence for every $a\in\left(\Z^2_{\geq 0}\right)_{\text{prim}}$ there exists a unique gauge $\varphi_a$ such that $e^{\varphi_a}\ast 0=\tilde{\Phi}_a$ and $p(\varphi_a)=0$. To be more precise we should consider $\tilde{p}\defeq\varpi^*( p)$ as gauge fixing condition and similarly $\tilde{\iota}\defeq\varpi^*(\iota)$ and $\tilde{H}\defeq\varpi^*(H)$ as homotopy operator, however if we consider affine coordinates on $\tilde{A}$, these operators are equal to $p$, $\iota$ and $H$ respectively. In addition in affine coordinates on $\tilde{A}$ the solution $\tilde{\Phi}_a$ is also equal to $\Phi_a$. Therefore in the following computations we will always use the original operators and the affine coordinates on $\tilde{A}$. We compute the asymptotic behaviour of the gauge $\varphi_a$ in the following theorem:

\begin{theorem}\label{thm:asymptotic_gauge}
Let $\varphi_a\in \Omega^0(\tilde{A}_0,\End E\oplus TM)$ be the unique gauge such that $p(\varphi_a)=0$ and $e^{\varphi_a}\ast 0=\tilde{\Phi}_{a}$. Then the asymptotic behaviour of $\varphi_a$ is 
\[
\varphi_a^{(s)}\in\begin{cases}\sum_{l}
\left(B_l, b_l\partial_{n_a}\right)t^{s}\mathfrak{w}^{l m_a}+\bigoplus_{l\geq 1}\mathcal{W}_0^{0}(\tilde{A}_0,\End E\oplus TM)\mathfrak{w}^{lm_a} & \text{on }  H_{m_a,+}\\
\bigoplus_{l\geq 1}\mathcal{W}_0^{-\infty}(\tilde{A}_0,\End E\oplus TM)\mathfrak{w}^{lm_a} & \text{on }  H_{m_a,-}
\end{cases}\]
for every $s\geq 0$, where $\big(B_l, b_l\partial_{n_a}\big)\mathfrak{w}^{lm_a}\in\tilde{\mathfrak{h}}$.
\end{theorem}
\begin{remark}
Notice that, from Theorem \ref{thm:asymptotic_gauge} the gauge $\varphi_a$ is asymptotically an element of the DGLA $\tilde{\mathfrak{h}}$. Hence the saturated scattering diagram (see Definition \ref{def:D_infty}) is strictly contained in the mirror DGLA $G$ (see Definition \ref{def:mirror dgLa}). 
\end{remark}
We first need the following lemma (for a proof see Lemma 5.27 \cite{MCscattering}
) which gives the explicit asymptotic behaviour of each component of the Lie bracket $\lbrace\cdot,\cdot\rbrace_\sim$:
\begin{lemma}\label{lem:bracket natural}
Let $\mathfrak{d}_m$ and $\mathfrak{d}_{m'}$ be two rays on $U$ such that $\mathfrak{d}_{m_a}\cap \mathfrak{d}_{m_{a'}}=\lbrace \xi_0\rbrace$. 

If $(A\mathfrak{w}^{m},\alpha\mathfrak{w}^{m}\partial_n)\in\mathcal{W}_{0}^s(U,\End E\oplus TM)\mathfrak{w}^{m}$ and $(B\mathfrak{w}^{m'},\beta\mathfrak{w}^{m'}\partial_{n'})\in\mathcal{W}_{0}^r(U,\End E\oplus TM)\mathfrak{w}^{m'}$, then 
\begin{equation*}
\begin{split}
&\mathbf{H}\left(\lbrace (A\delta_m\mathfrak{w}^{m},\alpha\delta_m\mathfrak{w}^{m}\partial_n), (B\delta_{m'}\mathfrak{w}^{m'},\beta\delta_{m'}\mathfrak{w}^{m'}\partial_{n'})\rbrace_{\natural}\right)\in\mathcal{W}_{\mathfrak{d}_{m+m'}}^{s+r+1}(U,\End E\oplus TM)\mathfrak{w}^{m+m'}\\
&\mathbf{H}\left(\lbrace (\delta_m\mathfrak{w}^{m},\alpha\delta_m\mathfrak{w}^{m}\partial_n), (B\delta_{m'}\mathfrak{w}^{m'},\beta\delta_{m'}\mathfrak{w}^{m'}\partial_{n'})\rbrace_{\flat}\right)\in\mathcal{W}_{\mathfrak{d}_{m+m'}}^{s+r}(U,\End E\oplus TM)\mathfrak{w}^{m+m'}\\
&\mathbf{H}\left(\lbrace (A\delta_m\mathfrak{w}^{m},\alpha\delta_m\mathfrak{w}^{m}\partial_n), (B\delta_{m'}\mathfrak{w}^{m'},\beta\delta_{m'}\mathfrak{w}^{m'}\partial_{n'})\rbrace_{\sharp}\right)\in\mathcal{W}_{\mathfrak{d}_{m+m'}}^{s+r}(U,\End E\oplus TM)\mathfrak{w}^{m+m'}\\
&\mathbf{H}\left(\lbrace (A\delta_m\mathfrak{w}^{m},\alpha\delta_m\mathfrak{w}^{m}\partial_n), (B\delta_{m'}\mathfrak{w}^{m'},\beta\delta_{m'}\mathfrak{w}^{m'}\partial_{n'})\rbrace_{\star}\right)\in\mathcal{W}_{\mathfrak{d}_{m+m'}}^{s+r-1}(U,\End E\oplus TM)\mathfrak{w}^{m+m'}.
\end{split}
\end{equation*}
\end{lemma}
\begin{remark}
The homotopy operators $H$ and $\mathbf{H}$ are different. However it is not a problem because the operator $H$ produce a solution of Maurer-Cartan and not of equation \eqref{eq:Phi}.\footnote{Recall that we were looking for a solution $\Phi$ of Maurer-Cartan of the form $\Phi=\bar{\Pi}+\Xi$ and since $d_W(\bar{\Pi})=0$, the correction term $\Xi$ is a solution of $d_W\Xi=-\frac{1}{2}\lbrace\Phi,\Phi\rbrace_\sim$. At this point we have introduced the homotopy operator $\mathbf{H}$ in order to compute $\Xi$ and we got $\Xi=-\frac{1}{2}\mathbf{H}(\lbrace\Phi,\Phi\rbrace_\sim)$.}     
\end{remark}
\begin{proof}{[Theorem \ref{thm:asymptotic_gauge}]}
First of all recall that for every $s\geq 0$
\begin{equation}
\varphi_a^{(s+1)}=-H\left(\tilde{\Phi}_a+\sum_{k\geq 1}\frac{\mathsf{ad}^k_{\varphi_a^s}}{k!}d_W\varphi_a^s\right)^{(s+1)}
\end{equation}
where $H$ is the homotopy operator defined in \eqref{def:homotopy_1wall} with the same choice of the path $\varrho$ as in \eqref{varro}. In addition as in the proof of Proposition \ref{prop:asymp1},
\begin{equation}
-H\left(\sum_{k\geq 1}\frac{\mathsf{ad}_{\varphi^s_a}^k}{k!}d_W\varphi_a^s\right)^{(s+1)}\in\bigoplus_{l\geq 1}\mathcal{W}_{0}^0(\tilde{A}_0,\End E\oplus TM)\mathfrak{w}^{lm_a}
\end{equation}
hence we are left to study the asymptotic of $H(\tilde{\Phi}_a^{(s+1)})$. By definition $\tilde{\Phi}_a^{(s+1)}$ is the sum over all $k\leq s$-trees such that they have outgoing vertex with label $m_T=lm_a$ for some $l\geq 1$. 

We claim the following:
\begin{equation}
H(\tilde{\Phi}_a^{(s+1)})\in H(\nu_{T_\natural})+\bigoplus_{l\geq 1}\mathcal{W}_{0}^r(\tilde{A}_0,\End E\oplus TM)\mathfrak{w}^{lm_a}
\end{equation}
for every $T\in\mathbb{LRT}_{k,0}$ such that $\mathfrak{t}_{k,T}(\bar{\Pi},\cdots ,\bar{\Pi})$ has Fourier mode $m_T=lm_a$, for every $k\leq s$ and for some $r\leq 0$. Indeed if $k=1$ the tree has only one root and there is nothing to prove because there is no label. In particular 
\[H(\nu_{T})=H(\bar{\Pi}^{(s+1)})=H(\tilde{\Phi}_{(1,0)}^{(s+1)}+\tilde{\Phi}_{(0,1)}^{(s+1)})\]
and we will explicitly compute it below. Then at $k\geq 2$, every tree can be considered as a $2$-tree where the incoming edges are the roots of two sub-trees $T_1$ and $T_2$, not necessary in $\mathbb{LRT}_{0}$, such that their outgoing vertices look like 
\[
\begin{split}
\nu_{T_1}&=(A_k\delta_{m_{a'}}\mathfrak{w}^{km_{a'}},\alpha_k\delta_{m_{a'}}\mathfrak{w}^{km_{a'}}\partial_{n_{a'}})\in\bigoplus_{k\geq 1}\mathcal{W}_{\mathfrak{d}_{m_{a'}}}^{r'}(\tilde{A}_0,\End E\oplus TM)\mathfrak{w}^{km_{a'}}\\
\nu_{T_2}&=(B_{k''}\delta_{m_{a''}}\mathfrak{w}^{k''m_{a''}},\beta_{k''}\delta_{m_{a''}}\mathfrak{w}^{k''m_{a''}}\partial_{n_{a''}})\in\bigoplus_{k''\geq 1}\mathcal{W}_{\mathfrak{d}_{m_{a''}}}^{r''}(\tilde{A}_0,\End E\oplus TM)\mathfrak{w}^{k''m_{a''}}
\end{split}
\]
where $k'm_{a'}+k''m_{a''}=lm_a$. Thus it is enough to prove the claim for a $2$-tree with ingoing vertex $\nu_{T_1}$ and $\nu_{T_2}$ as above. If $T\in\mathbb{LRT}_2$, then $\nu_T=\nu_{T_\natural}+\nu_{T_\flat}+\nu_{T_\sharp}+\nu_{T_\star}
$ and we explicitly compute $H(\nu_{T_\flat}), H(\nu_{T_\sharp})$ and $H(\nu_{T_\star})$. 
\begin{equation*}
\begin{split}
H(\nu_{T_\flat})&=-\frac{1}{2}H\left(\mathbf{H}\left(\lbrace\nu_{T_1},\nu_{T_2}\rbrace_\flat\right)\right)\\
&=-\frac{1}{2}H\Big(\mathbf{H}\Big(\lbrace(A_k\delta_{m_{a'}}\mathfrak{w}^{km_{a'}},\alpha_k\delta_{m_{a'}}\mathfrak{w}^{km_{a'}}\delta_{n_{a'}}), (B_{k''}\delta_{m_{a''}}\mathfrak{w}^{k''m_{a''}},\beta_{k''}\delta_{m_{a''}}\mathfrak{w}^{k''m_{a''}}\partial_{n_{a''}})\rbrace_{\flat}\Big)\Big)\\
&=-\frac{1}{2}i\hbar H(\mathbf{H}\Big(\beta_{k''}n_2^q\frac{\partial}{\partial x_q}(A_k\delta_{m_{a''}})\wedge\delta_{m_{a'}}\mathfrak{w}^{km_{a'}+k''m_{a''}},\\
&\quad\beta_{k''}n_2^q\delta_{m_{a''}}\wedge\frac{\partial}{\partial x_q}\left(\alpha_k\delta_{m_{a'}}\right)\mathfrak{w}^{km_{a'}+k''m_{a''}}\partial_{n_1}\Big))\\
&=\frac{1}{2}i\hbar H(\Big(\left(\int_0^{u_{m_a}}\beta_{k''}n_2^q\frac{\partial A_k}{\partial x_q}\frac{e^{-\frac{s^2}{\hbar}}}{\sqrt{\pi\hbar}}ds\right)\delta_{m_a}\mathfrak{w}^{lm_a}, \\
&\quad\left(\int_0^{u_{m_a}}\beta_{k''}\frac{\partial\alpha_k}{\partial x_q}n_2^q\frac{e^{-\frac{s^2}{\hbar}}}{\sqrt{\pi\hbar}}ds\right)\delta_{m_a}\mathfrak{w}^{km_{a'}+k''m_{a''}}\partial_{n_{a'}}\Big))+\\
&\quad+\frac{1}{2}i\hbar H(\left(\int_0^{u_{m_a}}\frac{e^{-\frac{s^2}{\hbar}}}{\sqrt{\pi\hbar}}\left(\beta_{k''}n_2^qA_k, \beta_{k''}n_2^q\partial_{n_{a'}}\right)\frac{2\gamma_q(s)}{\hbar}ds\right)\delta_{m_a}\mathfrak{w}^{lm_a})\\
&\in\bigoplus_{l\geq 1}\mathcal{W}_{0}^{r'+r''-2}(\tilde{A}_0)\mathfrak{w}^{lm_a}+\mathcal{W}_{0}^{r'+r''-1}(\tilde{A}_0)\mathfrak{w}^{lm_a}
\end{split}
\end{equation*}
where we assume $k'm_{a'}+k''m_{a''}=lm_a$ and in the last step we use Lemma \ref{lem:H(W^s)} to compute the asymptotic behaviour of $H(\delta_{m_a})$.
We denote by $\gamma^q(s)$ the coordinates $u_{m_a^\perp}$ written as functions of $x^q(s)$. In particular, since $\tilde{\Phi}_{(1,0)}^{(1)}\in\bigoplus_{k_1\geq 1}\mathcal{W}_{\mathfrak{d}_{m_1}}^1(\tilde{A}_0,\End E\oplus TM)\mathfrak{w}^{k_1m_1}$ and $\tilde{\Phi}_{(0,1)}^{(1)}\in\bigoplus_{k_2\geq 1}\mathcal{W}_{\mathfrak{d}_{m_2}}^1(\tilde{A}_0,\End E\oplus TM)\mathfrak{w}^{k_2m_2}$, we have $H(\tilde{\Phi}_{(a_1,a_2)}^{(2)})\in\bigoplus_{l\geq 1}\mathcal{W}_{0}^0\mathfrak{w}^{lm_a}$. 
The same holds true for $H(\nu_{T_\sharp})$ by permuting $A,\alpha$ and $B,\beta$.

Then we compute the behaviour of $H(\nu_{T_\star})$:
\begin{equation*}
\begin{split}
H(\nu_{T_\star})&=-\frac{1}{2}H\left(\mathbf{H}\left(\lbrace\nu_{T_1},\nu_{T_2}\rbrace_\star\right)\right)\\
&=H\left(\mathbf{H}\left(\lbrace (A_k\delta_{m_{a'}}\mathfrak{w}^{km_{a'}},\alpha_k\delta_{m_{a'}}\mathfrak{w}^{km_{a'}}\delta_{n_{a'}}), (B_{k''}\delta_{m_{a''}}\mathfrak{w}^{k''m_{a''}},\beta_{k''}\delta_{m_{a''}}\mathfrak{w}^{k''m_{a''}}\partial_{n_{a''}})\rbrace_{\star}\right)\right)\\
&=i\hbar H( \mathbf{H}\left((\alpha_kn_1^qB_{k''}A_q(\phi)\delta_{m_{a''}}\wedge\delta_{m_{a'}}-n_2^q\beta_{k''}A_q(\phi)A_k\delta_{m_{a'}}\wedge\delta_{m_{a''}}, 0)\right))\\
&=i\hbar H\Bigg(\Big(\Bigg(\int_0^{u_{m_{a}}}\alpha_kn_1^qB_{k''}A_q(\phi)\frac{e^{-\frac{s^2}{\hbar}}}{\sqrt{\pi\hbar}}ds-\int_0^{u_{m_{a}}}n_2^q\beta_{k''}A_q(\phi)A_k\frac{e^{-\frac{s^2}{\hbar}}}{\sqrt{\pi\hbar}}ds \Bigg)\delta_{m_a}\mathfrak{w}^{lm_a}, 0\Big)\Bigg)\\
&\in\bigoplus_{l\geq 1}\mathcal{W}_{\mathfrak{d}_{m_{a}}}^{r'+r''-2}(\tilde{A}_0)\mathfrak{w}^{lm_a}
\end{split}
\end{equation*}
where we denote by $m_a$ the primitive vector such that for some $l\geq1$ $lm_a=k'm_{a'}+k''m_{a''}$. 
Finally let us compute $H(\nu_{T_\natural})$: at $k=1$ there is only a $1$-tree, hence $H(\nu_{T_\natural})=H(\tilde{\Phi}_{(1,0)}^{(s+1)}+\tilde{\Phi}_{(0,1)}^{(s+1)})$ and for all $k_1,k_2\geq 1$
\begin{equation*}
\begin{split}
H(\tilde{\Phi}_{(1,0)}^{(s+1)})&\in(A_{s+1,k_1}t^{s+1}, a_{s+1,k_1}t^{s+1}\partial_{n_1})H(\delta_{m_1}\mathfrak{w}^{k_1m_1})+\mathcal{W}_{\mathfrak{d}_{m_1}}^0(\tilde{A}_0,\End E\oplus TM)\mathfrak{w}^{k_1m_1} \\
&\in (A_{s+1,k_1}t^{s+1}, a_{s+1,k_1}t^{s+1}\partial_{n_1})\mathfrak{w}^{k_1m_1}+\mathcal{W}_{\mathfrak{d}_{m_1}}^{-1}(\tilde{A}_0\End E\oplus TM)\mathfrak{w}^{k_1m_1}\\
& \\
H(\tilde{\Phi}_{(0,1)}^{(s+1)})&\in (A_{s+1,k_2}t, a_{s+1,k_2}t\partial_{n_2})H(\delta_{m_2}\mathfrak{w}^{k_2m_2})+\mathcal{W}_{\mathfrak{d}_{m_2}}^{-1}(\tilde{A}_0,\End E\oplus TM)\mathfrak{w}^{k_2m_2}\\
&\in(A_{s+1,k_2}t^{s+1}, a_{s+1,k_2}t^{s+1}\partial_{n_2})\mathfrak{w}^{k_2m_2}+\mathcal{W}_{\mathfrak{d}_{m_2}}^{-1}(\tilde{A}_0,\End E\oplus TM)\mathfrak{w}^{k_2m_2}
\end{split}
\end{equation*} 
Then every other $k$-tree ($k\leq s$) can be decomposed in two sub-trees $T_1$ and $T_2$ as above, and we can further assume $T_1, T_2\in\mathbb{LRT}_0$, because if either $T_1$ or $T_2$ contains at least a label different from $\natural$ then by Lemma \ref{lem:bracket natural} their asymptotic behaviour is of higher order in $\hbar$.
We explicitly compute $H(\nu_{T_\natural})$ at $s=1$: 
\begin{equation*}
\begin{split}
H(\nu_{T_\natural})&=H(-\frac{1}{2}\mathbf{H}\left(\lbrace\bar{\Pi}^{(1)},\bar{\Pi}^{(1)}\rbrace_\natural\right))\\
&=-\frac{1}{2}H\left(\mathbf{H}\left(\lbrace\bar{\Pi}_1^{(1)},\bar{\Pi}_2^{(1)}\rbrace_{\natural}+\lbrace\bar{\Pi}_2^{(1)},\bar{\Pi}_1^{(1)}\rbrace_{\natural}\right)\right)\\
&=-H\Big(\mathbf{H}\Big(\big(a_{1,k_1}A_{1,k_2}\langle n_1,k_2m_2\rangle- a_{1,k_2}A_{1,k_1}\langle n_2,k_1m_1\rangle+ [A_{1,k_1},A_{1,k_2}],\\
&\quad a_{1,k_1}a_{1,k_2}\partial_{\langle n_1,k_2m_2\rangle n_2+\langle n_2,k_1m_1\rangle n_1}\big)t^2\delta_{m_1}\wedge\delta_{m_2}\mathfrak{w}^{k_1m_1+k_2m_2})\Big)\Big)
\end{split}
\end{equation*}

\begin{equation*}
\begin{split}
&=-H\Big(\Big(a_{1,k_1}A_{1,k_2}\langle n_1,k_2m_2\rangle- a_{1,k_2}A_{1,k_1}\langle n_2,k_1m_1\rangle+ [A_{1,k_1},A_{1,k_2}],\\
&a_{1,k_1}a_{1,k_2}\partial_{\langle n_1,k_2m_2\rangle n_2+\langle n_2,k_1m_1\rangle n_1}\Big)t^2\left(\int_0^{u_{m_a}}\frac{e^{-\frac{s^2}{\hbar}}}{\sqrt{\hbar\pi}}ds\right)\delta_{m_a}\mathfrak{w}^{lm_a}\Big)\\
&=-\Big(a_{1,k_1}A_{1,k_2}\langle n_1,k_2m_2\rangle- a_{1,k_2}A_{1,k_1}\langle n_2,k_1m_1\rangle+ [A_{1,k_1},A_{1,k_2}],\\
&\quad a_{1,k_1}a_{1,k_2}\partial_{\langle n_1,k_2m_2\rangle n_2+\langle n_2,k_1m_1\rangle n_1}\Big)t^2 H\left(\left(\int_0^{u_{m_a}}\frac{e^{-\frac{s^2}{\hbar}}}{\sqrt{\hbar\pi}}ds\right)\delta_{m_a}\mathfrak{w}^{lm_a}\right)\\
&=-\Big(a_{1,k_1}A_{1,k_2}\langle n_1,k_2m_2\rangle- a_{1,k_2}A_{1,k_1}\langle n_2,k_1m_1\rangle+ [A_{1,k_1},A_{1,k_2}],\\
&\quad a_{1,k_1}a_{1,k_2}\partial_{\langle n_1,k_2m_2\rangle n_2+\langle n_2,k_1m_1\rangle n_1}\Big)t^2 \mathfrak{w}^{lm_a}H_{lm_a}\left(\left(\int_0^{u_{m_a}}\frac{e^{-\frac{s^2}{\hbar}}}{\sqrt{\hbar\pi}}ds\right)\delta_{m_a}\right)
\end{split}
\end{equation*}
Now 
\begin{equation*}
\mathfrak{w}^{lm_a}H_{lm_a}\left(\left(\int_0^{u_{m_a}}\frac{e^{-\frac{s^2}{\hbar}}}{\sqrt{\hbar\pi}}ds\right)\delta_{m_a}\right)\in\mathfrak{w}^{lm_a}+\mathcal{W}_0^0(\tilde{A}_0)\mathfrak{w}^{lm_a}
\end{equation*}
Therefore the leading order term of $H(\nu_{T_\natural})$ with labels $m_T=lm_a=k_1m_1+k_2m_2$ at $s=1$ is
\[
\big(a_{1,k_1}A_{1,k_2}\langle n_1,k_2m_2\rangle- a_{1,k_1}A_{1,k_2}\langle n_2,k_1m_1\rangle+ [A_{1,k_1},A_{1,k_2}],a_{1,k_1}a_{1,k_2}\partial_{\langle n_1,k_2m_2\rangle n_2+\langle n_2,k_1m_1\rangle n_1}\big)t^2\mathfrak{w}^{lm_a}
\]
At $s\geq 2$, every other $k$-tree $T$ ($k\leq s$) can be decomposed in two sub-trees, say $T_1$ and $T_2$ such that $\nu_{T_\natural}=-\mathbf{H}(\lbrace\nu_{T_{1,\natural}},\nu_{T_{2,\natural}}\rbrace_\natural)+\bigoplus_{l\geq  1}\mathcal{W}_{\mathfrak{d}_{m_a}}^1(\tilde{A}_0)\mathfrak{w}^{lm_a}$. 

Notice that the leading order term of $H\left(\mathbf{H}(\lbrace\bar{\Pi}_1^{(1)},\bar{\Pi}_2^{(1)}\rbrace_{\natural})\right)$ is the Lie bracket of \[\lbrace(A_{1,k_1}\mathfrak{w}^{k_1m_1}, a_{1,k_1}\mathfrak{w}^{k_1m_1}\partial_{n_1}),(A_{1,k_2}\mathfrak{w}^{k_2m_2},a_{1,k_2}\mathfrak{w}^{k_2m_2}\partial_{n_2})\rbrace_{\tilde{\mathfrak{h}}}\] hence the leading order term of $H(\nu_{T_\natural})$ belongs to $\tilde{\mathfrak{h}}$.  

\end{proof}

Notice that at any order in the formal parameter $t$, there are only a finite number of terms which contribute to the solution $\Phi$ in the sum \eqref{def:Phitree}, hence we define the set $\mathbb{W}(N)$ as
\begin{equation}
\mathbb{W}(N)\defeq\lbrace a\in\left(\Z^2_{\geq 0}\right)_{\text{prim}}\vert lm_a=m_T \text{ for some } l\geq 0 \text{and } T\in\mathbb{LRT}_k \text{with } 1\leq j_T\leq N\rbrace.
\end{equation}
\begin{definition}[Scattering diagram $\mathfrak{D}^{\infty}$]\label{def:D_infty}
The order $N$ scattering diagram $\mathfrak{D}_N$ associated to the solution $\Phi$ is \[\mathfrak{D}_N\defeq\big\lbrace\mathsf{w}_1, \mathsf{w}_2\big\rbrace\cup\big\lbrace\mathsf{w}_a=\big(m_a, \mathfrak{d}_{m_a}, \theta_a\big)\big\rbrace_{a\in\mathbb{W}(N)}\] where
\begin{itemize}
    \item $m_a=a_1m_1+a_2m_2$;
    \item $\mathfrak{d}_{m_a}=\xi_0+m_a\R_{\geq 0}$
    \item $\log(\theta_a)$ is the leading order term of the unique gauge $\varphi_a$, as computed in Theorem \eqref{thm:asymptotic_gauge}.
\end{itemize}
The scattering diagram $\mathfrak{D}^{\infty}\defeq\varinjlim_N\mathfrak{D}_N$.
\end{definition}

\subsubsection{Consistency of  $\,\,\mathfrak{D}^{\infty}$}
We are left to prove consistency of the scattering diagram $\mathfrak{D}^{\infty}$ associated to the solution $\Phi$. In order to do that we are going to use a monodromy argument (the same approach was used in \cite{MCscattering}).

Let us define the following regions
\begin{align}
&\tilde{\mathbb{A}}\defeq\lbrace(r,\vartheta)\vert \vartheta_0-\epsilon_0+2\pi<\vartheta<\vartheta_0+2\pi\rbrace\\
&\tilde{\mathbb{A}}-2\pi\defeq\lbrace(r,\vartheta)\vert \vartheta_0-\epsilon_0<\vartheta<\vartheta_0\rbrace.
\end{align}  
for small enough $\epsilon_0>0$, such that $\tilde{\mathbb{A}}-2\pi$ is away from all possible walls in $\mathfrak{D}^{\infty}$.

\begin{theorem}\label{thm:consistentD}
Let $\mathfrak{D}^{\infty}$ be the scattering diagram defined in \eqref{def:D_infty}. Then it is consistent, i.e. $\Theta_{\mathfrak{D}^{\infty}, \gamma}=\text{Id}$ for any closed path $\gamma$ embedded in $U\setminus\lbrace \xi_0\rbrace$, which intersects $\mathfrak{D}^{\infty}$ generically. 
\end{theorem} 
\begin{proof}
It is enough to prove that $\mathfrak{D}_N$ is consistent for any fixed $N>0$. First of all recall that $\Theta_{\gamma,\mathfrak{D}_N}=\prod_{a\in\mathbb{W}(N)}^\gamma \theta_a $.  
Then let us prove that the following identity \begin{equation}\label{eq:step1}
\prod_{a\in\mathbb{W}(N)}^\gamma e^{\varphi_a}\ast 0=\sum_{a\in\mathbb{W}(N)}\tilde{\Phi}_{a}
\end{equation}
holds true. 
Indeed \[\left(e^{\varphi_a}\ast e^{\varphi_{a'}}\right)\ast 0=e^{\varphi_a}\ast\left(\tilde{\Phi}_{a'}\right)=\tilde{\Phi}_{a'}-\sum_k\frac{\textsf{ad}^k_{\varphi_a}}{k!}\left(d_W\varphi_a+\lbrace\varphi_a,\tilde{\Phi}_{a'}\rbrace_\sim\right).\]
For degree reason $\lbrace \varphi_a,\tilde{\Phi}_{a'}\rbrace_\sim=0$ and by definition \[-\sum_k\frac{\textsf{ad}^k_{\varphi_a}}{k!}\big(d_W\varphi_a\big)=e^{\varphi_a}\ast 0=\tilde{\Phi}_{a} .\] 
Iterating the same procedure for more than two rays, we get the result. 

Recall that if $\varphi$ is the unique gauge such that $p(\varphi)=0$ and $e^\varphi\ast0=\Phi$, then $e^{\varpi^*(\varphi)}\ast 0=\varpi^*(\Phi)$ on $\tilde{A}$. Hence
$e^{\varpi^*(\varphi)}\ast 0=\varpi^*(\Phi)=\sum_{a\in\mathbb{W}(N)}\varpi^*(\Phi_{a})=\sum_{a\in\mathbb{W}(N)}\tilde{\Phi}_{a}$ and by equation \eqref{eq:step1}
\[e^{\varpi^*(\varphi)}\ast 0=\prod_{a\in\mathbb{W}(N)}^\gamma e^{\varphi_a}\ast 0.\]
In particular, by uniqueness of the gauge, $e^{\varpi^*(\varphi)}=\prod_{a\in\mathbb{W}(N)}^\gamma e^{\varphi_a}.$ Since $\varpi^*(\varphi)$ is defined on all $U$, $e^{\varpi^*(\varphi)}$ is monodromy free, i.e.
\[\prod_{a\in\mathbb{W}(N)}^\gamma e^{\varphi_a}\vert_{\tilde{\mathbb{A}}}=\prod_{a\in\mathbb{W}(N)}^\gamma e^{\varphi_a}\vert_{\tilde{\mathbb{A}}-2\pi}.\]
Notice that $\tilde{\mathbb{A}}-2\pi$ does not contain the support of $\varphi_a$ $\forall a\in\left(\Z_{\geq 0}^2\right)_{\text{prim}}$, therefore \[\prod_{a\in\mathbb{W}(N)}^\gamma e^{\varphi_a}\vert_{\tilde{\mathbb{A}}-2\pi}=\prod_{a\in\mathbb{W}(N)}^\gamma e^{0}=\text{Id}.\]

\end{proof}
\chapter{Relation with the wall-crossing formulas in coupled $2d$-$4d$ systems}\label{sec:WCF}

We are going to show how wall-crossing formulas in coupled $2d$-$4d$ systems, introduced by Gaiotto, Moore and Nietzke in \cite{WCF2d-4d},  can be interpreted in the framework we were discussing before. Let us first recall the setting for the $2d$-$4d$ WCFs: 

\begin{itemize}
	\item let $\Gamma$ be a lattice, whose elements are denoted by $\gamma$;
	\item  define an antisymmetric bilinear form $\langle\cdot ,\cdot\rangle_D\colon\Gamma\times\Gamma\to\Z$, called the Dirac pairing;
	\item  let $\Omega\colon\Gamma\to\Z$ be a homomorphism;
	\item denote by $\mathcal{V}$ a finite set of indices, $\mathcal{V}=\lbrace i,j,k,\cdots\rbrace$;
    \item define a $\Gamma$-torsor $\Gamma_i$, for every $i\in\mathcal{V}$. Elements of $\Gamma_i$ are denoted by $\gamma_i$ and the action of $\Gamma$ on $\Gamma_i$ is $\gamma+\gamma_i=\gamma_i+\gamma$;
    \item define another $\Gamma$-torsor $\Gamma_{ij}\defeq\Gamma_i-\Gamma_j$ whose elements are formal differences $\gamma_{ij}\defeq\gamma_i-\gamma_j$ up to equivalence, i.e. $\gamma_{ij}=(\gamma_{i}+\gamma)-(\gamma_j+\gamma)$ for every $\gamma\in\Gamma$. If $i=j$, then $\Gamma_{ii}$ is identified with $\Gamma$. The action of $\Gamma$ on $\Gamma_{ij}$ is $\gamma_{ij}+\gamma=\gamma+\gamma_{ij}$. Usually it is not possible to sum elements of $\Gamma_{ij}$ and $\Gamma_{kl}$, for instance $\gamma_{ij}+\gamma_{kl}$ is well defined only if $j=k$ and in this case it is an element of $\Gamma_{il}$;
    \item let $Z\colon\Gamma\to\C$ be a homomorphism and define its extension as an additive map $Z\colon\amalg_{i\in\mathcal{V}}\Gamma_{i}\to\C$, such that $Z({\gamma+\gamma_i})=Z(\gamma)+Z({\gamma_i})$. In particular, by additivity $Z$ is a map from $\amalg_{i,j\in\mathcal{V}}\Gamma_{ij}$ to $\C$, namely $Z(\gamma_{ij})=Z(\gamma_i)-Z(\gamma_j)$. The map $Z$ is usually called the central charge;
    \item let $\sigma(a,b)\in\lbrace \pm 1\rbrace$ be a \textit{twisting} function defined whenever $a+b$ is defined for $a,b\in\Gamma\sqcup\amalg_i\Gamma_i\sqcup\amalg_{i\neq j}\Gamma_{ij}$ and valued in $\lbrace\pm 1\rbrace$. Moreover it satisfies the following conditions:
\begin{equation}
\begin{split}
&\text{(i)}\quad \sigma(a,b+c)\sigma(b,c)=\sigma(a,b)\sigma(a+b,c) \\
&\text{(ii)}\quad {\sigma(a,b)=\sigma(b,a)} \text{ if both $a+b$ and $b+a$ are defined} \\
&\text{(iii)}\quad \sigma(\gamma,\gamma')=(-1)^{\langle\gamma,\gamma'\rangle_D}\,\,\forall\gamma,\gamma'\in\Gamma;
\end{split}
\end{equation}
    \item let $X_a$ denote formal variables, for every $a\in\Gamma\sqcup\amalg_i\Gamma_i\sqcup\amalg_{i\neq j}\Gamma_{ij}$. There is a notion of associative product: 
\[
X_{a}X_{b}\defeq\begin{cases}
\sigma(a,b)X_{a+b} & \text{if the sum $a+b$ is defined }\\
0 & \text{otherwise}
\end{cases}
\]
\end{itemize}

The previous data fit well in the definition of a pointed groupoid $\mathbb{G}$, as it is defined in \cite{WCF2d-4d}. In particular $\Ob(\mathbb{G})=\mathcal{V}\sqcup\lbrace o\rbrace$ and $\Mor(\mathbb{G})=\amalg_{i,j\in\Ob(\mathbb{G})}\Gamma_{ij}$, where the torsor $\Gamma_i$ is identified with $\Gamma_{io}$ and elements of $\Gamma$ are identified with $\amalg_i\Gamma_{ii}$. The composition of morphism is written as a sum, and the formal variables $X_a$ are elements of the groupoid ring $\C[\mathbb{G}]$. 
In this setting, BPS rays are defined as 
\[
\begin{split}
&l_{ij}\defeq Z(\gamma_{ij})\R_{>0}\\
&l\defeq Z(\gamma)\R_{>0}
\end{split}
\] and they are decorated with automorphisms of $\C[\mathbb{G}][\![ t ]\!]$ respectively of type $S$ and of type $K$, defined as follows: let $X_a\in\C[\mathbb{G}]$, then

\begin{equation}\label{eq:autS}
S_{\gamma_{ij}}^{\mu}(X_a)\defeq\big(1-\mu({\gamma_{ij}})tX_{\gamma_{ij}}\big)X_a\big(1+\mu(\gamma_{ij})tX_{\gamma_{ij}}\big)    
\end{equation}
where $\mu\colon\amalg_{i,j\in\mathcal{V}}\Gamma_{ij}\to\Z$ is a homomorphism;
\begin{equation}\label{eq:autK}
 K_{\gamma}^{\omega}(X_a)\defeq\big(1-X_{\gamma}t\big)^{-\omega(\gamma,a)}X_a   
\end{equation}
where $\omega\colon\Gamma\times\amalg_{i\in\mathcal{V}}\Gamma_i\to\Z$ is a homomorphism such that $\omega(\gamma,\gamma')=\Omega(\gamma)\langle\gamma,\gamma'\rangle_D$ and $\omega(\gamma,a+b)=\omega(\gamma,a)+\omega(\gamma,b)$ for $a,b\in\mathbb{G}$.

In particular under the previous assumption, the action of $S_{\gamma_{ij}}^\mu$ and $K_\gamma^\omega$ can be explicitly computed on variables $X_\gamma$ and $X_{\gamma_k}$ as follows:
\begin{equation}\label{eq:actionP}
\begin{split}
    S_{\gamma_{ij}}^\mu &\colon X_{\gamma'}\to X_{\gamma'}\\
    S_{\gamma_{ij}}^\mu &\colon X_{\gamma_k}\to\begin{cases}  X_{\gamma_k} & \text{if } k\neq j\\
    X_{\gamma_j}- \mu(\gamma_{ij})tX_{\gamma_{ij}}X_{\gamma_j} & \text{if } k=j
     \end{cases}\\
    K_\gamma^\omega &\colon X_{\gamma'}\to (1-tX_\gamma)^{-\omega(\gamma,\gamma')}X_{\gamma'}\\
    K_\gamma^\omega &\colon X_{\gamma_k}\to (1-tX_\gamma)^{-\omega(\gamma,\gamma_k)}X_{\gamma_k}
   \end{split}
\end{equation}

In order to interpret the automorphisms $S$ and $K$ as elements of $\exp(\tilde{\mathfrak{h}})$ we are going to introduce their infinitesimal generators. Let $\Der\left(\C[\mathbb{G}]\right)$ be the Lie algebra of the derivations of $\C[\mathbb{G}]$ and define:
\begin{equation}
    \mathfrak{d}_{\gamma_{ij}}\defeq \textsf{ad}_{X_{\gamma_{ij}}}
\end{equation}
where $\mathfrak{d}_{\gamma_{ij}}(X_a)\defeq \left( X_{\gamma_{ij}}X_a-X_aX_{\gamma_{ij}}\right)$, for every $X_a\in\C[\mathbb{G}]$; 
\begin{equation}
    \mathfrak{d}_\gamma\defeq \omega(\gamma,\cdot)X_\gamma
\end{equation}
where $\mathfrak{d}_\gamma(X_a)\defeq \left(\omega(\gamma,a)X_{\gamma}X_{a}\right)$, for every $X_a\in\C[\mathbb{G}]$. 
\begin{definition}
Let $\mathbf{L}_\Gamma$ be the $\C[\Gamma]$- module generated by $\mathfrak{d}_{\gamma_{ij}}$ and $\mathfrak{d}_\gamma$, for every $i\neq j\in\mathcal{V}$, $\gamma\in\Gamma$. 
\end{definition}
For instance a generic element of $\mathbf{L}_\Gamma$ is given by \[\sum_{i,j\in\mathcal{V}}\sum_{l\geq 1}c_l^{(\gamma_{ij})}X_{a_{l}}\mathfrak{d}_{\gamma_{ij}}+\sum_{\gamma\in\Gamma}\sum_{l\geq 1}c_l^{(\gamma)}X_{a_l}\mathfrak{d}_{\gamma}\] 
where $c_l^{(\bullet)}X_{a_l}\in\C[\Gamma]$.  
\begin{lemma}\label{lem:CGamma_module}
Let $\mathbf{L}_\Gamma$ be the $\C[\Gamma]$-module defined above. Then, it is a Lie ring\footnote{A Lie ring $\mathbf{L}_\Gamma$ is an abelian group $(L,+)$ with a bilinear form $[,]\colon:L\times L\to L$ such that 
\begin{enumerate}
\item $[a,b+c]=[a,b]+[a,c]$ and $[a+b,c]=[a,c]+[b,c]$;
\item $[\cdot,\cdot]$ is antisymmetric, i.e. $[a,b]=-[b,a]$;
\item $[\cdot,\cdot]$ satisfy the Jacobi identity. 
\end{enumerate}
} with the the Lie bracket $[\cdot, \cdot]_{\Der(\C[\mathbb{G}])}$ induced by $\Der(\C[\mathbb{G}])$\footnote{$\mathbf{L}_\Gamma$ is not a Lie algebra over $\C[\Gamma]$ because the bracket induced from $\Der(\C[\mathbb{G}])$ is not $\C[\Gamma]-$linear.}.  
\end{lemma}
\begin{proof}
It is enough to prove that $\mathbf{L}_\Gamma$ is a Lie sub-algebra of $\Der(\C[\mathbb{G}])$, i.e. it is closed under $[\cdot,\cdot]_{\Der(\C[\mathbb{G}])}$. By $\C$-linearity it is enough to prove the following claims:
\begin{itemize}
    \item[$(1)$] $[X_{\gamma}\mathfrak{d}_{\gamma_{ij}}, X_{\gamma'}\mathfrak{d}_{\gamma_{kl}}]\in \mathbf{L}_\Gamma$: indeed 
    \begin{equation*}
        \begin{split}
        [X_{\gamma}\mathfrak{d}_{\gamma_{ij}}, X_{\gamma'}\mathfrak{d}_{\gamma_{kl}}]&=X_{\gamma}\mathsf{ad}_{\gamma_{ij}}(X_\gamma'\mathsf{ad}_{\gamma_{kl}})-X_{\gamma'}\mathsf{ad}_{\gamma_{kl}}(X_\gamma\mathsf{ad}_{\gamma_{ij}})\\
        &=X_{\gamma}X_{\gamma'}\mathsf{ad}_{\gamma_{ij}}(\ad_{\gamma_{kl}})-X_{\gamma'}X_{\gamma}\mathsf{ad}_{\gamma_{kl}}(\mathsf{ad}_{\gamma_{ij}})\\
        &={\sigma(\gamma_{ij},\gamma_{kl})}X_{\gamma}X_{\gamma'}\ad_{\gamma_{ij}+\gamma_{kl}}-{\sigma(\gamma_{kl},\gamma_{ij})}X_{\gamma}X_{\gamma'}\ad_{\gamma_{kl}+\gamma_{ij}};
        \end{split}
    \end{equation*}
    \item[$(2)$] $[X_\gamma\mathfrak{d}_{\gamma_{ij}}, X_{\gamma'}\mathfrak{d}_{\gamma''}]\in\mathbf{L}_{\Gamma}$:indeed for every $X_a\in\C[\mathbb{G}]$
    \begin{equation*}
        \begin{split}
        [X_\gamma\mathfrak{d}_{\gamma_{ij}}, X_{\gamma'}\mathfrak{d}_{\gamma''}]X_a&=X_\gamma\mathfrak{d}_{\gamma_{ij}}\left(X_{\gamma'}\mathfrak{d}_{\gamma''}X_a\right)-X_{\gamma'}\mathfrak{d}_{\gamma''}\left(X_\gamma\mathfrak{d}_{\gamma_{ij}}X_a\right)\\
        &=X_{\gamma}\mathfrak{d}_{\gamma_{ij}}\left(X_{\gamma'}\omega(\gamma'',a)X_{\gamma''}X_{a}\right)-X_{\gamma'}\mathfrak{d}_{\gamma''}\left(X_{\gamma}X_{\gamma_{ij}}X_{a}-X_{\gamma}X_{a}X_{\gamma_{ij}}\right)\\
        &=\omega(\gamma'',a)X_{\gamma}\left(X_{\gamma_{ij}}X_{\gamma'}X_{\gamma''}X_{a}-X_{\gamma'}X_{\gamma''}X_{a}X_{\gamma_{ij}}\right)\\
        &-X_{\gamma'}\omega(\gamma'',\gamma+\gamma_{ij}+a)X_{\gamma''}X_{\gamma}X_{\gamma_{ij}}X_{a}+X_{\gamma'}\omega(\gamma'',\gamma+a+\gamma_{ij})X_{\gamma''}X_{\gamma}X_{a}X_{\gamma_{ij}}\\
        &=\left(\omega(\gamma'',a)-\omega(\gamma'',\gamma+\gamma_{ij})-\omega(\gamma'',a)\right)X_{\gamma}X_{\gamma'}X_{\gamma''}X_{\gamma_{ij}}X_{a}\\
        &-\left(\omega(\gamma'',a)-\omega(\gamma'',a)-\omega(\gamma'',\gamma+\gamma_{ij})\right)X_{\gamma}X_{\gamma'}X_{\gamma''}X_{a}X_{\gamma_{ij}}\\
        &=-\omega(\gamma'',\gamma+\gamma_{ij})X_{\gamma}X_{\gamma'}X_{\gamma''}\mathfrak{d}_{\gamma_{ij}}(X_a);
       \end{split}
    \end{equation*}
    \item[$(3)$] $[X_\gamma\mathfrak{d}_{\gamma'}, X_{\gamma''}\mathfrak{d}_{\gamma'''}]\in L$;indeed for every $X_a\in\C[\mathbb{G}]$

\begin{equation*}
\begin{split}
[X_\gamma\mathfrak{d}_{\gamma'}, X_{\gamma''}\mathfrak{d}_{\gamma'''}]X_a &= X_{\gamma} \mathfrak{d}_{\gamma'}\left(X_{\gamma''}\omega({\gamma'''} ,a) X_{{\gamma'''}}X_{a} \right)-X_{\gamma''}\mathfrak{d}_{\gamma'''}\left(\omega(\gamma',a)X_{\gamma'}X_{a}\right)\\ 
&=\omega(\gamma''',a)\omega(\gamma',\gamma'''+\gamma''+a)X_{\gamma}X_{\gamma'}X_{\gamma''}X_{\gamma'''}X_{a}+\\
&-\omega(\gamma',a)\omega(\gamma''',\gamma+\gamma'+a)X_{\gamma''}X_{\gamma'''}X_{\gamma}X_{\gamma'}X_{a}\\
&=\left(\omega(\gamma''' ,a)\left(\omega(\gamma' ,a)+\omega(\gamma',\gamma'''+\gamma'')\right)\right)X_{\gamma}X_{\gamma'}X_{\gamma''}X_{\gamma'''}X_{a}+\\
&-\left(\omega(\gamma' ,a)\left(\omega(\gamma''' ,a)+\omega(\gamma''',\gamma+\gamma')\right)\right)X_{\gamma}X_{\gamma'}X_{\gamma''}X_{\gamma'''}X_{a}\\
&=\omega(\gamma',\gamma'''+\gamma'')X_{\gamma'}X_{\gamma''}X_{\gamma}\mathfrak{d}_{\gamma'''}(X_a)-\omega(\gamma''',\gamma+\gamma')X_{\gamma'''}X_{\gamma''}X_{\gamma}\mathfrak{d}_{\gamma'}(X_a).
\end{split}
\end{equation*}
\end{itemize}

\end{proof}

We can now define the infinitesimal generators of $S_{\gamma_{ij}}^\mu$ and $K_\gamma^\omega$ as elements of $\mathbf{L}_\Gamma$: we first define  
\begin{equation}
    \mathfrak{s}_{\gamma_{ij}}\defeq -\mu(\gamma_{ij})t\mathfrak{d}_{\gamma_{ij}}
\end{equation}
then  $\exp(\mathfrak{s}_{\gamma_{ij}})=S_{\gamma_{ij}}^\mu$, indeed
\begin{equation*}
    \begin{split}
     \exp(\mathfrak{s}_{\gamma_{ij}})(X_a)&=\sum_{k\geq0 }\frac{1}{k!}\mathfrak{s}_{\gamma_{ij}}^k(X_a)\\
     &=\sum_{k\geq0}\frac{(-1)^k}{k!}t^k\mu(\gamma_{ij})^k\mathsf{ad}_{X_{\gamma_{ij}}}^k(X_a)\\
     &=X_a-\mu(\gamma_{ij})t\ad_{\gamma_{ij}}(X_a)+\frac{1}{2}t^2\mu(\gamma_{ij})^2\mathsf{ad}_{\gamma_{ij}}^2(X_a),
  \end{split}
\end{equation*}
where $\mathsf{ad}_{\gamma_{ij}}(X_a)=X_{\gamma_{ij}}X_{a}-X_{a}X_{\gamma_{ij}}$. Hence 

\[\mathsf{ad}_{\gamma_{ij}}^2(X_a)=-2X_{\gamma_{ij}}X_aX_{\gamma_{ij}}-t^2X_{\gamma_{ij}}X_{a}X_{\gamma_{ij}}\]

and since $\gamma_{ij}$ can not be composed with $\gamma_{ij}+a+\gamma_{ij}$, then  $\ad_{\gamma_{ij}}^3(X_a)=0$. Moreover if $a\in\Gamma$ then $\ad_{\gamma_{ij}}X_a=0$, while if $a=\gamma_{ok}$ then $\mathsf{ad}^2X_a=0$ and we recover the formulas \eqref{eq:actionP}.
 
Then we define 
\begin{equation}
    \mathfrak{k}_{\gamma}\defeq\sum_{l\geq 1}\frac{1}{l}t^{l}X_{\gamma}^{(l-1)}\mathfrak{d}_\gamma
\end{equation}
and we claim $\exp(\mathfrak{k}_\gamma)=K_{\gamma}^\omega$, indeed
     
\begin{equation*}
    \begin{split}
     \exp(\mathfrak{k}_\gamma)(X_a)&=\sum_{k\geq 0} \frac{1}{k!}\mathfrak{k}_\gamma^k(X_a)  \\
&=\sum_{k\geq0}\frac{1}{k!}\left(\sum_{l_k\geq1}\frac{1}{l_k}t^{l_k}X_{\gamma}^{l_k}\omega(\gamma,\cdot)\left(\cdots\left(\sum_{l_2\geq1}t^{l_2}\frac{1}{l_2}X_{\gamma}^{l_2}\omega(\gamma,\cdot)\left(\sum_{l_1\geq1}\frac{1}{l_1}t^{l_1}\omega(\gamma,a)X_{l_1\gamma}X_{a}\right)\right)\cdots\right)\right)\\
&=\sum_{k\geq0}\frac{1}{k!}\left(\sum_{l_k\geq1}\frac{1}{l_k}t^{l_k}X_{\gamma}^{l_k}\omega(\gamma,\cdot)\left(\cdots\left(\sum_{l_2\geq1}\sum_{l_1\geq1}\frac{1}{l_1l_2}t^{l_1+l_2}\omega(\gamma,l_1\gamma+a)\omega(\gamma,a)X_{\gamma}^{l_2}X_{\gamma}^{l_1}X_{a}\right)\cdots\right)\right)\\
&=\sum_{k\geq0}\frac{1}{k!}\left(\sum_{l_k\geq1}\frac{1}{l_k}t^{l_k}X_{\gamma}^{l_k}\omega(\gamma,\cdot)\left(\cdots\left(\sum_{l_2\geq1}\sum_{l_1\geq1}\frac{1}{l_1l_2}t^{l_1+l_2}\omega(\gamma,a)\omega(\gamma,a)X_{\gamma}^{l_2+l_1}X_{a}\right)\cdots\right)\right)\\
     &=\sum_{k\geq0}\frac{1}{k!}\omega(\gamma+a)^k\left(\sum_{l\geq1}\frac{1}{l}t^lX_{\gamma}^l\right)^kX_a\\ 
&=\exp\left(-\omega(\gamma,a)\log(1-tX_\gamma)\right)X_a.
    \end{split}
\end{equation*}

From now on we are going to assume that $\Gamma\cong\Z^2\cong\Lambda$. We distinguish between polynomial in $\C[\Gamma]$ and $\C[\Lambda]$ by writing $X_\gamma$ for a variable in $\C[\Gamma]$ and $\mathfrak{w}^\gamma$ as a variable in $\C[\Lambda]$.   
\begin{remark}
The group ring $\C[\Gamma]$ is isomorphic to $\C[\Lambda]$ even if there two different products: on $\C[\Gamma]$ the product is $X_{\gamma}X_{\gamma'}\defeq\sigma(\gamma,\gamma')X_{\gamma+\gamma'}=(-1)^{\langle\gamma,\gamma'\rangle_D}X_{\gamma+\gamma'}$ while the product in $\C[\Lambda]$ is defined by $\mathfrak{w}^{\gamma}\mathfrak{w}^{\gamma'}=\mathfrak{w}^{\gamma+\gamma'}$. In particular the isomorphism depends on the choice of $\sigma$. 
\end{remark}

Let us choose an element $e_{ij}\in\Gamma_{ij}$ for every $i\neq j\in\mathcal{V}$ and set $e_{ii}\defeq0\in\Gamma$ for every $i\in\mathcal{V}$. Under this assumption $\mathbf{L}_\Gamma$ turns out to be generated by $ \mathfrak{d}_{e_{ij}}$ for all $i\neq j\in\mathcal{V}$ and by $\mathfrak{d}_{\gamma}$ for every $\gamma\in\Gamma$. Indeed every $\gamma_{ij}\in\Gamma_{ij}$ can be written as $e_{ij}+\gamma$ for some $\gamma\in\Gamma$ and $\mathfrak{d}_{\gamma_{ij}}=\mathfrak{d}_{e_{ij}+\gamma}=X_\gamma\mathfrak{d}_{e_{ij}}$. Then, we define an additive map \[
\begin{split}
m&\colon\amalg_{i,j\in\mathcal{V}}\Gamma_{ij}\to\Gamma\\
m&(\gamma_{ij})\defeq \gamma_{ij}-e_{ij}
\end{split}
\]
In particular, notice that $m(\gamma_{ii})=\gamma_{ii}-e_{ii}=\gamma_{ii}$, hence, since $\Gamma=\amalg_i\Gamma_{ii}$, $m(\Gamma)=\Gamma$.  

We now define a $\C[\Gamma]$-module in the Lie algebra $\tilde{\mathfrak{h}}$: 
\begin{definition}
Define $\tilde{\mathbf{L}}$ as the $\C[\Lambda]$-module generated by $\tilde{\mathfrak{l}}_{\gamma_{ij}}\defeq\left(E_{ij}\mathfrak{w}^{m(\gamma_{ij})},0\right)$ for every $i\neq j\in\mathcal{V}$ and $\tilde{\mathfrak{l}}_\gamma\defeq\left(0, \Omega(\gamma)\mathfrak{w}^{\gamma}\partial_{n_\gamma}\right)$ for every $\gamma\in\Gamma$, where $E_{ij}\in\mathfrak{gl}(r)$ is an elementary matrix with all zeros and a $1$ in position $ij$.
\end{definition}
 
\begin{lemma}\label{lem:CLambdamodule}
The $\C[\Lambda]$-module $\tilde{\mathbf{L}}$ is a Lie ring with respect to the Lie bracket induced by $\tilde{\mathfrak{h}}$ (see Definition \eqref{eq:Liebracket}).\footnote{$\tilde{\mathbf{L}}$ is not a Lie sub-algebra of $\tilde{\mathfrak{h}}$ because the Lie bracket is not $\C[\Lambda]-$linear.}
\end{lemma}
\begin{proof}
As we have already comment in the proof of Lemma \ref{lem:CGamma_module}, since the bracket is induced by the Lie bracket $[\cdot,\cdot]_{{\tilde{\mathfrak{h}}}}$, we are left to prove that $\tilde{\mathbf{L}}$ is closed under $[\cdot,\cdot]_{\tilde{\mathfrak{h}}}$. In particular by $\C$-linearity it is enough to show the following: 
\begin{align*}
(1)&\, [\mathfrak{w}^{\gamma}\left(E_{ij}\mathfrak{w}^{m(\gamma_{ij})},0\right), \mathfrak{w}^{\gamma'}\left(E_{kl}\mathfrak{w}^{m(\gamma_{kl})},0\right)]\in\tilde{\mathbf{L}}\\
(2)&\, [\mathfrak{w}^{\gamma}\left(E_{ij}\mathfrak{w}^{m(\gamma_{ij})},0\right), \mathfrak{w}^{\gamma'}\left(0, \Omega(\gamma)\mathfrak{w}^{\gamma'}\partial_{n_{\gamma'}}\right)]\in \tilde{\mathbf{L}}\\
(3)&\, [\mathfrak{w}^{\gamma}\left(0, \Omega(\gamma)\mathfrak{w}^{\gamma'}t\partial_{n_{\gamma'}}\right), \mathfrak{w}^{\gamma''}\left(0, \Omega(\gamma''')\mathfrak{w}^{\gamma'''}\partial_{n_{\gamma'''}}\right)]\in \tilde{\mathbf{L}}
\end{align*}
and they are explicitly computed below:
\begin{align*}
(1)&\,[\mathfrak{w}^{\gamma}\left(E_{ij}t\mathfrak{w}^{m(\gamma_{ij})},0\right), \mathfrak{w}^{\gamma'}\left(E_{kl}t\mathfrak{w}^{m(\gamma_{kl})},0\right)]=\left(\mathfrak{w}^{\gamma}\mathfrak{w}^{\gamma'}[E_{ij},E_{kl}]_{\mathfrak{gl}(n)}\mathfrak{w}^{m(\gamma_{ij})}\mathfrak{w}^{m(\gamma_{kl})},0\right) \\
(2)&\,[\mathfrak{w}^{\gamma}\left(E_{ij}\mathfrak{w}^{m(\gamma_{ij})},0\right), \mathfrak{w}^{\gamma''}\left(0, \Omega(\gamma)\mathfrak{w}^{\gamma'}\partial_{n_{\gamma'}}\right)]=\left(-E_{ij}\Omega(\gamma')\langle\gamma+m(\gamma_{ij}),n_{\gamma'}\rangle \mathfrak{w}^{m(\gamma_{ij})}\mathfrak{w}^{\gamma}\mathfrak{w}^{\gamma''}\mathfrak{w}^{\gamma'},0\right)\\
(3)&\,[\mathfrak{w}^{\gamma}\left(0, \Omega(\gamma)\mathfrak{w}^{\gamma'}\partial_{n_{\gamma'}}\right), \mathfrak{w}^{\gamma''}\left(0, \Omega(\gamma''')\mathfrak{w}^{\gamma'''}\partial_{n_{\gamma'''}}\right)]=\Big(0,\Omega(\gamma)\Omega(\gamma''')\mathfrak{w}^{\gamma}\mathfrak{w}^{\gamma'}\mathfrak{w}^{\gamma''}\mathfrak{w}^{\gamma'''}\cdot\\
&\qquad\qquad\cdot\left(\langle\gamma''+\gamma''',n_{\gamma'}\rangle\partial_{n_{\gamma'''}}-\langle\gamma+\gamma',n_{\gamma'''}\rangle\partial_{n_{\gamma'}}\right)\Big).
\end{align*}

\end{proof}

\begin{theorem}\label{thm:homomLiering}Let $\left(\mathbf{L}_\Gamma,[\cdot,\cdot]_{\Der(\C[\mathbb{G}])}\right)$ and $\left(\tilde{\mathbf{L}},[\cdot,\cdot]_{\tilde{\mathfrak{h}}}\right)$ be the $\C[\Gamma]$-modules defined before. 
Assume $\omega(\gamma,a)=\Omega(\gamma)\langle{ a},n_{\gamma}\rangle$, then there exists a homomorphism of $\C[\Gamma]$-modules and of Lie rings $\Upsilon\colon \mathbf{L}_\Gamma\to\tilde{\mathbf{L}}$, which is defined as follows:
\begin{equation}
\begin{split}
&\Upsilon(X_{\gamma}\mathfrak{d}_{\gamma_{ij}})\defeq\mathfrak{w}^{\gamma}\left(-E_{ij}\mathfrak{w}^{m(\gamma_{ij})},0\right), \forall i\neq j\in\mathcal{V}, \forall\gamma\in\Gamma;\\
&\Upsilon(X_{\gamma'}\mathfrak{d}_\gamma)\defeq \mathfrak{w}^{\gamma'}\left(0, \Omega(\gamma)\mathfrak{w}^{\gamma}\partial_{n_\gamma}\right), \forall \gamma',\gamma\in\Gamma
\end{split} 
\end{equation} 
\end{theorem}

\begin{remark}
The assumption on $\omega$ is compatible with its Definition \eqref{eq:autK}, indeed by linearity of the pairing $\langle\cdot,\cdot\rangle$, 
\[\omega(\gamma,a+b)=\Omega(\gamma)\langle {a+b},n_\gamma\rangle=\Omega(\gamma)\langle{a},n_\gamma\rangle+\Omega(\gamma)\langle {b},n_\gamma\rangle=\omega(\gamma,a)+\omega(\gamma, b).\] 
Moreover notice that by the assumption on $\omega$, $\mathbf{L}_\Gamma$ turns out to be the $\C[\Gamma]$-module generated by $\mathfrak{d}_{e_{ij}}$ for every $i\neq j\in\mathcal{V}$ and by $\mathfrak{d}_{\gamma}$ for every primitive $\gamma\in\Gamma$. Indeed if $\gamma'$ is not primitive, then there exists a $\gamma\in\Gamma_{\text{prim}}$ such that $\gamma'=k\gamma$. Hence $\mathfrak{d}_{k\gamma}=\omega(k\gamma,\cdot)X_{\gamma}^{(k-1)}X_\gamma=CX_{(k-1)\gamma}\mathfrak{d}_{\gamma}$, where $C=\frac{k\Omega(k\gamma)}{\Omega(\gamma)}$.  In particular, if $\gamma,\gamma'$ are primitive vectors in $\Gamma$, then $\omega(\gamma,\gamma')=\Omega(\gamma)\langle {\gamma'},n_{\gamma}\rangle=\Omega(\gamma)\langle \gamma,\gamma'\rangle_D$.     
\end{remark}

\begin{proof}

We have to prove that $\Upsilon$ preserves the Lie-bracket, i.e. that for every $l_1,l_2\in L$, then $\Upsilon\left(\left[l_1,l_2\right]_{\mathbf{L}_\Gamma}\right)=\left[\Upsilon(l_1),\Upsilon(l_2)\right]_{\tilde{\mathbf{L}}}$. In particular, by $\C$-linearity it is enough to prove the following identities:
\begin{equation*}
\begin{split}
(1)& \Upsilon\left(\left[X_\gamma\mathfrak{d}_{\gamma_{ij}},X_{\gamma'}\mathfrak{d}_{\gamma_{kl}}\right]_{\mathbf{L}_\Gamma}\right)=\left[\Upsilon(X_\gamma\mathfrak{d}_{\gamma_{ij}}),\Upsilon(X_{\gamma'}\mathfrak{d}_{\gamma_{kl}})\right]_{\tilde{\mathbf{L}}}\\
(2)& \Upsilon\left(\left[X_\gamma\mathfrak{d}_{\gamma_{ij}},X_{\gamma'}\mathfrak{d}_{\gamma''}\right]_{\mathbf{L}_\Gamma}\right)=\left[\Upsilon(X_{\gamma}\mathfrak{d}_{\gamma'}),\Upsilon(X_{\gamma''}\mathfrak{d}_{\gamma_{kl}})\right]_{\tilde{\mathbf{L}}}\\
(3)& \Upsilon\left(\left[X_\gamma\mathfrak{d}_{\gamma'},X_{\gamma''}\mathfrak{d}_{\gamma'''}\right]_{\mathbf{L}_\Gamma}\right)=\left[\Upsilon(X_{\gamma}\mathfrak{d}_{\gamma'}),\Upsilon(X_{\gamma''}\mathfrak{d}_{\gamma'''})\right]_{\tilde{\mathbf{L}}}.
\end{split}
\end{equation*}
The identity $(1)$ is proved below:
\begin{equation*}
\begin{split}
   \text{LHS}&=\Upsilon\left(X_{\gamma}X_{\gamma'}{\sigma(\gamma_{ij},\gamma_{kl})}\mathsf{ad}_{\gamma_{ij}+\gamma_{kl}}-X_{\gamma}X_{\gamma'}{\sigma(\gamma_{kl},\gamma_{ij})}\mathsf{ad}_{\gamma_{kl}+\gamma_{ij}}\right)\\
            &=\mathfrak{w}^{\gamma}\mathfrak{w}^{\gamma'}\left(E_{ij}E_{kl}\mathfrak{w}^{m(\gamma_{ij})}\mathfrak{w}^{m(\gamma_{kl})}-E_{kl}E_{ij}\mathfrak{w}^{m(\gamma_{kl})}\mathfrak{w}^{m(\gamma_{ij})},0\right)\\
            &=\left(\mathfrak{w}^{\gamma}\mathfrak{w}^{\gamma'}[E_{ij}\mathfrak{w}^{m(\gamma_{ij})},E_{kl}\mathfrak{w}^{m(\gamma_{kl})}],0\right)\\
            \text{RHS}&=\left[\left(\mathfrak{w}^{\gamma} E_{ij}\mathfrak{w}^{m(\gamma_{ij})},0\right),\left(\mathfrak{w}^{\gamma'}E_{kl}\mathfrak{w}^{m(\gamma_{kl})},0\right)\right]_{\tilde{\mathfrak{h}}}\\
            &=\left(\mathfrak{w}^{\gamma}\mathfrak{w}^{\gamma'}[E_{ij}\mathfrak{w}^{m(\gamma_{ij})},E_{kl}\mathfrak{w}^{m(\gamma_{kl})}]_{\mathfrak{gl}(n)},0\right)\\
            &=\left(\mathfrak{w}^{\gamma}\mathfrak{w}^{\gamma'}[E_{ij}\mathfrak{w}^{m(\gamma_{ij})},E_{kl}\mathfrak{w}^{m(\gamma_{kl})}],0\right).
        \end{split}
    \end{equation*}
Then the second identity can be proved as follows:    
\begin{equation*}
    \begin{split}
            \text{LHS}&=\Upsilon\left(\omega(\gamma'',\gamma+\gamma_{ij})X_{\gamma+\gamma'+\gamma''}\mathfrak{d}_{\gamma_{ij}}\right)\\
            &=-\omega(\gamma'',\gamma+\gamma_{ij})\mathfrak{w}^{\gamma}\mathfrak{w}^{\gamma'}\mathfrak{w}^{\gamma''}\left(E_{ij}\mathfrak{w}^{m(\gamma_{ij})},0\right)\\
            \text{RHS}&=\left[\left(\mathfrak{w}^{\gamma} E_{ij}\mathfrak{w}^{m(\gamma_{ij})},0\right),\left(0,\mathfrak{w}^{\gamma'}\Omega(\gamma'')\mathfrak{w}^{\gamma''}\partial_{n_{\gamma''}}\right)\right]_{\tilde{\mathfrak{h}}}\\
            &=-\left(\mathfrak{w}^{\gamma}\mathfrak{w}^{\gamma'}E_{ij}\Omega(\gamma'')\langle m(\gamma_{ij})+\gamma, n_{\gamma''} \rangle \mathfrak{w}^{m(\gamma_{ij})}\mathfrak{w}^{\gamma''},0\right).
        \end{split}
    \end{equation*}
Finally the third identity is proved below:    
\begin{equation*}
            \begin{split}
            \text{LHS}&=\Upsilon\left(\omega(\gamma',\gamma''+\gamma''')X_{\gamma}X_{\gamma''}X_{\gamma'''}\mathfrak{d}_{\gamma'''}-\omega(\gamma''',\gamma+\gamma')X_{\gamma'''}X_{\gamma''}X_{\gamma}\mathfrak{d}_{\gamma'}\right)\\
            &=\left(0,\omega(\gamma',\gamma''+\gamma''')\mathfrak{w}^{\gamma}\mathfrak{w}^{\gamma'}\mathfrak{w}^{\gamma''}\Omega(\gamma''')\mathfrak{w}^{\gamma'''}\partial_{n_{\gamma'''}}-\omega(\gamma''',\gamma+\gamma')\mathfrak{w}^{\gamma'''}\mathfrak{w}^{\gamma''}\mathfrak{w}^{\gamma}\Omega(\gamma')\mathfrak{w}^{\gamma'}\partial_{n_{\gamma'}}\right)\\
            \text{RHS}&=\left[\left(0,\mathfrak{w}^{\gamma}\Omega(\gamma')\mathfrak{w}^{\gamma'}\partial_{n_{\gamma'}} \right),\left(0,\mathfrak{w}^{\gamma''}\Omega(\gamma''')\mathfrak{w}^{\gamma'''}\partial_{n_{\gamma'''}}\right)\right]_{\tilde{\mathfrak{h}}}\\
            &=\left(0, \Omega(\gamma')\Omega(\gamma''')\left[\mathfrak{w}^{\gamma}\mathfrak{w}^{\gamma'}\partial_{n_{\gamma'}},\mathfrak{w}^{\gamma''}\mathfrak{w}^{\gamma'''}\partial_{n_{\gamma'''}}\right]_{\tilde{\mathfrak{h}}}\right)\\
            &=\left(0, \Omega(\gamma')\Omega(\gamma''')\mathfrak{w}^{\gamma}\mathfrak{w}^{\gamma'}\mathfrak{w}^{\gamma''}\mathfrak{w}^{\gamma'''}\left(\langle\gamma''+\gamma''', n_{\gamma'}\rangle\partial_{n_{\gamma'''}}-\langle\gamma+\gamma', n_{\gamma'''}\rangle\partial_{n_{\gamma'}}\right)\right).
        \end{split}
    \end{equation*}    
\end{proof}

Let us now show which is the correspondence between WCFs in coupled $2d$-$4d$ systems and scattering diagrams which come from solutions of the Maurer-Cartan equation for deformations of holomorphic pairs:
\begin{enumerate}
    \item to every BPS ray $l_a=Z(a)\R_{>0}$ we associate a ray $\mathfrak{d}_{a}=m(a)\R_{>0}$ if either $\mu(a)\neq\mu(a)'$ or $\omega(a,\cdot)\neq\omega(a,\cdot)'$. Conversely we associate a line $\mathfrak{d}_a=m(a)\R$;
    \item to the automorphism $S_{\gamma_{ij}}^\mu$ we associate an automorphism $\theta_S\in\exp(\tilde{\mathfrak{h}})$ such that $\log(\theta_S)=\Upsilon(\mathfrak{s}_{\gamma_{ij}})=\left(-\mu(\gamma_{ij})tE_{ij}\mathfrak{w}^{m(\gamma_{ij})},0\right)$;
\item to the automorphism $K_\gamma^\omega$ we associate an automorphism $\theta_K\in\exp(\tilde{\mathfrak{h}})$ such that $\log(\theta_K)=\Upsilon(\mathfrak{k}_\gamma)=\left(0,\Omega(\gamma)\sum_l\frac{1}{l}t^l\mathfrak{w}^{l\gamma}\partial_{n_{\gamma}}\right)$. 
\end{enumerate}
\begin{remark}
If $m(\gamma_{ij})=m(\gamma_{il}+\gamma_{lj})$ then $\log(\theta_S)=\Upsilon(\mathfrak{s}_{\gamma_{ij}})=\Upsilon(\mathfrak{s}_{\gamma_{il}})\Upsilon(\mathfrak{s}_{\gamma_{lj}})$. Analogously if $m(\gamma_{ij}')=m(\gamma_{ij})+k\gamma$ then $\log(\theta_S')=\Upsilon(\mathfrak{s}_{\gamma_{ij}'})=t^k\Upsilon(\mathfrak{s}_{\gamma_{ij}})$. 
\end{remark}
In the following examples we will show this correspondence in practice: we consider two examples of WCFs computed in \cite{WCF2d-4d} and we construct the corresponding consistent scattering diagram. 

\subsection{Example 1}\label{Ex:1}
Let $\mathcal{V}=\lbrace i,j,k=l\rbrace$ and set $\gamma_{kk}=\gamma\in\Gamma$. Assume $\omega(\gamma,\gamma_{ij})=-1$ and $\mu(\gamma_{ij})=1$, then the wall-crossing formula (equation 2.39 in \cite{WCF2d-4d}) is 
\begin{equation}
    K_\gamma^\omega S_{\gamma_{ij}}^\mu=S_{\gamma_{ij}}^{\mu'}S_{\gamma_{ij}+\gamma}^{\mu'}K_\gamma^{\omega'} 
\end{equation}
with $\mu'(\gamma_{ij})=1$, $\mu'(\gamma+\gamma_{ij})=-1$ and $\omega'=\omega$.

Since $\mu'(\gamma_{ij})=\mu(\gamma_{ij})$ and $\omega'=\omega$ the initial scattering diagram has two lines. In addition, since $-1=\omega(\gamma,\gamma_{ij})=\Omega(\gamma)\langle m(\gamma_{ij}), n_\gamma\rangle$, we can assume $\Omega(\gamma)=1$, $m(\gamma_{ij})=(1,0)$ and $\gamma=(0,1)$. Therefore the initial scattering diagram is   
\[
\mathfrak{D}=\left\lbrace\mathsf{w}_S=\left(m_S=m(\gamma_{ij}), \mathfrak{d}_S, \theta_S\right),\mathsf{w}_K=\left(m_K=\gamma, \mathfrak{d}_K, \theta_K\right) \right\rbrace\] where $\log\theta_S=\left(-tE_{ij}\mathfrak{w}^{m(\gamma_{ij})}, 0\right)$ and $\log\theta_K=\left(0, \sum_{l\geq 1} \frac{1}{l}t^l \mathfrak{w}^{l \gamma}\partial_{n_\gamma}\right)$. Then the wall crossing formula says that the complete scattering diagram $\mathfrak{D}^{\infty}$ has one more S-ray, $\mathfrak{d}_{S+K}=(\gamma+m(\gamma_{ij}))\R_{\geq 0}$ and wall-crossing factor $\log\theta_{S+K}={\left(t^2E_{ij}\mathfrak{w}^{\gamma+m(\gamma_{ij})},0\right)}$. 
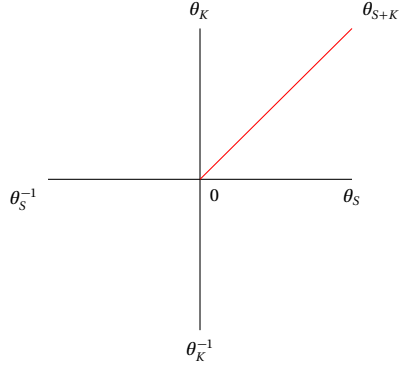
\begin{figure}[h!]
    \centering
    \begin{tikzpicture}
    \draw (0,2) -- (4,2);
    \draw (2,0) -- (2,4);
    \draw[red](2,2) -- (4,4);
    \node[font=\tiny, below right] at (2,2) {0};
    \node[font=\tiny, below] at (4,2) {$\theta_S$};
    \node[font=\tiny, above] at (2,4) {$\theta_{K}$};
    \node[font=\tiny, below] at (2,0) {$\theta_{K}^{-1}$};
    \node[font=\tiny, below left] at (0,2) {$\theta_S^{-1}$};
    \node[font=\tiny, above right] at (4,4) {$\theta_{S+K}$};
    \end{tikzpicture}
    \caption{The complete scattering diagram with K and S rays.}
    \label{fig:es2}
\end{figure}

We can check that $\mathfrak{D}^{\infty}$ is consistent (see Definition \ref{def:pathorderedprod}). In particular we need to prove the following identity:
\begin{equation}
\theta_{K}\circ\theta_S\circ\theta_{K^{-1}}=\theta_S\circ\theta_{S+K}
\end{equation}
\begin{equation*}
\begin{split}
\text{RHS}&=\theta_S\circ\theta_{S+K}\\
&=\exp(\log\theta_S)\circ\exp(\log\theta_{S+K})\\
&=\exp(\log\theta_S{\bullet}_{\text{BCH}}\log\theta_{S+K})\\
&=\exp\left(\log\theta_S+\log\theta_{S+K}\right)\\
\text{LHS}&=\theta_{K}\circ\theta_S\circ\theta_{K^{-1}}\\
&=\exp(\log\theta_K\bullet_{\text{BCH}}\log\theta_S\bullet_{\text{BCH}}\log\theta_{K^{-1}})\\
&=\exp\left(\log\theta_S+\sum_{l\geq1}\frac{1}{l!}\ad_{\log\theta_K}^l\log\theta_S\right)\\
&=\exp\left(\log\theta_S+[\log\theta_K,\log\theta_s]+\sum_{l\geq2}\frac{1}{l!}\ad_{\log\theta_K}^l\log\theta_S\right)\\
&=\exp\left(\log\theta_S-\left(E_{ij}\sum_{k\geq1}\frac{1}{k}\mathfrak{w}^{m(\gamma_{ij})+k\gamma}\langle m(\gamma_{ij}), n_{\gamma}\rangle,0\right)+\sum_{l\geq2}\frac{1}{l!}\ad_{\log\theta_K}^l\log\theta_S\right)\\
&=\exp\left(\log\theta_S+\left(t^2E_{ij}\mathfrak{w}^{m(\gamma_{ij})+\gamma},0\right)+\left(E_{ij}\sum_{k\geq2}\frac{1}{k}t^{k+1}\mathfrak{w}^{m(\gamma_{ij})+k\gamma},0\right)+\sum_{l\geq2}\frac{1}{l!}\ad_{\log\theta_K}^l\log\theta_S\right).
\end{split}
\end{equation*}
We claim that 
\begin{equation}\label{comput:claim}
-\left(E_{ij}\sum_{k\geq2}\frac{1}{k}t^{k+1}\mathfrak{w}^{m(\gamma_{ij})+k\gamma},0\right)=\sum_{l\geq2}\frac{1}{l!}\ad_{\log\theta_K}^l\log\theta_S
\end{equation}
and we compute it explicitly:
\begin{equation}\label{comput:ex1}
\begin{split}
\sum_{l\geq2}\frac{1}{l!}\ad_{\log\theta_K}^l\log\theta_S&=\sum_{l\geq2}\frac{1}{l!}\left(-tE_{ij}(-1)^l\sum_{k_1,\cdots,k_l\geq1}\frac{1}{k_1\cdots k_l}t^{k_1+\cdots+k_l}\mathfrak{w}^{(k_1+\cdots+k_l)\gamma+m(\gamma_{ij})},0\right)\\
&=-\sum_{l\geq2}\frac{(-1)^l}{l!}\left(tE_{ij}\mathfrak{w}^{m(\gamma_{ij})}\left(\sum_{k_1,\cdots,k_l\geq1}\frac{1}{k_1\cdots k_l}t^{k_1+\cdots+k_l}\mathfrak{w}^{(k_1+\cdots+k_l)\gamma}\right),0\right)\\
&=-\sum_{l\geq2}\frac{(-1)^l}{l!}\left(tE_{ij}\mathfrak{w}^{m(\gamma_{ij})}\left(\sum_{k\geq1}\frac{1}{k}t^k\mathfrak{w}^{k\gamma}\right)^l,0\right)\\
&=-\left(tE_{ij}\mathfrak{w}^{m(\gamma_{ij})}\sum_{l\geq2}\frac{(-1)^l}{l!}\left(\sum_{k\geq 1}\frac{1}{k}t^k\mathfrak{w}^{k\gamma}\right)^l,0\right)\\
&=-\left(tE_{ij}\mathfrak{w}^{m(\gamma_{ij})}\left(\exp\left(-\sum_{k\geq1}\frac{1}{k}t^k\mathfrak{w}^{k\gamma}\right)+\sum_{k\geq1}\frac{1}{k}t^k\mathfrak{w}^{k\gamma}-1\right),0\right)\\
&=-\left(tE_{ij}\mathfrak{w}^{m(\gamma_{ij})}\left( \left(1-t\mathfrak{w}^{\gamma}\right)+\sum_{k\geq 1}\frac{1}{k}t^k\mathfrak{w}^{k\gamma}-1\right),0\right)\\
&=-\left(tE_{ij}\mathfrak{w}^{m(\gamma_{ij})} \sum_{k\geq 2}\frac{1}{k}t^k\mathfrak{w}^{k\gamma},0\right).
\end{split}
\end{equation}

\subsection{Example 2}\label{Ex:2}
Finally let us give a example with only $S$-rays: assume $\mathcal{V}={i=l,j=k}$, then the wall-crossing formula (equation (2.35) in \cite{WCF2d-4d}) is
\begin{equation}
S_{\gamma_{ij}}^\mu S_{\gamma_{il}}^\mu S_{\gamma_{jl}}^\mu=S_{\gamma_{jl}}^\mu S_{\gamma_{il}}^{\mu'} S_{\gamma_{ij}}^\mu
\end{equation}
with $\gamma_{il}\defeq\gamma_{ij}+\gamma_{jl}$ and $\mu'(\gamma_{il})=\mu(\gamma_{il})-\mu(\gamma_{ij})\mu(\gamma_{jl})$. Let us further assume that $\mu(il)=0$, then the associated initial scattering has two lines:
\begin{equation*}
\mathfrak{D}=\left\lbrace\mathsf{w}_1=\left(m_1=m(\gamma_{ij}),\mathfrak{d}_1=m_1+\R,\theta_{S_1}\right), \mathsf{w}_2=\left\lbrace m_2=m(\gamma_{jk}), \mathfrak{d}_2=m_2\R, \theta_{S_2}\right\rbrace\right\rbrace
\end{equation*}
with \begin{equation*}
\begin{split}
\log\theta_{S_1}&=-\left(\mu(\gamma_{ij})tE_{ij}\mathfrak{w}^{m(\gamma_{ij})},0\right)\\
\log\theta_{S_2}&=-\left(\mu(\gamma_{jl})tE_{jl}\mathfrak{w}^{m(\gamma_{jl})},0\right).\\
\end{split}
\end{equation*}
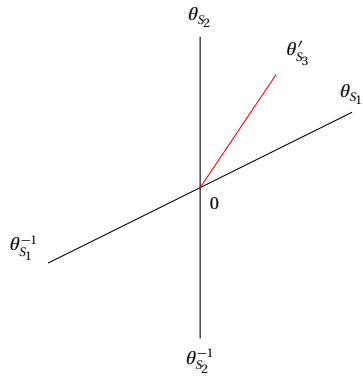
\begin{figure}[h]
    \centering
    \begin{tikzpicture}
    \draw (0,1) -- (4,3);
    \draw (2,0) -- (2,4);
    \draw [red](2,2) -- (3,3.5);
    \node[font=\tiny, below right] at (2,2) {0};
    \node[font=\tiny, below] at (4,3.5) {$\theta_{S_1}$};
    \node[font=\tiny, above] at (2,4) {$\theta_{S_2}$};
    \node[font=\tiny, below] at (2,0) {$\theta_{S_2}^{-1}$};
    \node[font=\tiny, below left] at (0,1.5) {$\theta_{S_1}^{-1}$};
    \node[font=\tiny, above right] at (3,3.5) {$\theta_{S_3}'$};
    \end{tikzpicture}
    \caption{The complete scattering diagram with only S rays.}
    \label{fig:es3}
\end{figure} 

Its completion has one more ray $\mathfrak{d}_3'=(m_1+m_2)\R_{\geq0}$ decorated with the automorphism $\theta_{3}'$ such that  
\[\log\theta_3'=\left(\mu(\gamma_{ij})\mu(\gamma_{jl})t^2E_{il}\mathfrak{w}^{m(\gamma_{il})},0\right).\] 
 
Since the path order product involves matrix commutators, the consistency of $\mathfrak{D}^{\infty}$ can be easily verified. 

\begin{remark}
In the latter example we assume $\mu(\gamma_{il})=0$ in order to have an initial scattering diagram with only two rays, as in our construction of solution of Maurer--Cartan equation in Section \ref{sec:two_walls}. However the general formula can be computed with the same rules, by adding the wall \[\mathsf{w}=\left(-(m_1+m_2),\log\theta_3=\left(-\mu(\gamma_{il})t^2E_{il}\mathfrak{w}(\gamma_{il})\right),0\right)\] to the initial scattering diagram.   
\end{remark}

\chapter{Gromov--Witten invariants in the extended tropical vertex}\label{cha:enumertaive geom}

In this chapter all scattering diagrams are defined in the extended tropical vertex group $\tilde{\mathbb{V}}$. We start with some preliminaries, following the approach of \cite{GPS}. 

\begin{notation}
We denote $\mathfrak{gl}(r,\C)_e\subset\mathfrak{gl}(r,\C)$ the subset of elementary matrices in $\mathfrak{gl}(r,\C)$, namely the matrices with only a non zero entry in position $ij$ for $i<j$. 
\end{notation}

\textbf{Deformation technique}: let $\mathfrak{D}=\left\lbrace(\mathfrak{d}_i=m_i\R,\overrightarrow{f}_{i})\vert i=1,... ,n \right\rbrace$ be a scattering diagram which consists of walls through the origin, as in figure \ref{fig:Dloop}. Assume each $\overrightarrow{f}_i$ can be factored as $\overrightarrow{f}_i=\prod_j\overrightarrow{f}_{ij}=\prod_j\left((1+A_{ij}z^{m_{i}'}), f_{ij}\right)$. Then replace each line $(\mathfrak{d}_i,\overrightarrow{f}_{i})$ with the collection of lines $\lbrace(\xi_{ij}+\mathfrak{d}_i, \overrightarrow{f}_{ij})\rbrace$, where $\xi_{ij}\in \Lambda_{\R}$ is chosen generically. This new scattering diagram is denoted by $\tilde{\mathfrak{D}}$. 

As an example let us consider the scattering diagram: 
\[
\begin{split}
\mathfrak{D}=&\lbrace\left((1,0)\R, \left(1,(1+u_{11}x)\right)\left(1,(1+u_{12}x)\right)\left(1,(1-u_{11}u_{12}x^2)\right)\right), \\
&\left((0,1)\R, \left((1+Au_{21}y),(1+u_{21}y)\right)\left((1+Au_{22}y),(1+u_{22}y)\right)\left(1, (1-u_{21}u_{22}y^2)\right)\right)\rbrace
\end{split}
\]   
where $A\in\mathfrak{gl}_e(r,\C)$. Then if we apply the deformation techniques we get the following diagram of Figure \ref{fig:D_deform}.

\begin{figure}
\center
\begin{tikzpicture}
\draw (2,0) -- (2,6);
\draw (3,0) -- (3,6);
\draw (5.5,0)-- (5.5,6);
\draw (0,1) -- (6,1);
\draw (0, 2.5) -- (6,2.5);
\draw (0, 4) -- (6,4);
\node [below left, font=\tiny] at (2,0) {$\left((1+Au_{21}y),(1+u_{21}y)\right) $};
\node [below , font=\tiny] at (3.5,0) {$\left((1+Au_{22}y),(1+u_{22}y)\right)$};
\node [below right, font=\tiny] at (5.5,0) {$\left(1,(1-u_{21}u_{22}y^2)\right)$};
\node [below, font=\tiny] at (0,1) {$\left(1,(1+u_{11}x)\right)$};
\node [below , font=\tiny] at (0,2.5) {$\left(1,(1+u_{12}x)\right)$};
\node [below,  font=\tiny] at (0,4) {$\left(1,(1-u_{11}u_{12}x^2)\right)$};

\end{tikzpicture}
\caption{Deformed scattering diagram}
\label{fig:D_deform}
\end{figure}
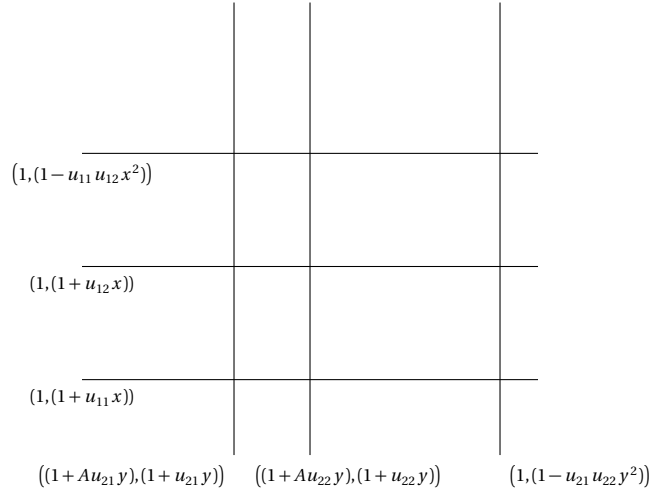

\begin{definition}
Let $\mathfrak{D}$ be a scattering diagram, then the \textit{asymptotic scattering diagram} $\mathfrak{D}_{as}$ is obtained by replacing each ray $(\xi_i+m_i\R_{\geq 0},\overrightarrow{f}_i)$ with $(\R_{\geq 0}m_i,\overrightarrow{f}_{i})$ and every line $(\xi_i+m_i\R,\overrightarrow{f}_i)$ with a parallel line through the origin $(m_i\R,\overrightarrow{f}_i )$. If two lines/rays in $\mathfrak{D}$ have the same slope given by the vector $m_i$ and functions $\overrightarrow{f}_i, \overrightarrow{f}_{i'}$, then the function $\overrightarrow{f}$ attached to the line/ray through the origin is the product of the $\overrightarrow{f}_i$ and $\overrightarrow{f}_{i'}$. 
\end{definition}


\begin{notation}
Let $m_1,m_2$ be two primitive vectors in $\Lambda$ with coordinates $m_1=(m_{11}, m_{12})$ and $m_2=(m_{21},m_{22})$. Let also choose the anticlockwise orientation on $\Lambda_{\R}$ and assume $m_1$ and $m_2$ are positive oriented. We denote by
\begin{align}
|m_1\wedge m_2|=|m_{11}m_{22}-m_{12}m_{21}| 
\end{align}
Moreover, let $n_i=(-m_{i2}, m_{i1})$ be the primitive vector orthogonal to $m_i$ and positive oriented, then  
\begin{align*}
\langle m_2, n_1\rangle=-\langle m_1,n_2\rangle = |m_1\wedge m_2|
\end{align*}
where $\langle-,-\rangle$ is the pairing between $\Lambda$ and $\Lambda^*$.
Finally if $m'=wm$ for some primitive $m$ and some integer $w$, then $w$ is called \textit{index} of $m'$.  
\end{notation}

\begin{notation}
Let $\mathfrak{d}_1, \mathfrak{d}_2$ be two rays (or lines) in $\tilde{\mathfrak{D}}_k$ such that $\mathfrak{d}_1\cap\mathfrak{d}_2\neq \emptyset$ and let $\mathfrak{d}_{\mathrm{out}}$ a third ray emanating from $\mathfrak{d}_1$ and $\mathfrak{d}_2$. Then we define 
\begin{equation}
\text{Parents}(\mathfrak{d}_{\mathrm{out}})=\begin{cases}\lbrace\mathfrak{d}_1,\mathfrak{d}_2\rbrace & \text{ if } \mathfrak{d}_1,\mathfrak{d}_2 \text{ are rays}\\
\emptyset & \text{ otherwise}
\end{cases}
\end{equation}
in which case $\mathfrak{d}_{\mathrm{out}}\in\text{Child}(\mathfrak{d}_1), \text{Child}(\mathfrak{d}_2)$;
\begin{equation}
\text{Ancestors}(\mathfrak{d}_{\mathrm{out}})=\lbrace\mathfrak{d}_{\mathrm{out}}\rbrace\cup \bigcup_{\mathfrak{d}'\in\text{Parents }(\mathfrak{d}_{\mathrm{out}})} \text{ Ancestors}(\mathfrak{d}')
\end{equation}
and 
\begin{equation}
\text{Leaves}(\mathfrak{d}_{\mathrm{out}})=\lbrace \mathfrak{d}'\in\text{Ancestor }(\mathfrak{d}_{\mathrm{out}})\vert \mathfrak{d}' \text{ is a line}\rbrace. 
\end{equation}
\end{notation} 

\begin{definition}
A standard scattering diagram $\mathfrak{D}=\left\lbrace \mathsf{w}_i=\left(\mathfrak{d}_i, \theta_i\right), 1\leq i\leq n\right\rbrace$ over $R=\C [\![t_1,... ,t_n]\!]$, consists of finite collection of lines $\mathfrak{d}_i=m_i\R$ through the origin and automorphisms $\theta_i$ such that $\log\theta_i\in \C[z^{m_i}][\![t_i]\!]\cdot\left(\mathfrak{gl}(r,\mathbb{C}),\partial_{n_i}\right)$. We allow $m_{j}=m_{j'}$ for $j\neq j'$, but we keep $t_j\neq t_{j'}$. 
\end{definition}

Given a standard scattering diagram, there is a method to compute its completion iteratively over the ring $R_N=\C [t_1,... ,t_n]/(t_1^{N+1},... , t_n^{N+1})$. We start with a preliminary lemma

\begin{lemma}\label{lem:1.9}
Let 
\[
\mathfrak{D}=\left\lbrace\left(m_1\R, \overrightarrow{f}_1=\left(1+A_1t_1z^{l_1m_1},1+c_1t_1z^{l_1m_1}\right)\right), \overrightarrow{f}_2=\left(m_2\R, \left(1+A_2t_2z^{l_2m_2}, 1+c_2t_2z^{l_2m_2}\right)\right) \right\rbrace
\]
with $A_1,A_2\in\mathfrak{gl}(r,\C)$, $m_1,m_2$ primitive vector in $\Lambda$, $l_1,l_2\in\Z_{> 0}$ and $c_1,c_2\in\C$. Then the consistent scattering diagram $\mathfrak{D}^{\infty}$ on $R_2$ is obtained by adding a single wall 
\[\mathsf{w}_{\mathrm{out}}=\left(\mathfrak{d}_{\mathrm{out}}=\R_{\geq 0}(l_1m_1+l_2m_2),\overrightarrow{f}_{\mathrm{out}}\right)\] 
with function
\begin{equation}\label{eq:1.9}
\begin{split}
\overrightarrow{f}_{\mathrm{out}}=&\Big(1+([A_1,A_2]+A_2c_1l_2|m_1\wedge m_2|+A_1c_2l_1|m_1\wedge m_2|)t_1t_2z^{l_1m_1+l_2m_2},\\
& 1+c_1c_2t_1t_2l_{\mathrm{out}}|m_1\wedge m_2|z^{l_1m_1+l_2m_2}\Big)
\end{split}
\end{equation}
where $l_{\mathrm{out}}$ is the index of $l_1m_1+l_2m_2$.
If $m_{\mathrm{out}}=0$, then no wall is added since $|m_1\wedge m_2|=0$ and $A_i$ commutes with itself.  
\end{lemma}
\begin{proof}
Let $\gamma$ be a generic loop around the origin, such that the path order product $\Theta_{\gamma,\mathfrak{D}^{\infty}}=\theta_2^{-1}\theta_1^{-1}\theta_2\theta_1$, where
\begin{align*}
&\log\theta_1=(\left(A_1t_1z^{l_1m_1}, c_1t_1z^{l_1m_1}\partial_{n_1}\right)\\
&\log\theta_2=\left(A_2t_2z^{l_2m_2}, c_2t_2z^{l_2m_2}\partial_{n_2}\right).
\end{align*} 

We now compute $\theta_2^{-1}\theta_1^{-1}\theta_2\theta_1$: we begin with $\theta_1^{-1}\theta_2\theta_1$
\begin{align*}
\theta_1^{-1}\theta_2\theta_1&=\exp\left(-\log\theta_1\bullet\log\theta_2\bullet\log\theta_1\right)\\
&=\exp(\left(A_2t_2z^{l_2m_2}, c_2t_2z^{l_2m_2}\partial_{n_2}\right)+\sum_{l\geq 1}\frac{1}{l!}\ad^l_{-\log\theta_1}\left(\left(A_2t_2z^{l_2m_2}, c_2t_2z^{l_2m_2}\partial_{n_2}\right)\right))
\end{align*}

\begin{align*}
&=\exp(\left(A_2t_2z^{l_2m_2}, c_2t_2z^{l_2m_2}\partial_{n_2}\right)+\\
&\qquad+((-[A_1,A_2]-A_2c_1l_2|m_1\wedge m_2|-A_1c_2l_1|m_1\wedge m_2|)t_1t_2z^{l_1m_1+l_2m_2},\\
&\qquad-c_1c_2t_1t_2w_{\mathrm{out}}|m_1\wedge m_2|z^{l_1m_1+l_2m_2})+\\
&\qquad+\sum_{l\geq 2}\frac{1}{l!}\ad^l_{(-\log\theta_1)}\left(\left(A_2t_2z^{l_2m_2}, c_2t_2z^{l_2m_2}\partial_{n_2}\right)\right)))\\
&\exp\Big(\left(A_2t_2z^{l_2m_2}, c_2t_2z^{l_2m_2}\partial_{n_2}\right)+\\
&\qquad+\big((-[A_1,A_2]-A_2c_1l_2|m_1\wedge m_2|-A_1c_2l_1|m_1\wedge m_2|)t_1t_2z^{l_1m_1+l_2m_2},\\
\qquad&-c_1c_2t_1t_2l_{\mathrm{out}}|m_1\wedge m_2|z^{l_1m_1+l_2m_2}\partial_{n_{\mathrm{out}}}\big)\Big)
\end{align*}
where in the last step $\sum_{l\geq 2}\frac{1}{l!}\ad^l_{(-\log\theta_1)}\left(\left(A_2t_2z^{l_2m_2}, c_2t_2z^{l_2m_2}\partial_{n_2}\right)\right)$ vanish on $R_2$. Then 
\begin{align*}
\theta_2^{-1}\theta_1^{-1}\theta_2\theta_1&=\exp\bigg(-\log\theta_2\bullet\Big(\left(A_2t_2z^{l_2m_2}, c_2t_2z^{l_2m_2}\partial_{n_2}\right)+\\
&\qquad+\big((-[A_1,A_2]-A_2c_1l_2|m_1\wedge m_2|-A_1c_2l_1|m_1\wedge m_2|)t_1t_2z^{l_1m_1+l_2m_2},\\
&\qquad-c_1c_2t_1t_2l_{\mathrm{out}}|m_1\wedge m_2|z^{l_1m_1+l_2m_2}\partial_{n_{\mathrm{out}}}\big)\Big)\bigg)\\
&\exp\Big(-\left(A_2t_2z^{l_2m_2}, c_2t_2z^{lw_2m_2}\partial_{n_2}\right)+\left(A_2t_2z^{l_2m_2}, c_2t_2z^{l_2m_2}\partial_{n_2}\right)+\\
&\qquad-\big(([A_1,A_2]+A_2c_1l_2|m_1\wedge m_2|+A_1c_2l_1|m_1\wedge m_2|)t_1t_2z^{l_1m_1+l_2m_2},\\
&\qquad c_1c_2t_1t_2l_{\mathrm{out}}|m_1\wedge m_2|z^{l_1m_1+l_2m_2}\partial_{n_{\mathrm{out}}}\big)\Big)\\
&=\exp\Big(-\big(([A_1,A_2]+A_2c_1l_2|m_1\wedge m_2|+A_1c_2l_1|m_1\wedge m_2|)t_1t_2z^{l_1m_1+l_2m_2},\\ &\qquad c_1c_2t_1t_2l_{\mathrm{out}}|m_1\wedge m_2|z^{l_1m_1+l_2m_2}\partial_{n_{\mathrm{out}}}\big)\Big)
\end{align*}
where we omit the other terms in the BCH product because they vanish on $R_2$. Hence the automorphism of the new wall is 
\begin{equation}
\begin{split}
&\log\theta_{\mathrm{out}}=\big(([A_1,A_2]+A_2c_1l_2|m_1\wedge m_2|+A_1c_2l_1|m_1\wedge m_2|)t_1t_2z^{l_1m_1+l_2m_2},\\ &c_1c_2t_1t_2l_{\mathrm{out}}|m_1\wedge m_2|z^{l_1m_1+l_2m_2}\partial_{n_{\mathrm{out}}}\big)
\end{split}
\end{equation}
and the corresponding function $\overrightarrow{f}_{\mathrm{out}}$ is as \eqref{eq:1.9}.
\end{proof}

Let us introduce the ring $\tilde{R}_N$
\begin{equation}
\tilde{R}_N=\frac{\C[\lbrace u_{ij}\vert 1\leq i\leq n, 1\leq j\leq N\rbrace]}{\langle u_{ij}^2 \vert 1\leq i\leq n, 1\leq j\leq N \rangle}
\end{equation} 
while $\overrightarrow{f}_i$ may not factor on $R_N$, it does on $\tilde{R}_N$ by replacing $t_i$ with $t_i=\sum_{j=1}^N u_{ij}$. For instance, let 
\begin{equation}
\log\overrightarrow{f}_i=\sum_{j=1}^N\sum_{l\geq 1}\left(A_{ijl}t_i^jz^{lm_i}, a_{ijl}t_{i}^jz^{lm_i}\right)\quad \text{on } R_N
\end{equation}
then, since \[t_i^j=\sum_{\substack{J\subset\lbrace 1,...  N\rbrace\\ |J|=j}}j!\prod_{j'\in J}u_{ij'}\] 
we get:
\begin{equation}\label{eq:f_fact}
\log\overrightarrow{f}_i=\sum_{j=1}^N\sum_{l\geq 1}\sum_{\substack{J\subset\lbrace 1,... , N\rbrace\\ |J|=j}}j!\left(A_{ijl}\left(\prod_{j'\in J}u_{ij'}\right)z^{lm_i},a_{ijl}\left(\prod_{j'\in J}u_{ij'}\right)z^{lm_i}\right)\quad \text{on } \tilde{R}_N
\end{equation}
Hence we can define a new scattering diagram  
\[\tilde{\mathfrak{D}}_N=\left\lbrace (\mathfrak{d}_{iJl}, \theta_{iJl})\vert 1\leq i\leq n, l\geq 1, J\subset\lbrace 1,... , N\rbrace, \# J\geq 1 \right\rbrace\] such that for a generic choice of $\xi_{il}\in\R^2$, $\mathfrak{d}_{iJl}=\xi_{il}+m_i\R$ is a line parallel to $\mathfrak{d}_i$ and the associated function $\overrightarrow{f}_{iJl} $ is a single term in  Equation \eqref{eq:f_fact}, i.e.
\begin{align*}
\overrightarrow{f}_{ijl}=\left(1+ (\#J)! A_{i(\#J)l}\left(\prod_{j'\in J}u_{ij'}\right)z^{lm_i}, 1+ (\#J)! a_{i(\#J)l}\left(\prod_{j'\in J}u_{ij'}\right)z^{lm_i}\right).
\end{align*}
Then we produce a sequence of scattering diagrams $\tilde{\mathfrak{D}}_N=\tilde{\mathfrak{D}}_N^0, \tilde{\mathfrak{D}}_N^1, \tilde{\mathfrak{D}}_N^2, ...  $ which eventually stabilizes, and $\tilde{\mathfrak{D}}_N^{\infty}=\tilde{\mathfrak{D}}_N^i$ for $i$ large enough. Inductively we assume the following:
\begin{enumerate}
\item each wall in $\tilde{\mathfrak{D}}_N^i $ is of the form $\mathsf{w}=\left(\mathfrak{d},\overrightarrow{f}_{\mathfrak{d}}\right)$ with  \[\log\overrightarrow{f}_{\mathfrak{d}}=\left(A_{\mathfrak{d}}u_{I(\mathfrak{d})}z^{m_\mathfrak{d}}, a_\mathfrak{d} u_{I(\mathfrak{d})}z^{m_\mathfrak{d}}\partial_{n_{\mathfrak{d}}}\right)\]
with $u_{I(\mathfrak{d})}=\prod_{(i,j')\in I(\mathfrak{d})}u_{ij'}$ for some index set $I(\mathfrak{d})\subset\lbrace1,... ,n\rbrace\times\lbrace1,... , N\rbrace$;
\item for each $\xi\in\text{Sing}(\tilde{\mathfrak{D}}_k^i)$, there are no more than two rays emanating from $\xi$ and if $\mathfrak{d}_1,\mathfrak{d}_2$ are these rays, then $I(\mathfrak{d}_1)\cap I(\mathfrak{d}_2)=\emptyset$. 
\end{enumerate} 
These assumptions hold for $\tilde{\mathfrak{D}}_N^0$. Then to construct $\tilde{\mathfrak{D}}_N^{i+1}$ from $\tilde{\mathfrak{D}}_N^i$ we look at every pair of lines or rays $\mathfrak{d}_1,\mathfrak{d}_2\in\tilde{\mathfrak{D}}_N^i$ such that 
\begin{itemize}
\item[(i)] $\lbrace\mathfrak{d}_1,\mathfrak{d}_2\rbrace\notin\tilde{\mathfrak{D}}_N^{i-1}$;
\item[(ii)] $\mathfrak{d}_1\cap\mathfrak{d}_2=\lbrace \xi\rbrace$ and $\xi\neq\lbrace\partial\mathfrak{d}_1, \partial\mathfrak{d}_2\rbrace$;
\item[(iii)] $I(\mathfrak{d}_1)\cap I(\mathfrak{d}_2)=\emptyset$. 
\end{itemize}
Notice that by Lemma \ref{lem:1.9}, if $\mathfrak{d}_1\cap\mathfrak{d}_2=\lbrace \xi\rbrace$ and $I(\mathfrak{d}_1)\cap I(\mathfrak{d}_2)=\emptyset$, then there is only one ray emanating from $\xi$: $\mathfrak{d}_{\mathrm{out}}=\mathfrak{d}(\mathfrak{d}_1,\mathfrak{d}_2)=\xi+\left(l_{\mathfrak{d}_1}m_{\mathfrak{d}_1}+l_{\mathfrak{d}_2}m_{\mathfrak{d}_2}\right)\R_{\geq 0}=\xi+l_{\mathfrak{d}_{\mathrm{out}}}m_{\mathfrak{d}_{\mathrm{out}}}\R_{\geq 0}$ with function
\begin{equation}\label{eq:1.4}
\begin{split}
\overrightarrow{f}_{\mathrm{out}}&=\Big(1+([A_{\mathfrak{d}_1},A_{\mathfrak{d}_2}]+A_{\mathfrak{d}_2}a_{\mathfrak{d}_1}l_{\mathfrak{d}_2}|m_{\mathfrak{d}_1}\wedge m_{\mathfrak{d}_2}|+A_{\mathfrak{d}_1}a_{\mathfrak{d}_2}l_{\mathfrak{d}_1}|m_{\mathfrak{d}_1}\wedge m_{\mathfrak{d}_2}|)u_{I(\mathfrak{d}_1)\cup I(\mathfrak{d}_2)}z^{l_{\mathfrak{d}_1}m_{\mathfrak{d}_1}+l_{\mathfrak{d}_2}m_{\mathfrak{d}_2}},\\
& 1+a_{\mathfrak{d}_1}a_{\mathfrak{d}_2}u_{I(\mathfrak{d}_1)\cup I(\mathfrak{d}_2)}l_{\mathfrak{d}_{\mathrm{out}}}|m_{\mathfrak{d}_1}\wedge m_{\mathfrak{d}_2}|z^{l_{\mathfrak{d}_1}m_{\mathfrak{d}_1}+l_{\mathfrak{d}_2}m_{\mathfrak{d}_2}})\Big).
\end{split}
\end{equation}
Hence we take
\begin{equation}
\tilde{\mathfrak{D}}_N^{i+1}=\tilde{\mathfrak{D}}_N^{i}\cup\left\lbrace\mathsf{w}=\left(\mathfrak{d}(\mathfrak{d}_1,\mathfrak{d}_2),\overrightarrow{f}_{\mathfrak{d}}\right)\vert\mathfrak{d}_1,\mathfrak{d}_2 \text{satsify } \text{(i)-(iii)}\right\rbrace.
\end{equation}

Of course $\tilde{\mathfrak{D}}_N^{i+1}$ satisfies assumption $(1)$ with $I(\mathfrak{d})=\bigcup_{\mathfrak{d}'\in\text{Parents}(\mathfrak{d})}I(\mathfrak{d}')$ for any $\mathfrak{d}\in\tilde{\mathfrak{D}}_N^{i+1}\setminus\tilde{\mathfrak{D}}_N^i$. Then let $\xi\in \text{Sing}(\tilde{\mathfrak{D}}_N^{i+1})$ and assume that $(2)$ does not hold, i.e. there are at least $3$ rays from $\xi$. Since $I(\mathfrak{d}_1)\cap I(\mathfrak{d}_2)=\emptyset$ then $ \text{Leaves}(\mathfrak{d}_1)\cap \text{Leaves}(\mathfrak{d}_1)=\emptyset$ hence their lines can be slightly moved independently. In particular by the assumption $\tilde{\mathfrak{D}}_N$ is generic we can move the lines from the point $\xi\in\mathfrak{d}_1\cap\mathfrak{d}_2$ violating assumption $(2)$. 
This procedure stops since at any step $I(\mathfrak{d})$ is the union of the $I(\mathfrak{d}_i)$ with $\mathfrak{d}_i\in\text{Leaves}(\mathfrak{d})$ and the maximal cardinality of $I(\mathfrak{d})$ is $nN$.   

Let us explain how the previous construction works in practice in the following example.

\begin{example}\label{ex:diagram1}

Let us consider the initial scattering diagram
\[
\mathfrak{D}=\left\lbrace\left(\mathfrak{d}_1=(0,1)\R, \left(1+At_1y, 1+t_1y\right)\right), \left(\mathfrak{d}_2=(1,0)\R, \left(1, 1+t_2x\right)\right) \right\rbrace
\] 
 with $A\in\mathfrak{gl}_e(r,\C)$. By the deformation techniques we get $\tilde{\mathfrak{D}}_2=\tilde{\mathfrak{D}}_2^0$ as in Figure \ref{fig:D_deform}.
 
\begin{figure}
\center
\begin{tikzpicture}
\draw (2,0) -- (2,7);
\draw (3,0) -- (3,7);
\draw (5.5,0)-- (5.5,7);
\draw (0,1) -- (7,1);
\draw (0, 2.5) -- (7,2.5);
\draw (0, 4) -- (7,4);
\node [below left, font=\tiny] at (2,0) {$\left((1+Au_{11}y),(1+u_{11}y)\right) $};
\node [below , font=\tiny] at (3.5,0) {$\left((1+Au_{12}y),(1+u_{12}y)\right)$};
\node [below right, font=\tiny] at (5.5,0) {$\left(1,(1-u_{11}u_{12}y^2)\right)$};
\node [below, font=\tiny] at (0,1) {$\left(1,(1+u_{21}x)\right)$};
\node [below , font=\tiny] at (0,2.5) {$\left(1,(1+u_{22}x)\right)$};
\node [below,  font=\tiny] at (0,4) {$\left(1,(1-u_{21}u_{22}x^2)\right)$};
\node at (2,1){$\bullet$};
\node at (2,2.5){$\bullet$};
\node at (2,4){$\bullet$};
\node at (3,1){$\bullet$};
\node at (3,2.5){$\bullet$};
\node at (3,4){$\bullet$};
\node at (5.5,1){$\bullet$};
\node at (5.5,2.5){$\bullet$};
\node at (5.5,4){$\bullet$};
\draw [red](2,1) -- (6,5);
\draw [red](3,1) -- (6,4);
\draw [red](5.5,4)-- (7,5.5);
\draw [red](2,2.5) -- (5,5.5);
\draw [red](3, 2.5) -- (6,5.5);
\draw [cyan](5.5,1) -- (6.5,3);
\draw [cyan](5.5,2.5) -- (6.5,4.5);
\draw [blue](2,4) -- (6,6);
\draw [blue](3,4) -- (7,6);
\end{tikzpicture}
\caption{Construction of $\tilde{\mathfrak{D}}_2^1$. We color distinctly rays with different slope: slope $1$ are red, slope $2$ are cyan and slope $1/2$ are blue. }
\label{fig:D_deform1}
\end{figure}
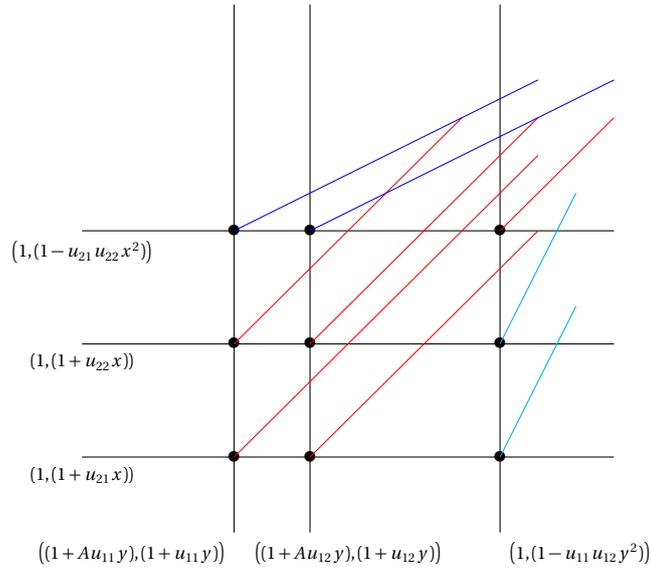

In order to construct $\tilde{\mathfrak{D}}_2^1$ we consider the marked points in Figure \ref{fig:D_deform1}, that are intersection of lines which satisfy conditions (i)-(iii). For each such point we draw a new ray as prescribed by Lemma \ref{lem:1.9}. We collect the functions $\overrightarrow{f}_{i}$ associated to the new rays below:
\begin{enumerate}
\item[\textbf{Slope 1}] $(1+Au_{21}u_{11}xy, 1+u_{21}u_{11}xy)$, $(1+Au_{21}u_{12}xy, 1+u_{21}u_{12}xy)$, $(1+Au_{22}u_{11}xy, 1+u_{22}u_{11}xy)$, $(1+Au_{22}u_{12}xy, 1+u_{22}u_{12}xy)$ and $(1, 1+2u_{21}u_{11}u_{22}u_{12}x^2y^2)$. 
\item[\textbf{Slope 1/2}] $(1-Au_{21}u_{22}u_{11}x^2y,1-u_{21}u_{22}u_{11}x^2y)$ and $(1-Au_{22}u_{21}u_{12}x^2y,1-u_{21}u_{22}u_{12}x^2y)$.
\item[\textbf{Slope 2}] $(1,1-u_{21}u_{11}u_{12}xy^2)$ and  $(1,1-u_{22}u_{12}u_{11}xy^2)$.
\end{enumerate}
 
\begin{figure}
\center
\begin{tikzpicture}
\draw (2,0) -- (2,7);
\draw (3,0) -- (3,7);
\draw (5.5,0)-- (5.5,7);
\draw (0,1) -- (7,1);
\draw (0, 2.5) -- (7,2.5);
\draw (0, 4) -- (7,4);
\node [below left, font=\tiny] at (2,0) {$\left((1+Au_{11}y),(1+u_{11}y)\right) $};
\node [below , font=\tiny] at (3.5,0) {$\left((1+Au_{12}y),(1+u_{12}y)\right)$};
\node [below right, font=\tiny] at (5.5,0) {$\left(1,(1-u_{11}u_{12}y^2)\right)$};
\node [below, font=\tiny] at (0,1) {$\left(1,(1+u_{21}x)\right)$};
\node [below , font=\tiny] at (0,2.5) {$\left(1,(1+u_{22}x)\right)$};
\node [below,  font=\tiny] at (0,4) {$\left(1,(1-u_{21}u_{22}x^2)\right)$};
\node at (3.5,2.5){$\bullet$};
\node at (4.5,2.5){$\bullet$};
\node at (6.25,2.5){$\bullet$};
\node at (3,3.5){$\bullet$};
\node at (3,4.5){$\bullet$};
\node at (3,2) {$\bullet$};
\draw [red](2,1) -- (6,5);
\draw [red](3,1) -- (6,4);
\draw [red](5.5,4)-- (7,5.5);
\draw [red](2,2.5) -- (5,5.5);
\draw [red](3, 2.5) -- (6,5.5);
\draw [cyan](5.5,1) -- (6.5,3);
\draw [cyan](5.5,2.5) -- (6.5,4.5);
\draw [blue](2,4) -- (6,6);
\draw [blue](3,4) -- (7,6);
\draw [thick, red] (3,4.5) -- (5,6.5);
\draw [thick, red] (6.25,2.5) -- (7,3.25);
\draw [thick, blue] (3.5,2.5) -- (6.5,4);
\draw [thick, blue] (4.5,2.5) -- (6.5,3.5);
\draw [thick, cyan] (3,3.5) -- (4.5,6.5);
\draw [thick, cyan] (3,2) -- (5,6);
\end{tikzpicture}
\caption{Construction of $\tilde{\mathfrak{D}}_2^2$. We color distinctly rays with different slope: slope $1$ are red, slope $2$ are cyan and slope $1/2$ are blue. Last rays added are thicker. }
\label{fig:D_deform2}
\end{figure} 
There is one step we can go further, since marked points in Figure \ref{fig:D_deform2} are the only intersection points of either lines or rays for which conditions (i)-(iii) hold true. The functions $\overrightarrow{f}_i$ associate to the new rays are:
\begin{enumerate}
\item[\textbf{Slope 1}] $(1,1-4u_{21}u_{22}u_{11}u_{12}x^2y^2)$ and $(1-4Au_{21}u_{22}u_{11}u_{12}x^2y^2,1-4u_{21}u_{22}u_{11}u_{12}x^2y^2)$
\item[\textbf{Slope 1/2}] $(1+Au_{21}u_{22}u_{11}x^2y, 1+u_{21}u_{11}u_{22}x^2y)$ and $(1+Au_{21}u_{12}u_{22}x^2y, 1+u_{21}u_{12}u_{22}x^2y) $
\item[\textbf{Slope 2}] $(1+2Au_{21}u_{11}u_{12}xy^2, 1+u_{21}u_{11}u_{12}xy^2)$ and  $(1+2Au_{12}u_{11}u_{22}xy^2, 1+u_{12}u_{11}u_{22}xy^2)$.
\end{enumerate}
 
The last step is represented as in Figure \ref{fig:D_deform3} and it contains one more ray of slope one and with function
\begin{equation}
\left(1+4Au_{21}u_{22}u_{11}u_{12}x^2y^2,1+4u_{11}u_{12}u_{21}u_{22}x^2y^2\right).
\end{equation}

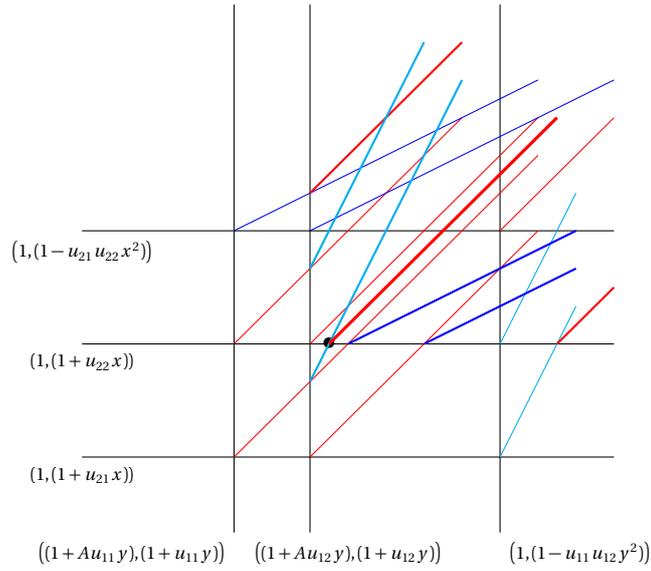
\begin{figure}
\center
\begin{tikzpicture}
\draw (2,0) -- (2,7);
\draw (3,0) -- (3,7);
\draw (5.5,0)-- (5.5,7);
\draw (0,1) -- (7,1);
\draw (0, 2.5) -- (7,2.5);
\draw (0, 4) -- (7,4);
\node [below left, font=\tiny] at (2,0) {$\left((1+Au_{11}y),(1+u_{11}y)\right) $};
\node [below , font=\tiny] at (3.5,0) {$\left((1+Au_{12}y),(1+u_{12}y)\right)$};
\node [below right, font=\tiny] at (5.5,0) {$\left(1,(1-u_{11}u_{12}y^2)\right)$};
\node [below, font=\tiny] at (0,1) {$\left(1,(1+u_{21}x)\right)$};
\node [below , font=\tiny] at (0,2.5) {$\left(1,(1+u_{22}x)\right)$};
\node [below,  font=\tiny] at (0,4) {$\left(1,(1-u_{21}u_{22}x^2)\right)$};
\node at (3.25,2.5){$\bullet$};
\draw [red](2,1) -- (6,5);
\draw [red](3,1) -- (6,4);
\draw [red](5.5,4)-- (7,5.5);
\draw [red](2,2.5) -- (5,5.5);
\draw [red](3, 2.5) -- (6,5.5);
\draw [cyan](5.5,1) -- (6.5,3);
\draw [cyan](5.5,2.5) -- (6.5,4.5);
\draw [blue](2,4) -- (6,6);
\draw [blue](3,4) -- (7,6);
\draw [thick, red] (3,4.5) -- (5,6.5);
\draw [thick, red] (6.25,2.5) -- (7,3.25);
\draw [thick, blue] (3.5,2.5) -- (6.5,4);
\draw [thick, blue] (4.5,2.5) -- (6.5,3.5);
\draw [thick, cyan] (3,3.5) -- (4.5,6.5);
\draw [thick, cyan] (3,2) -- (5,6);
\draw [very thick, red] (3.25,2.5)-- (6.25,5.5);
\end{tikzpicture}
\caption{Construction of $\tilde{\mathfrak{D}}_2^3$. We color distinctly rays with different slope: slope $1$ are red, slope $2$ are cyan and slope $1/2$ are blue. Last ray added is thicker. }
\label{fig:D_deform3}
\end{figure}   

Now if we take the product of all functions many factors cancel and we finally get the scattering diagram $\mathfrak{D}_2^3$ with two more rays:
\begin{equation}\label{eq:D_2^3}
\mathfrak{D}_2^3\setminus\mathfrak{D}=\left\lbrace\left((1,1)\R_{\geq 0},(1+At_1t_2xy, 1+t_1t_2xy)\right), \left((1,2)\R_{\geq 0}, (1+Axy^2t_2t_1^2,1)\right)\right\rbrace
\end{equation}
Notice that the ray of slope $2$ has only the matrix contribution, which was not there in the tropical vertex group $\mathbb{V}$ (see the analogous Example 1.11 of \cite{GPS}).
\begin{figure}[h]
\center
\begin{tikzpicture}
\draw (2,0) -- (2,4);
\draw (0,2) -- (4,2);
\draw (2,2) -- (4,4);
\draw (2,2) -- (3, 4);
\node [below right, font=\tiny] at (2,2) {0};
\end{tikzpicture}
\caption{The asymptotic scattering diagram $\mathfrak{D}_2^3$.}
\label{fig:D_2^3}
\end{figure}
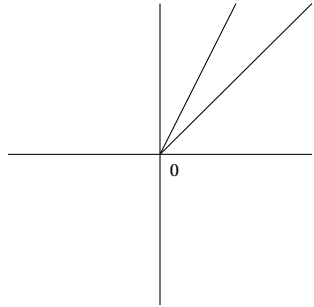
\end{example}

\section{Tropical curve count}\label{sec:tropical curve count}

\begin{notation}
Let $\overline{\Gamma}$ be a weighted, connected, finite graph without divalent vertices and denote the set of vertices by $\overline{\Gamma}^{[0]}$ and the set of edges by $\overline{\Gamma}^{[1]}$. The weight function $w_{\overline{\Gamma}}\colon\overline{\Gamma}^{[1]}\to \Z_{> 0}$ assigns a weight to each edge. We denote the set of univalent vertices by $\overline{\Gamma}^{[0]}_{\infty}$ and we define $\Gamma=\overline{\Gamma}\setminus\overline{\Gamma}^{[0]}_\infty$. The set of edges and vertices of $\Gamma$ is denoted by $\Gamma^{[0]}$ and $\Gamma^{[1]}$ respectively. The edges of $\Gamma$ which are not compact are called \textit{unbounded} and denoted by $\Gamma^{[1]}_\infty\subset\Gamma^{[1]}$.   
\end{notation}

\begin{definition}
A proper map $h\colon\Gamma\to\Lambda_\R$ is called a \emph{parametrized tropical curve} if it satisfies the following conditions:
\begin{enumerate}
\item[(1)] for every edge $E\in\Gamma^{[1]}$, $h\vert_E$ in an embedding with image $h(E)$ contained in an affine line with rational slope;
\item[(2)] for every vertex $V\in\Gamma^{[0]}$, the following \emph{balancing condition} holds true: let $E_1,... ,E_s\in\Gamma^{[1]}$ be the adjacent edges of $V$, let $w_1,...  w_s$ be the weights of $E_1,... , E_s$ and let $m_1,...  ,m_s$ be the primitive integer vectors at the point $h(V)$ in the direction of $h(E_1),...  ,h(E_s)$, then 
\begin{equation}
\sum_{j=1}^nw_{\Gamma}(E_j)m_j=0
\end{equation} 
is the balancing condition.
\end{enumerate}
\end{definition}
Two parametrized tropical curves $h\colon\Gamma\to\Lambda_\R$ and $h'\colon\Gamma'\to\Lambda_\R$ are isomorphic if there exists a homeomorphism $\Phi\colon\Gamma\to\Gamma'$ respecting the weights such that $h'=\Phi\circ h$. 
\begin{definition}
A tropical curve is an isomorphism class of parametrized tropical curves.
\end{definition}
\begin{definition}
The \emph{genus} of a tropical curve $h\colon\Gamma\to\Lambda_\R$ is the first Betti number of the underling graph $\overline{\Gamma}$. Genus zero tropical curves are called rational.
\end{definition}  
\begin{definition}
Let $h\colon\Gamma\to\Lambda_{\R}$ be a tropical curve such that $\overline{\Gamma}$ has only vertices of valency one and three. The \emph{multiplicity of a vertex} $V\in\Gamma^{[0]}$ in $h$ is 
\begin{equation}
\mathrm{Mult}_V(h)=w_1w_2|m_1\wedge m_2|=w_1w_3|m_1\wedge m_3|=w_2w_3|m_2\wedge m_3|
\end{equation}
where $E_1,E_2,E_3\in\Gamma^{[1]}$ are the edges containing $V$ with $w_i=w_{\Gamma}(E_i)$ and $m_i\in\Lambda$ is a primitive vector in the direction $h(E_i)$ emanating from $h(V)$. The equality of three expressions follows from the balancing condition. 
\end{definition}
\begin{definition}
The multiplicity of a tropical curve $h$ is 
\begin{equation*}
\mathrm{Mult}(h)=\prod_{V\in\Gamma^{[0]}}\mathrm{Mult}_V(h)
\end{equation*}
\end{definition}

\begin{theorem}\label{thm:2.4}
Let $\tilde{\mathfrak{D}}_N=\left\lbrace (\mathfrak{d}_{iJl},\overrightarrow{f}_{iJl})\vert 1\leq i\leq n, \, l\geq 1,\, J\in\lbrace 1,... ,N\rbrace\, \#J\geq 1 \right\rbrace$ be a scattering diagram, such that  
\begin{equation}
\overrightarrow{f}_{iJl}=\left(1+(\#J)!A_{i(\#J)l}z^{lm_i}\prod_{j\in J}u_{ij},1+(\#J)!a_{i(\#J)l}l\prod_{j\in J}u_{ij}z^{lm_i}\right)
\end{equation}
and assume $[A_{i(\#J)l}, A_{i'(\#J')l'}]=0$ for all $i,i'\in\lbrace1,...,n\rbrace$.

Then there is a bijective correspondence between elements in the complete scattering diagram $\tilde{\mathfrak{D}}_N^{\infty}$ and rational tropical curve $h\colon\Gamma\to \Lambda_{\R}$ such that:
\begin{enumerate}
\item there is an edge $E_{\mathrm{out}}\in\Gamma_{\infty}^{[1]}$ with $h(E_{\mathrm{out}})=\mathfrak{d}$;
\item if $E\in\Gamma_{\infty}^{[1]}\setminus\lbrace E_{\mathrm{out}}\rbrace$ or if $E_{\mathrm{out}}$ is the only edge of $\Gamma$ (in which case $E=E_{\mathrm{out}}$), then $h(E)$ is contained in some $\mathfrak{d}_{iJl}$, where $1\leq i\leq n$, $J\subset\lbrace 1,... , N\rbrace$ and $l\geq 1$. Moreover if $E\neq E_{\mathrm{out}}$, the unbounded direction of $h(E)$ is given by $-m_i$ and its weight is $w_{\Gamma}(E)=l$;
\item if $E,E'\in \Gamma_{\infty}^{[1]}\setminus\lbrace E_{\mathrm{out}}\rbrace$ and $h(E)\subset\mathfrak{d}_{iJl}$ and $h(E')\subset\mathfrak{d}_{iJ'l'}$, then $J\cap J'=\emptyset$.
\end{enumerate}

If $\mathfrak{d}$ is a ray, the corresponding curve $h$ is trivalent and
\begin{equation}\label{eq:2.1}
\overrightarrow{f}_{\mathfrak{d}}=\left(1+\mathrm{Mult}(h)\left(\sum_{i_pJ_pl_p\in\text{Leaves}(\mathfrak{d})}(\#J_p)!A_{i\#J_pl_p}\prod_{\substack{i_qJ_ql_q\in\text{Leaves}(\mathfrak{d})\\ q\neq p}}(\#J_q)! a_{i_q(\#J_q)l_q}\prod_{j\in J_{q}}u_{i_qj}\prod_{j\in J_p}u_{i_pj}\right)z^{l_{\mathfrak{d}}m_{\mathfrak{d}}}, f_\mathfrak{d}\right)
\end{equation}
where $f_\mathfrak{d}$ is given by
\begin{equation}
f_{\mathfrak{d}}=1+l_{\mathfrak{d}}\mathrm{Mult}(h)\left(\prod_{\mathfrak{d}_{iJl}\in\text{Leaves}(\mathfrak{d})}(\#J)!a_{iJl}\prod_{j'\in J} u_{ij'}\right)z^{l_{\mathfrak{d}}m_{\mathfrak{d}}}
\end{equation}
and $l_{\mathfrak{d}}m_{\mathfrak{d}}=\sum_{i=1}^n\sum_{l\geq 1}lm_i$, with $m_{\mathfrak{d}}$ primitive vector in $\Lambda$ and $l_{\mathfrak{d}}=w_{\Gamma}(E_{\mathrm{out}})$. 

\end{theorem}

This theorem is a generalization of Theorem $2.4$ in \cite{GPS} to scattering diagram in the extended tropical vertex group $\tilde{\mathbb{V}}$, thus in the proof we follow the same path of \cite{GPS}. 

\begin{proof}
Let $\mathfrak{d}=\mathfrak{d}_{iJl}$ be a line in $\tilde{\mathfrak{D}}_N$, then $\Gamma_{\infty}^{[1]}$ is defined by a single edge $E_{\mathfrak{d}}$ such that $h(E_{\mathfrak{d}})$ has slope $-m_i$ and weight $w_{\Gamma}(E_{\mathfrak{d}})=l$. Conversely, if $h\colon\Gamma\to\Lambda_{\R}$ is a tropical curve satisfying $(1)-(3)$, which consists of a single edge $E_{\mathrm{out}}=\Gamma_{\infty}^{[1]}$, then from $(2)$ $h(E_{\mathrm{out}}\subset\mathfrak{d}_{iJl}$, and by assumption $(\mathfrak{d}_{iJl},\overrightarrow{f}_{iJl})\in\tilde{\mathfrak{D}}_N^{\infty}$. 
Now let $\mathfrak{d}\in\tilde{\mathfrak{D}}_N^{\infty}\setminus\tilde{\mathfrak{D}}_N$ be a ray, then we define the graph $\Gamma$ by
\begin{align*}
\Gamma^{[0]}&=\left\lbrace V_{\mathfrak{d}'}\vert\mathfrak{d}'\in\text{Ancestors}(\mathfrak{d})\,\, \text{and  } \mathfrak{d}'\, \text{ ray}\right\rbrace\\
\Gamma^{[1]}&=\left\lbrace E_{\mathfrak{d}'}\vert \mathfrak{d}'\in\text{Ancestors}(\mathfrak{d})\right\rbrace.
\end{align*}
If $\mathfrak{d}'\in\text{Ancestors}(\mathfrak{d})$ then one of the following options can occur:
\begin{enumerate}
\item[(a)] $\mathfrak{d}'\neq\mathfrak{d}$ and $\mathfrak{d}'$ is a ray, then $E_{\mathfrak{d}'}$ have vertices $V_{\mathfrak{d}'}$ and $V_{\text{Child}(\mathfrak{d}')}$;
\item[(b)] $\mathfrak{d}=\mathfrak{d}'$, then $E_{\mathfrak{d}'}$ is an unbounded edge with vertex $V_{\mathfrak{d}}$;
\item[(c)] $\mathfrak{d}'$ is a line, then $E_{\mathfrak{d}'}$ is an unbounded edge with vertex $V_{\text{Child}(\mathfrak{d}')}$. 
\end{enumerate}
Thus, for any $\mathfrak{d}'\in\text{Ancestors}(\mathfrak{d})$, $f_{\mathfrak{d}'}=1+c_{\mathfrak{d}'}z^{l_{\mathfrak{d}'}m_{\mathfrak{d}'}}$ with $m_{\mathfrak{d}'}\in\Lambda$ primitive, and we define the weight $w_{\Gamma_{\mathfrak{d}}}(E_{\mathfrak{d}'})=l_{\mathfrak{d}'}$. The tropical curve $h$ is defined by mapping $E_{\mathfrak{d}'}$ in 
\begin{enumerate}
\item[-] a line of slope $m_{\mathfrak{d}'}$ joining $\mathrm{Init}(\mathfrak{d}')$ and $\mathrm{Init}(\text{Child}(\mathfrak{d}))$, if $\mathfrak{d}'$ in case (a);
\item[-] the line $\mathfrak{d}$, if $\mathfrak{d}'$ in case (b);
\item[-] the ray $\mathrm{Init}(\text{Child})(\mathfrak{d}')+\R_{\geq 0}m_{\mathfrak{d}'}$, if $\mathfrak{d}'$ in case (c). 
\end{enumerate}
Since $\Gamma_{\mathfrak{d}}$ is trivalent, the genus zero condition is satisfied. The balancing condition follows from equation \eqref{eq:1.4}: indeed let $\text{Parents}({\mathfrak{d}})=\lbrace\mathfrak{d}_1,\mathfrak{d}_2\rbrace$, then $l_{\mathfrak{d}}m_{\mathfrak{d}}=l_{\mathfrak{d}_1}m_{\mathfrak{d}_1}+l_{\mathfrak{d}_2}m_{\mathfrak{d}_2}$ and since $h(E_{\mathfrak{d}_1}), h(E_{\mathfrak{d}_2})$ are incoming edges while $h(E_{\mathfrak{d}})$ is outgoing, we get 
\begin{equation}\label{eq:balcond}
w_{\Gamma_{\mathfrak{d}}}(E_{\mathfrak{d}})m_{\mathfrak{d}}=-w_{\Gamma_{\mathfrak{d}}}(E_{\mathfrak{d}_1})m_{\mathfrak{d}_1}-w_{\Gamma_{\mathfrak{d}}}(E_{\mathfrak{d}_2})m_{\mathfrak{d}_2}.
\end{equation}

We can prove the expression \eqref{eq:2.1} by induction: indeed if $\mathfrak{d}=\mathfrak{d}_{iJl}$ is a line in $\tilde{\mathfrak{D}}_N$, then $\Gamma_{\mathfrak{d}}$ has only an edge $E_{\mathrm{out}}=\Gamma^{[1]}_{\infty}$, hence $\mathrm{Mult}(h)=1$ because no trivalent vertexes occur. 
The inductive step is the following: let $\mathfrak{d}$ be a ray and assume \eqref{eq:2.1} holds for $\mathfrak{d}_1,\mathfrak{d}_2\in\text{Parents}(\mathfrak{d})$, and let $h_1,h_2$ be the tropical curves associated respectively to $\mathfrak{d}_1,\mathfrak{d}_2$. Then, by equation \eqref{eq:1.4}   
\begin{align*}
\overrightarrow{f}_{\mathfrak{d}}&=\Big(1+\mathrm{Mult}(h_1)\mathrm{Mult}(h_2)\sum_{\substack{\mathfrak{d}_{i_rJ_rl_r}\in\\\text{Leaves}({\mathfrak{d}_1})}}\sum_{\substack{\mathfrak{d}_{i_r'J_r'l_r'}\in\\\text{Leaves}({\mathfrak{d}_2})}}(\#J_r)!(\#J_r')!\left[A_{i_r(\#J_r)l_r}, A_{i_r'(\#J_r')l_r'}\right]\prod_{\substack{i_qJ_ql_q\\ q\neq r}}(\#J_q)!a_{i_q(\#J_q)l_q}\cdot\\
&\qquad\cdot\prod_{\substack{i_q'(\#J_q')l_q'\\ q'\neq r'}}(\#J_q')!a_{i_q'(\#J_q')l_q'}\prod_{j\in J_q}u_{i_qj}\prod_{j\in J_r}u_{i_rj}\prod_{j'\in J_q'}u_{i_q'j'}\prod_{j'\in J_r'}u_{i_r'j'}+\\
&\qquad+\mathrm{Mult}(h_1)\mathrm{Mult}(h_2)\left(\sum_{\mathfrak{d}_{{i_rJ_rl_r}}\in\text{Leaves}({\mathfrak{d}_2})}A_{i_r(\#J_r)l_r}(\#J_r)!\prod_{\substack{i_qJ_ql_q\\ q\neq r}}(\#J_q)!a_{i_q(\#J_q)l_q}\prod_{j\in J_q}u_{i_qj}\prod_{j\in J_r}u_{i_rj}\right)\cdot\\
&\qquad\cdot\left(l_{\mathfrak{d}_1}\prod_{\mathfrak{d}_{iJl}\in\text{Leaves}(\mathfrak{d}_1)}(\#J)!a_{iJl}\prod_{j\in J}u_{ij}\right)l_{\mathfrak{d}_2}|m_{\mathfrak{d}_1}\wedge m_{\mathfrak{d}_2}|z^{l_{\mathfrak{d}_1}m_{\mathfrak{d}_1}+l_{\mathfrak{d}_2}m_{\mathfrak{d}_2}}+\\
&\qquad+\mathrm{Mult}(h_1)\mathrm{Mult}(h_2)\left(\sum_{\mathfrak{d}_{{i_rJ_rl_r}}\in\text{Leaves}({\mathfrak{d}_1})}A_{i_r(\#J_r) l_r}(\#J_r)!\prod_{\substack{i_qJ_ql_q\\ q\neq r}}(\#J_q)!a_{i_q(\#J_q)l_q}\prod_{j\in J_q}u_{i_qj}\prod_{j\in J_r}u_{i_rj}\right)\cdot\\
&\qquad\cdot\left(l_{\mathfrak{d}_2}\prod_{\mathfrak{d}_{iJl}\in\text{Leaves}(\mathfrak{d}_2)}(\#J)!a_{i(\#J)l}\prod_{j\in J}u_{ij}\right)l_{\mathfrak{d}_1}|m_{\mathfrak{d}_1}\wedge m_{\mathfrak{d}_2}|z^{l_{\mathfrak{d}_1}m_{\mathfrak{d}_1}+l_{\mathfrak{d}_2}m_{\mathfrak{d}_2}},
\end{align*}

\begin{align*}
&\quad 1+l_{\mathfrak{d}_1}\mathrm{Mult}(h_1)\left(\prod_{\mathfrak{d}_{iJl}\in\text{Leaves}(\mathfrak{d}_1)}(\#J)!a_{i(\#J)l}\prod_{j\in J}u_{ij}\right)l_{\mathfrak{d}_2}\mathrm{Mult}(h_2)\left(\prod_{\mathfrak{d}_{iJl}\in\text{Leaves}(\mathfrak{d}_2)}(\#J)!a_{i(\#J)l}\prod_{j\in J}u_{ij}\right)\cdot\\
&\qquad\cdot l_{\mathfrak{d}}|m_{\mathfrak{d}_1}\wedge m_{\mathfrak{d}_2}|z^{l_{\mathfrak{d}_1}m_{\mathfrak{d}_1}+l_{\mathfrak{d}_2}m_{\mathfrak{d}_2}}\Big)\\
&=\Bigg(1+\mathrm{Mult}(h_1)\mathrm{Mult}(h_2)\mathrm{Mult}_{V_{\mathfrak{d}}}(h)\Bigg(\left(\sum_{\substack{\mathfrak{d}_{{i_rJ_rl_r}}\in\\\text{Leaves}({\mathfrak{d}_2})}}(\#J_r)!A_{i_r(\#J_r)l_r}\prod_{\substack{{i_qJ_ql_q}\\
\in\text{Leaves}(\mathfrak{d})\\ q\neq r}}(\#J_q)!a_{i_q(\#J_q)l_q}\prod_{j\in J_q}u_{i_qj}\prod_{j\in J_r}u_{i_rj}\right)+\\
&\qquad+\left(\sum_{\mathfrak{d}_{{i_rJ_rl_r}}\in\text{Leaves}({\mathfrak{d}_1})}(\#J_r)!A_{i_r(\#J_r)l_r}\prod_{\substack{\mathfrak{d}_{{i_qJ_ql_q}}\in\text{Leaves}(\mathfrak{d})\\ q\neq r}}(\#J_q)!a_{i_q(\#J_q)l_q}\prod_{j\in J_q}u_{i_qj}\prod_{j\in J_r}u_{i_rj}\right)\Bigg)z^{l_{\mathfrak{d}_1}m_{\mathfrak{d}_1}+l_{\mathfrak{d}_2}m_{\mathfrak{d}_2}},\\
&\quad 1+\mathrm{Mult}(h_1)\mathrm{Mult}(h_2)\mathrm{Mult}_{V_{\mathfrak{d}}}(h)l_{\mathfrak{d}}\left(\prod_{\mathfrak{d}_{iJl}\in\text{Leaves}(\mathfrak{d})}(\#J)!a_{i(\#J)l}\prod_{j\in J}u_{ij}\right) z^{l_{\mathfrak{d}_1}m_{\mathfrak{d}_1}+l_{\mathfrak{d}_2}m_{\mathfrak{d}_2}}\Bigg)\\
&=\Bigg(1+\mathrm{Mult}(h)\left(\sum_{\mathfrak{d}_{{i_rJ_rl_r}}\in\text{Leaves}({\mathfrak{d}})}(\#J_r)!A_{i_r(\#J_r)l_r}\prod_{\substack{\mathfrak{d}_{{i_qJ_ql_q}}\\ q\neq r}}(\#J_q)!a_{i_q(\#J_q)l_q}\prod_{j\in J_q}u_{i_qj}\prod_{j\in J_r}u_{i_rj}\right)z^{l_{\mathfrak{d}_1}m_{\mathfrak{d}_1}+l_{\mathfrak{d}_2}m_{\mathfrak{d}_2}},\\
&\quad 1+\mathrm{Mult}(h)l_{\mathrm{out}}(\mathfrak{d})\left(\prod_{\mathfrak{d}_{iJl}\in\text{Leaves}(\mathfrak{d})}(\#J)!a_{i(\#J)l}\prod_{j\in J}u_{ij}\right) z^{l_{\mathfrak{d}_1}m_{\mathfrak{d}_1}+l_{\mathfrak{d}_2}m_{\mathfrak{d}_2}}\Bigg)
\end{align*} 

where in the second step we use that $[A_{i_r(\#J_r)l_r}, A_{i_r'(\#J_r')l_r'}]$ vanishes by assumption.

Conversely, if $h\colon\Gamma\to \Lambda_{\R}$ is a tropical curve satisfying $(1)-(3)$ and $\mathfrak{d}$ is a ray, then by the previous computations we have $(\mathfrak{d},\overrightarrow{f}_{\mathfrak{d}})\in\tilde{\mathfrak{D}}_N^{\infty}$.  
\end{proof}

We are now going to introduce invariants to count tropical curves: let $\mathbf{w}=(w_{1},... , w_s)$ be a s-tuple of non-zero vectors $w_i\in\Lambda$ and fix a set of points $\xi=(\xi_{1},...  ,\xi_s)$. Then a parametrized tropical curve $h\colon \Gamma\to\Lambda_{\R}$ \textit{of type} $(\mathbf{w},\xi)$ is the datum of
\begin{itemize}
\item $\Gamma_{\infty}^{[1]}=\lbrace E_{r}\vert 1\leq r\leq s \rbrace\cup\lbrace E_{\mathrm{out}}\rbrace$;
\item $h(E_{r})$ asymptotically coincide with the ray $\mathfrak{d}_{r}=\xi_{r}-\R_{\geq 0}w_r$ and $w_{\Gamma}(E_{r})=|w_{r}|$;
\item $h(E_{\mathrm{out}})$ pointing in the direction of $w_{\mathrm{out}}\defeq\sum_{r=1}^sw_r$ and $w_{\Gamma}(E_{\mathrm{out}})=|w_{\mathrm{out}}|$. 
\end{itemize}
We denote by $T_{\mathbf{w},\xi}$ the set of tropical curves $h\colon\Gamma\to\Lambda_{\R}$ as above. Then we define 
\begin{equation}\label{def:N^(trop)}
N_{\mathbf{w}}^{\mathrm{trop}}\defeq \sum_{h\in T_{\mathbf{w},\xi}}\text{Mult}(h)
\end{equation}
as the number of tropical curve in $T_{\mathbf{w},\xi}$ counted with multiplicity. 

The definition is well-posed since according to to Proposition 4.13 of \cite{Mik05} the set $T_{\mathbf{w},\xi}$ is finite. In addition the number $N_{\mathbf{w}}^{\mathrm{trop}}$ does not depend on the generic choice of the vectors $\xi_j$. 

\begin{notation}\label{not:k partition}
Let $P=(P_1,...  ,P_n)$ be a n-tuple of integer, and let $\mathbf{k}=(k_1,...  ,k_n)$ be a n-tuple of vectors such that $k_j=(k_{jl})_{l\geq 1}$ is a partition of $P_j$ with $\sum_{l\geq 1}lk_{jl}=P_j$. We denote the partition $\mathbf{k}$ by $\mathbf{k}\vdash \mathbf{P}$.

Given a partition $\mathbf{k}\vdash\mathbf{P}$, we define 
\begin{equation}
s(\mathbf{k})\defeq\sum_{j=1}^n\sum_{l\geq 1}k_{jl}
\end{equation}
and the $s(\mathbf{k})$-tuple $\mathbf{w}(\mathbf{k})=\left(w_1(\mathbf{k}),...  ,w_{s(\mathbf{k})}(\mathbf{k})\right)$
of non zero vectors in $\Lambda$, such that 

\[w_r(\mathbf{k})\defeq lm_j,\] 

for every $1+\sum_{j'=1}^j\sum_{l'=1}^{l-1}k_{l'j'}\leq r\leq k_{jl}+\sum_{j'=1}^j\sum_{l'=1}^{l-1}k_{l'j'} $. Notice that $\sum_{r=1}^{s(\mathbf{k})} w_r(\mathbf{k})=l_{\mathfrak{d}}m_{\mathfrak{d}}$.
\end{notation}

We can now state the first result which provide a link between consistent scattering diagrams in the extended tropical vertex group $\tilde{\mathbb{V}}$ and tropical curves count.  
 
\begin{theorem}\label{thm:2.8}
Let $\mathfrak{D}=\left\lbrace\mathsf{w}_i=(\mathfrak{d}_i=m_i\R,\overrightarrow{f}_i)\vert 1\leq i\leq n\right\rbrace$ such that for every $i=1,...  ,n$ 

\[\overrightarrow{f}_i=\left(1+A_it_iz^{m_i} , 1+t_iz^{m_i}\right)\] 

on $\C[\![t_1,...  ,t_n]\!]$ and assume $[A_i,A_{i'}]=0$ for all index $i, i'\in\lbrace 1,...,n\rbrace$.
 
Then for every wall $\left(\mathfrak{d}=m_{\mathfrak{d}}\R_{\geq 0}, \overrightarrow{f}_{\mathfrak{d}}\right)\in\mathfrak{D}^{\infty}\setminus\mathfrak{D}$ where $m_{\mathfrak{d}}\in\Lambda$ is a primitive non-zero vector, the function $\overrightarrow{f}_{\mathfrak{d}}$ is explicitly given by the following expression:

\begin{multline}\label{eq:f_trop}
\log\overrightarrow{f}_{\mathfrak{d}}=\sum_{l_{\mathfrak{d}}\geq 1}\sum_{\mathbf{P}}\sum_{\mathbf{k}\vdash \mathbf{P}}N_{\mathbf{w}(\mathbf{k})}^{\mathrm{trop}}\Bigg(\sum_{i=1}^n\sum_{l\geq 1}lA_{i}^lk_{il}\left(\prod_{1\leq i\leq n}\prod_{l'\geq 1}\left(\frac{(-1)^{l'-1}}{(l')^2}\right)^{k_{ij'}}\frac{1}{k_{il'}!}t_i^{P_i}\right)z^{\sum_iP_im_i}, \\
 l_{\mathfrak{d}}\left(\prod_{1\leq i\leq n}\prod_{l\geq 1}\left(\frac{(-1)^{l-1}}{l^2}\right)^{k_{il}}\frac{1}{k_{il}!}t_i^{P_i}\right)z^{\sum_iP_im_i}
\Bigg)
\end{multline}  

where the second sum is over all n-tuple $\mathbf{P}=(P_1,...  ,P_n)\in\N^n$ such that $\sum_{i=1}^nP_im_i=l_{\mathfrak{d}}m_{\mathfrak{d}}$. 
\end{theorem}
\begin{proof}
Recall that since $\mathfrak{D}$ is a standard scattering diagram we can consider the associate deformed diagram $\tilde{\mathfrak{D}}$, and on $R_N$ the complete scattering diagrams $\mathfrak{D}^{\infty}_N$  and $(\tilde{\mathfrak{D}}_N^{\infty})_{as}$ are equivalent. In particular $\tilde{\mathfrak{D}}_N$ is defined by taking the logarithmic expansion of $\overrightarrow{f}_i\in\mathfrak{D}_N$ and substituting $t_i^l=\sum_{\substack{J_{il}\in\lbrace1,... ,N\rbrace\\ \#J_{il}=l}}l!\prod_{j\in J_{il}}u_{ij}$:
\begin{align*}
\overrightarrow{f}_{iJ_{il}l}=\left(1+\frac{(-1)^{l-1}}{l}A_i^l(l)!\prod_{j'\in J_{il}}u_{ij'}z^{lm_i}, 1+l \frac{(-1)^{l-1}}{l^2}(l)!\prod_{j'\in J_{il}}u_{ij'} z^{lm_i}\right)
\end{align*}
and the diagram $\tilde{\mathfrak{D}}_N$ is 
\begin{align*}
\tilde{\mathfrak{D}}_N=\left\lbrace (\mathfrak{d}_{iJ_{il}l}, \overrightarrow{f}_{iJ_{il}l})\vert i=1,...  ,n,\, 1\leq l\leq N,\, J_{il}\subset\lbrace 1,... , N\rbrace,\, \#J_{il}=l \right\rbrace
\end{align*}
where $\mathfrak{d}_{iJ_{il}l}\defeq\xi_{iJ_{il}l}-\R_{\geq}m_i$ for a generic choice $\xi_{iJ_{il}l}$. Now, for every $\mathfrak{d}'\in\tilde{\mathfrak{D}}_N^{\infty}\setminus \tilde{\mathfrak{D}}_N$ there is a unique tropical curve $h\colon\Gamma\to\Lambda_\R$ defined as in Theorem \ref{thm:2.4}. Assuming $\mathfrak{d}'=\xi'+m_{\mathfrak{d}'}\R_{\geq 0}$ for some generic $\xi'\in\Lambda_\R$ and some primitive vector $m_{\mathfrak{d}'}\in\Lambda$, and $w_{\Gamma}(E_{\mathrm{out}})=l_{\mathfrak{d}'}\geq 1$, let us define $\mathbf{P}=(P_1,... , P_n)\in\N^n$ such that $\sum_{j=1}^nP_jm_j=l_{\mathfrak{d}'}m_{\mathfrak{d}'}$ and $\mathbf{k}\vdash \mathbf{P}$, as in Notation \ref{not:k partition}. Then the function $\overrightarrow{f}_{\mathfrak{d}'}$ is
\begin{multline}
\log\overrightarrow{f}_{\mathfrak{d}'}=\Bigg(\mathrm{Mult}(h)\sum_{i=1}^n\sum_{l\geq 1}\sum_{J_{il}}(l)!A_{i}^l\frac{(-1)^{l-1}}{l}\left(\prod_{1\leq i'\leq n}\prod_{l'\geq 1}\prod_{\substack{J_{i'l'}\\J_{i'l'}\cap J_{il}=\emptyset}}(l')! \frac{(-1)^{l'-1}}{(l')^2}\right)\cdot \\
\cdot\left(\prod_{j'\in J_{i'l'}}u_{i'j'}\prod_{j\in J_{1l}}u_{ij}\right)z^{\sum_iP_im_i}, l_{\mathfrak{d}}\mathrm{Mult}(h)\left(\prod_{1\leq i\leq n}\prod_{l\geq 1}\prod_{J_{il}}(l)!\frac{(-1)^{l-1}}{l^2}\prod_{j'\in J_{ij}} u_{ij'}\right)z^{\sum_iP_im_i}
\Bigg)
\end{multline}
where both the products and the sums over $J_{il}$ are on the subset $J_{il}\subset\lbrace 1,...  ,N\rbrace$ of size $l$ such that $\mathfrak{d}_{iJ_{il}l}\in\mathrm{Leaves}(\mathfrak{d}')$.

For every $1\leq i\leq n$ and $l\geq 1$, $k_{il}$ counts how many subset $J_{il}\subset\lbrace 1,...  ,N\rbrace$ of size $l$ are in $\mathrm{Leaves}(\mathfrak{d}')$, hence we can write $\log\overrightarrow{f}_{\mathfrak{d}'}$

\begin{align*}
\log\overrightarrow{f}_{\mathfrak{d}'}&=\Bigg(\mathrm{Mult}(h)\sum_{i=1}^n\sum_{l\geq 1}lA_{i}^l\left((l)!\frac{(-1)^{l-1}}{l^2}\right)\sum_{J_{il}}\left(\prod_{l'\geq 1}\prod_{\substack{J_{il'}\\ J_{il'}\cap J_{il}=\emptyset}}\left((l')! \frac{(-1)^{l'-1}}{(l')^2}\right)\prod_{j'\in J_{il'}}u_{ij'}\right)\prod_{j\in J_{il}}u_{ij}\cdot \\
&\qquad\cdot \left(\prod_{i'\neq i}\prod_{l'\geq 1}\left((l')! \frac{(-1)^{l'-1}}{(l')^2}\right)^{k_{i'j'}}\prod_{J_{i'j'}}\prod_{j'\in J_{i'l'}}u_{i'j'}\right)z^{\sum_iP_im_i},\\
&\qquad\qquad\qquad\qquad\qquad\qquad l_{\mathfrak{d}}\mathrm{Mult}(h)\left(\prod_{1\leq i\leq n}\left(\prod_{l\geq 1}(l)!\frac{(-1)^{l-1}}{l^2}\right)^{k_{il}}\prod_{J_{il}}\prod_{j'\in J_{ij}} u_{ij'}\right)z^{\sum_iP_im_i}
\Bigg)
\end{align*}

\begin{align*}
&=\Bigg(\mathrm{Mult}(h)\sum_{i=1}^n\sum_{l\geq 1}lA_{i}^l\left((l)!\frac{(-1)^{l-1}}{l^2}\right)\left((l)! \frac{(-1)^{l-1}}{(l)^2}\right)^{k_{il}-1}\sum_{J_{il}}\left(\prod_{\substack{J_{il}'\\ J_{il}'\cap J_{il}=\emptyset}}\prod_{j'\in J_{il}'}u_{ij'}\right)\prod_{j\in J_{il}}u_{ij}\cdot\\
&\qquad\cdot\left(\prod_{l'\neq l}\left((l')! \frac{(-1)^{l'-1}}{(l')^2}\right)^{k_{il'}}\prod_{j'\in J_{il'}}u_{ij'}\right)\left(\prod_{i'\neq i}\prod_{l'\geq 1}\left((l')! \frac{(-1)^{l'-1}}{(l')^2}\right)^{k_{i'j'}}\prod_{J_{i'j'}}\prod_{j'\in J_{i'l'}}u_{i'j'}\right)z^{\sum_iP_im_i},\\
&\qquad\qquad\qquad\qquad\qquad\qquad l_{\mathfrak{d}}\mathrm{Mult}(h)\left(\prod_{1\leq i\leq n}\left(\prod_{l\geq 1}(l)!\frac{(-1)^{l-1}}{l^2}\right)^{k_{il}}\prod_{J_{il}}\prod_{j'\in J_{ij}} u_{ij'}\right)z^{\sum_iP_im_i}
\Bigg)\\
&=\Bigg(\mathrm{Mult}(h)\sum_{i=1}^n\sum_{l\geq 1}lA_{i}^l\left((l)!\frac{(-1)^{l-1}}{l^2}\right)^{k_{il}}k_{il}\left(\prod_{J_{il}}\prod_{j'\in J_{il}}u_{ij'}\right)\left(\prod_{l'\neq l}\left((l')! \frac{(-1)^{l'-1}}{(l')^2}\right)^{k_{il'}}\prod_{j'\in J_{il'}}u_{ij'}\right)\cdot \\
&\qquad\cdot \left(\prod_{i'\neq i}\prod_{l'\geq 1}\left((l')! \frac{(-1)^{l'-1}}{(l')^2}\right)^{k_{i'j'}}\prod_{J_{i'j'}}\prod_{j'\in J_{i'l'}}u_{i'j'}\right)z^{\sum_iP_im_i},\\
&\qquad\qquad\qquad\qquad\qquad\qquad l_{\mathfrak{d}}\mathrm{Mult}(h)\left(\prod_{1\leq i\leq n}\left(\prod_{l\geq 1}(l)!\frac{(-1)^{l-1}}{l^2}\right)^{k_{il}}\prod_{J_{il}}\prod_{j'\in J_{ij}} u_{ij'}\right)z^{\sum_iP_im_i}
\Bigg)\\
&=\Bigg(\mathrm{Mult}(h)\sum_{i=1}^n\sum_{l\geq 1}lA_{i}^lk_{il}\left(\prod_{l'\geq 1}\left((l')! \frac{(-1)^{l'-1}}{(l')^2}\right)^{k_{il'}}\prod_{J_{il'}}\prod_{j'\in J_{il'}}u_{ij'}\right)\cdot\\
&\qquad\cdot\left(\prod_{i'\neq i}\prod_{l'\geq 1}\left((l')! \frac{(-1)^{l'-1}}{(l')^2}\right)^{k_{i'j'}}\prod_{J_{i'j'}}\prod_{j'\in J_{i'l'}}u_{i'j'}\right)z^{\sum_iP_im_i}, \\
&\qquad\qquad\qquad\qquad\qquad\qquad l_{\mathfrak{d}}\mathrm{Mult}(h)\left(\prod_{1\leq i\leq n}\left(\prod_{l\geq 1}(l)!\frac{(-1)^{l-1}}{l^2}\right)^{k_{il}}\prod_{J_{il}}\prod_{j'\in J_{ij}} u_{ij'}\right)z^{\sum_iP_im_i}
\Bigg)
\\
&=\Bigg(\mathrm{Mult}(h)\sum_{i=1}^n\sum_{l\geq 1}lA_{i}^lk_{il}\left(\prod_{1\leq i'\leq n}\prod_{l'\geq 1}\left((l')! \frac{(-1)^{l'-1}}{(l')^2}\right)^{k_{i'j'}}\prod_{J_{i'j'}}\prod_{j'\in J_{i'l'}}u_{i'j'}\right)z^{\sum_iP_im_i}, \\
&\qquad\qquad\qquad\qquad\qquad\qquad l_{\mathfrak{d}}\mathrm{Mult}(h)\left(\prod_{1\leq i\leq n}\left(\prod_{l\geq 1}(l)!\frac{(-1)^{l-1}}{l^2}\right)^{k_{il}}\prod_{J_{il}}\prod_{j'\in J_{ij}} u_{ij'}\right)z^{\sum_iP_im_i}
\Bigg)
\end{align*}

In addition, defining the $s(\mathbf{k})$-tuple of vectors $\mathbf{w}(\mathbf{k})$ as in Notation \ref{not:k partition}, then 
\begin{align*}
w_{\Gamma}(E_{\mathrm{out}})=\sum_{r=1}^{s(\mathbf{k})}w_r(\mathbf{k})
\end{align*} 
and there are $k_{il}$ unbounded edges $E_{iJ_{il}l}$ such that $h(E_{iJ_{il}l})\subset\xi_{iJ_{il}l}+lm_{i}\R_{\geq 0}$; thus $\Gamma$ is of type $\left(\mathbf{w}(\mathbf{k}),\xi\right)$. 

Recall, by construction of $\tilde{\mathfrak{D}}_N^{\infty}$ that for every $\mathfrak{d}\in\mathfrak{D}_N^{\infty}\setminus{\mathfrak{D}}_N$ such that $\mathfrak{d}=m_{\mathfrak{d}'}\R_{\geq 0}$ 

\begin{align*}
\log\overrightarrow{f}_{\mathfrak{d}}=\sum_{l_{\mathfrak{d}}\geq 1}\sum_{\substack{\mathfrak{d}'\in\tilde{\mathfrak{D}}_N^{\infty}\\\mathfrak{d}'=\xi'+l_{\mathfrak{d}} m_{\mathfrak{d}\R_{\geq 0}}}}\log\overrightarrow{f}_{\mathfrak{d}'}.
\end{align*} 
At fixed $\mathbf{P}$ and $\mathbf{k}\vdash \mathbf{P}$: for every $1\leq i\leq n $, $l\geq 1$ let $\A_{il}$ be the set of $k_{il}$ disjoint subsets $J_{il}\subset\lbrace 1,... , N\rbrace$ of size $l$. Then, by plugging-in the contribution from each $\mathfrak{d}'$ we have

\begin{align*}
\log\overrightarrow{f}_{\mathfrak{d}}=\sum_{l_{\mathfrak{d}}\geq 1}\sum_{\mathbf{P}}\sum_{\mathbf{k}\vdash \mathbf{P}}\sum_{\A_{il}}\sum_{h\in T_{\mathbf{w}(\mathbf{k}),\xi'}}&\Bigg(\mathrm{Mult}(h)\sum_{i=1}^n\sum_{l\geq 1}lA_{i}^lk_{il}\Bigg(\left(\prod_{1\leq i'\leq n}\prod_{l'\geq 1}\left((l')! \frac{(-1)^{l'-1}}{(l')^2}\right)^{k_{i'j'}}\right)\cdot\\
&\qquad\qquad\cdot \prod_{J_{i'j'}\in\A_{i'l'}}\prod_{j'\in J_{i'l'}}u_{i'j'}\Bigg)z^{\sum_iP_im_i},\\
& l_{\mathfrak{d}}\mathrm{Mult}(h)\left(\prod_{1\leq i\leq n}\prod_{l\geq 1}\left((l)!\frac{(-1)^{l-1}}{l^2}\right)^{k_{il}}\prod_{J_{il}\in\A_{il}}\prod_{j'\in J_{ij}} u_{ij'}\right)z^{\sum_iP_im_i}
\Bigg)
\end{align*}
where the second sum is over all n-tuple $\mathbf{P}=(P_1,... , P_n)\in\N^n$ such that $\sum_iP_im_i=l_{\mathfrak{d}}m_{\mathfrak{d}}$. Given $\A_{il}$, we can rearrange the sum in the following way: let $B_i\defeq\bigcup_{J_{il}\in\A_{il}}J_{il}$, then $B_i$ is of size $\sum_{l\geq 1}lk_{il}=P_i$ and there are  
\[
\frac{P_i!}{\prod_{l\geq 1}k_{il}!(l!)^{k_{il}}}
\]
different ways of writing $B_i$ as disjoint union of $k_{il}$ subsets of size $l$. Therefore, $\log\overrightarrow{f}_{\mathfrak{d}}$ can be rewritten as

\begin{align*}
\log\overrightarrow{f}_{\mathfrak{d}}=\sum_{l_{\mathfrak{d}}\geq 1}\sum_{\mathbf{P}}\sum_{\mathbf{k}\vdash \mathbf{P}}N_{\mathbf{w}(\mathbf{k})}^{\mathrm{trop}}\Bigg(\sum_{i=1}^n\sum_{l\geq 1}lA_{i}^lk_{il} \prod_{\substack{1\leq i'\leq n\\ l'\geq 1}}\left((l')! \frac{(-1)^{l'-1}}{(l')^2}\right)^{k_{i'j'}}\frac{P_{i'}!}{\prod_{l\geq 1}k_{i'l}!(l!)^{k_{i'l}}}\cdot\\
\cdot\left(\sum_{\substack{B_{i'}\subset\lbrace1,... ,N\rbrace\\|B_{i'}|=P_{i'}}}\prod_{j'\in B_{i'}} u_{i'j'}\right)z^{\sum_iP_im_i}, \\
l_{\mathfrak{d}}\prod_{1\leq i\leq n}\prod_{l\geq 1}\left((l)!\frac{(-1)^{l-1}}{l^2}\right)^{k_{il}}\frac{P_i!}{\prod_{l\geq 1}k_{il}!(l!)^{k_{il}}}\left(\sum_{\substack{B_i\subset\lbrace1,... ,N\rbrace\\|B_i|=P_i}}\prod_{j'\in B_i} u_{ij'}\right)z^{\sum_iP_im_i}\Bigg)
\end{align*}
and by recalling the combinatorics of $t_i^{P_i}$ as sum of $u_{ij}$ we finally get the expected result:

\begin{multline}
\log\overrightarrow{f}_{\mathfrak{d}}=\sum_{l_{\mathfrak{d}}\geq 1}\sum_{\mathbf{P}}\sum_{\mathbf{k}\vdash \mathbf{P}}N_{\mathbf{w}(\mathbf{k})}^{\mathrm{trop}}\sum_{i=1}^n\sum_{l\geq 1}lA_{i}^lk_{il}\left(\prod_{1\leq i'\leq n}\prod_{l'\geq 1}\left(\frac{(-1)^{l'-1}}{(l')^2}\right)^{k_{i'l'}}\frac{1}{k_{i'l'}!}t_{i'}^{P_{i'}}\right)z^{\sum_iP_im_i}, \\
 l_{\mathfrak{d}}\left(\prod_{1\leq i\leq n}\prod_{l\geq 1}\left(\frac{(-1)^{l-1}}{l^2}\right)^{k_{il}}\frac{1}{k_{il}!}t_i^{P_i}\right)z^{\sum_iP_im_i}
\Bigg).
\end{multline}
\end{proof}

\section{Gromov--Witten invariants}\label{sec:GW}


Let $\mathbf{m}=(m_1,...  ,m_n)$ be an n-tuple of primitive non-zero vectors $m_j\in\Lambda$ and consider the toric surface $\overline{Y}_{\mathbf{m}}$ whose fan in $\Lambda_\R$ has rays $-\R_{\geq 0}m_1,...  ,-\R_{\geq 0}m_n$. If the fan is not complete (i.e. its rays do not span $\Lambda_\R$) we can add some more rays and we still denote the compact toric surface by $\overline{Y}_{\mathbf{m}}$\footnote{Adding extra rays is irrelevant, because Gromov--Witten invariants are equivalent under birational transformations.}. Let $D_{m_1},...  ,D_{m_n}$ be the toric divisors corresponding to the rays $-\R_{\geq 0}m_1,...  ,-\R_{\geq 0}m_n$: we blow-up a point $\xi_j$ in general position on the divisor $D_{m_j}$, for $j=1,... , n$. Since we allow $m_j=m_k$ for $k\neq j$, we may blow-up more than one distinct points on the same toric divisor $D_{m_j}$. We denote by $E_j$ the exceptional divisor of the blow-up of the point $\xi_j$ and by $Y_{\mathbf{m}}$ the resulting projective surface. The strict transform of the toric boundary divisor $\partial Y_{\mathbf{m}}$ is an anti-canonical cycle of rational curves and the pair $(Y_{\mathbf{m}}, \partial Y_{\mathbf{m}})$ is a log Calabi--Yau pair. 

Following \cite{Pierrick} we introduce genus $0$ Gromov--Witten invariants both for a projective surface $Y_{\mathbf{m}}$ relative to the divisor $\partial Y_{\mathbf{m}}$ and for the toric surface $\overline{Y}_{\mathbf{m}}$ relative to the full toric boundary divisor $\bar{\partial}Y_{\mathbf{m}}=D_1\cup...\cup D_n$. In the following sections, we review the definitions and we recall the relation with tropical curve count (Proposition \ref{thm:trop-toric}) and the degeneration formula (Proposition \ref{thm:deg}).

\subsection{Gromov--Witten invariants for $(Y_{\mathbf{m}},\partial Y_{\mathbf{m}})$}\label{sec:blowup GW}

Let $\mathbf{P}=(P_1,...  ,P_n)$ be a n-tuple $\mathbf{P}\in\N^n$. We assume \[l_{\mathbf{P}}m_{\mathbf{P}}\defeq\sum_{j=1}^nP_jm_j\neq 0\] for some primitive $m_{\mathbf{P}}\in\Lambda$. The vector $m_{\mathbf{P}}$ can be written as a combination of vectors $m_i$ which generates the fan of $\overline{Y}_{\mathbf{m}}$: 
\[
m_{\mathbf{P}}=a_{L,\mathbf{P}}m_{L}+a_{R,\mathbf{P}}m_R
\] 
$a_{L,\mathbf{P}}, a_{R,\mathbf{P}}\in\C$. Then, denote by $D_{m_L}$ and $D_{m_R}$ the toric divisors corresponding to $m_L$ and $m_R$. We are going to define a class $\beta_{\mathbf{P}}\in H_2(\overline{Y}_{\mathbf{m}},\Z)$ which represents curves in $Y_{\mathbf{m}}$ with tangency conditions prescribed by $\mathbf{P}$: let $\beta$ be the class whose intersection numbers are:
\begin{itemize}
\item for every divisor $D_{m_j}$, $j=1,... ,n$, distinct from $D_{m_R}$ and $D_{m_L}$
\[\beta\cdot D_{m_j}=\sum_{j':D_{m_{j'}}=D_{m_j}}P_{j'}\]
\item for $D_{m_L}$
\[\beta\cdot D_{m_L}=l_{\mathbf{P}}a_{L,\mathbf{P}}+\sum_{j':D_{m_{j'}}=D_{m_L}}P_{j'}\]
\item for $D_{m_R}$
\[\beta\cdot D_{m_R}=l_{\mathbf{P}}a_{R,\mathbf{P}}+\sum_{j':D_{m_{j'}}=D_{m_R}}P_{j'}\]
\item for every divisor $D$ different from all $D_{m_L}, D_{m_R}$ and all $D_{m_i}$, then $\beta\cdot D=0$.  
\end{itemize}


Then we define $\beta_{\mathbf{P}}\in H_2(Y_{\mathbf{m}},\Z)$ as 
\[\beta_{\mathbf{P}}\defeq\nu^*\beta-\sum_{j=1}^n{P_j}[E_j]\]
where $\nu\colon \overline{Y}_{\mathbf{m}}\to Y_{\mathbf{m}}$ is the blow-up morphism. 
Let $M_{g,\mathbf{P}}(Y_{\mathbf{m}}/\partial Y_{\mathbf{m}})$ be the moduli space of stable log maps of genus $g$ and class $\beta_\mathbf{P}$ to a target log space $Y_{\mathbf{m}}$ that is log smooth over $\partial Y_{\mathbf{m}}$. In \cite{logGW}, the authors prove that $M_{g,\mathbf{P}}(Y_{\mathbf{m}}/\partial Y_{\mathbf{m}})$ is a proper Deligne--Mumford stack and it admits a virtual fundamental class 
\[[M_{g,\mathbf{P}}(Y_{\mathbf{m}}/\partial Y_{\mathbf{m}})]^{\vir}\in A_g(M_{g,\mathbf{P}}(Y_{\mathbf{m}},\partial Y_{\mathbf{m}}), \Q).\]  
In particular for genus $0$, $M_{0,\mathbf{P}}(Y_{\mathbf{m}})$ has virtual dimension zero, hence the log Gromov--Witten invariants of $Y_{\mathbf{m}}$ are defined as
\begin{equation}\label{eq:N_(0,P)}
N_{0,\mathbf{P}}(Y_{\mathbf{m}})\defeq\int_{[M_{0,\mathbf{P}}(Y_{\mathbf{m}}/\partial Y_{\mathbf{m}})]^{\vir}} 1
\end{equation}
where $1\in A^0(\overline{M}_{0,\mathbf{P}}(Y_{\mathbf{m}}/\partial Y_{\mathbf{m}}),\Q)$ is the dual class of a point.

\subsection{Gromov--Witten invariants for $(\overline{Y}_{\mathbf{m}},\partial \overline{Y}_{\mathbf{m}})$}\label{sec:toric GW}

We now introduce Gromov--Witten invariants for the toric surface $\overline{Y}_{\mathbf{m}}$. Let $s$ be an integer number and let $\mathbf{w}=(w_1,... , w_s)$ be a $s$-tuple of \textit{weight} vectors such that for every $r=1,... , s$ there is an index $i\in\lbrace 1,... ,n\rbrace$ such that $-m_i\R_{\geq 0}=-w_r\R_{\geq 0}$. In particular $-w_r\R_{\geq 0}$ is contained in the fan of $\overline{Y}_{\mathbf{m}}$ and we denote by $D_{w_r}$ the corresponding divisor in $\partial\overline{Y}_{\mathbf{m}}$. We also assume $\sum_{r=1}^sw_i\neq 0$.  
In order to ``count'' curves in $\overline{Y}_{\mathbf{m}}$ meeting $\partial\overline{Y}_{\mathbf{m}}$ in $s$ prescribed points with multiplicity $|w_r|$ and in a single unprescribed point with multiplicity $|\sum_{r=1}^s w_r|$, we are going to define a suitable curve class $\beta_{\mathbf{w}}\in H_2(\overline{Y}_{\mathbf{m}},\Z)$. Let $m_{\mathbf{w}}\in\Lambda$ be a primitive vector in $\Lambda$, such that 
\[
\sum_{r=1}^sw_r=l_{\mathbf{w}}m_{\mathbf{w}}
\]
where $l_{\mathbf{w}}=\left|\sum_{r=1}^sw_r \right|$. In particular $m_{\mathbf{w}}$ belongs to a cone of the fan of $\overline{Y}_{\mathbf{m}}$ and it can be written uniquely as 
\[ m_{\mathbf{w}}=a_{L}m_L+a_Rm_R\]
where $m_L,m_R$ are primitive generator of the rays of the fan and $a_L,a_R\in\N$.  

Then $\beta_{\mathbf{w}}$ is determined by the following intersection numbers:
\begin{itemize}
\item for every divisor $D_{w_r}$, $r=1,... ,s$, distinct from $D_{m_R}$ and $D_{m_L}$
\[\beta_{\mathbf{w}}\cdot D_{w_r}=\sum_{r':D_{w_{r'}}=D_{w_r}}|w_r'|\]
\item for both the divisors $D_{m_R}$ and $D_{m_L}$
\[\beta_{\mathbf{w}}\cdot D_{m_L}=l_{\mathbf{w}}a_L+\sum_{r':D_{w_{r'}}=D_{m_L}}|w_r'|\]
\[\beta_{\mathbf{w}}\cdot D_{m_R}=l_{\mathbf{w}}a_R+\sum_{r':D_{w_{r'}}=D_{m_R}}|w_r'|\]
\item for every divisor $D$ different from all $D_{w_r}$ and all $D_{m_i}$, then $\beta_{\mathbf{w}}\cdot D=0$.  
\end{itemize} 
The existence of this class follows from toric geometry: since $\overline{Y}_{\mathbf{m}}$ is complete, $A_1(\overline{Y}_{\mathbf{m}})$ is generated by the class of the divisors associated to the rays of the fan. 
Let $M_{0,\mathbf{w}}(\overline{Y}_{\mathbf{m}}, \partial \overline{Y}_{\mathbf{m}})$ be the moduli space of stable log maps of genus $0$ and class $\beta_{\mathbf{w}}$; it is a proper Deligne--Mumford stack of virtual dimension $s$. In addition, there are $s$ evaluation maps $\ev_1,...,\ev_s$ such that 
\begin{align*}
\ev_r&\colon M_{0,\mathbf{w}}(\overline{Y}_{\mathbf{m}},\partial\overline{Y}_{\mathbf{m}})\to D_{w_r}
\end{align*}
and the Gromov--Witten invariants are defined as follows
\begin{equation}\label{eq:N_w}
N_{0,\mathbf{w}}({\overline{Y}_{\mathbf{m}}})\defeq\int_{[M_{0,\mathbf{w}}(\overline{Y}_{\mathbf{m}},\partial \overline{Y}_{\mathbf{m}})]^{\vir}} \prod_{r=1}^s\ev^*_r(\pt_r)
\end{equation}
where $\pt_r\in A^1(D_{w_r})$ is the dual class of a point. 

\begin{prop}\label{thm:trop-toric}
Let $\mathbf{m}=(m_1...  ,m_n)$ be a n-tuple of non-zero primitive vectors in $\Lambda$. Then for every n-tuple $\mathbf{P}=(P_1,... ,P_n)\in\N^n$ and every partition $\mathbf{k}\vdash \mathbf{P}$
\begin{equation}\label{eq:N_trop N_rel}
N_{0,\mathbf{w}(\mathbf{k})}^{\mathrm{trop}}=N_{0,\mathbf{w}(\mathbf{k})}(\overline{Y}_{\mathbf{m}})\prod_{r=1}^n\prod_{l\geq 1}l^{k_{rl}}
\end{equation}
\end{prop}

\begin{prop}\label{thm:deg}
Let $\mathbf{m}=(m_1,...  m_n)$ be a n-tuple of primitive, non zero vectors in $\Lambda$, and let $\mathbf{P}=(p_1,...  ,p_n)\in\N^n$ be a n-tuple of positive integers, then 
\begin{equation}\label{eq:degeneration}
N_{0,\mathbf{P}}(Y_{\mathbf{m}})=\sum_{\mathbf{k}\,\vdash\,\mathbf{P}}N_{0,\mathbf{w}(\mathbf{k})}(\overline{Y}_{\mathbf{m}})\prod_{j=1}^n\prod_{l=1}^{l_j}\frac{l^{k_{jl}}}{k_{jl}!}(R_l)^{k_{jl}}
\end{equation}
where the sum is over all partition $\mathbf{k}$ of $\mathbf{P}$ and $R_l=\frac{(-1)^{l-1}}{l^2}$. 
\end{prop}

A first proof of this result is Proposition 5.3 \cite{GPS}, and it has been computed by applying Li's degeneration formula. In \cite{Pierrick}, the author proves a general version of \eqref{eq:degeneration} (see Proposition 11 of \cite{Pierrick}) using a more sophisticated approach. However in genus $0$ the two approaches give the same formula, as Gromov--Witten invariants are the same in Li's theory and in the logarithmic theory.

\section{Gromov--Witten invariants from commutators in $\tilde{\mathbb{V}}$}\label{sec:GW and commutators}

In this section we finally collect together the previous results and we get a generating function for genus $0$ relative Gromov--Witten invariants in terms of consistent scattering diagrams in the extended tropical vertex group.

\begin{theorem}\label{thm:m_out}
Let $\mathbf{m}=(m_1,...,m_n)$ be n primitive non zero vectors in $\Lambda$ and let $\mathfrak{D}$ be a standard scattering diagram over $\C[\![t_1,...,t_n]\!]$ with
\begin{align*}
\mathfrak{D}=\left\lbrace \left(\mathfrak{d}_i=m_i\R,\overrightarrow{f}_{\mathfrak{d}_i}=\left(1+A_it_iz^{m_i},1+t_iz^{m_i}\right)\right)\vert 1\leq i\leq n\right\rbrace
\end{align*}
where $A_{i}\in\mathfrak{gl}(r,\C)$ and assume $[A_i,A_{i'}]=0$ for all $i,i'\in\lbrace 1,...,n\rbrace$.
Then for every wall $(\mathfrak{d}=m_{\mathfrak{d}}\R_{\geq 0},\overrightarrow{f}_{\mathfrak{d}})\in\mathfrak{D}^{\infty}\setminus\mathfrak{D}$: 
\begin{multline}\label{eq:GW_rel}
\log\overrightarrow{f}_{\mathfrak{d}}=\Bigg(\sum_{l_{\mathfrak{d}}\geq 1}\sum_{\mathbf{P}}\sum_{\mathbf{k}\vdash \mathbf{P}}N_{0,\mathbf{w}(\mathbf{k})}(\overline{Y}_{\mathbf{m}}){\sum_{i=1}^n\sum_{l\geq 1}lA_{i}^lk_{il}\left(\prod_{i'=1}^n\prod_{l'\geq 1}\left(\frac{(-1)^{l'-1}}{l'}\right)^{k_{i'l'}}\frac{1}{k_{i'l'}!}\right)}\prod_{i=1}^nt_{i}^{P_{i}}z^{\sum_iP_im_i}, \\
 \sum_{l_{\mathfrak{d}}\geq 1}\sum_{\mathbf{P}} l_{\mathfrak{d}}N_{0,\mathbf{P}}(Y_{\mathbf{m}})\left(\prod_{i=1}^nt_i^{P_i}\right)z^{\sum_iP_im_i}
\Bigg)
\end{multline}
where the second is over all n-tuples $\mathbf{P}=(P_1,...,P_n)$ satisfying $\sum_{i=1}^nP_im_i=l_{\mathfrak{d}}m_{\mathfrak{d}}$.  
\end{theorem}

%
%
%

\begin{proof}
The proof is a consequence of Theorem \ref{thm:2.8}, Theorem \ref{thm:deg} and Theorem \ref{thm:trop-toric}. Indeed from equation \eqref{eq:degeneration} and Proposition \ref{thm:trop-toric}

\begin{align}\label{eq:N_P-N_trop}
N_{0,\mathbf{P}}(Y_{\mathbf{m}})&=\sum_{\mathbf{k}\vdash\mathbf{P}}N_{0,\mathbf{w}(\mathbf{k})}(\overline{Y}_{\mathbf{m}})\prod_{i=1}^n\prod_{l\geq 1}\frac{l^{k_{il}}}{k_{il}!}R_{l_i}^{k_{il}}\\
&=\sum_{\mathbf{k}\vdash\mathbf{P}}N_{\mathbf{w}(\mathbf{k})}^{\mathrm{trop}}\prod_{i=1}^n\prod_{l\geq 1}\frac{1}{k_{il}!}R_l^{k_{il}}
\end{align}

Then, for every every wall $(\mathfrak{d}=m_{\mathfrak{d}}\R_{\geq 0},\overrightarrow{f}_{\mathfrak{d}})\in\mathfrak{D}^{\infty}\setminus\mathfrak{D}$ is explicitly given by \eqref{eq:f_trop}:
\begin{multline}
\log\overrightarrow{f}_{\mathfrak{d}}=\sum_{l_{\mathfrak{d}}\geq 1}\sum_{\mathbf{P}}\sum_{\mathbf{k}\vdash \mathbf{P}}N_{\mathbf{w}(\mathbf{k})}^{\mathrm{trop}}\Bigg(\sum_{i=1}^n\sum_{l\geq 1}lA_{i}^lk_{il}\left(\prod_{i'=1}^n\prod_{l'\geq 1}\left(\frac{(-1)^{l'-1}}{(l')^2}\right)^{k_{i'l'}}\frac{1}{k_{i'l'}!}t_{i'}^{P_{i'}}\right)z^{\sum_iP_im_i}, \\
 l_{\mathfrak{d}}\left(\prod_{i=1}^n\prod_{l\geq 1}\left(\frac{(-1)^{l-1}}{l^2}\right)^{k_{il}}\frac{1}{k_{il}!}t_i^{P_i}\right)z^{\sum_iP_im_i}
\Bigg)
\end{multline}
  
and plugging-in equation \ref{eq:N_trop N_rel}, we get

\begin{align*}
\log\overrightarrow{f}_{\mathfrak{d}}=\Bigg(\sum_{l_{\mathfrak{d}}\geq 1}\sum_{\mathbf{P}}\sum_{\mathbf{k}\vdash \mathbf{P}}N_{0,\mathbf{w}(\mathbf{k})}(\overline{Y}_{\mathbf{m}})\sum_{i=1}^n\sum_{l\geq 1}lA_{i}^lk_{il}\left(\prod_{i'=1}^n\prod_{l'\geq 1}\left(\frac{(-1)^{l'-1}}{l'}\right)^{k_{i'l'}}\frac{1}{k_{i'l'}!}t_{i'}^{P_{i'}}\right)z^{\sum_iP_im_i}, \\
 \sum_{l_{\mathfrak{d}}\geq 1}\sum_{\mathbf{P}} l_{\mathfrak{d}}N_{0,\mathbf{P}}(Y_{\mathbf{m}})\left(\prod_{i=1}^nt_i^{P_i}\right)z^{\sum_iP_im_i}
\Bigg)
\end{align*}
where we use the relation between tropical curve counting and open Gromov--Witten of Theorem \ref{thm:trop-toric}, namely $N_{0,\mathbf{w}(\mathbf{k})}^{\mathrm{trop}}\prod_{r=1}^n\prod_{l\geq 1}\frac{1}{l^{k_{rl}}}=N_{0,\mathbf{w}(\mathbf{k})}(\overline{Y}_{\mathbf{m}})$.
\end{proof}

\begin{example}
Let $m_1=(0,1)$, $m_2=(1,0)$ and consider the scattering diagram
\begin{align*}
\mathfrak{D}=\left\lbrace \left(\mathfrak{d}_1=m_1\R,\overrightarrow{f}_{\mathfrak{d}_1}=\left(1+A_1t_1z^{m_1},1+t_1z^{m_1}\right)\right), \left(\mathfrak{d}_2=m_2\R,\overrightarrow{f}_{\mathfrak{d}_2}=\left(1, 1+t_2z^{m_2}\right)\right)\right\rbrace
\end{align*}
with $A_1\in\mathfrak{gl}(r,\C)_e$, which corresponds to the initial standard scattering diagram of Example \ref{ex:diagram1}.   
Let us consider the ray $\mathfrak{d}=(1,2)\R_{\geq 0}\in\mathfrak{D}^{\infty}\setminus \mathfrak{D}$. From Theorem \ref{thm:m_out} we have
\begin{multline}
\log\overrightarrow{f}_{(1,2)}=\sum_{l_{\mathfrak{d}}\geq 1}\sum_{\mathbf{P}}\Bigg(\sum_{\mathbf{k}\vdash \mathbf{P}}A_1k_{11}N_{0,\mathbf{w}(\mathbf{k})}(\overline{Y}_{\mathbf{m}})\prod_{i=1}^2\left(\prod_{l'\geq 1}\left(\frac{(-1)^{l'-1}}{l'}\right)^{k_{il'}}\frac{1}{k_{il'}!}\right)t_i^{P_i}z^{\sum_iP_im_i}, \\
  l_{\mathfrak{d}}N_{0,\mathbf{P}}(Y_{\mathbf{m}})\left(\prod_{i=1}^2t_i^{P_i}\right)z^{\sum_iP_im_i}
\Bigg)
\end{multline}
where the second sum is over all n-tuples $\mathbf{P}=(2 l_{\mathfrak{d}},l_{\mathfrak{d}})$. Now as a consequence of the pentagon identity for the di--logarithm we know that $N_{0,\mathbb{P}}(Y_{\mathbf{m}})=0$ unless $\mathbf{P}=\alpha(1,1)$ for some $\alpha\geq 1$, hence we have

\begin{align*}
\log\overrightarrow{f}_{(1,2)}=\left(\sum_{l_{\mathfrak{d}}\geq 1}\sum_{\mathbf{P}=(2l_{\mathfrak{d}},l_{\mathfrak{d}})}\sum_{\mathbf{k}\vdash \mathbf{P}}A_1k_{11}N_{0,\mathbf{w}(\mathbf{k})}(\overline{Y}_{\mathbf{m}})\prod_{i=1}^2\prod_{l'\geq 1}\left(\frac{(-1)^{l'-1}}{l'}\right)^{k_{il'}}\frac{1}{k_{il'}!}t_1^{2l_{\mathfrak{d}}}t_2^{l_{\mathfrak{d}}}x^{l_{\mathfrak{d}}}y^{2l_{\mathfrak{d}}}, 0\right).
\end{align*}

For $l_{\mathfrak{d}}=1$ we can explicitly compute the partitions $\mathbf{k}\vdash (2,1)$ and the weights $\mathbf{w}(\mathbf{k})$:
\begin{itemize}
\item[(a)] $k_1=(k_{1l})_{l\geq 1}=(2)$ and $k_2=(k_{2l})_{l\geq 1}=(1)$ with weight $\mathbf{w}(\mathbf{k})=(m_1 ,m_1, m_2)$;
\item[(b)] $k_1=(k_{1l})_{l\geq 1}=(0,1)$ and $k_2=(k_{2l})_{l\geq 1}=(1)$ with weight $\mathbf{w}(\mathbf{k})=(2m_1, m_2)$.
\end{itemize}
Therefore the non trivial contributions to $\log\overrightarrow{f}_{(1,2)}$ are
\begin{align*}
&A_1\left(\cdot 2 \cdot N_{0,\mathbf{w}(\mathbf{k})_{(a)}}(\overline{Y}_{\mathbf{m}})\frac{1}{2!}\cdot 1 +0 \cdot N_{0,\mathbf{w}(\mathbf{k})_{(b)}}(\overline{Y}_{\mathbf{m}})\frac{(-1)}{2}\cdot 1\right) t_1^2t_2xy^2\\
&=A_1N_{0,\mathbf{w}(\mathbf{k})_{(a)}}(\overline{Y}_{\mathbf{m}})t_12t_2xy^2
\end{align*}
and comparing it with the function $\overrightarrow{f}_{(1,2)}$ of Example \ref{ex:diagram1} we get $N_{0,\mathbf{w}(\mathbf{k})_{(a)}}(\overline{Y}_{\mathbf{m}})=1$. In this simple case $N_{0,\mathbf{w}(\mathbf{k})_{(a)}}(\overline{Y}_{\mathbf{m}})$ can be also explicitly computed by applying formula \eqref{eq:N_P-N_trop}: indeed
\[
N_{0,\mathbf{w}(\mathbf{k})_{(a)}}(\overline{Y}_{\mathbf{m}})=\frac{N_{\mathbf{w}(\mathbf{k})_{(a)}}^{\mathrm{trop}}}{\prod_{1\leq i\leq 2}\prod_{l\geq 1}l^{k_{il}}}=\frac{N_{\mathbf{w}(\mathbf{k})_{(a)}}^{\mathrm{trop}}}{1\cdot 1}=N_{\mathbf{w}(\mathbf{k})_{(a)}}^{\mathrm{trop}}
\]
and the there is a unique tropical curve, counted with multiplicity with weight $\mathbf{w}(\mathbf{k})_{(a)}=(m_1, m_1, m_2)$, drawn in Figure \ref{fig:trop2}. 
\begin{figure}[h]
\begin{tikzpicture}
\draw (1,0) -- (1,1);
\draw (0,1) -- (1,1);
\draw (1,1) -- (2,2);
\draw (2,0) -- (2,2);
\draw (2,2) -- (3,4);
\node [font=\tiny, left] at (1,0.2) {$w_{1}=1$};
\node [font=\tiny,above] at (0.5,1) {$w_{3}=1$};
\node [font=\tiny, left] at (2,0.5) {$w_{2}=1$};
\node [font=\tiny, left] at (2.7,3.5) {$w_{\mathrm{out}}=1$};
\end{tikzpicture}
\caption{Case $(a)$}
\label{fig:trop2}
\end{figure} 

In conclusion, the non trivial contribution in the matrix component of $\overrightarrow{f}_{(1,2)}$ allow to compute a Gromov--Witten invariant for the toric surface associated to the fan with rays $(-1,0), (0,-1), (1,2)$. 
\end{example}

\subsection{The generating function of $N_{0,\mathbf{w}}(\overline{Y}_{\mathbf{m}})$}

In this section we introduce a new setting in order to study when the automorphism associated to a ray of the consistent scattering diagram is a generating function for the invariants $N_{0,\mathbf{w}}(\overline{Y}_{\mathbf{m}})$ (as in Theorem \ref{thm:m_out}). In particular, let $\ell_1,\ell_2\geq 1$ and let us now consider the initial scattering diagram 
\begin{align*}
\mathfrak{D}=\left\lbrace \left(\mathfrak{d}_i=m_1\R,\overrightarrow{f}_{\mathfrak{d}_i}=\left(1+A_it_iz^{m_1},1+t_iz^{m_1}\right)\right) , \left(\mathfrak{d}_i=m_2\R,\overrightarrow{f}_{\mathfrak{d}_j}=\left(1+Q_js_jz^{m_2},1+s_jz^{m_2}\right)\right)\right\rbrace_{\substack{1\leq i\leq \ell_1 \\ 1\leq j\leq \ell_2}}
\end{align*}
with $[A_i,Q_j]=[A_i,A_{i'}]=[Q_j,Q_{j'}]=0$ for all $i,i'=1,...,\ell_1$ and $j,j'=1,...,\ell_2$. We will not put any restrictions on the rank of the matrices $A_i,Q_j$. As a consequence of Theorem \ref{thm:m_out}, for every ray $(\mathfrak{d}=m_{\mathfrak{d}}\R,\overrightarrow{f}_{\mathfrak{d}})\in\mathfrak{D}^{\infty}\setminus\mathfrak{D}$, the function $\overrightarrow{f}_{\mathfrak{d}}$ can be written as follows:

\begin{multline}\label{eq:GW_rel GPS}
\log\overrightarrow{f}_{\mathfrak{d}}=\Bigg(\sum_{l_{\mathfrak{d}}\geq 1}\sum_{\mathbf{P}}\sum_{\mathbf{k}\vdash \mathbf{P}}N_{0,\mathbf{w}(\mathbf{k})}(\overline{Y}_{\mathbf{m}})\left(\sum_{i=1}^{\ell_1}\sum_{l\geq 1}lA_{i}^lk_{il}+\sum_{i=1}^{\ell_2}\sum_{l\geq 1}lQ_{i}^lk_{(i+\ell_1)l}\right)\cdot\\
\left(\prod_{i'=1}^{\ell_1+\ell_2}\prod_{l'\geq 1}\left(\frac{(-1)^{l'-1}}{l'}\right)^{k_{i'l'}}\frac{1}{k_{i'l'}!}\right)\prod_{i=1}^{\ell_1}t_i^{P_i}\prod_{i=1}^{\ell_2}s_i^{P_{i+\ell_1}}z^{\sum_iP_im_i},\\
 \sum_{l_{\mathfrak{d}}\geq 1}\sum_{\mathbf{P}} l_{\mathfrak{d}}N_{0,\mathbf{P}}(Y_{\mathbf{m}})\prod_{i=1}^{\ell_1}t_i^{P_i}\prod_{i=1}^{\ell_2}s_i^{P_{i+\ell_1}} z^{\sum_iP_{i}m_i}
\Bigg)
\end{multline}
where for every $l_{\mathfrak{d}}$, $\sum_{i=1}^{\ell_1}P_im_1+\sum_{i=1}^{\ell_2}P_{\ell_1+i}m_2=l_{\mathfrak{d}}m_{\mathfrak{d}}$. In particular, $\log\overrightarrow{f}_{\mathfrak{d}}$ is a generating function for the invariants $N_{0,\mathbf{P}}(Y_{\mathbf{m}})$. We also expect that $\log\overrightarrow{f}_{\mathfrak{d}}$ is a generating function for $N_{0,\mathbf{w}}(\overline{Y}_{\mathbf{m}})$. 

Since the tangency conditions for $N_{0,\mathbf{w}}(\overline{Y}_{\mathbf{m}})$ are specified by weight vectors $\mathbf{w}$, it is natural to rewrite \eqref{eq:GW_rel GPS} summing over all $\mathbf{w}$ rather than $\mathbf{P}$. We do it below: let $\mathbf{w}=(w_1,...,w_s)$ such that $\sum_{i=1}^sw_i=l_{\mathfrak{d}}m_{\mathfrak{d}}$ and $w_i=|w_i|m_1$ for $i=1,...,s_1$ and $w_i=|w_i|m_2$ for $i=s_1+1,...,s_2+s_1$. Then we can rewrite \eqref{eq:GW_rel GPS} as follows:
\begin{multline}\label{eq:GW_rel GPS2}
\log\overrightarrow{f}_{\mathfrak{d}}=\Bigg(\sum_{l_{\mathfrak{d}}\geq 1}\sum_{\mathbf{w}}N_{0,\mathbf{w}}(\overline{Y}_{\mathbf{m}})\sum_{\mathbf{k}:\mathbf{w}(\mathbf{k})=\mathbf{w}}\left(\sum_{i=1}^{\ell_1}\sum_{l\geq 1}lA_{i}^lk_{il}+\sum_{i=1}^{\ell_2}\sum_{l\geq 1}lQ_{i}^lk_{(i+\ell_1)l}\right)\cdot\\
\cdot\left(\prod_{i'=1}^{\ell_1+\ell_2}\prod_{l'\geq 1}\left(\frac{(-1)^{l'-1}}{l'}\right)^{k_{i'l'}}\frac{1}{k_{i'l'}!}\right)\prod_{i=1}^{\ell_1}t_i^{\sum_{l\geq 1}lk_{il}}\prod_{i=1}^{\ell_2}s_i^{\sum_{l\geq 1}lk_{i+s_1 l}}z^{\sum_iw_i},\\
 \sum_{l_{\mathfrak{d}}\geq 1}\sum_{\mathbf{P}} l_{\mathfrak{d}}N_{0,\mathbf{P}}(Y_{\mathbf{m}})\prod_{i=1}^{\ell_1}t_i^{P_i}\prod_{i=1}^{\ell_2}s_i^{P_{i+\ell_1}} z^{\sum_iP_{i}m_i}
\Bigg)
\end{multline}
where for every $l_{\mathfrak{d}}$, $\sum_{i=1}^{\ell_1}P_im_1+\sum_{i=1}^{\ell_2}P_{\ell_1+i}m_2=l_{\mathfrak{d}}m_{\mathfrak{d}}$.

Let us define vectors $e_{\mathbf{k}(\mathbf{w})}\in\mathfrak{gl}(r,\C)\otimes_{\C}\C[t_1,...,t_{\ell_1},s_1,...,s_{\ell_2}]$ 

\begin{equation}
e_{\mathbf{k}(\mathbf{w})}\defeq\left(\sum_{i=1}^{\ell_1}\sum_{l\geq 1}lA_{i}^lk_{il}+\sum_{i=1}^{\ell_2}\sum_{l\geq 1}lQ_{i}^lk_{(i+\ell_1)l}\right)\prod_{i=1}^{\ell_1}t_i^{\sum_{l\geq 1}lk_{il}}\prod_{i=1}^{\ell_2}s_i^{\sum_{l\geq 1}lk_{i+s_1 l}}
\end{equation}

and set \[V_\mathbf{w}\defeq \sum_{\mathbf{k}:\mathbf{w}(\mathbf{k})=\mathbf{w}}\left(\prod_{i'=1}^{\ell_1+\ell_2}\prod_{l'\geq 1}\left(\frac{(-1)^{l'-1}}{l'}\right)^{k_{i'l'}}\frac{1}{k_{i'l'}!}\right)e_{\mathbf{k}(\mathbf{w})}.\]

To simplify the notation we rewrite the vectors $V_{\mathbf{w}}$, $e_{\mathbf{k}(\mathbf{w})}$ as follows:
\begin{notation}\label{not:base}
Assume $l_{\mathfrak{d}}m_{\mathfrak{d}}$ is fixed, then there are $\lbrace \mathbf{w}_i\rbrace_{i=1,...,s}$ vectors such that $\sum_{j=1}^s\mathbf{w}_j=l_{\mathfrak{d}}m_{\mathfrak{d}}$. Moreover for every $\mathbf{w}_i$ there are $\mathbf{k}_j=\mathbf{k}_j(\mathbf{w}_i)$ such that $\mathbf{w}_i(\mathbf{k}_j)=\mathbf{w}_i$, for $j=1,...,s_i$. In particular, for every $i=1,...,s$ there are $s_i$ vectors $e_{\mathbf{k}_j(\mathbf{w}_i)}$ which we denote by $e_j^{(i)}$. The vectors $V_{\mathbf{w}_i}$ are rewritten as follows:
\[
V_{\mathbf{w}_i}=\sum_{j=1}^{s_i}\lambda_j^{(i)}e_{j}^{(i)} 
\]
where $\lambda_{j}^{(i)}=\lambda_{\mathbf{k}_j(\mathbf{w}_i)}=\left(\prod_{i'=1}^{\ell_1+\ell_2}\prod_{l'\geq 1}\left(\frac{(-1)^{l'-1}}{l'}\right)^{k_{i'l'}}\frac{1}{k_{i'l'}!}\right)$.   
\end{notation} 

\begin{conjecture}\label{conj}
If $\ell_1,\ell_2, \lbrace A_i\rbrace_{1\leq i\leq \ell_1},\lbrace Q_j\rbrace_{1\leq j\leq \ell_2}$ are regarded as arbitrary parameters, and $l_{\mathfrak{d}}m_{\mathfrak{d}}$ is chosen such that $\mathfrak{d}=l_{\mathfrak{d}}m_{\mathfrak{d}}\R_{\geq 0}$ is a ray in the consistent scattering diagram $\mathfrak{D}^\infty$. Then $\overrightarrow{f}_{\mathfrak{d}}$ is a generating function for $N_{0,\mathbf{w}}(\overline{Y}_{\mathbf{m}})$.
\end{conjecture}


\begin{remark}
If the collection of $\lbrace e_j^{(i)}\rbrace_{\substack{i=1,...,s\\j=1,...,s_i}}$ is a basis, then Conjecture \ref{conj} will follow immediately; indeed for every $\lbrace c_{i}\rbrace_{i=1,...,s}\in\Q$
\begin{align*}
0=\sum_{i=1}^sc_{i}V_{\mathbf{w}_i}=\sum_{i=1}^sc_{i}\sum_{j=1}^{s_i}\lambda_{j}^{(i)}e_j^{(i)}=\sum_{i=1}^s\sum_{j=1}^{s_i}c_{i}\lambda_{j}^{(i)}e_j^{(i)} \Leftrightarrow c_{i}\lambda_{j}^{(i)}=0\,\,\forall i=1,...,s,\,\forall j=1,...,s_i.
\end{align*}
In particular since $\lambda_j^{(i)}\neq 0$ for all $j=1,...,s_i$, then $c_{i}=0$ for all $\mathbf{w}_i$, $i=1,...,s$.
\end{remark}

We do not have a proof of the conjecture yet, but in the following examples we show that even if for $\ell_1=\ell_2=1$ the vectors $e_{\mathbf{k}(\mathbf{w})}$ are linearly dependent, then for $\ell_1=\ell_2=2$ they are indeed a basis. Then we conclude the section with a partial result, which states how to compute some invariants from $\log\overrightarrow{f}_{\mathfrak{d}}$ (see Theorem \ref{thm:partial}).  

We are going to consider the initial scattering diagram 
\begin{align*}
\mathfrak{D}=\left\lbrace \left(\mathfrak{d}_i=m_1\R,\overrightarrow{f}_{\mathfrak{d}_i}=\left(1+A_it_iz^{m_1},1+t_iz^{m_1}\right)\right) , \left(\mathfrak{d}_i=m_2\R,\overrightarrow{f}_{\mathfrak{d}_j}=\left(1+Q_js_jz^{m_2},1+s_jz^{m_2}\right)\right)\right\rbrace_{\substack{1\leq i\leq \ell_1 \\ 1\leq j\leq \ell_2}}
\end{align*}
with $m_1=(1,0)$ and $m_2=(0,1)$ and such that $[A_i, A_{i'}]=[Q_j, Q_{j'}]=[A_i,Q_j]=0$ for every $i,i'=1,...,\ell_1$, $j,j'=1,...,\ell_2$. We remark that $\lbrace A_i\rbrace_{1\leq i\leq \ell_1},\lbrace Q_j\rbrace_{1\leq j\leq \ell_2}$ are regarded as formal parameters. 

\begin{example}\label{ex:1a}
Let $\ell_1=\ell_2=1$, then the consistent scattering diagram $\mathfrak{D}^{\infty}$ has the ray $\mathfrak{d}=(1,1)\R_{\geq 0}$. Let $l_{\mathfrak{d}}=2$: the possible vector $\mathbf{w}$ such that $\sum_jw_j=(2,2)$ are
\begin{enumerate}
\item[$(a)$] $\mathbf{w}=(m_1,m_1,m_2,m_2)$;
\item[$(b)$] $\mathbf{w}=(2m_1,m_2,m_2)$;
\item[$(c)$] $\mathbf{w}=(m_1,m_1,2m_2)$;
\item[$(d)$] $\mathbf{w}=(2m_1,2m_2)$
\end{enumerate}
and the partition $\mathbf{k}=(k_1,k_2)$ such that $\mathbf{w}(\mathbf{k})=\mathbf{w}$ are respectively
\begin{enumerate}
\item[$(a)$] $k_1=(2)$ and $k_2=(2)$;
\item[$(b)$] $k_1=(0,1)$ and $k_2=(2)$;
\item[$(c)$] $k_1=(2)$ and $k_2=(0,1)$;
\item[$(d)$] $k_1=(0,1)$ and $k_2=(0,1)$.
\end{enumerate}
Hence from equation \eqref{eq:GW_rel GPS2} we get
\begin{align*}
\log\overrightarrow{f}_{(2,2)}&=N_{\mathbf{w}(a)}(2A_1+2Q_1)\left(\frac{1}{4}\right)t_1^2s_1^2+N_{\mathbf{w}(b)}(2A_1^2+2Q_1)\left(\frac{-1}{4}\right)t_1^2s_1^2+\\
&+N_{\mathbf{w}(c)}(2A_1+2Q_1^2)\left(\frac{-1}{4}\right)t_1^2s_1^2+N_{\mathbf{w}(d)}(2A_1^2+2Q_1^2)\left(\frac{1}{4}\right)t_1^2s_1^2.
\end{align*}
According to Notation \ref{not:base} the vectors $e_{\mathbf{k}(\mathbf{w})}$ are respectively:
\begin{align*}
e_{(a)}=(2A_1+2Q_1)t_1^2s_1^2; \quad e_{(b)}=(2A_1^2+2Q_1)t_1^2s_1^2; \\
e_{(c)}=(2A_1+2Q_1^2)t_1^2s_1^2; \quad e_{(d)}=(2A_1^2+2Q_1^2)t_1^2s_1^2,
\end{align*} 
but they are not linearly independent vectors ($e_{(a)}=-e_{(d)}-e_{(c)}-e_{(d)} $) hence $ \log\overrightarrow{f}_{(2,2)}$ is not a generating function for $N_{\mathbf{w}}$.  
\end{example}

\begin{example}\label{ex:1b}
Let $\ell_1=\ell_2=2$, the consistent scattering diagram $\mathfrak{D}$ has many more rays, and we consider $\mathfrak{d}=(1,1)\R_{\geq0}$ with $l_{\mathfrak{d}}=2$ as in Example \ref{ex:1a}. The vectors $\mathbf{w}$ are the same as before but there are more combinations for the vectors $\mathbf{k}=(k_1^{(1)}, k_1^{(2)},k_2^{(1)}, k_2^{(2)})$, where we denote by $k_j^{(i)}$ the partition which corresponds to $A_i$ if $j=1$, $Q_i$ if $j=2$: 
\begin{enumerate}
\item[$(a)$] $k_1^{(i)}=(2)$ and $k_2^{(i')}=(2)$ for $i,i'=1,2$, $k_1^{(1)}=(1), k_1^{(2)}=(1)$ and $k_2^{(i)}=(2)$ for $i=1,2$, $k_1^{(i)}=(2)$ and $k_2^{(1)}=(1), k_2^{(2)}=(1)$ for $i=1,2$, $k_1^{(1)}=(1),k_1^{(2)}=(1)$ and $k_2^{(1)}=(1), k_2^{(2)}=(1)$; 
\item[$(b)$] $k_1^{(i)}=(0,1)$ and $k_2^{(i')}=(2)$ for $i,i'=1,2$, $k_1^{(i)}=(0,1)$ and $k_2^{(1)}=(1), k_2^{(2)}=(1)$ for $i=1,2$ ;
\item[$(c)$] $k_1^{(i)}=(2)$ and $k_2^{(i')}=(0,1)$ for $i,i'=1,2$, $k_1^{(1)}=(1), k_1^{(2)}=(1)$ and $k_2^{(i')}=(0,1)$ for $i'=1,2$ ;
\item[$(d)$] $k_1^{(i)}=(0,1)$ and $k_2^{(i')}=(0,1)$ for $i,i'=1,2$.
\end{enumerate}
Hence the vector $e_{\mathbf{k}(\mathbf{w})}$ are the following:
\begin{align*}
e_{(a)}&=\sum_{i,i'=1,2}(2A_i+2Q_{i'})t_i^2s_{i'}^2+\sum_{i=1,2}(A_1+A_2+2Q_i)t_1t_2s_i^2+\sum_{i=1,2}(2A_i+Q_1+Q_2)t_i^2s_1s_2+\\
&\qquad+(A_1+A_2+Q_1+Q_2)t_1t_2s_1s_2,\\  
e_{(b)}&=\sum_{i,i'=1,2}(2A_i^2+2Q_{i'})t_i^2s_{i'}^2+\sum_{i=1,2}(2A_i^2+Q_1+Q_2)t_i^2s_1s_2, \\
e_{(c)}&=\sum_{i,i'=1,2}(2A_i+2Q_{i'}^2)t_i^2s_{i'}^2+\sum_{i=1,2}(A_1+A_2+2Q_i^2)t_1t_2s_i^2, \\
e_{(d)}&=\sum_{i,i'=1,2}(2A_i^2+2Q_j^2)t_i^2s_{i'}^2. 
\end{align*} 
These vectors are linearly independent, indeed the monomial $A_1t_1t_2s_1s_2$ is only in $e_{(a)}$, the monomial $Q_1t_1^2s_1s_2$ is only in $e_{(b)}$ and the monomial $A_1t_1t_2s_1^2$ is only in $e_{(c)}$. 
\end{example}

\begin{example}\label{ex:2a}
Let $\ell_1=\ell_2=1$ and consider the ray $\mathfrak{d}=(1,2)\R_{\geq 0}$ with $l_{\mathfrak{d}}=2$. The vectors $\mathbf{w}$ and the possible vector $\mathbf{k}=(k_1,k_2)$ are written in the following tables:

\vspace{5mm}

\begin{tabular}{|c|c|c|c|c|c|}
\hline
$2m_1 / 4m_2$ & $(m_2,m_2,m_2,m_2)$ & $(m_2,3m_2)$ & $(m_2,m_2,2m_2)$ & $(2m_2,2m_2)$ & $(4m_2)$ \\
\hline 
$(m_1,m_1)$ & ${\mathbf{w}(a)}$ & ${\mathbf{w}(b)}$ & ${\mathbf{w}(c)}$ & ${\mathbf{w}(d)}$ & ${\mathbf{w}(e)}$\\
\hline 
$(2m_1)$ & ${\mathbf{w}(f)}$ & ${\mathbf{w}(g)}$ &${\mathbf{w}(h)}$ & ${\mathbf{w}(i)}$ & ${\mathbf{w}(l)}$\\
\hline
\end{tabular}

\vspace{5mm}

\begin{tabular}{|c|c|c|c|c|c|}
\hline
 $k_1 / k_2$ & $(4)$ & $(1,0,1)$ & $(2,1)$ & $(0,2)$ & $(0,0,0,1)$ \\
\hline 
$(2)$ & $\mathbf{k}(\mathbf{w}(a))$ & $\mathbf{k}({\mathbf{w}(b)})$ & $\mathbf{k}({\mathbf{w}(c)})$ & $\mathbf{k}(\mathbf{w}(d))$ & $\mathbf{k}({\mathbf{w}(e)})$\\
\hline 
$(0,1)$ & $\mathbf{k}({\mathbf{w}(f)})$ & $\mathbf{k}({\mathbf{w}(g)})$ &$\mathbf{k}({\mathbf{w}(h)})$ & $\mathbf{k}({\mathbf{w}(i)})$ & $\mathbf{k}({\mathbf{w}(l)})$\\
\hline
\end{tabular}

\vspace{5mm}
In this example we have $10$ vectors $e_{\bullet}$
\begin{align*}
e_{(a)}&= (2A_1+4Q_1)t_1^2s_1^4; \qquad e_{(b)}= (2A_1+Q_1+3Q_1^3)t_1^2s_1^4; \qquad e_{(c)}= (2A_1+ 2Q_1+2Q_1^2)t_1^2s_1^4; \\
e_{(d)}&= (2A_1+ 4Q_1^2)t_1^2s_1^4; \qquad e_{(e)}= (2A_1+ 4Q_1^4)t_1^2s_1^4; \qquad e_{(f)}= (2A_1^2+ 4Q_1)t_1^2s_1^4; \\
e_{(g)}&= (2A_1^2+ Q_1+3Q_1^3)t_1^2s_1^4; \qquad e_{(h)}= (2A_1^2+2Q_1+2Q_1^2)t_1^2s_1^4; \qquad e_{(i)}= (2A_1^2+  4Q_1^2)t_1^2s_1^4; \\ e_{(l)}&=(2A_1^2+ 4Q_1^4)t_1^2s_1^4
\end{align*}
but they are linear combination of $6$ monomials (namely $A_1t_1^2s_1^4$, $Q_1t_1^2s_1^4$, $A_1^2s_1^4t_1^2$, $Q_1^2t_1^2s_1^4$, 
$Q_1^3t_1^2s_1^4$, $Q_1^4t_1^2s_1^4 $), hence they are linearly dependent (e.g. $e_{(l)}=e_{(e)}-e_{(i)}-e_{(d)}$). 
\end{example}

\begin{example}\label{ex:2_b}
Let $\ell_1=\ell_2=2$ and consider the same vector $l_{\mathfrak{d}}m_{\mathfrak{d}}=(2,4)$ as in Example \ref{ex:2a}. The vectors $\mathbf{w}_{\bullet}$ are as before, while there are more possibilities for the vectors $\mathbf{k}=(k_1^{(1)},k_1^{(2)},k_2^{(1)}, k_2^{(2)})$ hence $e_{\bullet}$ are listed below: 
\begin{align*}
e_{(a)}&=\sum_{i,i'=1,2}(2A_i+4Q_{i'})t_i^2s_{i'}^4+\sum_{i=1,2}(2A_i+Q_{1}+3Q_2)t_i^2s_1s_2^3+\sum_{i=1,2}(2A_i+Q_{2}+3Q_1)t_i^2s_1^3s_2+\\
&+ \sum_{i=1,2}(2A_i+2Q_{1}+2Q_2)t_i^2s_1^2s_2^2+\sum_{i'=1,2}(A_1+A_2+4Q_{i'})t_1t_2s_{i'}^4+(A_1+A_2+Q_{1}+3Q_2)t_1t_2s_1s_2^3+\\
&+(A_1+A_2+Q_{2}+3Q_1)t_1t_2s_1^3s_2+ (A_1+A_2+2Q_{1}+2Q_2)t_1t_2s_1^2s_2^2; \\
e_{(b)}&= \sum_{i,i'=1,2}(2A_i+Q_{i'}+3Q_{i'}^3)t_i^2s_{i'}^4+ \sum_{i=1,2}(2A_i+Q_{1}+3Q_{2}^3)t_i^2s_1s_2^3+ \sum_{i=1,2}(2A_i+Q_{2}+3Q_{1}^3)t_i^2s_1^3s_2 \\
&+\sum_{i'=1,2}(A_1+A_2+Q_{i'}+3Q_{i'}^3)t_1t_2s_{i'}^4+ (A_1+A_2+Q_{1}+3Q_{2}^3)t_1t_2s_1s_2^3+(A_1+A_2+Q_{2}+3Q_{1}^3)t_1t_2s_1^3s_2;\\ 
e_{(c)}&=\sum_{i,i'=1,2}(2A_i+ 2Q_{i'}+2Q_{i'}^2)t_{i}^2s_{i'}^4+\sum_{i=1,2}(2A_i+ 2Q_{1}+2Q_{2}^2)t_{i}^2s_1^2s_2^2+\sum_{i=1,2}(2A_i+ 2Q_{2}+2Q_{1}^2)t_{i}^2s_1^2s_2^2+\\
&+\sum_{i'=1,2}(A_1+A_2+ 2Q_{i'}+2Q_{i'}^2)t_1t_2s_{i'}^4+(A_1+A_2+ 2Q_{1}+2Q_{2}^2)t_1t_2s_1^2s_2^2+\\
&+(A_1+A_2+ 2Q_{2}+2Q_{1}^2)t_1t_2s_1^2s_2^2+ (A_1+A_2+ Q_{1}+Q_2+2Q_{2}^2)t_1t_2s_1s_2^3+\\ 
&+(A_1+A_2+ Q_{2}+Q_1+2Q_{1}^2)t_1t_2s_1^3s_2+\sum_{i=1,2}(2A_i+ Q_{1}+Q_2+2Q_{2}^2)t_{i}^2s_1s_2^3+\\
&+\sum_{i=1,2}(2A_i+ Q_{2}+Q_1+2Q_{1}^2)t_{i}^2s_1^3s_2;\\
e_{(d)}&=\sum_{i,i'=1,2}(2A_i+ 4Q_{i'}^2)t_{i}^2s_{i'}^4+\sum_{i=1,2}(2A_i+ 2Q_{1}^2+2Q_2^2)t_{i}^2s_1^2s_2^2+\sum_{i'=1,2}(A_1+A_2+ 4Q_{i'}^2)t_1t_2s_{i'}^4+ \\
&+(A_1+A_2+ 2Q_{1}^2+2Q_2^2)t_1t_2s_1^2s_2^2;\\
e_{(e)}&=\sum_{i,i'=1,2}(2A_i+ 4Q_{i'}^4)t_i^2s_{i'}^4+\sum_{i'=1,2}(A_1+A_2+ 4Q_{i'}^4)t_1t_2s_{i'}^4;\\ 
e_{(f)}&=\sum_{i,i'=1,2} (2A_i^2+ 4Q_{i'})t_i^2s_{i'}^4+\sum_{i=1,2} (2A_i^2+ Q_{1}+3Q_2)t_i^2s_1s_2^3+\sum_{i=1,2} (2A_i^2+ 3Q_1+Q_2)t_i^2s_1^3s_2+\\
&+\sum_{i=1,2} (2A_i^2+ 2Q_{1}+2Q_2)t_i^2s_1^2s_2^2;\\
e_{(g)}&=\sum_{i,i'=1,2} (2A_{i}^2+ Q_{i'}+3Q_{i'}^3)t_i^2s_{i'}^4+\sum_{i=1,2} (2A_{i}^2+ Q_{1}+3Q_{2}^3)t_i^2s_1s_2^3+\sum_{i=1,2} (2A_{i}^2+ Q_{2}+3Q_{1}^3)t_i^2s_1^3s_2;\\
e_{(h)}&=\sum_{i,i'=1,2} (2A_i^2+2Q_{i'}+2Q_{i'}^2)t_i^2s_{i'}^4+\sum_{i=1,2} (2A_i^2+2Q_{1}+2Q_{2}^2)t_i^2s_1^2s_2^2+\sum_{i=1,2} (2A_i^2+2Q_{2}+2Q_{1}^2)t_i^2s_1^2s_2^2+\\
&+\sum_{i=1,2} (2A_i^2+Q_{1}+Q_2+2Q_{2}^2)t_i^2s_1s_2^3+\sum_{i=1,2} (2A_i^2+Q_{2}+Q_1+2Q_{1}^2)t_i^2s_1^3s_2;\\
e_{(i)}&=\sum_{i,i'=1,2} (2A_i^2+  4Q_{i'}^2)t_i^2s_{i'}^4+\sum_{i=1,2} (2A_i^2+  2Q_{1}^2+2Q_2^2)t_i^2s_1^2s_2^2;\\
e_{(l)}&= \sum_{i,i'=1,2} (2A_i^2+ 4Q_{i'}^4)t_i^2s_{i'}^4.
\end{align*}
It is possible to check they are all linearly independent vectors. 
\end{example}

\begin{theorem}
\label{thm:partial}
Let $\ell_1,\ell_2$ be positive integer greater or equal than $1$ and let
\begin{align*}
\mathfrak{D}=\left\lbrace \left(\mathfrak{d}_i=m_1\R,\overrightarrow{f}_{\mathfrak{d}_i}=\left(1+A_it_iz^{m_1},1+t_iz^{m_1}\right)\right) , \left(\mathfrak{d}_i=m_2\R,\overrightarrow{f}_{\mathfrak{d}_j}=\left(1+Q_js_jz^{m_2},1+s_jz^{m_2}\right)\right)\right\rbrace_{\substack{1\leq i\leq \ell_1 \\ 1\leq j\leq \ell_2}}
\end{align*}
be the initial scattering diagram with $[A_i,Q_j]=0$ for every $i=1,...,\ell_1$, $j=1,...,\ell_2$. If $\mathfrak{d}=m_{\mathfrak{d}}\R_{\geq 0}$ is a ray in the consistent scattering diagram $\mathfrak{D}^{\infty}$ and $l_{\mathfrak{d}}m_{\mathfrak{d}}=\ell_1m_1+\ell_2m_2$, then $\overrightarrow{f}_{\mathfrak{d}}$ allows to compute the relative Gromov--Witten invariants $N_{0,\mathbf{w}}(\overline{Y}_{\mathbf{m}})$ with the following tangency conditions:
\begin{enumerate}
\item[$(a)$] $\mathbf{w}(a)=(\underbrace{m_1,...,m_1}_{\ell_1-\text{times}},\underbrace{m_2,...,m_2}_{\ell_2-\text{times}})$, namely rational curve through $\ell_1$ distinct points on the divisor $D_{m_1}$, $\ell_2$ distinct points on the divisor $D_{m_2}$ and tangent of order $\ell_1$ to an unspecified point on $D_{m_1}$ and of order $\ell_2$ to an unspecified point on $D_{m_2}$;
\item[$(b_j)$] $\mathbf{w}(b_j)=(jm_1, \underbrace{m_1,...,m_1}_{(\ell_1-j)-\text{times}},\underbrace{m_2,...,m_2}_{\ell_2-\text{times}})$ for $j=2,...,\ell_1$. These are rational curves through $\ell_1-j$ distinct points on the divisor $D_{m_1}$, $\ell_2$ distinct points on the divisor $D_{m_2}$, tangent of order $j$ to a given point on $D_{m_1}$ (distinct from the previous $\ell_1-j$ points), tangent of order $\ell_1$ to an unspecified point on $D_{m_1}$ and of order $\ell_2$ to an unspecified point on $D_{m_2}$;
\item[$(c_j)$] $\mathbf{w}(c_j)=(\underbrace{m_1,...,m_1}_{\ell_1-\text{times}},jm_2,\underbrace{m_2,...,m_2}_{(\ell_2-j)-\text{times}})$ for $j=2,...,\ell_2$. These are rational curves through $\ell_1$ distinct points on the divisor $D_{m_1}$, $\ell_2-j$ distinct points on the divisor $D_{m_2}$, tangent of order $j$ to a given point on $D_{m_2}$ (distinct from the previous $\ell_2-j$ points), tangent of order $\ell_1$ to an unspecified point on $D_{m_1}$ and of order $\ell_2$ to an unspecified point on $D_{m_2}$.
\end{enumerate}   
\end{theorem}

\begin{proof}
Being an element of $\tilde{\mathfrak{h}}$, $\log\overrightarrow{f}_{\mathfrak{d}}$ is a polynomial in the variables $A_1,...,A_{\ell_1}$, $Q_{1},...,Q_{\ell_2}$, $ t_1,...,t_{\ell_1}$, $s_1,...,s_{\ell_2}$ and, according to formula \eqref{eq:GW_rel GPS2}, the invariants $N_{0,\mathbf{w}}(\overline{Y}_{\mathbf{m}})$ are the coefficients of the polynomial $V_{\mathbf{w}}$. In particular there are some monomials which only appear in $V_{\mathbf{w}}$ for a given $\mathbf{w}$, hence from the expression of $\log \overrightarrow{f}_{\mathfrak{d}}$ we can compute the invariants $N_{0,\mathbf{w}}(\overline{Y}_{\mathbf{m}})$ looking at the coefficient of these monomials. From the definition of $V_\mathbf{w}$, the vectors $\mathbf{k}$ (such that $\mathbf{k}(\mathbf{w})=\mathbf{w}$) govern the possible polynomials which appears in $V_{\mathbf{w}}$. We collect the possible partitions of $\ell_1$ and $\ell_2$ in the following table (see Figure \ref{table}). Then $t_1\cdot...\cdot t_{\ell_1}$ can only appear when $\mathbf{w}=(\underbrace{m_1,...,m_1}_{\ell_1-\text{times}}, *)$ because they are all distinct, and for the same reason $s_1\cdot ...\cdot s_{\ell_2}$ can only appear when $\mathbf{w}=(*,\underbrace{m_2,...,m_2}_{\ell_1-\text{times}})$. In addition this forces $A_1,...,A_{\ell_1}$ and $Q_1,...,Q_{\ell_2}$ not being of higher powers. Hence $A_1t_1\cdot...\cdot t_{\ell_1}s_1\cdot...\cdot s_{\ell_2}$ only contributes to $N_{0,\mathbf{w}(a)}$.

\begin{figure}
\begin{center}
\begin{tabular}{|c|c|c|c|c|c|c|}
\hline
$\ell_1m_1 / \ell_2m_2$ & $(\underbrace{m_2,...,m_2}_{\ell_2-\text{times}})$ & $(2m_2,\underbrace{m_2,...,m_2}_{(\ell_2-2)-\text{times}})$ & $\cdots$ & $(\ell_2m_2)$  & $(2m_2,2m_2,\underbrace{m_2,...,m_2}_{(\ell_2-4)-\text{times}} )$ &$\cdots$ \\
\hline 
$(\underbrace{m_1,...,m_1}_{\ell_1-\text{times}})$  & ${\mathbf{w}(a)}$ & ${\mathbf{w}(c_2)}$ & $\cdots$ & ${\mathbf{w}(c_{\ell_2})}$ & $\mathbf{w}(\spadesuit)$ & \\
\hline 
$(2m_1,\underbrace{m_1,...,m_1}_{(\ell_1-2)-\text{times}})$ & ${\mathbf{w}(b_2)}$ & $\mathbf{w}(*)$  &$\cdots$ & $\mathbf{w}(*)$  & &\\
\hline 
$\vdots$ & $\vdots$ & $\vdots$  &  & $\vdots$  & &\\
\hline 
$(\ell_1m_1)$  &${\mathbf{w}(b_{\ell_1})}$  & $\mathbf{w}(*)$  &$\cdots$ & $\mathbf{w}(*)$      & &\\
&  &  & &  &  & \\
\hline 
$(2m_1,2m_1,\underbrace{m_1,...,m_1}_{(\ell_1-4)-\text{times}} )$ & $\mathbf{w}(\clubsuit)$ &  & &  & &\\
\hline 
$\vdots$ &  &  & &  & &\\
\hline
\end{tabular}
\end{center}
\caption{Table of the possible invariants appearing in the commutator formula.}\label{table}
\end{figure}
\vspace{5mm}

As soon as we consider $t_1^jt_2\cdot...\cdot t_{\ell_1-j}$ for $j\in\lbrace 2,...,\ell_1\rbrace$ it is not possible to uniquely determine the corresponding partition of $\ell_1$: these monomials appear in $(rm_1,\underbrace{m_1,...,m_1}_{\ell_1-r})$ for $r\in\lbrace 1,...,j-1\rbrace$ by choosing non distinct values (e.g. if $r=1$ and $k_1^{(1)}=(j), k_1^{(i)}=(1)$ for $i=2,...,\ell_1-j$). However $A_1^jt_1^jt_2\cdot...\cdot t_{\ell_1-j}$ can only occur when $\mathbf{w}=(jm_1,\underbrace{m_1,...,m_1}_{(\ell_1-j)-\text{times}},*)$ because $jm_1$ forces $k_1=(\ell_1-j,...,1)$ with $1$ in the $j$-th position and consequently $A_1^jt_1^j$ comes from $k_1^{(1)}=(0,...,1)$ with $1$ in the $j$-th position. If $\mathbf{w}=(rm_1,\underbrace{m_1,...,m_1}_{\ell_1-r},*)$ for $r\in\lbrace 1,...,j-1\rbrace$ then the higher power of $A_1$ would have been $A_1^r$. Therefore $A_1^jt_1^jt_2\cdot...\cdot t_{\ell_1-j}s_1\cdot...\cdot s_{\ell_2}$ only contributes to $N_{0,\mathbf{w}(b_j)}$. Reversing the role of $k_1$ and $k_2$ the same arguments apply to $N_{0,\mathbf{w}(c_j)}$ which is the coefficient of $Q_1^jt_1\cdot...\cdot t_{\ell_1}s_1^js_2\cdot...\cdot s_{\ell_2-j}$. 

What goes wrong with the other $N_{0,\mathbf{w}}$ is that we are not able to isolate a unique monomial which corresponds to $\mathbf{w}$: for the $\mathbf{w}(*)$ in the table, we would need both $A_1^jt_1^jt_2\cdot...\cdot t_{\ell_1-j}$ and $Q_1^rs_1^rs_2\cdot...\cdot s_{\ell_2-r}$ but $\left(A_1^jt_1^jt_2\cdot...\cdot t_{\ell_1-j}\cdot Q_1^rs_1^rs_2\cdot...\cdot s_{\ell_2-r},\bullet\right)$ is not a monomial in $\log \overrightarrow{f}_{\mathfrak{d}}\subset\tilde{\mathfrak{h}}$. For the $\mathbf{w}(\clubsuit)$ we may choose $t_1^2t_2^2t_3,...,t_{\ell_1-4}$ but as for $\mathbf{w}(b_2)$ we need to consider also $A_1^2$ and $A_2^2$; however $\left(A_1^2A_2^2t_1^2t_2^2t_3\cdot...\cdot t_{\ell_1-4}\cdot s_1\cdot...\cdot s_{\ell_2},\bullet\right)$ is not a monomial in $\log \overrightarrow{f}_{\mathfrak{d}}\subset\tilde{\mathfrak{h}}$. The same argument applies to $N_{\mathbf{w}(\spadesuit)}$.




\end{proof}

\newpage

\backmatter

\bibliographystyle{amsalpha}
\bibliography{biblio_tesi}
\printindex

\end{document}